\documentclass[11pt]{amsart}

\usepackage{amsmath, amssymb, amsthm}

\input xy
\xyoption{all}

\usepackage{hyperref}

\swapnumbers
\numberwithin{equation}{section}

\theoremstyle{plain}
\newtheorem{theorem}[subsubsection]{Theorem}
\newtheorem{lemma}[subsubsection]{Lemma}
\newtheorem{prop}[subsubsection]{Proposition}
\newtheorem{cor}[subsubsection]{Corollary}
\newtheorem{conj}[subsubsection]{Conjecture}

\theoremstyle{definition}
\newtheorem{defn}[subsubsection]{Definition}

\newtheorem{remark}[subsubsection]{Remark}
\newtheorem{exam}[subsubsection]{Example}

\def\AA{\mathbb{A}}
\def\BB{\mathbb{B}}
\def\CC{\mathbb{C}}

\def\FF{\mathbb{F}}
\def\GG{\mathbb{G}}

\def\PP{\mathbb{P}}
\def\QQ{\mathbb{Q}}

\def\SS{\mathbb{S}}
\def\TT{\mathbb{T}}
\def\UU{\mathbb{U}}

\def\WW{\mathbb{W}}

\def\ZZ{\mathbb{Z}}

\def\calA{\mathcal{A}}
\def\calB{\mathcal{B}}
\def\calC{\mathcal{C}}
\def\calD{\mathcal{D}}
\def\calE{\mathcal{E}}
\def\calF{\mathcal{F}}
\def\calG{\mathcal{G}}
\def\calH{\mathcal{H}}

\def\calK{\mathcal{K}}
\def\calL{\mathcal{L}}

\def\calO{\mathcal{O}}

\def\calT{\mathcal{T}}
\def\calU{\mathcal{U}}
\def\calV{\mathcal{V}}

\def\calZ{\mathcal{Z}}

\def\bA{\mathbf{A}}

\def\bG{\mathbf{G}}

\def\bI{\mathbf{I}}

\def\bK{\mathbf{K}}
\def\bL{\mathbf{L}}
\def\bM{\mathbf{M}}

\def\bP{\mathbf{P}}

\def\bR{\mathbf{R}}

\def\bV{\mathbf{V}}

\newcommand\frI{\mathfrak{I}}
\newcommand\frT{\mathfrak{T}}

\newcommand\frg{\mathfrak{g}}
\newcommand\frt{\mathfrak{t}}

\newcommand\frh{\mathfrak{h}}
\newcommand\frk{\mathfrak{k}}
\renewcommand\frm{\mathfrak{m}}

\newcommand\tilW{\widetilde{W}}
\newcommand\tilA{\widetilde{A}}

\newcommand\tilL{\widetilde{L}}
\newcommand\tilX{\widetilde{X}}

\newcommand\tilv{\widetilde{v}}

\def\dG{\widehat{G}}
\def\dT{\widehat{T}}

\def\htheta{\widehat{\theta}}

\newcommand\ab{\textup{ab}}

\newcommand\an{\textup{an}}

\newcommand\AS{\textup{AS}}

\newcommand{\Bun}{\textup{Bun}}
\newcommand{\chk}{\textup{char}(k)}

\newcommand{\coker}{\textup{coker}}
\newcommand\cont{\textup{cont}}
\newcommand{\Cox}{\textup{Cox}}
\newcommand\cusp{\textup{cusp}}

\newcommand\ev{\textup{ev}}

\newcommand\fin{\textup{fin}}

\newcommand\Frob{\textup{Frob}}
\newcommand{\Fun}{\textup{Fun}}
\newcommand\Gal{\textup{Gal}}

\newcommand\geom{\textup{geom}}

\newcommand{\Gr}{\textup{Gr}}
\newcommand{\GR}{\textup{GR}}

\newcommand{\Herm}{\textup{Herm}}

\newcommand\Hk{\textup{Hk}}
\newcommand\IC{\textup{IC}}
\newcommand\id{\textup{id}}
\renewcommand{\Im}{\textup{Im}}
\newcommand{\Ind}{\textup{Ind}}

\newcommand\Kl{\textup{Kl}}
\newcommand\Lie{\textup{Lie}\ }
\newcommand\Loc{\textup{Loc}}

\newcommand{\Nm}{\textup{Nm}}

\newcommand\Out{\textup{Out}}

\newcommand\Perv{\textup{Perv}}
\newcommand{\Pic}{\textup{Pic}}

\newcommand\pt{\textup{pt}}

\newcommand\rank{\textup{rank}}
\newcommand{\red}{\textup{red}}

\newcommand\Rep{\textup{Rep}}
\newcommand{\Res}{\textup{Res}}
\newcommand\res{\textup{res}}
\newcommand\Rig{\textup{Rig}}

\newcommand\Spec{\textup{Spec}\ }
\newcommand\St{\textup{St}}
\newcommand\st{\textup{st}}

\newcommand\Sw{\textup{Sw}}
\newcommand\Sym{\textup{Sym}}

\newcommand{\Tr}{\textup{Tr}}

\newcommand\triv{\textup{triv}}

\newcommand{\univ}{\textup{univ}}

\newcommand{\Vect}{\textup{Vec}}

\newcommand\Aut{\textup{Aut}}
\newcommand\Hom{\textup{Hom}}
\newcommand\End{\textup{End}}

\newcommand\uAut{\underline{\Aut}}
\newcommand\uHom{\underline{\Hom}}

\newcommand{\Isom}{\textup{Isom}}

\newcommand\GL{\textup{GL}}

\newcommand\PGL{\textup{PGL}}

\newcommand\SL{\textup{SL}}
\renewcommand\sl{\mathfrak{sl}}

\newcommand\GU{\textup{GU}}
\newcommand\SO{\textup{SO}}

\newcommand\PSO{\textup{PSO}}

\newcommand\Sp{\textup{Sp}}

\newcommand{\Gm}{\GG_m}
\def\Ga{\GG_a}

\newcommand{\ad}{\textup{ad}}
\newcommand{\Ad}{\textup{Ad}}
\renewcommand\sc{\textup{sc}}
\newcommand{\der}{\textup{der}}

\newcommand\xch{\mathbb{X}^*}
\newcommand\xcoch{\mathbb{X}_*}

\newcommand{\incl}{\hookrightarrow}
\newcommand{\isom}{\stackrel{\sim}{\to}}
\newcommand{\surj}{\twoheadrightarrow}
\newcommand{\leftexp}[2]{{\vphantom{#2}}^{#1}{#2}}

\newcommand{\twtimes}[1]{\stackrel{#1}{\times}}
\newcommand{\jiao}[1]{\langle{#1}\rangle}
\newcommand{\wt}[1]{\widetilde{#1}}
\newcommand{\un}[1]{\underline{#1}}
\newcommand\nth{^{\textup{th}}}
\newcommand\upH{\textup{H}}
\newcommand\ur{\textup{ur}}
\newcommand{\oll}[1]{\overleftarrow{#1}}
\newcommand{\orr}[1]{\overrightarrow{#1}}

\renewcommand{\l}{\lambda}
\newcommand{\ep}{\epsilon}
\newcommand{\om}{\omega}
\newcommand{\Om}{\Omega}
\renewcommand{\L}{\Lambda}
\renewcommand{\a}{\alpha}
\renewcommand{\b}{\beta}

\newcommand{\cohog}[2]{\textup{H}^{#1}({#2})}     
\newcommand{\cohoc}[2]{\textup{H}_{c}^{#1}({#2})}     

\newcommand{\diag}{\textup{diag}}
\newcommand{\Wt}{\textup{Wt}}
\newcommand{\Tt}{\textup{Tt}}
\newcommand{\Sat}{\textup{Sat}}

\newcommand\wta{\widetilde{\alpha}}

\newcommand\GQ{\Gal(\overline{\QQ}/\QQ)}

\newcommand\Gk{\Gal(\overline{k}/k)}

\newcommand{\Ql}{\QQ_{\ell}}
\newcommand{\Zl}{\ZZ_{\ell}}

\newcommand\Qbar{\overline{\QQ}}
\newcommand{\Qlbar}{\overline{\QQ}_\ell}

\newcommand{\kbar}{\overline{k}}
\newcommand{\Fbar}{\overline{F}}

\newcommand{\xbar}{\overline{x}}

\newcommand\cG{\mathcal{G}}
\newcommand\cK{\mathcal{K}}
\newcommand\cZ{\mathcal{Z}}
\newcommand\cT{\mathcal{T}}
\newcommand\cL{\mathcal{L}}
\newcommand\cD{\mathcal{D}}
\newcommand\cE{\mathcal{E}}
\newcommand\cF{\mathcal{F}}
\newcommand{\cLoc}{\mathcal{L}oc}
\renewcommand{\top}{\textup{top}}

\newcommand\hZZ{\widehat{\ZZ}}

\newcommand\CS{\textup{CS}_{1}}
\newcommand\cCS{\mathcal{CS}_{1}}

\newcommand\AG{G(F)\backslash G(\AA_{F})}
\newcommand\AZ{Z(F)\backslash Z(\AA_{F})}

\renewcommand{\c}{\circ}
\newcommand{\onat}{\omega^{\natural}}
\newcommand{\n}{\natural}

\newcommand\AZn{Z(F)\backslash Z(\AA_{F})^{\n}}

\newcommand\Cov{\textup{Cov}}
\newcommand\Serre{\textup{Serre}}
\newcommand\coev{\textup{coev}}

\newcommand\frG{\mathfrak{G}}
\newcommand\tfrI{\widetilde{\mathfrak{I}}}

\newcommand{\oHk}{\leftexp{\c}\Hk}
\newcommand{\oHko}{\leftexp{\c}{\Hk}^{\c}}
\newcommand{\Hko}{\Hk^{\c}}
\newcommand{\oo}[1]{\leftexp{\c}{#1}^{\c}}
\newcommand{\oGR}{\leftexp{\c}\GR}

\newcommand{\tbM}{\wt{\bM}}
\newcommand{\Sram}{S-\textup{ram}}
\newcommand{\bAd}{\textup{Ad}^{\der}}

\title{Rigidity in automorphic representations and local systems}
\author{Zhiwei Yun}
\thanks{Supported by the Packard Foundation and the NSF grant DMS-1302071.}
\address{Department of Mathematics, Stanford University, 450 Serra Mall, Building 380, Stanford, CA 94305}
\email{zwyun@stanford.edu}
\date{}

\begin{document}

\begin{abstract}
We introduce the notion of rigidity for automorphic representations of groups over global function fields. We construct the Langlands parameters of rigid automorphic representations explicitly as local systems over open curves. We expect these local systems to be rigid. Examples of rigid automorphic representations from previous work are reviewed and more examples for $\GL_{2}$ are discussed in details.
\end{abstract}

\maketitle

\tableofcontents

\section{Introduction}
In this article, we shall study rigid local systems over algebraic curves from the point of view of the Langlands correspondence.

\subsection{The goal} 
Fix a smooth,  projective and connected algebraic curve $X$ over an algebraically closed field $k$. Let $S\subset X$ be a finite set of closed points. A local system $\calF$ over $X-S$ is {\em physically rigid} if it is determined up to isomorphism by its local monodromy around points $x\in S$. A local system $\calF$ over $X-S$ is {\em cohomologically rigid} if $\cohog{1}{X,j_{*}\End^{\circ}(\calF)}=0$, where $\End^{\circ}(\calF)$ is the local system of trace-free endomorphisms of $\calF$, and $j_{*}$ means the sheaf (not derived) push-forward along $j:X-S\incl X$. These notions were defined and studied in depth by N. Katz \cite{Katz}. The main result of \cite{Katz} is an algorithmic description of tame local systems.  

We are interested in local systems in a broader sense. Let $H$ be a connected reductive algebraic group over $\Qlbar$. An $H$-local system on $U=X-S$ is a continuous homomorphism of the \'etale fundamental group $\pi_{1}(U,u)$ (for some base point $u$) into $H(\Qlbar)$. This specializes to the notion of rank $n$ local systems when $H=\GL_{n}$. Both notions of rigidity can be easily extended to $H$-local systems. Details are discussed in \S\ref{s:ls}.

We would like to construct many examples of $H$-local systems that are rigid, or, at least, expected to be rigid. The tool we use for our construction is the Langlands correspondence over function fields. 

Now let $X$ be a smooth, projective and geometrically connected curve over a finite field $k$. Let $G$ be a connected reductive group over the function field $F$ of $X$, which for simplicity is assumed to be split. In this case, Langlands philosophy predicts that there should be a finite-to-one correspondence $\pi\mapsto \rho_{\pi}$ from automorphic representations $\pi$ of $G(\AA_{F})$ to continuous representations $\rho: W_{F}\to\dG(\Qlbar)$, where $W_{F}$ is the Weil group of the function field $F$ and $\dG$ is the Langlands dual group to $G$. When the automorphic representation $\pi$ is unramified outside a finite set of places $S$, $\rho_{\pi}$ should also be unramified outside $S$, and the datum of $\rho_{\pi}$ is the same as a $\dG$-local system $\cF_{\pi}$ over the open curve $U=X-S$. The correspondence should satisfy the following property: for each closed point $x\notin S$ of $X$, the Satake parameter of the spherical representation $\pi_{x}$ of $G(F_{x})$ should coincide with the conjugacy class of $\rho_{\pi}(\Frob_{x})$.

The main strategy of our construction may be summarized as follows.
\begin{itemize}
\item There should be a notion of rigidity for automorphic representations of $G(\AA_{F})$. For a rigid automorphic representation $\pi$ of $G(\AA_{F})$, the corresponding $\dG$-local system $\cF_{\pi}$ should also be rigid.
\item Rigid automorphic representations should be {\em easier} to construct than rigid local systems. Once a rigid automorphic representations is known, there should be a way to construct the corresponding local system via the {\em geometric} Langlands correspondence. 
\end{itemize}

\subsection{Applications} 
Before describing how we implement these ideas, let us list a few applications of our construction. In fact, it is these applications that convinced the author that a systematic study of rigid local system from the point of view of Langlands correspondence was meaningful.

\subsubsection{Local systems with exceptional monodromy groups}
Deligne showed in \cite{Deligne-ExpSum} that the classical Kloosterman sums (or hyper Kloosterman sums) are obtained as the Frobenius trace function of a local system over $\PP^{1}_{\FF_{p}}-\{0,\infty\}$, the Kloosterman sheaf $\Kl_{n}$.  Katz showed that for $p>2$, the Zariski closure of the monodromy of $\Kl_{n}$ is either $\SL_{n}$ (when $n$ is odd) or $\Sp_{n}$ (when $n$ is even). In joint work with Heinloth and Ng\^o \cite{HNY}, we construct a $\dG$-local system $\Kl_{\dG}$ on $\PP^{1}_{\FF_{p}}-\{0,\infty\}$ for every almost simple group $\dG$. The local monodromy of $\Kl_{\dG}$ resembles that of the classical Kloosterman sheaf $\Kl_{n}$, and the Zariski closure of their global monodromy is a large subgroup of $\dG$. For example, when $\dG$ is of type  $E_{7}, E_{8}, F_{4}$ or $G_{2}$, the monodromy is Zariski dense in $\dG$. These give the first examples of motivic local systems with Zariski dense monodromy in exceptional groups other than $G_{2}$ (the $G_{2}$ case was constructed earlier by Katz \cite{Katz-DE}). In \cite{Y-GenKloo} we give further generalizations of Kloosterman sheaves.

Our construction in \cite{HNY} was inspired by the construction of simple supercuspidal representations by Gross and Reeder \cite{GR}, an observation of Gross \cite{Gross-Prescribe} on the global realization of such representations and the work of Frenkel and Gross \cite{FG} on rigid irregular connections. Gross showed in \cite{Gross-Prescribe} that when $G$ is simply-connected, the automorphic representations for $G$ over the rational function field $F=k(t)$ which is Steinberg at $0$ and simple supercuspidal at $\infty$ (and unramified elsewhere) should be unique. He then conjectures that when $G=\GL_{n}$, the Satake parameters of this automorphic representation should give the classical Kloosterman sums. Our work \cite{HNY} confirms this conjecture and generalizes it to other reductive groups.

\subsubsection{Motives over number fields with exceptional motivic Galois groups}\label{ss:Serre Q}
In early 1990s, Serre asked the following question \cite{Serre-motive}: Is there a motive over a number field whose motivic Galois group is of exceptional type such as $G_{2}$ or $E_{8}$?

A motive $M$ over a number field $K$ is, roughly speaking, part of the cohomology $\upH^{i}(X)$ for some (smooth projective) algebraic variety $X$ over $K$ and some integer $i$, which is cut out by geometric operations (such as group actions). The definition of the motivic Galois group of $M$ relies on the validity of standard conjectures in algebraic geometry. However, one can use the following alternative definition which is believed to give the same group. For each prime $\ell$, the motive $M$ has the associated $\ell$-adic cohomology $\upH_{\ell}(M)\subset \upH^{i}(X_{\overline{K}},\Ql)$, which admits a Galois action:
\begin{equation*}
\rho_{M,\ell}:\Gal(\overline{K}/K)\to \GL(\upH_{\ell}(M))
\end{equation*}
The {\em $\ell$-adic motivic Galois group } $G_{M,\ell}$ of $M$ is the Zariski closure of the image of $\rho_{M,\ell}$. This is an algebraic group over $\Ql$. Classical groups appear as $\ell$-adic motivic Galois groups of abelian varieties (see \cite{Milne}).  However, it is proved in \cite{Milne} that abelian varieties do not have exceptional motivic Galois groups; nor is there a Shimura variety of type $G_{2}$ or $E_{8}$. This is why Serre raised the question for exceptional groups, and remarked that it was ``plus hasardeuse''. Until recently, the only known case of Serre's question was $G_{2}$, by the work of Dettweiler and Reiter \cite{DR}.

In \cite{Y-motive}, we give a {\em uniform} construction of local systems on $\PP^{1}_{\QQ}-\{0,1,\infty\}$ with Zariski dense monodromy in exceptional groups $E_{7},E_{8}$ and $G_{2}$, which come from cohomology of families of varieties over $\PP^{1}-\{0,1,\infty\}$. As a consequence of this construction, we give an affirmative answer to the $\ell$-adic version of Serre's question for $E_{7}, E_{8}$ and $G_{2}$: these groups can be realized as the $\ell$-adic motivic Galois groups for motives over number fields (in fact the number field is either $\QQ$ or $\QQ(i)$). With a bit more work, one can also realize $F_{4}$ as a motivic Galois group over $\QQ$.

\subsubsection{Inverse Galois Problem}
The inverse Galois problem over $\QQ$ asks whether every finite group can be realized as the Galois group of some Galois extension $K/\QQ$. The problem is still open for many finite simple groups, especially those of Lie type.   The same rigid local systems over $\PP^{1}_{\QQ}-\{0,1,\infty\}$ constructed to answer Serre's question can be used to solve new cases of the inverse Galois problem. We show in \cite{Y-motive} that for sufficiently large primes $\ell$, the finite simple groups $G_{2}(\FF_{\ell})$ and $E_{8}(\FF_{\ell})$ can be realized as Galois groups over $\QQ$. With a bit more work, one can also prove the same statement for $F_{4}(\FF_{\ell})$. 

In inverse Galois theory, people use the ``rigidity method'' to prove certain finite groups $H$ are Galois groups over $\QQ$. The rigidity method is an analog of rigid local systems with finite monodromy group. Although the idea of rigidity has long been used in the inverse Galois theory, the connection with Langlands correspondence and automorphic forms has not been explored before. Our result shows that this connection can be useful in solving the inverse Galois problem, and it even sheds some light to the rigidity method itself. In fact, our construction of the local system over $\PP^{1}-\{0,1,\infty\}$ suggests a triple in $E_{8}(\FF_{\ell})$ which should be a rigid triple (see \cite[Conjecture 5.16]{Y-motive}). This has been confirmed by Guralnick and Malle \cite{GM}, where they used this triple to show that $E_{8}(\FF_{\ell})$ is a Galois group over $\QQ$ as long as $\ell\geq7$.

\subsection{Main results} 
\subsubsection{Rigid automorphic data} Let $X$ be a smooth, projective and geometrically connected curve over a finite field $k$. For simplicity we assume $G$ is split over $F$ and simply-connected. Fix a finite set $S$ of closed points of $X$. By an {\em automorphic datum} for $G$ with respect to $S$ we mean a triple $(\om,K_{S},\chi_{S})$ where $\om$ is a central character $\AZ\to\Qlbar^{\times}$, $K_{S}$ is a collection of compact open subgroups $K_{x}\subset G(F_{x})$ for each $x\in S$, and $\chi_{S}$ is a collection of characters $\chi_{x}:K_{x}\to\Qlbar^{\times}$. An automorphic representation $\pi$ of $G(\AA_{F})$ is called $(\om,K_{S},\chi_{S})$-typical if its central character is $\om$, it is unramified outside $S$ and for each $x\in S$, $\pi_{x}$ has an eigenvector under $K_{x}$ on which $K_{x}$ acts through the character $\chi_{x}$. 

We also introduce the notion of a {\em geometric automorphic datum} $(\Om,\bK_{S}, \cK_{S},\iota_{S})$ in \S\ref{ss:geom data}. The idea is to give more structure to an automorphic datum so that it makes sense to base change to an extension of the ground field $k$. In particular, we replace $K_{x}$ by a pro-algebraic subgroup $\bK_{x}$ of the loop group of $G$ at $x$, and replace the character $\chi_{x}$ by a rank one character sheaf $\cK_{x}$ on $\bK_{x}$.  The geometric automorphic datum $(\Om,\bK_{S},\cK_{S},\iota_{S})$ not only recovers the automorphic datum by the sheaf-to-function correspondence, but also gives an automorphic datum $(\om', K_{S'},\chi_{S'})$ for $G$ over $F\otimes_{k}k'$ for any finite extension $k'/k$. 

We introduce the notion of rigidity for a geometric automorphic datum $(\Om,\bK_{S},\cK_{S}, \iota_{S})$ in Definition \ref{def:geom rigid}. The datum $(\Om,\bK_{S},\cK_{S}, \iota_{S})$ is called {\em strongly rigid} if for every finite extension $k'/k$, there is a unique $(\om',K_{S'},\chi_{S'})$-typical automorphic representation $\pi'$ of $G(\AA_{F\otimes_{k}k'})$ for the automorphic datum $(\om',K_{S'},\chi_{S'})$ obtained from base change. We also introduce the notion of weak rigidity.  We expect that strong (resp. weak) rigidity of geometric automorphic data should correspond to the physical (resp. cohomological) rigidity of $\dG$-local systems under the Langlands correspondence. 

The first main result is Theorem \ref{th:wrigid}, which gives a sheaf-theoretic criterion for rigidity of geometric automorphic data, using ``relevant'' points on certain moduli stack $\Bun_{G}(\bK_{S})$ of $G$-torsors over $X$ with level structures.

\subsubsection{Construction of the local system} The Langlands correspondence $\pi\mapsto\cF_{\pi}$ for cuspidal automorphic representations $\pi$ has been established by recent work of V.Lafforgue \cite{VLaff}. However, the construction in \cite{VLaff} does not give an explicit description of the $\dG$-local system $\cF_{\pi}$. We would like to construct the $\dG$-local system $\cF_{\pi}$ explicitly in the rigid situation.

The main results in this direction are Theorem \ref{th:eigen} and Proposition \ref{p:desc eigen}. The former guarantees that under a certain rigidity assumption on the geometric automorphic datum $(\Om,\bK_{S},\cK_{S}, \iota_{S})$, the $\dG$-local system $\cF_{\pi}$ for an $(\om,K_{S}, \chi_{S})$-typical automorphic representation $\pi$ can be constructed. The latter gives a concrete description of $\cF_{\pi}$ as a direct summand of the direct image sheaf of a family of varieties over $X-S$. 

We also make a Conjecture \ref{conj:rigid} which relates the Artin conductor of $\cF_{\pi}$ at $x\in S$ and certain relative dimension of $\bK_{x}$ in a precise way. This conjecture has been verified for many examples in \S\ref{s:Kl} and \S\ref{s:3p}.

The above discussion is over-simplified. In the main body of the paper we consider quasi-split groups $G$ that are not assumed to be simply-connected. Complications arise in the general case because the moduli stack $\Bun_{G}(\bK_{S})$ has several connected components. The definition of rigidity for geometric automorphic data needs to be modified to take care of unramified twists of automorphic representations. In this generality, Theorem \ref{th:eigen} only guarantees the existence of $\cF_{\pi}$ as a $\dG$-local system in a weakened sense. With a bit extra structure we can construct an actual $\dG$-local system $\cF_{\pi}$ from $(\Om,\bK_{S},\cK_{S}, \iota_{S})$, see Theorem \ref{th:eigenvar}.

\subsubsection{The method} The method we use to construct $\dG$-local systems is the geometric Langlands correspondence, a program initiated by Drinfeld and Laumon. The geometric Langlands correspondence can be set up over a general field, and can be viewed as an upgraded version of the Langlands correspondence, in which functions are replaced by sheaves. The starting point is an observation of Weil which says that the double coset $\AG/\prod_{x\in |X|}G(\calO_{x})$ may be interpreted as the set of $G$-torsors over $X$. Let $\Bun_{G}$ be the moduli stack of $G$-torsors over $X$. By the sheaf-to-function correspondence of Grothendieck, $\Qlbar$-sheaves on $\Bun_{G}$ (for the \'etale topology) should be thought of as an upgraded version of automorphic forms, which are $\Qlbar$-valued functions on the automorphic space $\AG/\prod_{x\in|X|}G(\calO_{x})$ (in the everywhere unramified case). 

Given a geometric automorphic datum $(\Om,\bK_{S},\cK_{S}, \iota_{S})$, assuming each $\cK_{x}$ descends to a finite-dimesional quotient $\bL_{x}$ of $\bK_{x}$, then the automorphic sheaves for the situation are sheaves over $\Bun_{G}(\bK^{+}_{S})$ (the moduli stack of $G$-torsors over $X$ with $\bK^{+}_{x}=\ker(\bK_{x}\to\bL_{x})$ level structures at $x\in S$), together with certain equivariance structures dictated by $\Om$ and $\cK_{S}$. Such sheaves form a derived category $D_{G,\Om}(\bK_{S}, \cK_{S})$. Rigidity roughly means that the category $D_{G,\Om}(\bK_{S}, \cK_{S})$ contains a unique irreducible perverse sheaf $\calA$. 

The advantage of the geometric Langlands correspondence is that it allows us to apply geometric Hecke operators to $\calA$, and in this way we can construct the Hecke eigen $\dG$-local system $\cF$ explicitly. Details will be explained in \S\ref{ss:Hk}-\ref{ss:Hk eigen}.

\subsubsection{Examples}
There are three classes of examples of geometric automorphic data in the case $G=\GL_{2}$ that we work out in details at various places of the paper. The corresponding local systems in these examples are exactly the three types of hypergeometric sheaves of rank two constructed by Katz, of which the Kloosterman sheaf of rank two is a special case. We prove in \S\ref{ss:GL2} that these geometric automorphic data are strongly rigid. We show in \S\ref{ss:rigidlocGL2} that the hypergeometric sheaves of rank two essentially exhaust all rigid local systems of rank two on $\PP^{1}$ ramified at at least two points. In \S\ref{ss:desc eigen} we explain how these geometric automorphic data are related to hypergeometric sheaves via the geometric Hecke operators.

In \S\ref{s:Kl} and \S\ref{s:3p} we review the work \cite{HNY}, \cite{Y-GenKloo} and \cite{Y-motive}, which are the main examples that lead to the theory of rigid automorphic representations.

\subsection{Acknowledgment}
The author would like to thank the organizers of the Current Development of Mathematics conference held at Harvard University in November of 2013.

\section{Rigidity for automorphic representations}\label{s:auto}
\S\ref{ss:X}-\S\ref{ss:auto} contain background material on reductive groups and automorphic representations. \S\ref{ss:Weil}-\S\ref{ss:ff} provide background on the sheaf-theoretic interpretation of automorphic forms. More results on rank one character sheaves are contained in Appendix \ref{a:ch}. The key definitions appear in \S\ref{ss:geom data} and \S\ref{ss:rigid auto}, where we introduce geometric automorphic data and the notion of rigidity for them. The main result is Theorem \ref{th:wrigid} which gives a criterion for the rigidity of geometric automorphic data. Examples in $\GL_{2}$ are worked out in details in \S\ref{ss:GL2}.

\subsection{Function field}\label{ss:X} 
Let $k$ be a perfect field and fix an algebraic closure $\kbar$ of $k$. Let $X$ be a projective, smooth and geometrically connected curve $X$ over $k$. Let $F=k(X)$ be the field of rational functions on $X$. Let $\Gamma_{F}=\Gal(F^{s}/F)$ be the absolute Galois group of $F$, where $F^{s}$ is a separable closure of $F$. The field $F\otimes_{k}\kbar$ has absolute Galois group  $I_{F}:=\Gal(F^{s}_{\kbar}/F\otimes_{k}\kbar)\lhd\Gamma_{F}$ and $\Gamma_{F}/I_{F}\cong\Gk$. 

Let $|X|$ be the set of closed points of $X$. For each $x\in |X|$, let $F_{x}$ denote the completion of $F$ at the place $x$. The valuation ring and residue field of $F_{x}$ are denoted by $\calO_{x}$ and $k_{x}$. The maximal ideal of $\calO_{x}$ is denoted by $\frm_{x}$. The absolute Galois group of $F_{x}$ is denoted by $\Gamma_{x}$, inside of which we have the inertia group $I_{x}=\Gal(F^{s}_{x}/F^{\ur}_{x})\lhd\Gamma_{x}$ where $F^{\ur}_{x}$ is the maximal unramified extension of $F_{x}$ inside a separable closure $F^{s}_{x}$. The quotient group $\Gamma_{x}/I_{x}$ is $\Gal(\kbar/k_{x})$. Fixing an embedding $F^{s}\incl F^{s}_{x}$ for each $x$ gives an embedding of the Galois groups $\Gamma_{x}\incl\Gamma_{F}$ and $I_{x}\incl I_{F}$.

The ring of ad\`eles of $F$ is the restricted product
\begin{equation*}
\AA_{F}:={\prod_{x\in |X|}}'F_{x}
\end{equation*}
where for almost all $x$, the $x$-component of an element $a\in \AA_{F}$ lies in $\calO_{x}$. It is equipped with a natural topology: a neighborhood basis of $0$ is given by $K_{D}=\prod_{x\in|X|}\frm_{x}^{d_{x}}$ where $D=\sum_{x\in |X|}d_{x}\cdot x$ is an effective divisor on $X$.

When $k$ is a finite field,  $F$ is a global function field, and $|X|$ can be identified with the set of places of $F$. The ring of ad\`eles $\AA_{F}$ is locally compact. The Galois group $\Gk$ is topologically generated by the geometric Frobenius element $\Frob_{k}$. Recall the Weil group $W_{F}\subset \Gamma_{F}$ is the preimage of $\Frob^{\ZZ}_{k}$ under the quotient $\Gamma_{F}\surj\Gk$.

\subsection{Groups over a function field}  

\subsubsection{Tori} For any diagonalizable group $T$ over any field $K$, we denote the character and cocharacter lattices of $T\otimes_{K}\overline{K}$ by $\xch(T)$ and $\xcoch(T)$, where $\overline{K}$ is an algebraic closure of $K$.  They are discrete abelian groups with continuous actions of $\Gal(\overline{K}/K)$.

\subsubsection{Quasi-split groups}\label{sss:qsplit} We start with a split connected reductive group $\GG$ over $k$. Let $\ZZ\GG$ be the center of $\GG$ and $\GG^{\ad}=\GG/\ZZ\GG$ be the adjoint form of $\GG$. Fix a pinning of $\GG$, i.e., a split maximal torus $\TT$, a Borel subgroup $\BB$ containing $\TT$ (hence a based root system $\Phi\subset\xcoch(\TT)$ with simple roots $\Delta$) and an isomorphism $\UU_{\alpha}\cong\Ga$ for each simple root subgroup $\UU_{\alpha}\subset \GG$ ($\alpha\in\Delta$).  Let $\Aut^{\dagger}(\GG)$ be the pinned automorphism group of $\GG$, i.e., automorphisms of $\GG$ preserving the pinning. Then $\Aut^{\dagger}(\GG)$ is canonically isomorphic to the outer automorphism group $\Out(\GG)=\Aut(\GG)/\GG^{\ad}$.

Fix a homomorphism $\theta:\Gamma_{F}\to\Aut^{\dagger}(\GG)$. The image of $\theta$ is of the form $\Gamma=\Gal(F'/F)$ for a finite Galois extension $F'/F$ inside $F^{s}$. The extension of $F'/F$ corresponds to a Galois branched cover $\theta_{X}: X'\to X$ of $X$. We assume that {\em $\pi$ is tamely ramified}. Note that we allow base field extension, for example $X'$ can be $X\otimes_{k}k'$ for some finite extension $k'/k$. For any $F$-algebra $R$, we define
\begin{equation}\label{qs G}
G(R)=\{g\in \GG(F'\otimes_{F}R)|\leftexp{\gamma}{g}=\theta(\gamma)g, \forall\gamma\in\Gamma\}.
\end{equation}
Here $g\mapsto \leftexp{\gamma}{g}$ denotes the action of $\Gamma$ on $\GG(F'\otimes_{F}R)$ induced from its action on $F'$. The functor \eqref{qs G} is represented by a connected reductive group $G$ over $F$ which is a quasi-split form of $\GG$ and which becomes split over $F'$. 

The torus $\TT\subset \GG$ gives rise to a maximal torus $T\subset G$. We use $Z, G^{\der}, G^{\ad}, D$ to denote the center, the derived group, the adjoint quotient and the maximal torus quotient of $G$.

\subsubsection{The Langlands dual of $G$} Let $\dT=\xch(\TT)\otimes_{\ZZ}\GG_{m,\Qlbar}$ be a torus over $\Qlbar$. Let $\dG$ be the connected split reductive group over $\Qlbar$ with maximal torus $\dT$ and based root system $\Delta(\dG)\subset \Phi(\dG)$ identified with based coroot system $\Delta^{\vee}\subset\Phi^{\vee}$ of $\GG$ with respect to $\TT$. We extend the based root system of $\dG$ into a pinning, and denote the pinned automorphism of $\dG$ by $\Aut^{\dagger}(\dG)$. Then there is a canonical isomorphism  $\Aut^{\dagger}(\dG)\cong\Aut^{\dagger}(\GG)$. The homomorphism $\theta:\Gamma_{F}\to\Aut^{\dagger}(\GG)$ then induces $\hat{\theta}:\Gamma_{F}\to\Aut^{\dagger}(\dG)$ that factors through the finite quotient $\Gamma=\Gal(F'/F)$.

\subsubsection{Integral models of $G$}\label{sss:cG} An integral model of $G$ over $X$ is a smooth group scheme $\cG$ over $X$ together with an isomorphism of group schemes  $\iota:\cG|_{\Spec F}\isom G$ over $F$ (here we identify $\Spec F$ with the generic point of $X$). One can construct an integral model of $G$ as follows. Let $U_{\theta}=X-S_{\theta}$ where $S_{\theta}$ is the ramification locus of $\theta_{X}$. Then the same formula \eqref{qs G}, now with $R$ a commutative $k$-algebra with a morphism $\Spec R\to U_{\theta}$, gives a reductive group scheme $\cG_{U_{\theta}}$ over $U_{\theta}$. We then need to extend $\cG_{U_{\theta}}$ to a smooth group scheme  over $X$. One can do this locally at every closed point $x\in S_{\theta}$ and patch the results together. At $x\in S_{\theta}$, the theory of Bruhat and Tits gives several choices of smooth group schemes $\cG_{x}$ over $\Spec\calO_{x}$ extending $G|_{\Spec F_{x}}$ (the parahoric subgroups), and we may choose any of them. Once we choose a model $\cG$ for $G$, we can talk about $\cG(\calO_{x})$ for any $x\in|X|$. 

For the maximal torus $T$ of $G$ and the center $Z$ of $G$, we may use formulae similar to \eqref{qs G} to define their integral models $\calT$ and $\calZ$. Then $\calT$ is the finite-type version of the N\'eron model of $T$ over $X$. We make the following technical assumption (which may turn out to hold always, and which holds in all examples we consider):

{\em The group scheme $\cZ$ fits into an exact sequence $1\to \cZ\to \cT_{1}\to \cT_{2}\to1$, where $\cT_{1}$ and $\cT_{2}$ are smooth group schemes over $X$ whose generic fibers are tori.}

\subsubsection{Loop groups} For each $x\in |X|$, we may view $G(F_{x})$ as the set of $k$-points of a group indscheme $L_{x}G$ called {\em the loop group of $G$ at $x$}. The loop group of  $G$ at $x$ represents the functor $R\mapsto G(R\widehat{\otimes}_{k} F_{x})$ where $R$ is any $k$-algebra and $R\widehat{\otimes}_{k}F_{x}$ is the completion of $R\otimes_{k} F_{x}$ with respect to the $R\otimes\frm_{x}$-adic topology. Similarly we may define a pro-algebraic group over $k$ called the positive loop group $L^{+}_{x}\cG$, representing the functor $R\mapsto \cG(R\widehat{\otimes}_{k} \calO_{x})$. More generally, for a parahoric subgroup $\bP_{x}\subset G(F_{x})$ we may define the associated positive loop group which we still denote by $\bP_{x}$ (but viewed as a subgroup scheme of  $L_{x}G$).

{\bf Warning:} $L_{x}G$ and $L^{+}_{x}\cG$ are both over $k$ and not over $k_{x}$: alternatively we may define versions of loop groups over $k_{x}$ and take restriction of scalars to $k$. 

\subsection{Automorphic representations and automorphic data}\label{ss:auto} 
In this subsection we assume $k$ to be finite. We also fix a prime $\ell\neq\chk$, and choose an algebraic closure $\Qlbar$ of $\QQ_{\ell}$. We will consider automorphic representations of $G(\AA_{F})$ on $\Qlbar$-vector spaces. At this point one could replace $\Qlbar$ by $\CC$, but it will be convenient to use $\Qlbar$ directly because later we will relate automorphic forms with $\Qlbar$-sheaves. 

The group $G(\AA_{F})$ of $\AA_{F}$-points of $G$ is the restricted product of $G(F_{x})$ over $x\in|X|$ with respect to the subgroups $\cG(\calO_{x})$. Then $G(\AA_{F})$ is a locally compact topological group with a  neighborhood basis of the identity given by $\prod_{x\in S}K_{x}\times\prod_{x\notin S}\cG(\calO_{x})$, where $S\subset|X|$ is finite, $K_{x}\subset G(F_{x})$ is a compact open subgroup for all $x\in S$.  With this topology on $G(\AA_{F})$ we may talk about locally constant $\Qlbar$-valued functions on $G(\AA_{F})$. Let $C^{\infty}(\AG)$ be the space of locally constant $\Qlbar$-valued functions on $G(\AA_{F})$ that are left invariant under $G(F)$. By definition, we have
\begin{equation*}
C^{\infty}(\AG)=\varinjlim_{(S,\{K_{x}\}_{x\in S})}\Fun(\AG/(\prod_{x\in S}K_{x}\times\prod_{x\notin S}\cG(\calO_{x}))),
\end{equation*}
where $\Fun(-)$ means the space of $\Qlbar$-valued functions. The group $G(\AA_{F})$ acts on $C^{\infty}(\AG)$ via right translation: $(g\cdot f)(x)=f(xg)$ where $g\in G(\AA_{F}), f\in C^{\infty}(\AG)$ and $x\in\AG$.

\begin{defn}[See {\cite[Definition 5.8]{BorelJacquet}}]
\begin{enumerate}
\item A function $f \in C^{\infty}(\AG)$ is called an {\em automorphic form} if for some (equivalently any) $x\in |X|$, the $G(F_{x})$-module spanned by right $G(F_{x})$-translations of $f$ is admissible. Denote the space of automorphic forms by $\calA_{G}$. This is a $G(\AA_{F})$-module under right translation.
\item An {\em automorphic representation} of $G(\AA_{F})$ is an irreducible subquotient of $\calA_{G}$.
\end{enumerate}
\end{defn}

\subsubsection{Kottwitz homomorphism}\label{sss:Kott} For the local field $F_{x}$, Kottwitz defined a homomorphism \cite[\S7]{Kottwitz}
\begin{equation}\label{local Kott}
\kappa_{G,F_{x}}:G(F_{x})\to (\xch(Z\dG)_{I_x})^{\Frob_{x}}.
\end{equation}
Here  $Z\dG$ is the center of $\dG$, and $I_{x}\lhd\Gamma_{F}$ acts on $\dG$, and hence also on  $\xch(Z\dG)$ via $\hat{\theta}$. 
For the global field $F$ one can similarly define a homomorphism
\begin{equation}\label{Kott}
\kappa_{G,F}:G(\AA_{F})\to (\xch(Z\dG)_{I_{F}})^{\Frob_{k}}
\end{equation}
which is trivial on $G(F)$ and is compatible with \eqref{local Kott} in the obvious sense. Since \cite{Kottwitz} works in the local setting, we sketch the construction of \eqref{Kott} along the same lines.

To construct $\kappa_{G,F}$, we first work with $F\otimes_{k}\kbar$ instead of $F$ and try to construct $\kappa_{G,F\otimes_{k}\kbar}:G(\AA_{F\otimes_{k}\kbar})\to \xch(Z\dG)_{I_{F}}$, then take $\Frob_{k}$-invariants to get $\kappa_{G,F}$. 

In the sequel we assume $k=\kbar$. To construct $\kappa_{G,F}$, we follow these steps (see \cite[\S7.2-7.4]{Kottwitz}):
\begin{itemize}
\item  When $G=T$ is a torus,  it can be fit into an exact sequence $T_{1}\to T_{2}\to T\to1$ where $T_{1}$ and $T_{2}$ are induced tori and the corresponding sequence $T_{1}(F)\backslash T_{1}(\AA_{F})\to T_{2}(F)\backslash T_{2}(\AA_{F})\to T(F)\backslash T(\AA_{F})\to1$ is also exact. For an induced torus $T'=\Res^{E}_{F}\Gm$ (where $E$ is a finite Galois $F$-algebra), the map $\kappa_{T',F}: E^{\times}\backslash \AA^{\times}_{E}\to\ZZ=\ZZ[\Gal(E/F)]_{I_{F}}$ is given by the degree map on id\`eles. We then define $\kappa_{T,F}$ to be the unique arrow making the following diagram with exact rows commutative
\begin{equation*}
\xymatrix{T_{1}(F)\backslash T_{1}(\AA_{F})\ar[r]\ar[d]^{\kappa_{T_{1}, F}} & T_{2}(F)\backslash T_{2}(\AA_{F})\ar[r]\ar[d]^{\kappa_{T_{2}, F}} & T(F)\backslash T(\AA_{F})\ar[r]\ar[d]^{\kappa_{T, F}} & 1\\
\xcoch(T_{1})_{I_{F}}\ar[r] & \xcoch(T_{2})_{I_{F}}\ar[r] & \xcoch(T)_{I_{F}}\ar[r] & 0}
\end{equation*}

\item When $G^{\der}$ is simply-connected, $Z\dG$ is a torus and $\xch(Z\dG)=\xcoch(D)$. The map $\kappa_{G,F}$ is defined as the composition of $G(\AA_{F})\to D(\AA_{F})$ and $\kappa_{D, F}$.

\item In general, one can find a central extension $1\to C\to \wt{G}\to G\to1$ such that $\wt{G}^{\der}$ is simply-connected and $C$ is a torus. We have already defined $\kappa_{C,F}$ and $\kappa_{\wt{G},F}$ in the previous steps. We define $\kappa_{G, F}$ to be the unique arrow making the following diagram with exact rows commutative
\begin{equation*}
\xymatrix{C(\AA_{F})\ar[r]\ar[d]^{\kappa_{C, F}} & \wt{G}(\AA_{F})\ar[r]\ar[d]^{\kappa_{\wt{G}, F}} & G(\AA_{F})\ar[r]\ar[d]^{\kappa_{G, F}} & 1\\
\xcoch(\widehat{C})_{I_{F}}\ar[r] & \xcoch(Z\widehat{\wt{G}})_{I_{F}}\ar[r] & \xcoch(Z\dG)_{I_{F}}\ar[r] & 0}
\end{equation*}
\end{itemize}

\begin{defn} An {\em unramified character} of $G(\AA_{F})$ is  a smooth character $\chi:G(\AA_{F})\to \Qlbar^{\times}$ that factors through $\kappa_{G, F}$. Given an irreducible representation $\pi$ of $G(\AA_{F})$ and an unramified character $\chi$ of $G(\AA_{F})$, we obtain another irreducible representation $\pi\otimes\chi$, called an {\em unramified twist of $\pi$}.
\end{defn}

\begin{remark} Unramified characters of $G(\AA_{F})$ are in bijection with $(Z\dG)^{I_{F}}_{\Frob_{k}}$. Langlands correspondence for function fields predicts that to an automorphic representation $\pi$ of $G(\AA_{F})$ one should attach a  continuous cocycle $\rho:W_{F}\to \dG(\Qlbar)$ (with respect to the action $\htheta$ of $\Gamma_{F}$ on $\dG$) up to $\dG$-conjugacy, i.e., a class in $\upH^{1}_{\cont}(W_{F}, \dG)$. Since $W_{F}/I_{F}=\Frob^{\ZZ}_{k}$, there is a natural map $(Z\dG)^{I_{F}}_{\Frob_{k}}=\upH^{1}(\Frob^{\ZZ}_{k}, (Z\dG)^{I_{F}})\to\upH^{1}_{\cont}(W_{F}, Z\dG)$, and the latter acts on $\upH^{1}_{\cont}(W_{F}, \dG)$. Therefore there is a natural action of $(Z\dG)^{I_{F}}_{\Frob_{k}}$ on $\upH^{1}_{\cont}(W_{F}, \dG)$. The Langlands correspondence is expected to intertwine unramified twists of automorphic representations and the action of $(Z\dG)^{I_{F}}_{\Frob_{k}}$ on $\upH^{1}_{\cont}(W_{F}, \dG)$.
\end{remark}

\subsubsection{Central characters and restricted central characters} Let $Z\subset G$ be the center. For a continuous character $\omega:Z(F)\backslash Z(\AA_{F})\to\Qlbar^{\times}$, let $C^{\infty}_{\omega}(\AG)\subset C^{\infty}(\AG)$ be the subspace on which $Z(\AA_{F})$ acts through $\omega$ by right translation. Similarly let $\calA_{G,\omega}=\calA_{G}\cap C^{\infty}_{\omega}(\AG)$.

Let $Z(\AA_{F})^{\n}\subset Z(\AA_{F})$ be the kernel of the composition $Z(\AA_{F})\to G(\AA_{F})\xrightarrow{\kappa_{G,F}}\xch(Z\dG)^{\Frob_{k}}_{I_{F}}$. For each $x\in|X|$, let $Z(F_{x})^{\n}\subset Z(F_{x})$ be the kernel of $Z(F_{x})\to G(F_{x})\xrightarrow{\kappa_{G,F_{x}}}\xch(Z\dG)^{\Frob_{x}}_{I_{x}}$. Then $Z(F_{x})^{\n}$ is a compact open subgroup of $Z(F_{x})$, and for almost all $x\in |X|$, $Z(F_{x})^{\n}=\cZ(\calO_{x})$. Therefore, 
for sufficiently large $S$ and sufficiently small compact open subgroup $K_{Z, S}\subset \prod_{x\in S}Z(F_{x})$, $\prod_{x\notin S}\cZ(\calO_{x})\times K_{Z,S}$ is contained in $(\AG)^{\n}$ and the quotient $\AZn/\prod_{x\notin S}\cZ(\calO_{x})\times K_{Z,S}$ is finite. We call a continuous character $\onat:\AZ^{\natural}\to\Qlbar^{\times}$ a {\em restricted central character}. We can similarly define the spaces $C^{\infty}_{\onat}(\AG)$ and $\calA_{G,\onat}$, which are stable under unramified twists of $G(\AA_{F})$. For a restricted central character $\onat$, we may restrict it to $Z(F_{x})^{\n}$ and get a continuous character $\onat_{x}:Z(F_{x})^{\n}\to\Qlbar^{\times}$.

\subsubsection{Cusp forms} A function $f\in C^{\infty}(\AG)$ is called a {\em cusp form} if for every parabolic subgroup $P\subset G$ defined over $F$ with unipotent radical $N_{P}$, we have
\begin{equation*}
\int_{N_{P}(F)\backslash N_{P}(\AA_{F})}f(ng)dn=0
\end{equation*}
for all $g\in G(\AA_{F})$. For a central character $\omega$, the space of cusp forms in $\calA_{G,\omega}$ form a sub-$G(\AA_{F})$-module $\calA^{\cusp}_{G,\omega}\subset\calA_{G,\omega}$. It is known that $\calA^{\cusp}_{G,\omega}$ decomposes discretely into a direct sum of irreducible admissible $G(\AA_{F})$-modules, called {\em cuspidal automorphic representations}.

\begin{defn}\label{def:auto data} Let $S\subset |X|$ be finite. 
\begin{enumerate}
\item An {\em automorphic datum} for $G$ is  a triple $(\omega, K_{S}, \chi_{S})$ where
\begin{itemize}
\item $\omega:\AZ\to\Qlbar^{\times}$ is a central character;
\item $K_{S}$ is a collection of compact open subgroups $K_{x}\subset G(F_{x})$, one for each $x\in S$;
\item $\chi_{S}$ is a collection of smooth characters $\chi_{x}: K_{x}\to\Qlbar^{\times}$, one for each $x\in S$.
\end{itemize}
Such a triple should satisfy the following compatibility conditions.
\begin{itemize}
\item For each $x\in S$, $\omega_{x}|_{Z(F_{x})\cap K_{x}}=\chi_{x}|_{Z(F_{x})\cap K_{x}}$;
\item For each $x\notin S$, $\omega_{x}$ is trivial on $Z(F_{x})\cap \cG(\calO_{x})$.
\end{itemize}
\item A {\em restricted automorphic datum} is a triple $(\onat, K_{S},\chi_{S})$, where $K_{S}$ and $\chi_{S}$ are as above, $\onat:\AZn\to\Qlbar^{\times}$ is a restricted central character, and the compatibility conditions become
\begin{itemize}
\item For each $x\in S$, $\onat_{x}|_{Z(F_{x})^{\n}\cap K_{x}}=\chi_{x}|_{Z(F_{x})^{\n}\cap K_{x}}$;
\item For each $x\notin S$, $\onat_{x}$ is trivial on $Z(F_{x})^{\n}\cap \cG(\calO_{x})=Z(F_{x})\cap \cG(\calO_{x})$.
\end{itemize}
\end{enumerate}
\end{defn}

\begin{defn} Let $(\omega, K_{S},\chi_{S})$ be an automorphic datum. An automorphic representation $\pi=\otimes'_{x}\pi_{x}$ of $G(\AA_{F})$ is called {\em $(\omega, K_{S},\chi_{S})$-typical} if
\begin{itemize}
\item The central character of $\pi$ is $\omega$; 
\item For each $x\in S$, the local component $\pi_{x}$ contains a nonzero vector on which $K_{x}$ acts through the character $\chi_{x}$;
\item For each $x\notin S$, $\pi_{x}$ is spherical; i.e., $\pi_{x}^{\calG(\calO_{x})}\neq0$.
\end{itemize}
When $(\onat, K_{S},\chi_{S})$ is a restricted automorphic datum, we define {\em $(\onat,K_{S},\chi_{S})$-typical automorphic representations} in a similar way, except that the first condition changes to that the central character of $\pi$ restricted to $\AZn$ is $\onat$. Similarly, we define the notion of {\em $(K_{S},\chi_{S})$-typical automorphic representations} by dropping the central character condition.
\end{defn}

Let
\begin{equation}\label{Kchi fun}
C_{\calG,\omega}(K_{S},\chi_{S}):=C_{\omega}(\AG/\prod_{x\notin S}\calG(\calO_{x})\times\prod_{x\in S}(K_{x},\chi_{x})).
\end{equation}
denote the space of functions on $\AG$ which are eigenvectors under $Z(\AA_{F})$ with eigenvalue $\omega$,  invariant under $\prod_{x\notin S}\calG(\calO_{x})$ and are eigenvectors under the action of each $K_{x}$ with eigenvalue $\chi_{x}$, $x\in S$. Similarly one defines $C_{\calG,\omega^{\n}}(K_{S},\chi_{S})$ for a restricted automorphic datum $(\om^{\n}, K_{S},\chi_{S})$.

\begin{lemma}\label{l:cusp} Let $(\omega, K_{S}, \chi_{S})$ be an automorphic datum. If there is a subset $\Sigma\subset \AG$, compact modulo $Z(\AA_{F})$, such that all functions in $C_{\calG,\omega}(K_{S},\chi_{S})$ vanish outside $\Sigma$, then any $(\omega, K_{S}, \chi_{S})$-typical automorphic representation $\pi$ of $G(\AA_{F})$ is cuspidal. Moreover,
\begin{equation}\label{Csum}
C_{\cG,\omega}(K_{S},\chi_{S})=\bigoplus_{\pi}(\bigotimes_{x\in S}\pi^{(K_{x},\chi_{x})}_{x})\otimes(\bigotimes_{x\notin S}\pi^{\cG(\calO_{x})}_{x})
\end{equation}
where $\pi$ runs over $(\omega, K_{S}, \chi_{S})$-typical automorphic representations, and $\pi^{(K_{x},\chi_{x})}_{x}$ denotes the eigenspace of $\pi_{x}$ under $K_{x}$ on which it acts through $\chi_{x}$.
\end{lemma}
\begin{proof}
If every function in $C_{\calG,\omega}(K_{S},\chi_{S})$ is supported on some $\Sigma$ compact modulo $Z(\AA_{F})$, then the space $C_{\calG,\omega}(K_{S},\chi_{S})\subset\Fun(\AG/\prod_{x\notin S}\cG(\calO_{x})\times\prod_{x\in S}\ker(\chi_{x}))$ is finite-dimensional and stable under the action of the spherical Hecke algebra at all places $x\notin S$. Then we appeal to \cite[Lemme 7.16]{VLaff} to conclude that $C_{\calG,\omega}(K_{S},\chi_{S})$ consists of cusp forms. 

Now let $\pi$ be a $(\omega,K_{S}, \chi_{S})$-typical automorphic representation. By definition $\pi$ is a subquotient of the space $\calA_{G}$, therefore we may assume $\pi$ is a quotient of a $G(\AA_{F})$-submodule $\calA'\subset \calA_{G,\omega}$. Let $\calB=\calA'\cap C_{\calG,\omega}(K_{S},\chi_{S})$, then $\calB$ consists of cusp forms. By the semisimplicity of the action of the compact group   $\prod_{x\notin S}\cG(\calO_{x})\times\prod_{x\in S}K_{x}$, we have $\calB\surj (\otimes_{x\in S}\pi_{x}^{(K_{x},\chi_{x})})\otimes(\otimes_{x\notin S}\pi_{x}^{\cG(\calO_{x})})$. The $G(\AA_{F})$-module $\wt{\calB}$ generated by $\calB$ is then contained in the space of cusp forms, and we have $\wt{\calB}\surj \pi$. Therefore $\pi$ is also cuspidal. The equality \eqref{Csum} follows from the fact that the cuspidal spectrum is discrete.
\end{proof}

\subsection{Moduli of $G$-torsors}\label{ss:Weil} 
A.Weil observed that the double coset $\AG/\prod_{x\in |X|}\cG(\calO_{x})$ may be interpreted as the set of $k$-points of a certain moduli stack. We recall this interpretation in this subsection. 

\subsubsection{Moduli of $\cG$-torsors} Let $k$ be any field. For an integral model $\cG$ of $G$ over $X$, we may talk about $\cG$-torsors on $X$: these are schemes $\calE\to X$ with a right action of $\cG$ such that \'etale locally over $X$, $\calE$ is $\cG$-equivariantly isomorphic to $\cG$ for the right translation of $\cG$ on itself. Similarly one can define $\cG$-torsors over $X\otimes_{k}R$ for any $k$-algebra $R$. Let $\Bun_{\cG}$ be the moduli stack of $\cG$-torsors over $X$, i.e., $\Bun_{\cG}(R)$ is the groupoid of $\cG$-torsors over $X\otimes_{k}R$. Then $\Bun_{\cG}$ is an algebraic stack locally of finite type.

\begin{exam} When $G=\GL_{n}$, there is a natural way to assign a vector bundle $\calV$ of rank $n$ to a $\GL_{n}$-torsor $\calE$ and vice versa: $\calV=\calE\twtimes{\GL_{n}}\AA^{n}_{k}$ and $\calE=\underline{\Isom}(\calO_{X}^{n},\calV)$. Therefore we have an isomorphism of stacks 
\begin{equation*}
\Bun_{\GL_{n}}\cong \Bun_{n}
\end{equation*}
the latter being the moduli stack of vector bundles of rank $n$ on $X$.  In particular, for $n=1$, $\Bun_{\GL_{1}}\cong\Bun_{1}\cong\Pic_{X}$.

Similarly, $\Bun_{\SL_{n}}$ is equivalent to the groupoid of pairs $(\calV,\iota)$ where $\calV$ is a vector bundle of rank $n$ on $X$ and $\iota:\wedge^{n}\calV\isom\calO_{X}$ is a trivialization of the determinant of $\calV$. 

When $G=\PGL_{n}$, $\Bun_{\PGL_{n}}$ is isomorphic to the quotient stack $\Bun_{n}/\Pic_{X}$ with $\Pic_{X}$ acting on $\Bun_{n}$ via tensor product. The content of this statement is Tsen's theorem on the vanishing of the Brauer group for curves over an algebraically closed field.
\end{exam}

\subsubsection{The set $\upH^{1}_{\cG}(F,G)$} For each $x\in |X|$, we have a map $\upH^{1}(\Spec\calO_{x}, \cG)\to\upH^{1}(F_{x}, G)$ obtained by restricting a $\cG$-torsor over $\Spec\calO_{x}$ to $\Spec F_{x}$. Define  $\upH^{1}_{\cG}(F,G)\subset\upH^{1}(F,G)$ using the Cartesian diagram
\begin{equation*}
\xymatrix{\upH^{1}_{\cG}(F,G)\ar[r]\ar@{^{(}->}[d] & \prod_{x\in|X|}\Im(\upH^{1}(\Spec\calO_{x}, \cG)\to\upH^{1}(F_{x},G))\ar@{^{(}->}[d] \\
\upH^{1}(F,G)\ar[r] & \prod_{x\in|X|}\upH^{1}(F_{x},G)}
\end{equation*}
Note that $\upH^{1}(F\otimes_{k}\kbar, G)$ hence $\upH^{1}_{\cG}(F\otimes_{k}\kbar,G)$ is a singleton (see \cite[\S8.6]{BS}, \cite[II, 2.3, Remarque 1)]{Serre-Coho}). When $k$ is finite and $\cG$ has connected geometric fibers, $\upH^{1}(\Spec\calO_{x}, \cG)$ is trivial and $\upH^{1}_{\cG}(F,G)$ is the same as $\ker^{1}(F,G):=\ker(\upH^{1}(F,G)\to\prod_{x\in|X|}\upH^{1}(F_{x}, G))$. It is also known that when $G$ is split and $k$ is finite, $\ker^{1}(F,G)$ is a singleton (\cite[Remark 7.13]{VLaff}). In general, $\upH^{1}_{\cG}(F,G)$ measures how a $\cG$-torsor over $X$ can fail to be trivializable at the generic point of $X$. Any $\cG$-torsor $\calE$ over $X$ determines a class in $\upH^{1}_{\cG}(F,G)$ by restricting it to the generic point $\eta\in X$. 

For each class $\xi\in\upH^{1}_{\cG}(F,G)$, we choose a $\cG$-torsor $\calE_{\xi}$ with generic class $\xi$ (this is possible since the class $\xi$ gives a $\cG$-torsor $\calE_{U}$ over some nonempty Zariski open subset $U\subset X$ and we can extend $\calE_{U}$ across the missing points $x\in X-U$ because the restriction of $\xi$ to $\upH^{1}(F_{x}, G)$ lies in the image of $\upH^{1}(\Spec\calO_{x}, \cG)$). Define $\cG_{\xi}=\uAut_{\cG}(\calE_{\xi})$, the group scheme over $X$ of $\cG$-automorphisms of $\calE_{\xi}$. This is an inner form of $\cG$. We denote by $G_{\xi}$ the generic fiber of $\cG_{\xi}$, which is an inner form of $G$ over $F$.

\subsubsection{Weil's interpretation} Weil observed a natural bijection of groupoids
\begin{equation}\label{Weil e}
e:\bigsqcup_{\xi\in\upH^{1}_{\cG}(F,G)}G_{\xi}(F)\backslash G_{\xi}(\AA_{F})/\prod_{x\in|X|}\cG_{\xi}(\calO_{x})\isom\Bun_{\cG}(k).
\end{equation}
In other words, not only are the isomorphism classes of both sides in bijection, but for any coset $g=(g_{x})\in\prod'_{x\in |X|}G_{\xi}(F_{x})/\cG_{\xi}(\calO_{x})$, the automorphism group of $e(g)$ (as a $\cG_{\xi}$-torsor) is isomorphic to the stabilizer of $g$ under the left action of  $G(F)$.

We give the definition of the map $e$ on the part where $\xi$ is the trivial class, i.e., $\AG/\prod_{x\in|X|}\cG(\calO_{x})\to\Bun_{\cG}(k)$. We have
\begin{equation*}
G(F)\backslash G(\AA_{F})/\prod_{x\in |X|}\cG(\calO_{x})=\bigcup_{S\subset|X|,\textup{finite}}\cG(\calO_{X-S})\backslash \prod_{x\in S}G(F_{x})/\cG(\calO_{x}).
\end{equation*}
Thus it suffices to construct maps
\begin{equation*}
e_{S}:\cG(\calO_{X-S})\backslash \prod_{x\in S}G(F_{x})/\cG(\calO_{x})\to \Bun_{G}(k)
\end{equation*}
compatible with inclusions $S\incl S'$. To $g=(g_{x})\in \prod _{x\in S}G(F_{x})$, we assign a $\cG$-torsor $\calE_{g}$ by gluing the trivial torsor $\calE^{\triv}_{X-S}=\cG|_{X-S}$ over $X-S$ with the trivial $\cG$-torsor $\calE^{\triv}_{x}=\cG|_{\Spec\calO_{x}}$ over $\Spec\calO_{x}$ for each $x\in S$ along $\Spec F_{x}$. The gluing datum at $x$ is given by the isomorphism $G|_{\Spec F_{x}}=\calE^{\triv}_{x}|_{\Spec F_{x}}\isom\calE^{\triv}_{X-S}|_{\Spec F_{x}}=G|_{\Spec F_{x}}$ which is given by left multiplication by $g_{x}$. Changing the trivializations of $\calE^{\triv}_{X-S}$ and $\calE^{\triv}_{x}$ results in a right multiplication of $g_{x}$ by an element in $\cG(\calO_{x})$ and left multiplication by an element in $\cG(\calO_{X-S})$. Thus the isomorphism type of the resulting torsor $\cG$-torsor $\calE_{g}$ only depends on the double coset $\cG(\calO_{X-S})g\prod_{x\in S}\cG(\calO_{x})$. This defines the map $e_{S}$.

Working with the base field $\kbar$, noting that $\upH^{1}(F\otimes_{k}\kbar, G)$ vanishes, \eqref{Weil e} becomes an equivalence of groupoids
\begin{equation}\label{Weil e kbar}
G(F\otimes_{k}\kbar)\backslash G(\AA_{F\otimes_{k}\kbar})/\prod_{x\in X(\kbar)} \calG(\calO^{\ur}_{x})\cong\Bun_{\cG}(\kbar).
\end{equation}

\begin{exam}\label{ex:unitary} Unitary similitude groups. Assume $\textup{char}(k)\neq2$. Let $\pi:X'\to X$ be a double cover (possibly ramified) of $X$ which is generically \'etale with $X'$ smooth. Let $\sigma\in\Aut_{X}(X')$ be the nontrivial involution. Fix a line bundle $\calL$ on $X$. An $\calL$-Hermitian vector bundle of rank $n$ on $X'$ is a pair $(\calE',h)$ where  $\calE'$ is a vector bundle of rank $n$ on $X'$ and $h$ is a pairing
\begin{equation*}
h:\calE'\otimes_{\calO_{X'}}\sigma^{*}\calE'\to \pi^{*}\calL
\end{equation*}
such that 
\begin{enumerate}
\item $\sigma^{*}h=h\circ s$ where $s:\calE'\otimes_{\calO_{X'}}\sigma^{*}\calE'\to\sigma^{*}\calE'\otimes_{\calO_{X'}}\calE'$ switches the two factors. (Here we use the canonical isomorphism $\sigma^{*}\pi^{*}\calL\cong\pi^{*}\calL$) to identify the targets of $\sigma^{*}h$ and $h\circ s$.
\item The map $\calE'\to \uHom_{X'}(\sigma^{*}\calE',\pi^{*}\calL)=\sigma^{*}\calE'^{\vee}\otimes\pi^{*}\calL
$ induced by $h$ is an isomorphism.
\end{enumerate}
Let $\Bun^{\Herm}_{n, X'/X}$ be the moduli stack classifying triples $(\calL,\calE',h)$ as above, i.e., $\calL$ is a line bundle over $X$ and $(\calE',h)$ is an $\calL$-Hermitian vector bundles of rank $n$ over $X'$. 

Let $\calE'_{0}=\calO^{\oplus n}_{X'}$ and $h_{0}: \calE'_{0}\otimes\sigma^{*}\calE'_{0}\to\calO_{X'}$ be the direct sum of $n$ copies of rank one $\calO_{X}$-Hermitian bundles. We define a group scheme $\cG$ over $X$ in the following way: for each $k$-algebra $R$ and $x\in X(R)$, we denote the preimage of $x$ in $X'$ by $\Spec R'$, then the fiber $\cG_{x}(R)=\{(g,\l)\in \GL_{n}(R')\times R^{\times}|h_{0}(ge_{1},ge_{2})=\l h_{0}(e_{1},e_{2})\textup{ for all }e_{1},e_{2}\in R'^{n}\}$. The generic fiber $G$ of $\cG$ is the unitary similitude group $\GU(n,h_{0})$ over $F$ defined using the Hermitian form $h_{0}|_{\Spec F}$ on $F'^{n}$. 

The set $\upH^{1}(F,G)$ parametrizes $F^{\times}$-homothety classes of ($F'$-valued) Hermitian forms $h$ on $F'^{n}$ up to change of bases. The subset $\upH^{1}_{\cG}(F,G)\subset \upH^{1}(F,G)$ parametrizes those Hermitian forms $h$ on $F'^{n}$ whose base change to $F'^{n}_{x}$ for each $x\in |X|$ contains an $\calO'_{x}$-lattice $\L_{x}$ such that $\L^{\bot}_{x}=c\L_{x}$ for some $c\in F^{\times}_{x}$  (here $\L^{\bot}_{x}=\{v\in F'^{n}_{x}|h(v,\L_{x})\subset\calO'_{x}\}$). We call such a lattice {\em essentially self-dual under $h$}. Let $\Herm_{\textup{esd}}(n,F)$ be the set of Hermitian forms on $F'^{n}$ whose base change to each $F'^{n}_{x}$ contains an essentially self-dual lattice. Then $\upH^{1}_{\cG}(F,G)=\Herm_{\textup{esd}}(n,F)/(\GL_{n}(F')\times F^{\times})$, where $\GL_{n}(F')$ acts by changing bases for $F'^{n}$ and $F^{\times}$ acts by scaling the form.

For each $h\in \Herm_{\textup{esd}}(n,F)$, there exists a nonempty Zariski open subset $U\subset X$ such that $(F'^{n},h)$ extends to an $\calO_{U}$-Hermitian bundle $(\calO^{n}_{U'}, h_{U'})$ where $U'=\pi^{-1}(U)$. For each $x\in X-U$ we may choose an essentially self-dual lattice $\L^{h}_{x}\subset F'^{n}_{x}$ and $c_{x}\in F_{x}^{\times}/\calO^{\times}_{x}$ such that $\L_{x}^{h,\bot}=c_{x}\L^{h}_{x}$. Define $c_{x}=1$ for $x\in |U|$. Let $\calL_{c}$ be the line bundle on $X$ defined by the id\`ele class $c=(c_{x})_{x\in |X|}$. We may glue $(\calO^{n}_{U'}, h_{U'})$ with the lattices $\L^{h}_{x}$ to form an $\calL_{c}$-Hermitian vector bundle $(\calE'_{h}, h)$. Then we may define the unitary similitude group scheme over $X$ in a similar way as we defined $\cG$, using $(\calE'_{\xi}, h_{\xi})$ instead of $(\calO^{n}_{X'}, h_{0})$, and its generic fiber now becomes $\GU(n,h)$. The equivalence \eqref{Weil e} in this case states an equivalence of groupoids
\begin{equation}\label{Herm bun}
\bigsqcup_{[h]\in\Herm_{\textup{esd}}(n,F)/(\GL_{n}(F')\times F^{\times})}\GU(n,h, F)\backslash \GU(n,h, \AA_{F})/\prod_{x\in|X|}\textup{GAut}_{\calO'_{x}}(\L^{h}_{x}, h)\isom\Bun^{\Herm}_{n}(k).
\end{equation}
Here $\textup{GAut}_{\calO'_{x}}(\L^{h}_{x}, h)$ is the group of $\calO'_{x}$-linear automorphisms of $\L^{h}_{x}$ that preserve $h$ up to a scalar in $\calO^{\times}_{x}$. Note that the left side of \eqref{Herm bun} maps to $F^{\times}\backslash \AA^{\times}_{F}/\prod_{x}\calO^{\times}_{x}\cong\Pic_{X}(k)$ by taking the similitude character, which corresponds to the morphism $\Bun^{\Herm}_{n}\to\Pic_{X}$ sending $(\calL,\calE',h)$ to $\calL$.
\end{exam}

\subsubsection{Level structures}\label{sss:level str} One can generalize $\Bun_{\cG}$ to $\cG$-torsors with level structures. Fix a finite set $S\subset |X|$, and for each $x\in S$ let $\bK_{x}\subset L_{x}G$ be a pro-algebraic subgroup that is contained in some parahoric subgroup of $L_{x}G$, and has finite codimension therein. Then there is a group scheme $\cG(\bK_{S})$ obtained by modifying $\cG$ at each $x\in S$ such that $L^{+}_{x}\cG(\bK_{s})=\bK_{x}$. A $\cG$-torsors over $X$ with $\bK_{S}$-level structures is the same as a $\cG(\bK_{S})$-torsor. We shall denote the corresponding moduli stack by $\Bun_{\cG}(\bK_{S})$. Let $K_{x}=\bK_{x}(k)\subset G(F_{x})$. Then Weil's equivalence \eqref{Weil e} can be generalized to an equivalence of groupoids
\begin{equation}\label{Weil Bun general}
\bigsqcup_{\xi\in\upH^{1}_{\cG, \bK_{S}}(F,G)}G_{\xi}(F)\backslash G_{\xi}(\AA_{F})/(\prod_{x\notin S}\calG_{\xi}(\calO_{x})\times\prod_{x\in S}K_{x})\isom\Bun_{\calG}(\bK_{S})(k).
\end{equation}
Here $\upH^{1}_{\cG, \bK_{S}}(F,G)$ is the set defined similarly as $\upH^{1}_{\cG}(F,G)$, replacing $\cG$ with $\cG(\bK_{S})$.

\subsubsection{Connected components of $\Bun_{\cG}$}\label{sss:comp Bun} The construction of the Kottwitz map $\kappa_{G,F}$ \eqref{Kott} can be adapted to give a map on the level of moduli stacks. For $S$ containing $S_{\theta}$ and $\bK_{x}$ contained in some parahoric subgroup of $L_{x}G$, there is a morphism
\begin{equation}\label{geom Kott}
\kappa:\Bun_{\calG}(\bK_{S})\to \un{\xch(Z\dG)_{I_{F}}}.
\end{equation}
Here $\un{\xch(Z\dG)_{I_{F}}}$ means the finite \'etale group scheme over $k$ with geometric points $\xch(Z\dG)_{I_{F}}$ and the obvious $\Frob_{k}$-action on it. This map is always surjective. In \cite[Theorem 6]{Heinloth}, J. Heinloth showed that when $G$ is semisimple or a torus (and $\bK_{x}$ are contained in parahoric subgroups), the map $\kappa$ exhibits $\xch(Z\dG)_{I_{F}}$ as the set of  geometric connected components of the stack $\Bun_{\calG}(\bK_{S})$.

\subsubsection{The action of $\Bun_{\cZ}$}\label{sss:BunZ} Recall from \S\ref{sss:cG} that $Z$ admits an integral model $\cZ$ over $X$. We may define the moduli functor $\Bun_{\cZ}$ classifying $\cZ$-torsors over $X$. The assumption in \S\ref{sss:cG} guarantees that $\Bun_{\cZ}$ is representable by an algebraic stack. In fact, present $\cZ$ as the kernel of a surjection $\cT_{1}\surj \cT_{2}$ of smooth models of tori $T_{1}\surj T_{2}$ over $X$, $\Bun_{\cZ}$ is represented by the kernel of the morphism $\Bun_{\cT_{1}}\to\Bun_{\cT_{2}}$ between Picard stacks. Concretely, $\Bun_{\cZ}$ can be identified with the moduli stack of pairs $(\calE_{1}, \tau)$ where $\calE_{1}$ is a $\cT_{1}$-torsor over $X$ and $\tau$ is a trivialization of the associated $\cT_{2}$-torsor $\calE_{1}\twtimes{\cT_{1}}\cT_{2}$. 

More generally, for a pro-algebraic subgroup $\bK_{Z, x}\subset L^{+}_{x}\cZ$ of finite codimension (one for each $x\in S$), we may modify the integral model $\cZ$ at $S$ to get another integral model $\cZ(\bK_{Z,S})$. We can form the moduli stack $\Bun_{\cZ}(\bK_{Z, S})$ of $\cZ(\bK_{Z,S})$-torsors over $X$ (i.e., $\cZ$-torsors with $\bK_{Z,x}$-level structures at $x\in S$). This is a $\prod_{x\in S}L^{+}_{x}\cZ/\bK_{Z,x}$-torsor over $\Bun_{\cZ}$.

Now consider a certain level structure $\bK_{S}=\{\bK_{x}\}_{x\in S}$ for $G$. Let $\cG(\bK_{S})$ be the group scheme defined in \S\ref{sss:level str}. Let $\bK_{Z,x}=L_{x}Z\cap\bK_{x}$, which is necessarily contained in $L^{+}_{x}\cZ$. Therefore $\cZ(\bK_{Z,S})$ and $\Bun_{\cZ}(\bK_{Z, S})$ are defined. For every point $\calE\in\Bun_{\cG}(\bK_{S})$, the automorphism group scheme $\uAut_{\cG, \bK_{S}}(\calE)$ is a group scheme over $X$ which is an inner form of $\cG(\bK_{S})$. Since $\cZ(\bK_{Z,S})\subset\cG(\bK_{S})$, $\cZ(\bK_{Z,S})$ is canonically a central subgroup scheme of $\uAut_{\cG, \bK_{S}}(\calE)$. It therefore makes sense to twist $\calE$ by a $\cZ(\bK_{Z,S})$-torsor. In conclusion we get an action of the Picard stack $\Bun_{\cZ}(\bK_{Z,S})$ on $\Bun_{\cG}(\bK_{S})$.

\begin{exam} 
\begin{enumerate}
\item In Example \ref{ex:unitary}, the center $Z\subset\GU(n,h_{0})$ is the induced torus $\Res^{F'}_{F}\Gm$ with the integral model $\cZ=\Res^{X'}_{X}\Gm$. Therefore $\Bun_{\cZ}\cong\Pic_{X'}$. The action of $\calL'\in\Pic_{X'}$ on $\Bun^{\Herm}_{n}$ is given by $(\calL, \calE', h)\mapsto(\Nm_{X'/X}(\calL')\otimes \calL, \calE'\otimes\calL', h')$, where $h'$ is the obvious $\Nm_{X'/X}(\calL')\otimes \calL$-Hermitian structure on $\calE'\otimes\calL'$ induced from $h$ by tensoring with the natural isomorphism  $\calL'\otimes\sigma^{*}\calL'\cong\pi^{*}\Nm_{X'/X}(\calL')$.
\item Suppose $G$ is split (i.e., $G=\GG\otimes_{k}F$) and $X=\PP^{1}$. In this case, $\cG=\GG\times X$ and $\cZ=\ZZ\GG\times X$. Then $\Bun_{\cZ}$ is simply the classifying space $\BB(\ZZ\GG)$, because any $\ZZ\GG$-torsor over $\PP^{1}$ is trivial. The fact that there exists an action of $\Bun_{\cZ}=\BB(\ZZ\GG)$ on $\Bun_{\cG}$ simply says that the automorphism group of every point of $\Bun_{\cG}$ contains $\ZZ\GG$. 

In general, let $\bA_{Z,S}$ be the automorphism group of points in $\Bun_{\cZ}(\bK_{Z,S})$, then $\bA_{Z,S}$ is contained in the automorphism group of every point of $\Bun_{\cG}(\bK_{S})$.
\end{enumerate}
\end{exam}

\subsection{The sheaf-to-function correspondence}\label{ss:ff}  

\subsubsection{Convention}
We temporarily use $X$ to denote an algebraic stack locally of finite type over a field $k$. Let $D^{b}_{c}(X)$ denote the bounded derived category of constructible complex of $\Qlbar$-sheaves for the \'etale topology of $X$. For detailed definition see \cite{WeilII} for the case of schemes and \cite{LO} and \cite{LZ} for the case of stacks. All complexes of sheaves we consider will have $\Qlbar$-coefficients unless otherwise stated. All sheaf-theoretic functors are understood to be derived unless otherwise stated.

\subsubsection{The correspondence} For the rest of this subsection, $k$ is a finite field. Let $\calF\in D^{b}_{c}(X)$. For each $x\in X(k')$,  $\Gal(\kbar/k')$ acts on the geometric stalk $\calF_{\xbar}$ (a complex of $\Qlbar$-vector spaces). Let $\Frob_{k'}\in\Gal(\kbar/k')$ be the geometric Frobenius element. Define a function
\begin{eqnarray*}
f_{\calF, k'}: X(k')&\to& \Qlbar\\
x  &\mapsto& \sum_{i\in\ZZ}(-1)^{i}\Tr(\Frob_{k'}, \upH^{i}\calF_{\xbar})
\end{eqnarray*}
We get a family of $\Qlbar$-valued functions $\{f_{\calF,k'}\}_{k'/k}$ defined on $X(k')$ for varying finite extensions $k'/k$. Let $\Fun(S)$ denote the vector space of $\Qlbar$-valued functions on a set $S$. The map
\begin{eqnarray*}
\textup{Ob }D^{b}_{c}(X)&\to& \prod_{k'/k}\Fun(X(k'))\\
\calF&\mapsto& \{f_{\calF,k'}\}_{k'/k}
\end{eqnarray*}
is called the {\em sheaf-to-function} correspondence.

\subsubsection{Functorial properties of the sheaf-to-function correspondence} If $\phi: X\to Y$ is a schematic morphism of finite type between algebraic stacks over $k$ which induces a map $\phi(k'):X(k')\to Y(k')$ between finite sets, then for any $\calF\in D^{b}_{c}(X)$, we have
\begin{equation}\label{f!}
f_{\phi_{!}\calF,k'}=\phi(k')_{!}f_{\calF,k'}.
\end{equation}
Here $\phi_{!}$ is the derived push-forward functor with compact support, while $\phi(k')_{!}: \Fun(X(k'))\to \Fun(Y(k'))$ is ``summation along the fibers'' (in the case where $\phi$ is not schematic, summation along the fibers should be the weighted by the cardinalities of the automorphism groups of the $k$-points on the fiber).

A special case of \eqref{f!} is the Lefschetz trace formula. Namely if we take $Y=\Spec k$, $X$ a scheme over $k$ with $\phi$ the structure morphism and $k'=k$, we have on the one hand
\begin{equation}\label{Lef left}
f_{\phi_{!}\calF,k}=\sum_{i\in\ZZ}(-1)^{i}\Tr(\Frob_{k}, \cohoc{i}{X\otimes_{k}\kbar,\calF})
\end{equation}
and on the other
\begin{equation}\label{Lef right}
\phi(k)_{!}f_{\calF, k}=\sum_{x\in X(k)}f_{\calF,k}(x)=\sum_{x\in X(k)}\sum_{i\in\ZZ}(-1)^{i}\Tr(\Frob_{k},\upH^{i}\calF_{\xbar}).
\end{equation}
The equality of the right hand sides of \eqref{Lef left} and \eqref{Lef right} is the Lefschetz trace formula for the Frobenius morphism of $X$.

For any $\calG\in D^{b}_{c}(Y)$, we have
\begin{equation*}
f_{\phi^{*}\calG, k'}=\phi(k')^{*}f_{\calG,k'}
\end{equation*}
where $\phi(k')^{*}:\Fun(Y(k'))\to \Fun(X(k'))$ is the pullback of functions.

If $\calF,\calG\in D^{b}_{c}(X)$, then we have
\begin{equation*}
f_{\calF\otimes\calG, k'}=f_{\calF, k'}f_{\calG,k'}.
\end{equation*}
where the right hand side means pointwise multiplication of functions on $X(k')$.

\subsubsection{Character sheaves} In general, there is no canonical way to go from functions back to sheaves. However in many cases, starting from a function $f$ on $X(k)$, there is a natural candidate sheaf $\calF$ for which $f_{\calF,k}=f_{k}$. When $X=L$ is an algebraic group over $k$ and $f:L(k)\to\Qlbar^{\times}$ is a character, we will see in Appendix \ref{a:ch} that often times one can construct a {\em rank one character sheaf} on $L$ whose associated function is $f$. The correspondence between rank one character sheaves on $L$ and characters of $L(k)$ will play an important role in defining geometric automorphic data in the next subsection, and we shall frequently refer to the notations and results in Appendix \ref{a:ch}.

\subsection{Geometric automorphic data and automorphic sheaves}\label{ss:geom data}  
Let $k$ be any field.
We fix an integral model $\cG$ of $G$ over $X$ constructed as in  \S\ref{sss:cG}. Recall that $S_{\theta}\subset |X|$ is the locus where $\theta_{X}$ is ramified. The group scheme $\cG|_{X-S_{\theta}}$ is reductive with connected fibers. The exact choice of the model $\cG$ at $x\in S_{\theta}$ is not important at this point because we will impose level structures at these points. Let $S\subset|X|$ be a finite set containing $S_{\theta}$.

\subsubsection{Geometric (restricted) central character} Let $\cZ$ be the integral model of $Z$ over $X$ as in \S\ref{sss:cG}. For each $x\in S$, let $\bK^{+}_{Z, x}\subset L^{+}_{x}\cZ$ be a {\em connected} pro-algebraic subgroup. We may form the stack $\Bun_{\cZ}(\bK^{+}_{Z, S})$ of $\cZ$-torsors over $X$ with $\bK^{+}_{Z, x}$-level structures at $x$ for each $x\in S$, as in \S\ref{sss:BunZ}. Restricting the Kottwitz morphism $\kappa_{\Bun_{\cG}}$ in \S\ref{sss:comp Bun} to $\Bun_{\cZ}(\bK^{+}_{Z,S})$, we get a homomorphism $\Bun_{\cZ}(\bK^{+}_{Z,S})\to \underline{\xch(Z\dG)_{I_{F}}}$. Let $\Bun^{\n}_{\cZ}(\bK^{+}_{Z,S})$ be the kernel of this homomorphism. When $S\neq\varnothing$ and $\bK^{+}_{Z,x}$ is chosen small enough, $\Bun^{\n}_{\cZ}(\bK^{+}_{Z,S})$ is an algebraic group over $k$ (i.e., $\cZ$-torsors with $\bK^{+}_{Z,S}$-level structures have trivial automorphism groups, and $\Bun^{\n}_{\cZ}(\bK^{+}_{Z,S})$ is of finite type over $k$). We consider the category $\cCS(\Bun^{\n}_{\cZ}(\bK^{+}_{Z,S}))$ of rank one character sheaves on $\Bun^{\n}_{\cZ}(\bK^{+}_{Z,S})$. 

If we shrink $\bK^{+}_{Z,x}$ to a smaller level $\bK^{++}_{Z,x}$ for each $x\in S$, then we have a natural projection $\Bun^{\n}_{\cZ}(\bK^{++}_{S})\to\Bun^{\n}_{\cZ}(\bK^{+}_{S})$ which is a $\prod_{x\in S}\bK^{+}_{Z,x}/\bK^{++}_{Z,x}$-torsor. The pullback functor $\cCS(\Bun^{\n}_{\cZ}(\bK^{+}_{Z,S}))\to \cCS(\Bun^{\n}_{\cZ}(\bK^{++}_{Z,S}))$ is fully faithful since $\bK^{+}_{Z,x}/\bK^{++}_{Z,x}$ is geometrically connected. We may form the colimit over smaller and smaller $\bK^{+}_{Z,S}$ with respect to the fully faithful embeddings of categories
\begin{equation*}
\cCS^{\n}(Z;S):=\varinjlim_{\bK^{+}_{Z,S}}\cCS(\Bun^{\n}_{\cZ}(\bK^{+}_{Z,S})).
\end{equation*}
An object in the category $\cCS^{\n}(Z;S)$ is called a {\em geometric (restricted) central character}.

When $k$ is a finite field, let $K^{+}_{Z,x}=\bK^{+}_{Z,x}(k)$. Via the sheaf-to-function correspondence and the homomorphism $\AZn/(\prod_{x\notin S}\cZ(\calO_{x})\times\prod_{x\in S}K^{+}_{Z,x})\incl\Bun^{\n}_{\cZ}(\bK_{Z,S})(k)$ as in \eqref{Weil Bun general},  we get a homomorphism
\begin{equation*}
\CS(\Bun^{\n}_{\cZ}(\bK^{+}_{Z,S}))\to\Hom(\AZn/(\prod_{x\notin S}\cZ(\calO_{x})\times\prod_{x\in S}K^{+}_{Z,x}), \Qlbar^{\times}).
\end{equation*}
Here $\CS(-)$ means the group of isomorphism classes of objects in $\cCS(-)$. 
Passing to the colimit, we get a homomorphism
\begin{equation*}
\CS^{\n}(Z;S):=\textup{Ob } \cCS^{\n}(Z;S)\to\Hom_{\cont}(\AZn/\prod_{x\notin S}\cZ(\calO_{x}),\Qlbar^{\times}).
\end{equation*}
i.e., each geometric restricted central character gives a restricted central character.
 
For any choice of $\bK^{+}_{Z,S}$ and each  $x\in S$, there is a homomorphism $i_{x}:L_{x}Z\to L_{x}Z/\bK^{+}_{Z,x}\to\Bun_{\cZ}(\bK^{+}_{Z,S})$ by locally modifying the $\cZ$-torsors at $x$. Let $(L_{x}Z)^{\n}\subset L_{x}Z$ be the kernel of $L_{x}Z\to L_{x}G\to \underline{\xch(Z\dG)_{_{I_{x}}}}$ (the local Kottwitz homomorphism \eqref{local Kott}). Then $i_{x}$ restricts to a homomorphism $i^{\n}_{x}: (L_{x}Z)^{\n}\to (L_{x}Z)^{\n}/\bK^{+}_{Z,x}\to\Bun^{\n}_{\cZ}(\bK^{+}_{Z,S})$. Suppose $\Om\in\cCS^{\n}(Z;S)$ comes from an object in $\CS(\Bun^{\n}_{\cZ}(\bK^{+}_{Z,S}))$, we let
\begin{equation*}
\Om_{x}:=i^{\n,*}_{x}\Om\in\cCS((L_{x}Z)^{\n}),
\end{equation*}
which is a rank one character sheaf that descends to the algebraic group $(L_{x}Z)^{\n}/\bK^{+}_{Z,x}$. The object $\Om_{x}$ is independent of the choice of $\bK^{+}_{Z,S}$. Let $\Om_{S}=\boxtimes_{x\in S}\Om_{x}\in\cCS(\prod_{x\in S}(L_{x}Z)^{\n})$.

\begin{defn}\label{def:geom data} A quadruple $(\Om, \bK_{S},\calK_{S},\iota_{S})$ is a {\em geometric automorphic datum} with respect to $S$ if
\begin{itemize}
\item $\Om\in\cCS^{\n}(Z;S)$ is a geometric restricted central character;
\item $\bK_{S}$ is a collection $\{\bK_{x}\}_{x\in S}$, where  $\bK_{x}\subset L_{x}G$ is a pro-algebraic group contained in some parahoric subgroup of $L_{x}G$ with finite codimension. We require that each $\bK_{x}$ is generated by $\bK_{x}\cap L^{x}Z$ and $\bK_{x}\cap L_{x}G^{\der}$. We often use $\bK_{S}$ also to denote the product $\prod_{x\in S}\bK_{x}$.
\item $\calK_{S}$ is a collection $\{\calK_{x}\}_{x\in S}$ where each $\calK_{x}\in\cCS(\bK_{x})$ is a rank one character sheaf on $\bK_{x}$ that is the pullback of a rank one character sheaf from some finite dimensional quotient $\bK_{x}\surj \bL_{x}$.  We often use $\cK_{S}$ to denote the tensor product $\boxtimes_{x\in S}\cK_{x}\in\cCS(\bK_{S})$.
\item Note that $\bK_{Z,x}:=\bK_{x}\cap L_{x}Z$ is automatically contained in  $(L_{x}Z)^{\n}$, since $\bK_{x}$ was assumed to be contained in a parahoric subgroup. Let $\bK_{Z,S}:=\prod_{x\in S}\bK_{Z,x}$. Then $\iota_{S}$ is an isomorphism $\iota_{S}: \Om_{S}|_{\bK_{Z,S}}\cong\cK_{S}|_{\bK_{Z, S}}$ in the category $\cCS(\bK_{Z,S})$ (which is the same as a collection of isomorphisms $\iota_{x}:\Om_{x}|_{\bK_{Z,x}}\cong\cK_{x}|_{\bK_{Z, x}}$, one for each $x\in S$). 
\end{itemize}
\end{defn}

When $k$ is a finite field, a geometric automorphic datum $(\Om, \bK_{S}, \cK_{S},\iota_{S})$ gives rise to a restricted automorphic datum $(\om^{\n}, K_{S}, \chi_{S})$ as in Definition \ref{def:auto data}, by setting $K_{x}:=\bK_{x}(k)$ and using the sheaf-to-function correspondence to turn $\Om$ and $\cK_{x}$ into $\om^{\n}$ and $\chi_{x}$. However, passage from a geometric automorphic datum to a restricted automorphic datum loses information, especially that carried by $\iota_{S}$, as we shall see in \S\ref{sss:sec central}.

\subsubsection{A smaller level $\bK^{+}_{S}$}\label{sss:Kplus} Let $(\Om, \bK_{S}, \cK_{S},\iota_{S})$ be a geometric automorphic datum. For each $x\in S$ we choose a pro-algebraic normal subgroup $\bK^{+}_{x}\lhd\bK_{x}$ of finite codimension such that
\begin{itemize}
\item The character sheaf $\cK_{x}$ descends to the finite-dimensional quotient $\bL_{x}:=\bK_{x}/\bK^{+}_{x}$. 
\item Let $\bK^{+}_{Z,x}=\bK^{+}_{x}\cap L_{x}Z$. Then $\bK^{+}_{Z,x}$ is connected and $\Bun_{\cZ}(\bK^{+}_{Z, S})$ is a scheme (i.e., the automorphic group of every point is trivial), and the geometric central character $\Om$ descends to $\Bun^{\n}_{\cZ}(\bK^{+}_{Z,S})$.
\end{itemize}
By shrinking $\bK^{+}_{S}$ these conditions can always be satisfied. With this choice of $\bK^{+}_{S}$, we shall view $\cK_{x}$ as an object in $\cCS(\bL_{x})$, and view $\Om$ as an object in $\cCS(\Bun^{\n}_{\cZ}(\bK^{+}_{Z, S}))$. 

For each $x\in S$, let $\bK_{Z,x}=\bK_{x}\cap L_{x}Z$. Then $\bL_{Z,x}:=\bK_{Z,x}/\bK^{+}_{Z,x}$ is a central subgroup of $\bL_{x}$.  Let
\begin{equation*}
\bL_{S}:=\prod_{x\in S}\bL_{x}; \quad \cK_{S}:=\boxtimes_{x\in S}\cK_{x}\in\cCS(\bL_{S}); \quad \bL_{Z,S}:=\bL_{Z, x}
\end{equation*}
The morphism $\Bun_{\calG}(\bK^{+}_{S})\to \Bun_{\calG}(\bK_{S})$ is an $\bL_{S}$-torsor. Similarly, $\Bun^{\n}_{\cZ}(\bK^{+}_{Z,S})\to\Bun^{\n}_{\cZ}(\bK_{Z,S})$ is an $\bL_{Z,S}$-torsor. The group scheme $\Bun_{\cZ}(\bK^{+}_{Z,S})$ acts on $\Bun_{\cG}(\bK^{+}_{S})$ by the discussion in \S\ref{sss:BunZ}. Summarizing the situation we have a diagram
\begin{equation}\label{action Bun}
\xymatrix{\Bun^{\n}_{\cZ}(\bK^{+}_{Z, S})\ar@{^{(}->}[r]\ar[d]_{\bL_{Z,S}} & \Bun_{\cZ}(\bK^{+}_{Z, S})\ar@<-4ex>[d]_{\bL_{Z,S}}  \curvearrowright  \Bun_{\cG}(\bK^{+}_{S})\ar@<4ex>[d]^{\bL_{S}} \\
\Bun^{\n}_{\cZ}(\bK_{Z, S})\ar@{^{(}->}[r] & \Bun_{\cZ}(\bK_{Z, S})  \curvearrowright \Bun_{\cG}(\bK_{S})}
\end{equation}

\subsubsection{Secondary central character}\label{sss:sec central} The isomorphism $\iota_{S}$ is an extra piece of structure in the geometric automorphic datum that do not appear in the restricted automorphic datum $(\onat, K_{S}, \chi_{S})$. Choose $\bK^{+}_{S}$ as in \S\ref{sss:Kplus}, then $\iota_{S}$ is an isomorphism $\Om|_{\bL_{Z,S}}\cong\cK_{S}|_{\bL_{Z,S}}$. Let $\bA_{Z, S}$ be the automorphism group of the identity point of the Picard stack $\Bun_{\cZ}(\bK_{Z,S})$. Then we have a canonical homomorphism $i_{Z}: \bA_{Z,S}\to \bK_{Z,S}\incl\bK_{S}$. On the other hand, $\bA_{Z,S}$ is the kernel of the homomorphism $\bL_{Z,S}=\bK_{Z,S}/\bK^{+}_{Z,S}\incl\prod_{x\in S}L_{x}Z/\bK^{+}_{Z,S}\xrightarrow{(i_{x})_{x\in S}}\Bun_{\cZ}(\bK^{+}_{Z,S})$ for $\bK^{+}_{Z,S}$ as in \S\ref{sss:Kplus}. Since $\Om$ is a local system on $\Bun^{\n}_{\cZ}(\bK^{+}_{Z, S})$, the restriction $\Om_{S}|_{\bA_{Z,S}}$ admits a canonical trivialization. Composing this trivialization with the restriction of the isomorphism  $\iota_{S}: \Om_{S}|_{\bL_{Z,S}}\cong\cK_{S}|_{\bL_{Z,S}}\in\cCS(\bL_{Z,S})$ to $\bA_{Z,S}$, we get a trivialization
\begin{equation*}
\tau:\cK_{S}|_{\bA_{Z,S}}\cong\Qlbar\in\cCS(\bA_{Z,S}).
\end{equation*}
We would like to call $\tau$ the {\em secondary central character} attached to the geometric automorphic datum $(\Om, \bK_{S}, \cK_{S},\iota_{S})$. By Remark \ref{r:char}\eqref{char bc}, the trivializations of $\cK_{S}|_{\bA_{Z,S}}$ form a torsor under $\Hom(\pi_{0}(\bA_{Z,S})_{\Gk}, \Qlbar^{\times})$, so the secondary central character can add as much as  $\#\pi_{0}(\bA_{Z,S})_{\Gk}$ extra constraints in addition to the restricted automorphic datum $(\onat, K_{S}, \chi_{S})$. 

The above discussion also shows that when the coarse moduli space of $\Bun_{\cZ}(\bK_{Z,S})$ is a point, the secondary central character determines $(\Om,\iota_{S})$. We summarize this into a lemma.

\begin{lemma}\label{l:unique Om} Let $(\bK_{S},\chi_{S})$ be as in Definition \ref{def:geom data}. Let $\bK^{+}_{S}$ be chosen as in \S\ref{sss:Kplus}.
Suppose the coarse moduli space of $\Bun_{\cZ}(\bK_{Z,S})$ is a point, i.e., $\Bun_{\cZ}(\bK_{Z,S})\cong\BB\bA_{Z,S}$. Then the pairs $(\Om,\iota_{S})$ such that $(\Om, \bK_{S}, \cK_{S},\iota_{S})$ form a geometric automorphic datum correspond bijectively to the descents of $\cK_{S}$ to $\bL_{S}/\bA_{Z,S}$. 

In particular, if $\Bun_{\cZ}(\bK_{Z,S})\cong\BB\bA_{Z,S}$ and $\bA_{Z,S}$ is connected, there is up to isomorphism a unique pair $(\Om,\iota_{S})$ making $(\Om, \bK_{S}, \cK_{S},\iota_{S})$  a geometric automorphic datum. 
\end{lemma}

\begin{exam} Suppose $G$ is split and semisimple and $X=\PP^{1}$. Then $\cG=\GG\times X$ and $\cZ=\ZZ\GG\times X$ to are constant group schemes over $X$. Assume each $\bK_{x}$ contains $\ZZ\GG$, then $\Bun_{\cZ}(\bK_{Z, S})=\BB(\ZZ\GG)$ and $\bA_{Z,S}=\ZZ\GG$. Given $\cK_{S}$, the datum of $(\Om,\iota_{S})$ correspond bijectively to descents  $\overline{\cK}_{S}\in\cCS(\bK_{S}/\ZZ\GG)$ of $\cK_{S}$. Via the sheaf-to-function correspondence, $\overline{\cK}_{S}$ gives a character of $(\bK_{S}/\ZZ\GG)(k)$, which is in general a larger group than $\bK_{S}(k)/\ZZ\GG(k)$. On the other hand, the corresponding restricted  automorphic datum $(\onat, K_{S},\chi_{S})$ only gives a character on the smaller group $\bK_{S}(k)/\ZZ\GG(k)$. Therefore in general geometric automorphic data contain more information than the corresponding restricted automorphic data.
\end{exam}

\subsubsection{Base change of geometric automorphic data}\label{sss:bc} Let $k'/k$ be a field extension. Let $S'$ be the preimage of $S$ in $X\otimes_{k}k'$ (a closed point in $S$ may split into several in $S'$).  Given a geometric automorphic datum $(\Om, \bK_{S}, \cK_{S},\iota_{S})$ for $G$ over $F$ with respect to $S$, we may define a corresponding geometric automorphic datum $(\Om', \bK_{S'}, \cK_{S'},\iota_{S'})$ for the function field $F\otimes_{k}k'$ with respect to $S'$ as follows. 

The base change $\bK_{x}\otimes_{k}k'$ naturally decomposes as a product $\prod_{y\mapsto x}\bK_{y}$ where $y$ runs over the preimages of $x$ in $S'$. The pullback of $\cK_{x}$ to $\cK_{x}\otimes_{k}k'$ then takes the form $\boxtimes_{y}\cK_{y}$, whose tensor factors give the base change character sheaves $\cK_{y}\in\cCS(\bK_{y})$.

If $\Om$ is a geometric central character coming from $\Bun^{\n}_{\cZ}(\bK^{+}_{Z,S})$ for a certain level $\bK^{+}_{Z,S}$, we may similarly define the base change level $\bK^{+}_{Z,S'}$ over $k'$, so that $\Bun^{\n}_{\cZ_{k'}}(\bK^{+}_{Z,S'})=\Bun^{\n}_{\cZ}(\bK^{+}_{Z,S})\otimes_{k}k'$. Then define $\Om'\in\cCS(\Bun^{\n}_{\cZ_{k'}}(\bK^{+}_{Z,S'}))$ to be the pullback of $\Om$ via the morphism $\Bun^{\n}_{\cZ}(\bK^{+}_{Z,S})\otimes_{k}k'\to \Bun^{\n}_{\cZ}(\bK^{+}_{Z,S})$. The isomorphism $\iota_{S}$ induces $\iota_{S'}$.

When $k$ is finite, $k'/k$ a finite extension, and the restricted automorphic datum $(\omega^{\n}, K_{S}, \chi_{S})$ over $F$ comes from a geometric automorphic datum $(\Om,\bK_{S},\cK_{S})$,  we may base change it to an automorphic datum $(\omega'^{\n}, K_{S'}, \chi_{S'})$ over $F\otimes_{k}k'$ coming from the base change $(\Om', \bK_{S'}, \cK_{S'})$ of $(\Om,\bK_{S},\cK_{S})$ via the sheaf-to-function correspondence.  For example, $\omega'^{\n}$ is the composition $Z(F\otimes_{k}k')\backslash Z(\AA_{F\otimes_{k}k'})^{\n}\xrightarrow{\Nm}\AZn\xrightarrow{\onat}\Qlbar^{\times}$. However, the datum $(K_{S},\chi_{S})$ does not have a well-defined base change unless we specify a geometric automorphic datum $(\bK_{S}, \cK_{S})$ from which it comes.

Let $(\Om, \bK_{S}, \cK_{S},\iota_{S})$ be a geometric automorphic datum giving rise to a restricted automorphic datum $(\omega^{\n}, K_{S}, \chi_{S})$. We are interested in the sheaf-theoretic analog of the function space \eqref{Kchi fun}, under the sheaf-to-function correspondence.

\subsubsection{The category of automorphic sheaves} We form the group scheme
\begin{equation*}
\bM_{S}:=\Bun^{\n}_{\cZ}(\bK^{+}_{S})\twtimes{\bL_{Z,S}}\bL_{S}
\end{equation*}
which acts on $\Bun_{\cG}(\bK^{+}_{S})$ preserving each fiber of the geometric Kottwitz morphism (cf. the diagram \eqref{action Bun}). The quotient $[\Bun_{\cG}(\bK^{+}_{S})/\bM_{S}]$ should be thought of as the quotient of $\Bun_{\cG}(\bK_{S})$ by the Picard stack $\Bun^{\n}_{\cZ}(\bK_{Z,S})$, although we do not want to get into the issue of making quotients of a stack by another Picard stack. The isomorphism $\iota_{S}$ in the geometric automorphic datum gives a trivialization of the restriction of $\Om\boxtimes\cK_{S}\in\cCS(\Bun^{\n}_{\cZ}(\bK^{+}_{S})\times\bL_{S})$ to the anti-diagonally embedded $\bL_{Z,S}$. By Lemma \ref{l:descent cs}, such a trivialization gives a descent of $\Om\boxtimes\cK_{S}$ to $\cK_{S,\Om}\in\cCS(\bM_{S})$. We may then talk about the derived category $D^{b}_{(\bM_{S}, \cK_{S,\Om})}(\Bun_{\cG}(\bK^{+}_{S}))$ of $(\bM_{S}, \cK_{S, \Om})$-equivariant $\Qlbar$-complexes on $\Bun_{\calG}(\bK^{+}_{S})$, as we mentioned in \S\ref{sss:equiv sh}. 

\begin{remark} There is a subtlety in defining the category $D^{b}_{(\bM_{S}, \cK_{S,\Om})}(\Bun_{\cG}(\bK^{+}_{S}))$ because each connected component of $\Bun_{\calG}(\bK^{+}_{S})$ is not of finite type in general. To remedy, we can write $\Bun_{\calG}(\bK^{+}_{S})$ as a union of $\bM_{S}$-stable open substack $\Bun_{\xi}$ that are of finite type modulo the action of $\bM_{S}$ (depending on a parameter $\xi$ from a filtered set of indices, for example the set of Harder-Narasimhan polygons of $\cG$-torsors). The category $D^{b}_{(\bM_{S},\calK_{S,\Om})}(\Bun_{\calG}(\bK^{+}_{S}), \Qlbar)$ is defined as the colimit of $D^{b}_{(\bM_{S},\calK_{S,\Om})}(\Bun_{\xi}, \Qlbar)$ as $\Bun_{\xi}$ gets larger (with respect to the extension by zero functor). In practice, the categories $D^{b}_{(\bM_{S},\calK_{S,\Om})}(\Bun_{\xi}, \Qlbar)$ will often stabilize when $\xi$ is large enough (this is the sheaf-theoretic manifestation of the cuspidality of automorphic forms).
\end{remark}

If we shrink $\bK^{+}_{S}$ to another level $\bK^{++}_{S}$, we have the corresponding algebraic groups $\wt\bL_{S}=\prod_{x\in S}\bK_{x}/\bK^{++}_{x}$, $\bK^{++}_{Z, x}=\bK^{++}_{x}\cap L_{x}Z$ and $\wt\bL_{Z,S}=\prod_{x\in S}\bK_{Z, x}/\bK^{++}_{Z, x}$. We define similarly $\wt\bM_{S}:=\Bun_{\cZ}(\bK^{++}_{Z, S})\twtimes{\wt\bL_{Z,S}}\wt\bL_{S}$ and $\Om\boxtimes\cK_{S}$ descends to a rank one character sheaf $\wt\cK_{S,\Om}$ on $\wt\bM_{S}$. There is a projection $\wt\bM_{S}\to\bM_{S}$ whose kernel is $\bK^{+}_{S}/\bK^{++}_{S}$, and $\wt\cK_{S,\Om}$ is the pullback of $\cK_{S,\Om}$. The natural projection $\pi: \Bun_{\cG}(\bK^{++}_{S})\to\Bun_{\cG}(\bK^{+}_{S})$ is also a $\bK^{+}_{S}/\bK^{++}_{S}$-torsor. Therefore the quotient stacks $[\Bun_{\cG}(\bK^{++}_{S})/\wt\bM_{S}]$ and $[\Bun_{\cG}(\bK^{+}_{S})/\bM_{S}]$ are isomorphic, and the pullback along  $\pi$ induces an equivalence of categories
\begin{equation*}
D^{b}_{(\bM_{S}, \cK_{S,\Om})}(\Bun_{\cG}(\bK^{+}_{S}))\isom D^{b}_{(\wt\bM_{S}, \wt\cK_{S,\Om})}(\Bun_{\cG}(\bK^{++}_{S})).
\end{equation*}
We define
\begin{equation*}
D_{\cG,\Om}(\bK_{S},\cK_{S}):=\varinjlim_{\bK^{+}_{S}}D^{b}_{(\bM_{S},\calK_{S,\Om})}(\Bun_{\calG}(\bK^{+}_{S})).
\end{equation*}
where the transition functors are equivalences. Objects in the category $D_{\calG, \Om}(\bK_{S},\calK_{S})$ are called {\em $(\Om, \bK_{S}, \cK_{S}, \iota_{S})$-typical automorphic sheaves}.

Let $K^{+}_{x}:=\bK^{+}_{x}(k)$. For an object $\calF\in D_{\calG,\Om}(\bK_{S},\cK_{S})$ the corresponding function $f_{\calF, k}$ on $\Bun_{\cG}(\bK^{+}_{S})(k)$ is an eigenfunction under $\bM_{S}(k)$ with eigenvalues given by the character corresponding to $\cK_{S,\Om}$. In particular, $f_{\calF, k}$ is $(K_{S}/K^{+}_{S},\chi_{S})$-equivariant and $(\AZn,\onat)$-equivariant.  According to \eqref{Weil Bun general}, we may identify the double coset $\AG/\prod_{x\notin S}\cG(\calO_{x})\times\prod_{x\in S}K^{+}_{x}$ as a sub-groupoid of $\Bun_{\cG}(\bK^{+}_{S})(k)$. Restricting $f_{\calF, k}$ to $\AG/\prod_{x\notin S}\cG(\calO_{x})\times\prod_{x\in S}K^{+}_{x}$ gives a function in $C_{\calG,\omega}(K_{S}, \chi_{S})$. Thus we get an additive map
\begin{equation*}
\textup{Ob}D_{\calG,\Om}(\bK_{S},\cK_{S})\to C_{\calG,\omega^{\n}}(K_{S}, \chi_{S}). 
\end{equation*}

\subsubsection{Variants of the category of automorphic sheaves}\label{sss:DCa} For an extension $k'/k$, we may take the base-changed geometric automorphic datum $(\Om', \bK_{S'}, \cK_{S'},\iota_{S'})$ for $G(F\otimes_{k}k')$ and define the corresponding category of automorphic sheaves. We denote this category by $D_{\calG,\Om}(k'; \bK_{S},\cK_{S})$. Similarly, when $k'/k$ is a finite extension, we may define the function space $C_{\calG,\omega^{\n}}(k'; K_{S}, \chi_{S})$.

By the Kottwitz homomorphism \eqref{Kott}, we may decompose $\Bun_{\cG}(\bK_{S})$ into open and closed substacks $\Bun_{\cG}(\bK_{S})_{\a}$ according to the $\Frob_{k}$-orbits $\a$ in $\xch(Z\dG)_{I_{F}}$. Each $\Bun_{\cG}(\bK_{S})_{\a}$ is stable under the action of $\Bun^{\n}_{\cZ}(\bK_{Z,S})$. Let $D_{\calG,\Om}(\bK_{S},\cK_{S})_{\a}\subset D_{\calG,\Om}(\bK_{S},\cK_{S})$ be the full subcategory consisting of those sheaves supported on $\Bun_{\cG}(\bK_{S})_{\a}$. Similarly we may define $C_{\calG,\omega^{\n}}(K_{S}, \chi_{S})_{\a}$ for any $\a\in\xch(Z\dG)$ (nonzero only when $\a$ is fixed by $\Frob_{k}$).

\subsection{Rigidity of geometric automorphic data}\label{ss:rigid auto}   
In this subsection we introduce several notions of rigidity for geometric automorphic data, and give a criterion for the weak rigidity using the notion of relevant orbits.

\begin{defn}\label{def:auto strong rigid} A restricted automorphic datum $(\onat,K_{S},\chi_{S})$ is called {\em strongly rigid}, if there is a unique $(\onat,K_{S},\chi_{S})$-typical automorphic representation up to unramified twists.
\end{defn}

\begin{defn}\label{def:geom rigid} Let $k$ be a finite field and let $(\Om, \bK_{S}, \cK_{S},\iota_{S})$ be a geometric automorphic datum giving rise to a restricted automorphic datum $(\onat, K_{S}, \chi_{S})$. For any finite extension $k'/k$, let $(\om'^{\n}, K_{S'}, \chi_{S'})$ denote the automorphic datum for $G$ over $F'$ obtained from the base-changed geometric automorphic datum $(\Om', \bK_{S'}, \cK_{S'}, \iota_{S'})$ from $F$ to $F\otimes_{k}k'$. 
\begin{enumerate}
\item The geometric automorphic datum $(\Om, \bK_{S}, \cK_{S},\iota_{S})$ is called {\em strongly rigid}, if for {\em every} finite extension $k'/k$, the base-changed restricted automorphic datum $(\om'^{\n}, K_{S'}, \chi_{S'})$ is strongly rigid in the sense of Definition \ref{def:auto strong rigid}.
\item The geometric automorphic datum $(\Om, \bK_{S}, \cK_{S},\iota_{S})$ is called {\em weakly rigid}, if there is a constant $N$ such that for {\em every} finite extension $k'/k$, we have $\dim C_{\cG, \onat}(k';K_{S}, \chi_{S})_{\a}\leq N$ for any $\a\in\xch(Z\dG)_{I_{F}}$. For notation see \S\ref{sss:DCa}.
\end{enumerate}
\end{defn}

We will see several examples of strongly rigid restricted automorphic data for $G=\GL_{2}$ and $F=k(t)$ in \S\ref{ss:GL2}.

\begin{lemma} If  $(\Om, \bK_{S}, \cK_{S},\iota_{S})$ is weakly rigid, then there exists a number $N$ such that for each finite extension $k'/k$, there are at most $N$ isomorphism classes of $(\om'^{\n}, K_{S'}, \chi_{S'})$-typical automorphic representations $\pi'$ of $G(\AA_{F\otimes_{k}k'})$ up to unramified twists, and all of them are cuspidal (here, $(\om'^{\n}, K_{S'}, \chi_{S'})$ is the restricted automorphic datum for $G(F\otimes_{k}k')$ obtained from base change).
\end{lemma}
\begin{proof}
Let $G(\AA_{F})_{0}$ be the preimage of $0\in\xch(Z\dG)_{I_{F}}$ under the Kottwitz homomorphism. The dimension of $C_{\cG, \onat}(K_{S}, \chi_{S})_{0}$ is the number of double cosets in $Z(\AA_{F})^{\n}\AG_{0}/\prod_{x\notin S}\cG(\calO_{x})\times\prod_{x\in S}K_{x}$ that support functions with prescribed eigenproperties under the action of $\AZ^{\n}$ and $K_{S}$. Since $\dim C_{\cG, \onat}(K_{S}, \chi_{S})_{0}\leq N$, all functions in $C_{\cG, \onat}(K_{S}, \chi_{S})_{0}$ are supported on finitely many such double cosets, hence they are all supported on some compact-modulo-$Z(\AA_{F})^{\n}$ subset of $G(\AA_{F})_{0}$. Extending $\onat$ to a central character $\omega$, then functions in  $C_{\cG, \om}(K_{S}, \chi_{S})$ are all supported on compact-modulo-$Z(\AA_{F})$ subset of $G(\AA_{F})$. We then apply Lemma \ref{l:cusp} to conclude that all $(\om,K_{S},\chi_{S})$-typical automorphic representations are cuspidal, and there are at most $N$ of them. Since every $(\onat,K_{S}, \chi_{S})$-typical automorphic representation can be made  $(\om,K_{S}, \chi_{S})$-typical by an unramified twist, they are also cuspidal, and the number of them up to unramified twists are bounded by $\dim C_{\cG,\onat}(K_{S}, \chi_{S})_{0}\leq N$. Replacing $k$ with a finite extension $k'$, the same estimate holds. 
\end{proof}

\subsubsection{Relevant points}\label{sss:relevant}
One can talk about the stabilizer of $\Bun^{\n}_{\cZ}(\bK_{Z,S})$ at $\calE$: it is the Picard groupoid $A_{\calE}$ whose objects are pairs $(b,\beta)$ where $b\in\Bun^{\n}_{\cZ}(\bK_{Z,S})$ and $\beta\in\Isom_{\cG,\bK_{S}}(b\cdot\calE,\calE)$, with the obvious definition of composition of and isomorphisms between such pairs. Since the automorphism group of points in $\Bun^{\n}_{\cZ}(\bK_{Z,S})$ is a subgroup of the automorphism group of every point in $\Bun_{\cG}(\bK_{S})$,  it turns out that $A_{\calE}$ is a group scheme. Choosing a smaller level $\bK^{+}_{S}$ as in \S\ref{sss:Kplus},  a preimage $\calE^{+}\in\Bun_{\cG}(\bK^{+}_{S})$ of $\calE$ gives an isomorphism $i_{\calE^{+}}: A_{\calE}\cong\bM_{S,\calE^{+}}$, the latter being the stabilizer of $\calE^{+}$ under $\bM_{S}$. There is a canonical homomorphism $\bM_{S,\calE^{+}}\to\bM_{S}$ which may not be an embedding. Consider the composition
\begin{equation}\label{evK+}
\ev_{\bK^{+}_{S},\calE^{+}}: A_{\calE}\xrightarrow{i_{\calE^{+}}}\bM_{S,\calE^{+}}\to\bM_{S}.
\end{equation}
We then define $\cK_{\calE}\in\cCS(A_{\calE})$ to be the pullback of $\cK_{S,\Om}$ along $\ev_{\bK^{+}_{S},\calE^{+}}$. The object $\cK_{\calE}$ a priori depends on the choices of  $\bK^{+}_{S}$ and $\calE^{+}$, but we have

\begin{lemma}\label{l:well defined A}
\begin{enumerate}
\item The isomorphism type of $\calK_{\calE}\in\cCS(A_{\calE})$ is independent of the choice of $\bK^{+}_{S}$ and $\calE^{+}$.
\item For $\calE_{1}, \calE_{2}\in\Bun_{\cG}(\bK_{S})(\kbar)$ in the same $\Bun_{\cZ}(\bK_{Z,S})$-orbit, the pairs $(A_{\calE_{1}},\cK_{\calE_{1}})$ and $(A_{\calE_{2}},\cK_{\calE_{2}})$ are (non-canonically) isomorphic.
\end{enumerate}
\end{lemma}  
\begin{proof} 
(1) We denote the $\calK_{\calE}$ defined in \eqref{evK+} by $\cK_{\bK^{+}_{S},\calE^{+}}$ to emphasize its a priori dependence on auxiliary choices. For fixed choice of $\bK^{+}_{S}$, changing $\calE^{+}$ will change the map $\ev_{\bK^{+}_{S},\calE^{+}}$ by an inner automorphism of $\bM_{S}$. Since inner automorphisms act trivially on $\CS(\bM_{S})$, the isomorphism type of $\calK_{\bK^{+}_{S},\calE^{+}}$ is independent of  $\calE^{+}$.

Changing $\bK^{+}_{S}$ to an even smaller level $\bK^{++}_{S}$, and choosing a preimage $\calE^{++}$ of $\calE^{+}$ in $\Bun_{\cG}(\bK^{++}_{S})$, we get a surjection $\wt{\bM}_{S}\surj\bM_{S}$ (here $\wt\bM_{S}$ is the counterpart of $\bM_{S}$ for $\bK^{++}_{S}$), and a similar isomorphism $i_{\calE^{++}}:A_{\calE}\cong\wt\bM_{S,\calE^{++}}$ making a commutative diagram
\begin{equation*}
\xymatrix{A_{\calE}\ar[r]^{i_{\calE^{++}}}  \ar[dr]_{i_{\calE^{+}}} &\wt\bM_{S,\calE^{++}}\ar[r] & \wt\bM_{S}\ar[d]\\
& \bM_{S,\calE^{+}}\ar[r] & \bM_{S}}
\end{equation*}
Since $\cK_{S,\wt\Om}$ (the counterpart of $\cK_{S,\Om}$ for $\bK^{++}_{S}$)  is the pullback of $\cK_{S,\Om}$ along $\wt\bM_{S}\to\bM_{S}$, we have $\cK_{\bK^{+}_{S},\calE^{+}}=\ev^{*}_{\bK^{+}_{S}, \calE^{+}}\cK_{S,\Om}\cong\ev^{*}_{\bK^{++}_{S}, \calE^{++}}\cK_{S,\wt\Om}=\cK_{\bK^{++}_{S}, \calE^{++}}$. Therefore $\cK_{\calE}$ does not change under shrinking $\bK^{+}_{S}$ either.

(2) Fix a level $\bK^{+}_{S}$ as in \S\ref{sss:Kplus}. For $\calE_{1},\calE_{2}$ in the same $\Bun_{\cZ}(\bK_{Z, S})$-orbit, their arbitrary preimages $\calE^{+}_{1}$ and $\calE^{+}_{2}$ in $\Bun_{\cG}(\bK^{+}_{S})$ are in the same $\Bun_{\cZ}(\bK^{+}_{Z,S})\times\bL_{S}$-orbit. Therefore their stabilizers $\bM_{S,\calE_{1}}$ and $\bM_{S,\calE_{2}}$ are conjugate to each other in $\bM_{S}$. We then argue as in the first part of the proof of (1).
\end{proof}

\begin{defn}\label{def:relevant} A point $\calE\in\Bun_{\cG}(\bK_{S})(\kbar)$ is called {\em $(\Om, \cK_{S})$-relevant} if the restriction of  $\cK_{\calE}$ to the neutral component of $A_{\calE}$ is isomorphic to the constant sheaf. Otherwise $\cE$ is called {\em $(\Om, \cK_{S})$-irrelevant}.
\end{defn}

By Lemma \ref{l:well defined A}(2), a $\Bun_{\cZ}(\bK_{Z,S})$-orbit on $\Bun_{\cG}(\bK_{S})(\kbar)$ either consists entirely of $(\Om, \cK_{S})$-relevant points or entirely of $(\Om, \cK_{S})$-irrelevant points. Therefore we may talk about $(\Om, \cK_{S})$-relevant and irrelevant $\Bun_{\cZ}(\bK_{Z,S})$-orbits.

In practice we introduce a weaker relevance condition by ignoring the central character. Let $\calE\in\Bun_{\cG}(\bK_{S})(\kbar)$. Then we have an affine algebraic group $\Aut_{\cG, \bK_{S}}(\calE)$ over $\kbar$, the automorphism group of $\calE$ preserving the $\bK_{x}$-level structures. For each $x\in S(\kbar)$, the restriction $\calE|_{\Spec\calO_{x}}$ can be viewed as a $\bK_{x}$-torsor. Let $\Aut_{x}(\calE)$ be the pro-algebraic group of automorphisms of $\calE|_{\Spec\calO_{x}}$ as a $\bK_{x}$-torsor. Choosing a trivialization of $\calE|_{\Spec\calO_{x}}$ as a $\bK_{x}$-torsor, we get an isomorphism $\ep_{x}:\Aut_{x}(\calE)\cong\bK_{x}$; changing the trivialization changes the isomorphism $\ep_{x}$ by an inner automorphism of $\bK_{x}$. For each $x\in S(\kbar)$, restricting an automorphism of $\calE$ to $\Spec\calO_{x}$ gives a homomorphism of pro-algebraic groups
\begin{equation*}
\ev_{S,\calE}:\Aut_{\cG, \bK_{S}}(\calE)\to \prod_{x\in S(\kbar)} \Aut_{x}(\calE)\cong\bK_{S}\otimes_{k}\kbar
\end{equation*}
which is well-defined up to conjugacy. We can pullback the character sheaf $\cK_{S}$ to $\Aut(\calE)$ via $\ev_{S,\calE}$ and get a rank one character sheaf
\begin{equation*}
\cK_{S,\calE}:=\ev^{*}_{S,\calE}\cK_{S}\in\cCS(\Aut_{\cG,\bK_{S}}(\calE)).
\end{equation*}
This is well-defined because inner automorphisms act trivially on the isomorphism classes of rank one character sheaves. 

\begin{defn} A point $\calE\in \Bun_{\cG}(\bK_{S})(\kbar)$ is called {\em $\cK_{S}$-relevant} if the restriction of $\cK_{S,\calE}$ to the neutral component of $\Aut_{\cG,\bK_{S}}(\calE)$ is isomorphic to the constant sheaf.  Otherwise the point $\calE$ is called {\em $\cK_{S}$-irrelevant}. Again $\cK_{S}$-relevance is a property of a $\Bun_{\cZ}(\bK_{Z,S})$-orbit.
\end{defn}

\begin{lemma}\label{l:two rel} If a point $\calE\in \Bun_{\cG}(\bK_{S})(\kbar)$ is $(\Om,\cK_{S})$-relevant, then it is also $\cK_{S}$-relevant. Moreover, if $\bA_{Z,S}$ is connected, then a point $\calE\in \Bun_{\cG}(\bK_{S})(\kbar)$ is $\cK_{S}$-relevant if and only if it is $(\Om,\cK_{S})$-relevant.
\end{lemma}
\begin{proof}
If $\calE^{+}\in\Bun_{\cG}(\bK^{+}_{S})$ is a lifting of $\calE$, then $\bL_{S,\calE^{+}}$ is canonically identified with $\Aut_{\cG,\bK_{S}}(\calE)$, hence a homomorphism $\l:\Aut_{\cG,\bK_{S}}(\calE)\cong\bL_{S,\calE^{+}}\to\bM_{S,\calE^{+}}\cong A_{\cE}$. We have $\l^{*}\cK_{\cE}\cong \cK_{S,\cE}$, hence $(\Om,\cK_{S})$-relevance implies $\cK_{S}$-relevance.

Next we assume $\bA_{Z,S}$ is connected and deduce that $\cK_{S}$-relevance implies $(\Om,\cK_{S})$-relevance. Note that $\bA_{Z,S}\subset\bL_{Z,S}\subset\bL_{S}$ is the kernel of the homomorphism $\bL_{S}\to\bM_{S}$, and is always contained in $\Aut_{\cG, \bK_{S}}(\calE)$. Therefore, we have an exact sequence
\begin{equation*}
1\to\bA_{Z,S}\to\Aut_{\cG,\bK_{S}}(\calE)\xrightarrow{\l}A_{\calE}.
\end{equation*}
Since $\bA_{Z,S}$ is connected, the fact that $\cK_{S,\cE}$ is trivial on $\Aut^{\circ}_{\cG,\bK_{S}}(\calE)$ implies that $\cK_{\cE}$ is trivial when restricted to $\Aut^{\circ}_{\cG,\bK_{S}}(\calE)/\bA_{Z,S}$. To show that $\cK_{\cE}|_{A^{\circ}_{\cE}}$ is trivial, it suffices to show that $\Aut^{\circ}_{\cG,\bK_{S}}(\calE)/\bA_{Z,S}=A^{\circ}_{\calE}$, or to show that $\coker(\l)$ is finite. 

Since $\coker(\bL_{S}\to\bM_{S})$ is the coarse moduli space of $\Bun^{\n}_{\cZ}(\bK_{Z,S})$, $\coker(\l)$ is the coarse moduli space of the stabilizer of $\calE$ under $\Bun^{\n}_{\cZ}(\bK_{Z,S})$. Recall $D$ is the maximal torus quotient of $G$. Let $\bK_{D, S}=\prod_{x\in S}\bK_{D,x}$ where $\bK_{D,x}$ is the  image of $\bK_{S}$ in $L_{x}D$. By our assumption, $\bK_{x}=\bK_{Z,x}\cdot(\bK_{x}\cap L_{x}G^{\der})$, hence $\bK_{D,x}$ is also the image of $\bK_{Z,x}$ under $L_{x}Z\to L_{x}D$. We have maps $\Bun_{\cZ}(\bK_{Z,S})\to\Bun_{\cG}(\bK_{S})\to\Bun_{\calD}(\bK_{D,S})$. For $b\in \Bun_{\cZ}(\bK_{Z,S})$ and $\cE\in\Bun_{\cG}(\bK_{S})$, we denote their images in $\Bun_{\calD}(\bK_{D,S})$ by $b_{D}$ and $\cE_{D}$. Suppose $b\cdot\calE\cong\calE$, then  $b_{D}\calE_{D}\cong\calE_{D}$ in $\Bun_{\calD}(\bK_{D,S})$, which implies that $b$ lies in the kernel of $\zeta: \Bun_{\cZ}(\bK_{Z, S})\to\Bun_{\calD}(\bK_{D,S})$. Since $Z\to D$ is an isogeny, we have a factorization $[n]: Z\to D\to Z$ for some positive integer $n$. This induces a factorization $[n]: \Bun_{\cZ}(\bK_{Z, S})\xrightarrow{\zeta}\Bun_{\calD}(\bK_{D,S})\to\Bun_{\cZ}(\bK_{Z, S})$. Therefore $\ker(\zeta)^{c}\subset\ker([n])^{c}$ (where $(-)^{c}$ means taking the coarse moduli space). Hence $\coker(\l)\subset \ker([n])^{c}$, which is finite. This completes the proof.
\end{proof}

\begin{lemma}\label{l:irr}
\begin{enumerate}
\item Let $\calE\in \Bun_{\calG}(\bK_{S})(\kbar)$ be a $(\Om,\cK_{S})$-irrelevant point. Then for any object $\calF\in D_{\calG,\Om}(\bK_{S},\calK_{S})$, $i^{*}_{\calE}\calF=0$ and $i^{!}_{\calE}\calF=0$. Here $i_{\calE}$ denotes the inclusion map of the fiber of $\calE$ in $\Bun_{\calG}(\bK^{+}_{S})$.
\item Let $[g]\in\AG/\prod_{x\notin S}\cG(\calO_{x})\times\prod_{x\in S}K_{x}$ be a double coset whose corresponding point $\calE\in \Bun_{\calG}(\bK_{S})(k)$ is $\cK_{S}$-irrelevant. Then any function $f\in C_{\calG}(K_{S},\chi_{S})$ vanishes on the double coset $G(F)g(\prod_{x\notin S}\cG(\calO_{x})\times\prod_{x\in S}K_{x})$. Similar statement holds when $k$ is replaced by a finite extension $k'$.
\end{enumerate}
\end{lemma}
\begin{proof}
(1) We work over $\kbar$ without changing notation. Choose a smaller level $\bK^{+}_{S}$ as in \S\ref{sss:Kplus} and a preimage $\calE^{+}\in\Bun_{\cG}(\bK^{+}_{S})(\kbar)$ of $\calE$. Let $i_{\cE^{+}}:\Spec \kbar\to\Bun_{\cG}(\bK^{+}_{S})$ be the inclusion of $\calE^{+}$. It suffices to prove that $i^{*}_{\cE^{+}}\calF=0$ and $i^{!}_{\cE^{+}}=0$. Since $D_{\calG,\Om}(\bK_{S},\calK_{S})\cong D_{(\bM_{S},\cK_{S,\Om})}(\Bun_{\cG}(\bK^{+}_{S}))$, the stalks $i^{*}_{\calE^{+}}\calF$ and $i^{!}_{\calE^{+}}\calF$ are $\bM_{S,\calE^{+}}$-equivariant complexes over $\Spec\kbar$. By Lemma \ref{l:Ac}, this category is zero if $\calK_{S,\Om}|_{\bM^{\circ}_{S,\calE^{+}}}$ is not the constant sheaf, i.e., if $\calE$ is a $(\Om,\cK_{S})$-irrelevant point of $\Bun_{\cG}(\bK_{S})(\kbar)$.

(2) Let $A:=\Aut_{\cG,\bK_{S}}(\calE)$. Then $A(k)$ is the automorphism group of the double coset $[g]$ if we view $\AG/\prod_{x\notin S}\cG(\calO_{x})\times\prod_{x\in S}K_{x}$ as a groupoid. There is a function $f\in C_{\calG}(K_{S},\chi_{S})$ nonzero on the double coset represented by $\calE$ if and only if the character $\chi_{S, [g]}:A(k)\to \prod_{x\in S}K_{x}\xrightarrow{\prod\chi_{x}}\Qlbar^{\times}$ is trivial. Under the sheaf-to-function correspondence, $\chi_{S,[g]}$ corresponds to the rank one character sheaf $\calK_{S,\calE}$ on $A$. Since $\calE$ is irrelevant, $\calK_{S,\calE}|_{A^{\circ}}$ is nontrivial.  By Theorem \ref{th:char sh}, the character $\chi_{S,[g]}|_{A^{\circ}(k)}$ is also nontrivial, and therefore $f$ must vanish on $[g]$.
\end{proof}

\begin{theorem}\label{th:wrigid} Let $(\Om, \bK_{S},\cK_{S},\iota_{S})$ be a geometric automorphic datum. Consider the following statements.
\begin{enumerate}
\item\label{Kfin} $\Bun_{\cG}(\bK_{S})$ has only finitely many $\cK_{S}$-relevant $\Bun_{\cZ}(\bK_{Z,S})$-orbits over $\kbar$. 
\item\label{Omfin} $\Bun_{\cG}(\bK_{S})$ has only finitely many $(\Om, \cK_{S})$-relevant $\Bun_{\cZ}(\bK_{Z,S})$-orbits over $\kbar$. 
\item\label{Pervfin} For each $\a\in\xch(Z\dG)_{I_{F}}$, $D_{\cG,\Om}(\kbar; \bK_{S}, \cK_{S})_{\a}$ (see \S\ref{sss:DCa} for notation) contains only finitely many irreducible perverse sheaves up to isomorphism.
\item\label{weakrig} The geometric automorphic datum $(\Om, \bK_{S},\cK_{S},\iota_{S})$ is weakly rigid.
\end{enumerate}
Then $\eqref{Kfin}\Rightarrow\eqref{weakrig}\Rightarrow\eqref{Omfin}\Leftrightarrow\eqref{Pervfin}$. When $\bA_{Z,S}$ is connected, all the statements are equivalent.
\end{theorem}
\begin{proof} \eqref{Omfin}$\Rightarrow$\eqref{Pervfin} We base change all stacks to $\kbar$ without changing notation. Clearly \eqref{Omfin} is equivalent to that $\Bun_{\cG}(\bK^{+}_{S})_{\a}$ has only finitely many $(\Om, \cK_{S})$-relevant $\bM_{S}$-orbits over $\kbar$ for each $\a$. If this is the case, let $\calE^{+}_{1},\cdots, \calE^{+}_{N}$ be representatives of the relevant $\bM_{S}$-orbits in $\Bun_{\cG}(\bK^{+}_{S})(\kbar)_{\a}$, whose image in $\Bun_{\cG}(\bK_{S})$ we denote by $\calE_{1},\cdots, \calE_{N}$. By Lemma \ref{l:irr}, any object $\calF\in D_{\cG,\Om}(\bK_{S},\cK_{S})_{\a}\cong D^{b}_{(\bM_{S},\cK_{S,\Om})}(\Bun_{\cG}(\bK^{+}_{S})_{\a})$ can only have nonzero stalk and costalk along the $\bM_{S}$-orbits of $\calE^{+}_{i}$. If such an $\calF$ is an irreducible $(\bM_{S}, \cK_{S,\Om})$-equivariant perverse sheaf, it is the middle extension of some $(\bM_{S}, \cK_{S,\Om})$-equivariant local system on the $\bM_{S}$-orbit $O_{i}:=\bM_{S}\cdot\calE^{+}_{i}$. By Lemma \ref{l:Ac}, the category of $(\bM_{S}, \cK_{S,\Om})$-equivariant local system on $O_{i}\cong\bM_{S}/A_{\calE_{i}}$ is equivalent to the category of twisted representations of $\pi_{0}(A_{\calE_{i}})$ under some cocycle in $\cohog{2}{\pi_{0}(A_{\calE_{i}}),\Qlbar^{\times}}$ determined by $\cK_{\calE_{i}}$, which only has finitely many irreducible objects. Therefore there are only finitely many irreducible perverse sheaves in   $D_{\cG, \Om}(\kbar; \bK_{S}, \cK_{S})$ up to isomorphism.

\eqref{Pervfin}$\Rightarrow$\eqref{Omfin} For each $(\Om,\cK_{S})$-relevant $\bM_{S}$-orbit $O\cong\bM_{S}/A_{\calE}\subset\Bun_{\cG}(\bK^{+}_{S})_{\a}$, the category $\Loc_{(\bM_{S},\cK_{S, \Om})}(O)\cong\Rep_{\xi}(\pi_{0}(A_{\calE}))$ for some cocycle $\xi\in\cohog{2}{\pi_{0}(A_{\calE}),\Qlbar^{\times}}$. One can find a central extension $1\to C\to \tilA_{\calE}\to A_{\calE}\to1$ and a character $\chi_{C}:C\to\Qlbar^{\times}$ such that $\Rep_{\xi}(\pi_{0}(A_{\calE}))$ is equivalent to the category of finite-dimensional representations of $\tilA_{\calE}$ on which $C$ acts through $\chi_{C}$. For the argument see \S\ref{sss:centext}. In particular, $\Rep_{\xi}(\pi_{0}(A_{\calE}))$  is never the zero category. Take any nonzero irreducible object therein, we get an irreducible $(\bM_{S},\cK_{S,\Om})$-equivariant local system $\calL$ on $O$, and the middle extension of a suitable shift of $\calL$ is an irreducible perverse sheaf $\calF$. For different $(\Om,\cK_{S})$-relevant $\bM_{S}$-orbits we obtain different irreducible perverse sheaves in this way. Therefore, if there are only finitely many irreducible perverse sheaves in $D_{\cG,\Om}(\kbar; \bK_{S}, \cK_{S})_{\a}$, $\Bun_{\cG}(\bK^{+}_{S})_{\a}$ can only have finitely many $(\Om,\cK_{S})$-relevant $\bM_{S}$-orbits, i.e., \eqref{Omfin} holds.

\eqref{Kfin}$\Rightarrow$\eqref{weakrig} We will show that  $\dim C_{\calG,\onat}(K_{S},\chi_{S})_{\a}$ is bounded by a constant independent of the base field $k$ and $\a\in\xch(Z\dG)_{I_{F}}$. Choose the level $\bK^{+}_{S}$ small enough as in \S\ref{sss:Kplus} so that $\bK^{+}_{S}$ is pro-unipotent, hence $K_{S}/K^{+}_{S}=L_{S}:=\bK_{S}(k)$. Let $G(\AA_{F})_{\a}$ be the preimage of $\a$ under the Kottwitz homomorphism $\kappa_{G,F}$. Then $\dim C_{\calG,\onat}(K_{S},\chi_{S})_{\a}$ is the number of $\AZ^{\n}\times L_{S}$-orbits on $\AG_{\a}/\prod_{x\in S}K^{+}_{x}\times\prod_{x\notin S}\cG(\calO_{x})$ that can support $(\AZ^{\n}, \onat)$ and $(L_{S},\chi_{S})$-eigenfunctions. By Lemma \ref{l:irr}(2), only those cosets that are $\cK_{S}$-relevant can support such functions. By \eqref{Kfin}, there are only a finite number of $\Bun^{\n}_{\cZ}(\bK_{Z,S})$-orbits in $\Bun_{\cG}(\bK_{S})_{\a}$ over $\kbar$ (any this bound can be chosen to be independent of $\a$). Take such an orbit and denote its preimage in $\Bun_{\cG}(\bK^{+}_{S})_{\a}$ by $O$. Suppose $O$ contains a $k$-point (otherwise it does not contribute to the function space $C_{\calG,\onat}(K_{S},\chi_{S})_{\a}$). We shall bound the number of $\AZ^{\n}\times L_{S}$-orbits on $O(k)$ by a number that is unchanged upon passing $k$ to a finite extension.

Fix $\calE^{+}\in O(k)$, and let $S_{\calE^{+}}$ be the stabilizer of $\calE^{+}$ under $\Bun^{\n}_{\cZ}(\bK^{+}_{Z,S})\times\bL_{S}$. Let $P:=\Bun^{\n}_{\cZ}(\bK^{+}_{Z,S})(k)$. Since $O$ is a single orbit under $\Bun^{\n}_{\cZ}(\bK^{+}_{Z,S})\times\bL_{S}$ over $\kbar$, the number $P\times L_{S}$-orbits in $O(k)$ is bounded by $\#\cohog{1}{k,S_{\calE^{+}}}\leq\#\pi_{0}(S_{\calE^{+}})$.

Note that $\Bun_{\cG}(\bK^{+}_{S})_{\a}(k)$ is decomposed according to classes $\zeta\in\upH^{1}_{\cG,\bK^{+}_{S}}(F,G)$ as in \eqref{Weil Bun general}, and only the trivial class corresponds to $\AG_{\a}/\prod_{x\in S}K^{+}_{x}\times\prod_{x\notin S}\cG(\calO_{x})$, which is all we care about. Let $O(k)^{\heartsuit}$ be the subset of $O(k)$ that belongs to $\AG_{\a}/\prod_{x\in S}K^{+}_{x}\times\prod_{x\notin S}\cG(\calO_{x})$. We also have a map $P\to \upH^{1}_{\cZ, \bK^{+}_{Z,S}}(F,Z)$, the latter is a subgroup of $\cohog{1}{F,Z}$ and it acts on $\upH^{1}_{\cG,\bK^{+}_{S}}(F,G)$ compatibly with the action of $P$ on $\Bun_{\cG}(\bK^{+}_{S})_{\a}(k)$. Let $P_{0}\subset P$ be the subset which maps to $\ker_{S,\bK^{+}_{S}}(Z,G):=\ker(\upH^{1}_{\cZ, \bK^{+}_{Z,S}}(F,Z)\to \upH^{1}_{\cG,\bK^{+}_{S}}(F,G))$. Then each $P\times L_{S}$-orbit on $O(k)$ intersects $O(k)^{\heartsuit}$ in a single $P_{0}\times L_{S}$-orbit. Therefore the number of $P_{0}\times L_{S}$-orbits on $O(k)^{\heartsuit}$ is also bounded by $\#\pi_{0}(S_{\calE^{+}})$. 

Finally we bound the number of $\AZ^{\n}\times L_{S}$-orbits on $O(k)^{\heartsuit}$. For this we only need to bound the cokernel of $\AZ^{\n}\to P_{0}$, which is a subgroup of $\ker_{S,\bK^{+}_{S}}(Z, G)$. Therefore it suffices to bound $\#\ker_{S,\bK^{+}_{S}}(Z,G)$. For a group $H$ over $F$ we denote $\upH^{1}_{\Sram}(F,H)$ to be the kernel of $\upH^{1}(F,H)\mapsto\prod_{x\in |U|}\cohog{1}{F^{\ur}_{x}, H}$. Then $\ker_{S,\bK^{+}_{S}}(Z,G)\subset \ker(\upH^{1}_{\Sram}(F,Z)\to\upH^{1}_{\Sram}(F,G))=:\ker_{S}(Z,G)$. It suffices to bound the size of $\ker_{S}(Z,G)$. 

Over the finite Galois extension $F'/F$, $Z$ become constant diagonalizable groups (i.e., $\Gal(F^{s}/F')$ acts trivially on their character groups), and we may write $Z\otimes_{F}F'$ as $\Gm^{r}\times Z^{\fin}$ for some finite constant diagonalizable group over $F'$. Let $\Gamma=\Gal(F'/F)$, then we have exact sequences
\begin{equation*}
\xymatrix{1\ar[r] & \cohog{1}{\Gamma, Z(F')}\ar[r]\ar[d] & \cohog{1}{F,Z}\ar[r]\ar[d] & \cohog{1}{F',Z}^{\Gamma}\ar[d]\\
1\ar[r] & \cohog{1}{\Gamma, G(F')}\ar[r] & \cohog{1}{F,G}\ar[r] & \cohog{1}{F',G}^{\Gamma}}
\end{equation*}
The image of $\ker_{S}(Z,G)$ in $\cohog{1}{F',Z}^{\Gamma}$ certainly lies in $\upH^{1}_{\Sram}(F',Z)^{\Gamma}$.  Since $Z\otimes_{F}F'\cong\Gm^{r}\times Z^{\fin}$, $\upH^{1}_{\Sram}(F',Z)^{\Gamma}\subset\upH^{1}_{\Sram}(F',Z^{\fin})$. On the other hand, the kernel of  $\ker_{S}(Z,G)\to\cohog{1}{F',Z}^{\Gamma}$ lies in the kernel of $\cohog{1}{\Gamma, Z(F')}\to \cohog{1}{\Gamma, G(F')}$, which is surjected by $\cohog{1}{\Gamma, C(F')}$ where $C=\ker(Z\to G)$ is a finite diagonalizable group over $F$. The conclusion is that
\begin{equation*}
\#\ker_{S}(Z,G)\leq\#\upH^{1}_{\Sram}(F',Z^{\fin})^{\Gamma}\cdot\#\cohog{1}{\Gamma, C(F')}.
\end{equation*}
Both groups on the right side are finite and their cardinalities are bounded independent of extensions of $k$. This gives a bound on $\#\ker_{S,\bK^{+}_{S}}(Z, G)\leq\#\ker_{S}(Z,G)$, and hence on the number of $\AZ^{\n}\times L_{S}$-orbits on $O^{\heartsuit}(k)$ for any $\cK_{S}$-relevant $\Bun^{\n}_{\cZ}(\bK^{+}_{Z,S})$-orbit on $\Bun_{\cG}(\bK^{+}_{S})_{\a}$ for any $\a$. Therefore the dimension of $C_{\cG,\onat}(K_{S},\chi_{S})_{\a}$ is bounded independent of $\a$ and the field extension $k'/k$.

\eqref{weakrig}$\Rightarrow$\eqref{Omfin} Suppose $C_{\cG,\onat}(k';K_{S},\chi_{S})_{\a}\leq N$ for all finite extensions $k'/k$, we shall show that $\Bun_{\cG}(\bK^{+}_{S})$ contains at most $N$ $(\Om,\cK_{S})$-relevant $\bM_{S}$-orbits over $\kbar$. If not, let $O_{1},\cdots, O_{N+1}$ be distinct $(\Om,\cK_{S})$-relevant $\bM_{S}$-orbits, and let $\cE^{+}_{i}\in O_{i}(\kbar)$. Replacing $k$ by a large enough finite extension, we may assume that  all $\cE^{+}_{i}$ are defined over $k$. However, recall from \eqref{Weil Bun general} that not all $k$-points of $\Bun_{\cG}(\bK^{+}_{S})$ correspond to points in the double coset $\AG/\prod_{x\in S}K^{+}_{x}\times\prod_{x\notin S}\cG(\calO_{x})$. For each point $\cE^{+}_{i}$ there is a class $\zeta_{i}\in\cohog{1}{F,G}$ recording the isomorphism class of $\cE^{+}_{i}$ at the generic point of $X$. Since $\cohog{1}{F\otimes_{k}\kbar,G}$ vanishes, by enlarging $k$ to a finite extension, we may kill the classes $\zeta_{i}$, and hence make sure that $\cE^{+}_{i}$ do  represent certain double cosets $[g_{i}]\in\AG/\prod_{x\in S}K^{+}_{x}\times\prod_{x\notin S}\cG(\calO_{x})$ (which are necessarily distinct). By the argument of \eqref{Pervfin}$\Rightarrow$\eqref{Omfin}, there is a nonzero local system $\calL_{i}\in \Loc_{(\bM_{S},\cK_{S,\Om})}(O_{i})$ with support on the closure of the orbit $\bM_{S}\cdot \calE^{+}_{i}$ (enlarge $k$ if necessarily to make sure that $\calL_{i}$ is also defined over $k$). Then for large enough finite extensions $k'/k$, the function $f_{\calL_{i}, k'}$ attached to $\calL_{i}$ via the sheaf-to-function correspondence lies in $C_{\cG,\onat}(k';K_{S},\chi_{S})_{\a}$, and is nonzero at $[g_{i}]$ and vanishing outside the $\bM_{S}(k')$-orbit of $[g_{i}]$. Then $\{f_{\calL_{i}, k'}\}_{i=1,\cdots, N+1}$ gives at least $(N+1)$ linearly independent functions in  $C_{\cG,\onat}(k';K_{S},\chi_{S})_{\a}$, contradicting our original assumption.

Finally, when $\bA_{Z,S}$ is connected, we know from Lemma \ref{l:two rel} that \eqref{Kfin}$\Rightarrow$\eqref{Omfin}, hence all four statements are equivalent to each other.
\end{proof}

\subsubsection{Numerical condition for weak rigidity}\label{sss:numrig}
Let $(\Om,\bK_{S}, \cK_{S}, \iota_{S})$ be a geometric automorphic datum. For each $x\in S$ we define an integer $d(\bK^{\ad}_{x})$ as follows. Let $\bK^{\ad}_{x}=\bK_{x}/\bK_{Z,x}\subset L_{x}G^{\ad}$. The Lie algebra $\frk^{\ad}_{x}$ of $\bK^{\ad}_{x}$ is an $\calO_{x}$-lattice in $\frg^{\ad}(F_{x})$. On the other hand we have the special parahoric subgroup $L^{+}_{x}\cG^{\ad}$ whose Lie algebra is another lattice $\frg^{\ad}(\calO_{x})\subset \frg^{\ad}(F_{x})$. It makes sense to consider the relative dimension $\dim_{k}(\frg^{\ad}(\calO_{x}):\frk^{\ad}_{x})$ between  the two lattices $\frg^{\ad}(\calO_{x})$ and $\frk^{\ad}_{x}$: take another lattice $\L$ contained in both of them and define $\dim_{k}(\frg^{\ad}(\calO_{x}):\frk^{\ad}_{x})$ as $\dim_{k}(\frg^{\ad}(\calO_{x})/\L)-\dim_{k}(\frk^{\ad}_{x}/\L)$. Note that the inertia group $I_{x}$ acts on $\GG^{\ad}$ (as restricted from the homomorphism $\Gamma_{F}\xrightarrow{\theta}\Aut^{\dagger}(\GG)\to\Aut^{\dagger}(\GG^{\ad})$). Define
\begin{equation*}
d(\bK^{\ad}_{x}):=\dim\GG^{\ad}-\dim\GG^{\ad,I_{x}}+2\dim_{k}(\frg^{\ad}(\calO_{x}):\frk^{\ad}_{x}).
\end{equation*}

\begin{lemma}
\begin{enumerate}
\item For any choice of $\bK^{+}_{S}$ as in \S\ref{sss:Kplus}, the stack $[\bM_{S}\backslash \Bun_{\cG}(\bK^{+}_{S})_{\a}]$ is smooth of pure dimension $(g_{X}-1)\dim\GG^{\ad}+\frac{1}{2}\sum_{x\in S}d(\bK^{\ad}_{x})$.
\item Suppose there exists a point on $\Bun_{\cG}(\bK_{S})$ whose  stabilizer under $\Bun_{\cZ}(\bK_{Z,S})$ is finite. If $(\Om,\bK_{S}, \cK_{S}, \iota_{S})$ is weakly rigid, then
\begin{equation}\label{cond K}
\frac{1}{2}\sum_{x\in S}d(\bK^{\ad}_{x})=(1-g_{X})\dim\GG^{\ad}.
\end{equation}
\end{enumerate}
\end{lemma}
\begin{proof}
(1) Since $\dim[\bM_{S}\backslash \Bun_{\cG}(\bK^{+}_{S})_{\a}]=\dim\Bun_{\cG}(\bK^{+}_{S})_{\a}-\dim\bM_{S}=\dim\Bun_{\cG}(\bK_{S})-\dim\Bun_{\cZ}(\bK_{Z,S})$, we shall compute the dimension of $\Bun_{\cG}(\bK_{S})$ and $\Bun_{\cZ}(\bK_{Z,S})$ separately. One can define $d(\bK_{x})$ and $d(\bK_{Z,x})$ in a way analogous to $d(\bK^{\ad}_{x})$, with the group $G^{\ad}$ changed to $G$ or $Z$ (correspondingly $\GG^{\ad}$ changes to $\GG$ or $\ZZ\GG$). Clearly $d(\bK_{x})-d(\bK_{Z,x})=d(\bK^{\ad}_{x})$. We shall prove that
\begin{eqnarray}
\label{dimBunG}\Bun_{\cG}(\bK_{S})=(g_{X}-1)\dim\GG+\frac{1}{2}\sum_{x\in S}d(\bK_{x}),\\
\label{dimBunZ}\Bun_{\cZ}(\bK_{Z,S})=(g_{X}-1)\dim\ZZ\GG+\frac{1}{2}\sum_{x\in S}d(\bK_{Z,x}).
\end{eqnarray}
Taking the difference of these equations gives (1).

It is well-know that $\Bun_{\cG}(\bK_{S})$ is smooth of dimension $-\chi(X,\Lie\cG_{\bK_{S}})$, where $\Lie\cG_{\bK_{S}}$ is the coherent sheaf on $X$ obtained from $\Lie\cG$ by changing its local sections over $\Spec\calO_{x}$ from $\frg(\calO_{x})$ to $\frk_{x}=\Lie\bK_{x}$. A local calculation shows that $\deg(\Lie\cG)=-\frac{1}{2}\sum_{x\in S}\dim\GG/\dim\GG^{I_{x}}$. Moreover, $\deg\Lie\cG-\deg\Lie\cG_{\bK_{S}}=\sum_{x\in S}\dim_{k}(\frg(\calO_{x}):\frk_{x})$. Therefore $\deg\Lie\cG_{\bK_{S}}=-\sum_{x\in S}(\frac{1}{2}\dim\GG/\dim\GG^{I_{x}}+\dim_{k}(\frg(\calO_{x}):\frk_{x})=-\frac{1}{2}\sum_{x\in S}d(\bK_{x})$. By Riemann-Roch theorem, $-\chi(X,\Lie\cG_{\bK_{S}})=-\deg\Lie\cG_{\bK_{S}}+(g_{X}-1)\rank(\Lie\cG_{\bK_{S}})=\frac{1}{2}\sum_{x\in S}d(\bK_{x})+(g_{X}-1)\dim\GG$. This proves \eqref{dimBunG}. The proof of \eqref{dimBunZ} is similar. This finishes the proof of (1).

(2) The stabilizer of a point $\calE\in\Bun_{\cG}(\bK_{S})$ under $\Bun_{\cZ}(\bK_{Z,S})$ is the same as the stabilizer of any of its preimage $\cE^{+}\in\Bun_{\cG}(\bK^{+}_{S})$ under $\bM_{S}$. Let $\calU\subset\Bun_{\cG}(\bK^{+}_{S})_{\a}$ be the locus where the stabilizers under $\bM_{S}$ are finite. The given condition ensures that $\calU$ is nonempty for some $\a$. Also $\calU$ is stable under $\bM_{S}$ and consists of $(\Om,\cK_{S})$-relevant points (because the condition for relevance only has to do with the neutral component of stabilizers). Moreover $\calU$ is open because the dimension of the stabilizers under $\bM_{S}$ is upper semicontinuous. By Theorem \ref{th:wrigid}, $(\Om,\bK_{S}, \cK_{S}, \iota_{S})$ being weakly rigid implies that $\calU$ consists of finitely many $\bM_{S}$-orbits. Therefore $\dim[\bM_{S}\backslash \Bun_{\cG}(\bK^{+}_{S})_{\a}]=\dim[\bM_{S}\backslash \calU]=0$, and we get \eqref{cond K} from (1).
\end{proof}

\subsection{Rigid automorphic data for $\GL_{2}$}\label{ss:GL2} 
Let $k$ be a finite field.
Consider the split group $G=\GL_{2}$ over $F=k(t)$, the function field of $X=\PP^{1}_{k}$. Since $G$ is split, we take the constant group scheme over $X$ as the integral model $\cG$, and denote $\cG$ also by $G$.

\subsubsection{$S$ consists of three points}\label{sss:GL2S3} Let $S=\{0,1,\infty\}\subset|\PP^{1}_{k}|=|X|$. For $x\in S$,  let $\bI_{x}$ be the Iwahori subgroup of $L_{x}G$ such that
\begin{equation*}
\bI_{x}(k_{x})=\left\{\left(\begin{array}{cc} a & b \\ c & d\end{array}\right)|c\in\frm_{x}, a,b,d\in\calO_{x}\right\}.
\end{equation*}
The reductive quotient of $\bI_{x}$ is $\Gm^{2}$ given by sending $\left(\begin{array}{cc} a & b \\ c & d\end{array}\right)$ to $a\mod\frm_{x}$ and $d\mod\frm_{x}$. For each $x\in S$, we choose a character $\chi_{x}=(\chi_{x}^{(1)},\chi_{x}^{(2)}):k^{\times}\times k^{\times}\to \Qlbar^{\times}$, which corresponds to a rank one character sheaf $\cK_{x}=\cK^{(1)}_{x}\boxtimes\cK^{(2)}_{x}$ on $\Gm^{2}$. We may view $\cK_{x}$ as a rank one character sheaf on $\bI_{x}$ by pullback along $\bI_{x}\surj \Gm^{2}$. We temporarily use $I_{x}$ to denote $\bI_{x}(k_{x})$ (not to be confused with the inertia group $I_{x}$ in \S\ref{ss:X}). Then $(I_{S}, \chi_{S})$ comes from $(\bI_{S}, \cK_{S})$ via the sheaf-to-function correspondence.

\begin{prop}\label{p:GL2S3}
\begin{enumerate}
\item There is an $(I_{S},\chi_{S})$-typical automorphic representation of $G(\AA_{F})$ only if
\begin{equation}\label{det GL2}
\prod_{x\in S}\chi_{x}^{(1)}\chi_{x}^{(2)}=1.
\end{equation}
In this case there is unique way (up to isomorphism) to extend $(\bI_{S}, \cK_{S})$ into a geometric automorphic datum $(\Om, \bI_{S}, \cK_{S}, \iota_{S})$.
\item If, moreover, for any map $\ep:S\to\{1,2\} $ we have 
\begin{equation*}
\prod_{x\in S}\chi^{(\ep(x))}_{x}\neq1 (\textup{ the trivial character}),
\end{equation*}
then the geometric automorphic datum $(\Om, \bI_{S}, \cK_{S}, \iota_{S})$ is strongly rigid.
\end{enumerate}
\end{prop}
\begin{proof} 
(1) Let $\pi$ be an $(I_{S},\chi_{S})$-typical automorphic representation of $G(\AA_{F})$. Then its central character of $\om: F^{\times}\backslash \AA^{\times}_{F}\to\Qlbar^{\times}$ is compatible with $(I_{S},\chi_{S})$: $\omega|_{\calO^{\times}_{x}}$ is trivial for $x\notin S$ and $\omega_{x}|_{\calO^{\times}_{x}}$ is equal to $\calO^{\times}_{x}\surj k^{\times}_{x}\xrightarrow{\chi_{x}^{(1)}\chi_{x}^{(2)}}\Qlbar^{\times}$ for $x\in S$. Since $\otimes\omega_{x}$ has to be trivial on $F^{\times}$, and in particular $k^{\times}$, we must have  \eqref{det GL2}. When \eqref{det GL2} is satisfied, the central character $\omega$ is also unique up to an unramified twist, and in particular the restricted part $\onat$ is unique determined by  the characters $\{\chi_{x}\}_{x\in S}$.

Since in our case $\Bun^{\n}_{\cZ}(\bI_{Z,S})=\Pic^{0}(X)\cong\BB\Gm$, Lemma \ref{l:unique Om} implies that when \eqref{det GL2} holds, there is a unique pair $(\Om, \iota_{S})$ making $(\Om, \bI_{S}, \cK_{S},\iota_{S})$ into a geometric automorphic datum.

(2) We shall show that $(I_{S},\chi_{S})$-typical automorphic representations are unique up to unramified twists. For this we may fix a central character $\omega$ compatible with $(I_{S},\chi_{S})$, and argue that $(\omega, I_{S},\chi_{S})$-typical automorphic representations are unique up to unramified twists.

By \eqref{Weil e}, we have an equivalence of groupoids (note that $\cohog{1}{F,\GL_{2}}$ is trivial)
\begin{equation*}
G(F)\backslash G(\AA_{F})/\prod_{x\notin S}G(\calO_{x})\times\prod_{x\in S}I_{x}\isom \Bun_{2}(\bI_{S})(k).
\end{equation*}
Here $\Bun_{2}(\bI_{S})$ is the moduli stack classifying $(\calV,\ell_{S})$, where $\calV$ is a vector bundle of rank two on $X$ and $\ell_{S}$ is a collection of lines $\ell_{x}$ in the two-dimensional fiber $\calV_{x}$, one for each $x\in S$. The Kottwitz map in this case is $\kappa: \Bun_{2}(\bI_{S})\to\ZZ$ sending $(\calV,\ell_{S})$ to the degree of $\calV$. We denote $\kappa^{-1}(d)$ by $\Bun^{d}_{2}(\bI_{S})$. We shall determine the $\cK_{S}$-relevant points on each $\Bun^{d}_{2}(\bI_{S})$.

Evaluating an automorphisms of $(\calV,\ell_{S})\in\Bun_{2}(\bI_{S})(\kbar)$ at $x\in S$ gives a map
\begin{equation*}
\alpha_{x}=(\alpha^{(1)}_{x},\alpha^{(2)}_{x}):\Aut(\calV,\ell_{S})\to \Gm^{2}.
\end{equation*}
which sends an automorphism $\varphi\in\Aut(\calV,\ell_{S})$ to the scalars by which $\varphi$ acts on $\ell_{x}$ and on $\calV_{x}/\ell_{x}$. Then $(\calV,\ell_{S})$ is $\cK_{S}$-relevant if and only if $\cK_{S,\calV}:=\bigotimes_{x\in S}\alpha^{(1),*}_{x}\cK^{(1)}_{x}\otimes\alpha^{(2),*}_{x}\cK^{(2)}_{x}$ is trivial on the neutral component of $\Aut(\calV,\ell_{S})$. 

There is one situation in which we may immediately conclude that $(\calV, \ell_{S})$ is irrelevant, that is when $\calV=\calL'\oplus\calL''$ is a sum of two line bundles such that for each $x\in S$, either $\ell_{x}=\calL'_{x}$ or $\ell_{x}=\calL''$. We call such a point $(\calV, \ell_{S})$ {\em decomposable}. In this case, consider the subgroup $\Gm'\incl \Aut(\calV,\ell_{S})$ given by scaling only the direct summand $\calL'$. The map $\alpha_{x}|_{\Gm}$ is given by the inclusion of $\Gm$ into the first or second factor of $\Gm^{2}$. Therefore the sheaf $\cK_{S,\calV}|_{\Gm'}$ takes the form
\begin{equation*}
\cK_{S,\calV}|_{\Gm'}\cong\otimes_{x\in S}\cK^{(\ep(x))}_{x}
\end{equation*}
where $\ep(x)=1$ if $\ell_{x}=\calL'_{x}$ and $\ep(x)=2$ if $\ell_{x}=\calL''_{x}$. Under the sheaf-to-function correspondence, $\cK_{S,\calV}|_{\Gm'}$ goes to the character $\prod_{x\in S}\chi^{(\ep(x))}_{x}:k^{\times}\to\Qlbar^{\times}$, which is nontrivial by our assumption. Therefore, $\cK_{S,\calV}|_{\Gm'}$ is nontrivial, and $(\calV,\ell_{S})$ is irrelevant.

To determine the relevant points, we need to look for indecomposable $(\calV, \ell_{S})$'s. Any rank two vector bundles $\calV$ over $X=\PP^{1}$ is of the form $\calL_{1}\oplus\calL_{2}$ where $\calL_{1}\cong\calO(a), \calL_{2}\cong\calO(b)$ for integers $a\geq b$. If $a>b$, the automorphism group of $\calV$ consists of $2\times2$ matrices $\left(\begin{array}{cc} \l & \phi \\ 0 & \mu\end{array}\right)$ where $\phi\in\Hom(\calL_{2},\calL_{1})$. If $\ell_{x}\neq\calL_{1,x}$ then there is a unique value $n_{x}\in \Hom(\calL_{2,x},\calL_{1,x})$ such that $\left(\begin{array}{cc} 1 & n_{x} \\ 0 & 1\end{array}\right)\cdot\calL_{2,x}=\ell_{x}$. 

If $a-b\geq2$, there exists $\phi\in\Hom(\calL_{2},\calL_{1})\cong\upH^{0}(\PP^{1},\calO(a-b))$ such that $\phi(x)=n_{x}$ for all $x\in S$ whenever $\ell_{x}\neq\calL_{1,x}$ (since $\deg S=3$). Therefore, letting $\calL'=\calL_{1}$ and $\calL''=\left(\begin{array}{cc} 1 & \phi \\ 0 & 1\end{array}\right)\cdot\calL_{2}$, we get a new decomposition $\calV=\calL'\oplus\calL''$ such that for each $x\in S$, $\ell_{x}$ is either $\calL'_{x}$ or $\calL''_{x}$, i.e., $(\calV,\ell_{S})$ is decomposable hence $\cK_{S}$-irrelevant. 

The remaining possibilities are
\begin{enumerate}
\item If $a-b=1$, for $(\calV,\ell_{S})$ to be indecomposable, none of the $\ell_{x}$ should be contained in $\calL_{1}$ by the above discussion. When this is the case, we may use an automorphism of $\calV$ as above to obtain another decomposition $\calV=\calL'\oplus\calL''$ such that $\ell_{0}$ and $\ell_{\infty}$ are contained in the fibers of $\calL''$ while $\ell_{1}$ is spanned by $(1,1)$ under trivializations of $\calL'$ and $\calL''$ at $1$. This gives a unique $\cK_{S}$-relevant point $\star_{2a-1}\in\Bun^{2a-1}_{2}(\bI_{S})$ (which is also a $k$-point).
\item If $a=b$, trivializing $\PP(\calV)$ we may view $\ell_{x}$ as a point in $\PP^{1}$. The point $(\calV,\ell_{S}\in(\PP^{1})^{S})$ is indecomposable if and only if the three points $\ell_{x}\in \PP^{1}$ are distinct. Since $\Aut(\calV)=\GL_{2}$, all such $(\calV, \ell_{S})$ give the same point $\star_{2a}\in \Bun^{2a}_{2}(\bI_{S})$. This is  a $k$-point, and is the unique $\cK_{S}$-relevant point on $\Bun^{2a}_{2}(\bI_{S})$.
\end{enumerate}

To conclude, for each $d\in\ZZ$, $\Bun^{d}_{2}(\bI_{S})$ contains a unique $\cK_{S}$-relevant point $\star_{d}$, which is defined over $k$. For each $d\in\ZZ$, let $G(\AA_{F})_{d}$ be the subset consisting of $g=(g_{x})$ with $\deg\det(g)=d$, and let $C_{G}(I_{S},\chi_{S})_{d}=C\left(\AG_{d}/\prod_{x\notin S}G(\calO_{x})\times\prod_{x\in S}(I_{x},\chi_{x})\right)$. The above discussion shows that $\dim C_{G}(I_{S},\chi_{S})_{d}=1$, in which all functions are supported on the double coset given by $\star_{d}$. 

The action of the center $F^{\times}\backslash \AA^{\times}_{F}$ on $\AG$ identifies $\AG_{d}$ for all $d$ with the same parity. Therefore $\dim C_{G,\omega}(I_{S},\chi_{S})=2$, and the restriction map $C_{G,\omega}(I_{S},\chi_{S})\to C_{G}(I_{S},\chi_{S})_{0}\oplus C_{G}(I_{S},\chi_{S})_{1}$ is an isomorphism. Formula \eqref{Csum} implies that there are at most two $(\omega,I_{S},\chi_{S})$-typical automorphic representations, and they are cuspidal. All statements so far hold when $k$ is replaced by any finite extension $k'$ (with respect to the base-changed automorphic datum).

If there is only one $(\omega,I_{S},\chi_{S})$-typical automorphic representation $\pi$, there is nothing else to show. Suppose there are two $(\omega,I_{S},\chi_{S})$-typical automorphic representations $\pi$ and $\pi'$. By multiplicity one for $G=\GL_{2}$, $\pi$ is not isomorphic to $\pi'$ as $G(\AA_{F})$-modules. Consider the unramified quadratic character $\eta:G(F)\backslash G(\AA)/\prod_{x}G(\calO_{x})\xrightarrow{\deg\circ\det} \ZZ\surj\{\pm1\}$. Twisting by $\eta$ sends one $(\omega,I_{S},\chi_{S})$-typical automorphic representation to another. We claim that $\pi\otimes\eta=\pi'$. If not, then $\pi\otimes\eta=\pi$. By the global Langlands correspondence for $\GL_{2}$, there is a Galois representation $\rho: W_{F}\to\GL_{2}(\Qlbar)$ attached to $\pi$, and it satisfies $\rho\otimes\chi_{\eta}\cong\rho$, where $\chi_{\eta}$ is the unramified character $W_{F}\to\ZZ\surj\{\pm1\}$. This implies that $\rho=\Ind^{W_{F}}_{W_{E}}\zeta$ for some character $\zeta: W_{E}\to\Qlbar^{\times}$, where $E=F\otimes_{k}k'$ and $k'$ is the quadratic extension of $k$. In particular $\rho|_{W_{E}}$ is reducible. This means that the base change $\textup{BC}_{E/F}(\pi)$, as a $(\omega', I_{S'}, \chi_{S'})$-typical automorphic  representation of $G(\AA_{E})$, is not cuspidal, which is a contradiction. Therefore up to unramified twists, there is only one $(\omega,I_{S},\chi_{S})$-typical automorphic representation of $G(\AA_{F})$. The same argument applies when $k$ is replaced by any finite extension $k'$. Therefore $(\Om,\bI_{S},\cK_{S},\iota_{S})$ is strongly rigid. 
\end{proof}

\subsubsection{$S$ consists of two points--Case I}\label{sss:GL2S2I} Let $S=\{0,\infty\}$ and $\chk\neq2$. We consider the following geometric automorphic datum. At $x=0$, let $\bK_{0}$ be the Iwahori subgroup $\bI_{0}$. Let $\chi_{0}:\TT(k)=k^{\times}\times k^{\times}\to\Qlbar^{\times}$ be a character, viewed as a character of $I_{0}=\bI_{0}(k)$ via $I_{0}\to \TT(k)$. Let $\cK_{0}\in\cCS(\TT)$ be the corresponding character sheaf, also viewed as a character sheaf on $\bK_{0}=\bI_{0}$.

On the other hand, fix a nontrivial additive character $\psi:k\to\Qlbar^{\times}$. At $x=\infty$, let $\bK_{\infty}=\bI^{+}_{\infty}$, the pro-unipotent radical of the Iwahori subgroup at $\infty$, i.e.,
\begin{equation*}
\bK_{\infty}(k)=\left\{\left(\begin{array}{cc} a & b \\ c & d\end{array}\right)|c\in t^{-1}k[[t^{-1}]], a,b,d\in k[[t^{-1}]], a,b\equiv1\mod t^{-1}\right\}.
\end{equation*}
Consider the homomorphism
\begin{eqnarray}\label{GL2 I+}
\bK_{\infty}&\to&\Ga^{2}\\
\notag\left(\begin{array}{cc} a & b \\ c & d\end{array}\right) &\mapsto& (b_{0}, c_{-1}).
\end{eqnarray}
Here $b_{0}$ is the constant term of $b$ and $c_{-1}$ is the coefficient of $t^{-1}$ in $c$. Let $\phi:k^{2}\to k$ be a linear function. Then the character $\psi\phi:\Ga^{2}(k)=k\times k\to\Qlbar^{\times}$ corresponds to an Artin-Schreier sheaf $\AS_{\phi}$. Let $\cK_{\infty}=\AS_{\phi}$, viewed as a rank one character sheaf on $\bK_{\infty}$ via the map \eqref{GL2 I+}. Since $\Bun_{\cZ}(\bK_{Z,S})$ is a point in this case, by Lemma \ref{l:unique Om} there is a unique way (up to isomorphism) to extend $(\bK_{S}, \cK_{S})$ into a geometric automorphic datum $(\Om,\bK_{S}, \cK_{S},\iota_{S})$.

\begin{prop}\label{p:GL2S2I}  Suppose the linear function $\phi:k^{2}\to k$ is nontrivial on each coordinate, then the geometric automorphic datum $(\Om,\bK_{S}, \cK_{S},\iota_{S})$ is strongly rigid.
\end{prop}
\begin{proof}  Let $(\onat, K_{S},\chi_{S})$ be the restricted automorphic datum attached to $(\Om,\bK_{S}, \cK_{S},\iota_{S})$. To show strong rigidity, we need to show that there is a unique $(\onat, K_{S},\chi_{S})$-typical automorphic representation of $G(\AA_{F})$ up to unramified twists (and the argument will work when $k$ is replaced with any finite extension $k'$). 

Let $\pi$ be a $(\onat, K_{S},\chi_{S})$-typical automorphic representation. The moduli stack we should consider is $\Bun_{G}(\bK_{S})=\Bun_{2}(\bI_{0}, \bI^{+}_{\infty})$. It classifies quadruples $(\calV, \ell_{0}, v^{(1)}_{\infty}, v^{(2)}_{\infty})$, where $\calV$ is a rank two vector bundle over $X=\PP^{1}$, $\ell_{0}$ is a line  in the fiber $\calV_{0}$, $v^{(1)}_{\infty}$ is a nonzero vector in $\calV_{\infty}$ and $v^{(2)}_{\infty}$ is a nonzero vector in the quotient line $\calV_{\infty}/\jiao{v^{(1)}_{\infty}}$. Again the Kottwitz map gives a decomposition of $\Bun_{2}(\bI_{0}, \bI^{+}_{\infty})$ into $\Bun^{d}_{2}(\bI_{0}, \bI^{+}_{\infty})$ according to the degree of $\calV$. We shall show that each $\Bun^{d}_{2}(\bI_{0}, \bI^{+}_{\infty})$ contains a unique $\cK_{S}$-relevant point, and the rest of the argument is similar to that of Proposition \ref{p:GL2S3}.

Fix a $\kbar$-point $(\calV, \ell_{0}, v^{(1)}_{\infty}, v^{(2)}_{\infty})$ of $\Bun_{2}(\bI_{0}, \bI^{+}_{\infty})$. Let $A=\Aut(\calV, \ell_{0}, v^{(1)}_{\infty}, v^{(2)}_{\infty})$. For an automorphism $\varphi\in A$, we may evaluate it at $\infty$ to get $\varphi_{\infty}\in\GL(\calV_{\infty})$, which fixes $v^{(1)}_{\infty}$ and preserves $v^{(2)}_{\infty}$ up to a multiple of $v^{(1)}_{\infty}$. Then $\varphi_{\infty}(v^{(2)}_{\infty})-v^{(2)}_{\infty}=f_1(\varphi)v^{(1)}_{\infty}$ for some $f_{1}(\varphi)\in \kbar$. On the other hand, extending $v^{(1)}_{\infty}$ to a section $\tilv^{(1)}_{\infty}$ of $\calV$ on the first infinitesimal neighborhood of $\infty$, then $\varphi(\tilv^{(1)}_{\infty})-\tilv^{(1)}_{\infty}=f_{2}(\varphi)t^{-1}v^{(2)}_{\infty}$ modulo the span of $t^{-1}v^{(1)}_{\infty}$, for some $f_{2}(\varphi)\in \kbar$ independent of the choice of $\tilv^{(1)}_{\infty}$ lifting $v^{(1)}_{\infty}$. This way we get a homomorphism
\begin{equation*}
f=(f_{1},f_{2}): A\to\Ga^{2}.
\end{equation*}
The point $(\calV, \ell_{0}, v^{(1)}_{\infty}, v^{(2)}_{\infty})$ is $\cK_{S}$-relevant only if $f^{*}\cK_{\infty}$ is trivial on $A^{\circ}$ (in fact $A$ is connected). 

The vector bundle $\calV$ is isomorphic to $\calL'\oplus\calL''$, where $\calL'\cong\calO(a), \calL''\cong\calO(b)$ for integers $a\geq b$. When $a>b$,  as in the proof of Proposition \ref{p:GL2S3}, we may adjust the decomposition by an automorphism of $\calV$ so that $\ell_{0}$ is either $\calL'_{0}$ or $\calL''_{0}$, and at the same time either $v^{(1)}_{\infty}\in\calL'_{\infty}$ or $v^{(1)}_{\infty}\in\calL''_{\infty}$. 

If  $v^{(1)}_{\infty}\in\calL'_{\infty}$, there is an automorphism $\left(\begin{array}{cc} 1 & \phi \\ 0 & 1\end{array}\right)\in A$ where $\phi\in\Hom_{X}(\calL'',\calL')=\upH^{0}(\PP^{1},\calO(a-b))$ such that $\phi_{0}=0$ and $\phi_{\infty}\neq0$. This means that there is a subgroup $\Ga\subset A$ such that $f|_{\Ga}$ identifies it with the first factor of $\Ga^{2}$. Since $\cK_{\infty}$ is nontrivial on  the first factor of $\Ga^{2}$, $f^{*}\cK_{\infty}|_{\Ga}$ is nontrivial, and $(\calV, \ell_{0}, v^{(1)}_{\infty}, v^{(2)}_{\infty})$ is $\cK_{S}$-irrelevant. 

If $v^{(1)}_{\infty}\in\calL''_{\infty}$ and $a-b\geq2$, then we may consider the similar automorphism with $\phi_{0}=0, \phi_{\infty}=0$ but the coefficient of $t^{-1}$ in the local expansion of $\phi$ at $\infty$ is nonzero. This way we get a subgroup $\Ga\subset A$ such that $f|_{\Ga}$ identifies it with the second factor of $\Ga^{2}$, and since $\cK_{\infty}$ is nontrivial on the second factor of $\Ga^{2}$, we again conclude that $(\calV, \ell_{0}, v^{(1)}_{\infty}, v^{(2)}_{\infty})$ is $\cK_{S}$-irrelevant. When $a-b=1$ and $\ell_{0}=\calL'_{0}$ we may drop the requirement $\phi_{0}=0$ and the above argument still works.

The remaining cases are
\begin{enumerate}
\item If $a-b=1$, $\ell_{0}=\calL''_{0}$ and $v^{(1)}_{\infty}\in\calL''_{\infty}$. This is in fact a single point in $\Bun^{2a-1}_{2}(\bI_{0}, \bI^{+}_{\infty})$ defined over $k$ with trivial automorphism, and it is the unique $\cK_{S}$-relevant point therein.
\item If $a=b$, after twisting by $\calO(-a)$ we may reduce to the case $a=0$ and we have $\calV=V\otimes_{\kbar}\calO_{X}$ for the two-dimensional vector space $V=\cohog{0}{X,\calV}$. The lines $\ell_{0}$ and $\jiao{v^{(1)}_{\infty}}$ determine two points in $\PP(V)$. If these two points are the same, then $A$ is the unipotent radical $U$ of the Borel subgroup of $\GL(V)$ preserving $\ell_{0}$, which maps isomorphically to the first factor of $\Ga^{2}$ under $f$. Therefore in this case $(\calV, \ell_{0}, v^{(1)}_{\infty}, v^{(2)}_{\infty})$ is $\cK_{S}$-irrelevant. The only possibility for $(\calV, \ell_{0}, v^{(1)}_{\infty}, v^{(2)}_{\infty})$ to be relevant is when the lines $\ell_{0}$ and $\jiao{v^{(1)}_{\infty}}$  are distinct. This is in fact a single point in $\Bun^{2a}_{2}(\bI_{0}, \bI^{+}_{\infty})$ defined over $k$ with trivial automorphism, and it is the unique $\cK_{S}$-relevant point therein.
\end{enumerate}
\end{proof}

\subsubsection{$S$ consists of two points--Case II}\label{sss:GL2S2II} Let $S=\{0,\infty\}$ and $\chk\neq2$. We consider the following geometric automorphic datum. At $x=0$, let $\bK_{0}$ be the Iwahori subgroup $\bI_{0}$ with reductive quotient $\TT\cong\Gm^{2}$. Choose a characters $\chi_{0}=(\chi^{(1)}_{0}, \chi^{(2)}_{0}): \TT(k)\to\Qlbar^{\times}$, which correspond to character sheaves $\cK^{(1)}_{0}$ and $\cK^{(2)}_{0}$ on $\Gm$. Let $\cK_{0}=\cK^{(1)}_{0}\boxtimes\cK^{(2)}_{0}$  be the character sheaf on $\TT$, viewed also as a character sheaf on $\bI_{0}$. 

On the other hand, fix a nontrivial additive character $\psi:k\to\Qlbar^{\times}$. Let $\bK^{1}_{\infty}=\ker(L_{\infty}^{+}G\xrightarrow{\mod\frm_{\infty}} \GG)$ be the first congruence subgroup of $L^{+}_{\infty}G$, and let $\bK_{\infty}=\bK^{1}_{\infty}\cdot \TT$. There is a natural projection $\bK^{1}_{\infty}\to \frg$ (where $\frg=\Lie\GG$) given by reduction modulo $\frm^{2}_{\infty}$.  Let $\frt=\Lie\TT$. Choose $\phi=(\phi^{(1)}, \phi^{(2)})\in k^{2}$, which gives an additive character $\psi\phi:\frt\to \Qlbar^{\times}$ sending $\diag(x,y)$ to $\psi(\phi^{(1)}x+\phi^{(2)}y)$. We have the corresponding Artin-Schreier sheaf $\AS_{\phi}$ on $\frt$. Also choose characters $(\chi^{(1)}_{\infty}, \chi^{(2)}_{\infty}): \TT(k)\to\Qlbar^{\times}$, which correspond to a character sheaf $\cK^{(1)}_{\infty}\boxtimes\cK^{(2)}_{\infty}$ on $\TT$. Let $\cK_{\infty}=\cK^{(1)}_{\infty}\boxtimes\cK^{(2)}_{\infty}\boxtimes\AS_{\phi}\in\cCS(\TT\times\frt)$. Via the projection $\bK_{\infty}\to\TT\times\frg\to\TT\times\frt$ (the last projection $\frg\to\frt$ is the projection onto the diagonal matrices), we view $\cK_{\infty}$ as a character sheaf on $\bK_{\infty}$. The corresponding character of $K_{\infty}=\bK_{\infty}(k)$ is
\begin{equation*}
\chi_{\infty}: K_{\infty}\surj  \TT(k)\times \textup{Mat}_{2}(k)\surj k^{\times}\times k^{\times}\times k\times k\xrightarrow{(\chi^{(1)}_{\infty}, \chi^{(2)}_{\infty}, \psi\phi^{(1)}, \psi\phi^{(2)})}\Qlbar^{\times}.
\end{equation*}

\begin{prop}\begin{enumerate}
\item There is a $(K_{S},\chi_{S})$-typical automorphic representation of $G(\AA_{F})$ only if
\begin{equation*}
\prod_{x\in S}\chi_{x}^{(1)}\chi_{x}^{(2)}=1.
\end{equation*}
In this case, there is a unique way (up to isomorphism) to extend $(\bK_{S},\cK_{S})$ into a geometric automorphic datum $(\Om, \bK_{S}, \cK_{S}, \iota_{S})$.
\item If the following extra conditions are satisfied
\begin{itemize}
\item For any map $\ep:S\to\{1,2\} $ we have
\begin{equation}\label{S2 chigen}
\prod_{x\in S}\chi^{(\ep(x))}_{x}\neq1;
\end{equation}
\item $\phi^{(1)}\neq\phi^{(2)}$.
\end{itemize}
Then the geometric automorphic datum $(\Om, \bK_{S}, \cK_{S},\iota_{S})$ is strongly rigid.
\end{enumerate}
\end{prop}
\begin{proof}[Sketch of proof]
The basic idea of the proof is similar to that of Proposition \ref{p:GL2S3}. The only nontrivial part in proving (2) is to show that each $\Bun^{d}_{2}(\bK_{S})$ contains a unique $\cK_{S}$-relevant point. Let us sketch the argument.

The moduli stack $\Bun_{2}(\bK_{S})=\Bun_{2}(\bI_{0}, \bK_{\infty})$ in this situation classifies the data $(\calV, \ell_{0}, \ell^{(1)}_{\infty},\ell^{(2)}_{\infty})$ where $\calV$ is a rank two vector bundle over $X=\PP^{1}$, $\ell_{0}$ is a line in $\calV_{0}$, and $\ell^{(1)}_{\infty}$ and $\ell^{(2)}_{\infty}$ are independent lines in $\calV_{\infty}$. Fix a point $(\calV, \ell_{0}, \ell^{(1)}_{\infty},\ell^{(2)}_{\infty})\in\Bun_{2}(\bI_{0}, \bK_{\infty})(\kbar)$ and let $A=\Aut(\calV, \ell_{0}, \ell^{(1)}_{\infty},\ell^{(2)}_{\infty})$. Evaluating an automorphism $\varphi\in A$ at $0$ we get
\begin{equation*}
\a_{0}=(\a^{(1)}_{0}, \a^{(2)}_{0}): \Aut(\calV, \ell_{0}, \ell^{(1)}_{\infty},\ell^{(2)}_{\infty})\to\Gm^{2}=\TT
\end{equation*}
where the two coordinates are the scalars by which $\varphi_{0}$ acts on $\ell_{0}$ and on $\calV_{0}/\ell_{0}$. Let $v^{(i)}\in\ell^{(i)}_{\infty}-\{0\}$ for $i=1,2$. Expanding $\varphi$ near $\infty$ to the first order, we may write $\varphi(v^{(1)})=\a^{(1)}_{\infty}(\varphi)v^{(1)}+t^{-1}\b^{(1)}(\varphi)v^{(1)}$ modulo $t^{-1}v^{(2)}$ and $t^{-2}$; similarly  $\varphi(v^{(2)})=\a^{(2)}_{\infty}(\varphi)v^{(2)}+t^{-1}\b^{(2)}(\varphi)v^{(2)}$ modulo $t^{-1}v^{(1)}$ and $t^{-2}$. The functions $\a^{(1)}_{\infty}, \a^{(2)}_{\infty}, \b^{(1)}, \b^{(2)}$ are well-defined (independent of the choices of $v^{(1)}$ and $v^{(2)}$). They define homomorphisms
\begin{eqnarray*}
\a_{\infty}=(\a^{(1)}_{\infty}, \a^{(2)}_{\infty})&:& A\to\Gm^{2}=\TT\\
\b=(\b^{(1)}, \b^{(2)})&:& A\to\AA^{2}=\frt.\end{eqnarray*}
The point $(\calV, \ell_{0}, \ell^{(1)}_{\infty},\ell^{(2)}_{\infty})$ is $\cK_{S}$-relevant if and only if both character sheaves
\begin{equation*}
\cK_{\a}:=\bigotimes_{x\in S}\a^{(1),*}_{x}\cK^{(1)}_{x}\otimes\a^{(2),*}_{x}\cK^{(2)}_{x}; \textup{ and } \cK_{\b}:=\b^{*}\AS_{\phi}
\end{equation*}
are trivial on the neutral component $A^{\circ}$.

We write $\calV=\calL'\oplus\calL''$ where $\calL'\cong\calO(a)$, $\calL''\cong\calO(b)$ with $a\geq b$. If $a>b$, we may use automorphisms of $\calV$ to arrange so that one of the following happens
\begin{itemize}
\item $\ell^{(1)}_{\infty}=\calL'_{\infty}$, $\ell^{(2)}_{\infty}=\calL''_{\infty}$, $\ell_{0}=\calL'_{0}$. Then the diagonal torus $\diag(\l,1)$ belongs to $A$, and $\bK_{\a}$ is nontrivial on this torus by assumption \eqref{S2 chigen}. Therefore $(\calV, \ell_{0}, \ell^{(1)}_{\infty},\ell^{(2)}_{\infty})$ is $\cK_{S}$-irrelevant.
\item $\ell^{(1)}_{\infty}=\calL'_{\infty}$, $\ell^{(2)}_{\infty}=\calL''_{\infty}$, $\ell_{0}=\calL''_{0}$. Similar argument for the torus $\diag(1,\l)$ shows that $(\calV, \ell_{0}, \ell^{(1)}_{\infty},\ell^{(2)}_{\infty})$ is $\cK_{S}$-irrelevant.
\item $\ell^{(1)}_{\infty}=\calL''_{\infty}$, $\ell^{(2)}_{\infty}=\calL'_{\infty}$, $\ell_{0}=\calL'_{0}$ or $\calL''_{0}$. Similar argument as in the previous cases shows that $(\calV, \ell_{0}, \ell^{(1)}_{\infty},\ell^{(2)}_{\infty})$ is $\cK_{S}$-irrelevant.
\item $\ell^{(1)}_{\infty}=\calL''_{\infty}$, $\ell^{(2)}_{\infty}\neq\calL'_{\infty}$  or $\calL''_{\infty}$. If $a\geq b+2$, then there is an additive subgroup $\varphi:\Ga\incl A$ which is the identity when evaluated at $0$ (so that $\ell_{0}$ is preserved) and takes the form $\varphi(u)=\left(\begin{array}{cc} 1 & ut^{-1} \\ 0 & 1\end{array}\right)$ modulo $t^{-2}$, $u\in\Ga$. Calculation shows that $\b(\varphi(u))=(us,-us)$ for some $s\neq0$. Therefore $\cK_{\b}|_{\Ga}\cong m^{*}_{s}\AS_{\phi^{(1)}-\phi^{(2)}}$ (where $m_{s}:\Ga\to\Ga$ is the multiplication by $s$ map). Since $\phi^{(1)}-\phi^{(2)}\neq0$, $\cK_{\b}$ is nontrivial on $\varphi(\Ga)\subset A$, hence $(\calV, \ell_{0}, \ell^{(1)}_{\infty},\ell^{(2)}_{\infty})$ is $\cK_{S}$-irrelevant. If $a-b=1$ and $\ell_{0}\in\calL'_{0}$, then $\ell_{0}$ is preserved by any automorphism of $\calV$, and we still have an additive subgroup $\varphi:\Ga\incl A$ given by $\varphi(u)=\left(\begin{array}{cc} 1 & ut^{-1} \\ 0 & 1\end{array}\right)$. The argument above still works to show that $(\calV, \ell_{0}, \ell^{(1)}_{\infty},\ell^{(2)}_{\infty})$ is $\cK_{S}$-irrelevant.
\end{itemize}

The remaining cases are
\begin{itemize}
\item $a=b$. Up to twisting by $\calO(-a)$ we may assume $\calV\cong V\otimes_{\kbar}\calO_{X}$ for some two-dimensional vector space $V$. Then we may view $\ell_{0},\ell^{(1)}_{\infty}, \ell^{(2)}_{\infty}$ as lines in $V$.  To avoid possible $\Gm$ in $A$ on which $\cK_{\a}$ is nontrivial, the three lines $\ell_{0},\ell^{(1)}_{\infty}, \ell^{(2)}_{\infty}$ must be distinct. This defines the unique $\cK_{S}$-relevant point in $\Bun^{2a}_{2}(\bK_{S})$, which is defined over $k$ with automorphism group equal to the central $\Gm$.
\item $a=b+1$ and none of the three lines $\ell_{0}, \ell^{(1)}_{\infty}$ and $\ell^{(2)}_{\infty}$ are in $\calL'_{\infty}$. This defines the unique $\cK_{S}$-relevant point in $\Bun^{2a-1}_{2}(\bK_{S})$, which is defined over $k$ with automorphism group equal to the central $\Gm$.
\end{itemize}
The rest of the argument is similar to that of Proposition \ref{p:GL2S3}. 
\end{proof}

We will see in \S\ref{ss:rigidlocGL2} that the three types of rigid automorphic data we considered in this subsection for $\GL_{2}$ correspond, under the Langlands correspondence, to the three types of hypergeometric local systems of rank two over punctured $\PP^{1}$.

\section{Rigidity for local systems}\label{s:ls} 

In this section, we will review some basic facts about local systems in \'etale topology in \S\ref{ss:ls}, and introduce the notion of rigidity for them in \S\ref{ss:rls}, following Katz \cite{Katz}. Rigid local systems of rank two are studied in details in \S\ref{ss:rigidlocGL2}. We also review the notion of rigid tuples in Inverse Galois Theory in \S\ref{ss:InvGal} for comparison.

\subsection{Local systems}\label{ss:ls} 
Let $k$ be any field.
\subsubsection{Local systems in \'etale topology} Let $U$ be a scheme of finite type over $k$.  We recall some definitions from \cite[\S1.2, \S1.4.2, \S1.4.3]{SGA5VI}. A $\Zl$-local system on $U$ is a projective system $(\calF_{n})_{n\geq1}$ of locally constant locally free $\ZZ/\ell^{n}\ZZ$-sheaves $\calF_{n}$ of finite rank on $X$ under the \'etale topology such that the natural map $\calF_{n}\otimes\ZZ/\ell^{n-1}\ZZ\to\calF_{n-1}$ is an isomorphism for all $n$. Denote the category of $\Zl$-local systems on $U$ by $\Loc(U,\Zl)$, which is a $\Zl$-linear abelian category. The category of $\Ql$-local systems is by definition the abelian category $\Loc(U,\Zl)\otimes\Ql$ obtained by inverting $\ell$ in the Hom groups in $\Loc(U,\Zl)$. Similar definition gives $\Loc(U,\calO_{L})$ and $\Loc(U,L)$ for any finite extension $L$ of $\Ql$. Finally define $\Loc(U):=\Loc(U,\Qlbar)$ to be the colimit $\varinjlim_{L}\Loc(U,L)$ over all finite extensions $L$ of $\Ql$.

In the sequel we assume $U$ is normal and connected. Fix a geometric point $u\in U$. Let $\calF$ be a $\Qlbar$-local system on $U$ of rank $n$. The stalk $\calF_{u}$ is a $\Qlbar$-vector space of dimension $n$ which carries an action of the \'etale fundamental group $\pi_{1}(U,u)$ defined in \cite[V,\S7]{SGA1}. Thus $\calF$ determines a {\em continuous} homomorphism
\begin{equation*}
\rho_{\calF}: \pi_{1}(U,u)\to \Aut_{\Qlbar}(\calF_{u})\cong\GL_{n}(\Qlbar),
\end{equation*}
where  $\Qlbar$ is topologized as the colimit of finite extensions $E$ of $\Ql$, each with the $\ell$-adic topology. The last isomorphism above depends upon the choice of a basis of $\calF_{u}$. This way we get a functor
\begin{equation}\label{fiber u}
\omega_{u}:\Loc(U)\to\Rep_{\cont}(\pi_{1}(U,u))
\end{equation}
where $\Rep_{\cont}(\pi_{1}(U,u))$ is the category of continuous representation of $\pi_{1}(U,u)$ on finite-dimensional $\Qlbar$-vector spaces. Both sides of \eqref{fiber u} carry tensor structures and $\omega_{u}$ is in fact an equivalence of tensor categories (\cite[Proposition 1.2.5]{SGA5VI}).

Let $\pi^{\geom}_{1}(U,u):=\pi_{1}(U\otimes_{k}\kbar,u)$, and call it the 
{\em geometric fundamental group} of $U$ (with respect to the base point $u$). This is a normal subgroup of $\pi_{1}(U,u)$ which fits into an exact sequence (\cite[V, Proposition 6.13]{SGA1})
\begin{equation*}
1\to \pi^{\geom}_{1}(U,u)\to\pi_{1}(U,u)\to\Gal(k^{s}/k)\to1.
\end{equation*}

\subsubsection{$H$-local systems} Let $H$ be an affine algebraic group over $\Qlbar$. We may define the notion of $H$-local system on a connected normal scheme $U$ over $k$. There are two ways to do this.

First definition. Fix a geometric point $u\in U$. An $H$-local system on $U$ is a continuous homomorphism
\begin{equation}\label{loc H}
\rho:\pi_{1}(U,u)\to H(\Qlbar).
\end{equation}
Such homomorphisms form a category $\Loc_{H}(U)$, in which isomorphisms are given by $H(\Qlbar)$-conjugacy of representations. For example, $\Loc_{\GL_{n}}(U)$ is equivalent to the full subcategory of $\Loc(U)$ consisting of local systems of rank $n$.

Second (and more canonical) definition. Let $\Rep(H)$ be the tensor category of algebraic representations of $H$ on finite-dimensional $\Qlbar$-vector spaces.  We define an $H$-local system to be a tensor functor $\calF: \Rep(H)\to\Loc(U)$ (notation: $V\mapsto\calF_{V}$). We form the category of $H$-local system on $U$ by taking the category of such tensor functors
\begin{equation*}
\Loc_{H}(U):=\Fun^{\otimes}(\Rep(H),\Loc(U)).
\end{equation*}

The two notions of $H$-local systems are equivalent. Given a representation $\rho$ as in \eqref{loc H} and for $V\in\Rep(H)$,  the composition 
\begin{equation*}
\rho_{V}:\pi_{1}(U,u)\xrightarrow{\rho}H(\Qlbar)\to\GL(V)
\end{equation*}
is an object in $\Loc(U)$ of rank equal to $\dim V$. The assignment $V\mapsto \rho_{V}$ gives a tensor functor $\calF:\Rep(H)\to\Loc(U)$. Conversely, given a tensor functor $\calF:\Rep(H)\to\Loc(U)$, using the equivalence \eqref{fiber u}, it can be viewed as a tensor functor $\Rep(H)\to \Rep_{\cont}(\pi_{1}(U,u))$. The Tannakian formalism \cite{DM} then implies that such a tensor functor comes from a group homomorphism $\rho$ as in \eqref{loc H}, well-defined up to conjugacy.

\begin{defn}\label{define global mono} Let $\calF\in\Loc_{H}(U)$ be given by a continuous homomorphism $\rho:\pi_{1}(U,u)\to H(\Qlbar)$.
The {\em global geometric monodromy group} $H^{\geom}_{\calF}$ of $\calF$ is the Zariski closure of $\rho(\pi^{\geom}_{1}(U,u))$ in $H$. 
\end{defn}

Recall the following fundamental result of Deligne on the nature of the global geometric monodromy groups of local systems when $k$ is a finite field.

\begin{theorem}[Deligne]\label{th:Del ss} Let $k$ be a finite field and $\calF$ be an $H$-local system on $U$ (defined over $k$). Suppose that for some faithful representation $H\incl \GL(V)$, the associated local system $\calF_{V}$ is $\iota$-pure with respect to an embedding $\iota:\Qlbar\incl\CC$. Then the neutral component of $H^{\geom}_{\calF}$ is semisimple.
\end{theorem}
In fact, purity of $\calF_{V}$ implies that $\calF_{V}$ is semisimple over $U_{\kbar}$, and we then apply \cite[Corollaire 1.3.9]{WeilII} to get the above result. For the definition of $\iota$-purity, see \cite[\S1.2.6]{WeilII}.

\subsubsection{Local monodromy} Let $X$ be a projective, smooth and geometrically connected curve over a perfect field $k$. We shall use the notation from \S\ref{ss:X}. Fix a finite set of closed points $S\subset |X|$ and let $U=X-S$. An algebraic closure $\Fbar_{x}$ of $F_{x}$ gives a geometric generic point $\eta_{x}\in X$.  The morphism $\Spec F_{x}\to U$ then induces an injective homomorphism of fundamental groups
\begin{equation}\label{incl local Gal}
\Gamma_{x}=\Gal(F^{s}_{x}/F_{x})\incl \pi_{1}(U, \eta_{x})\cong\pi_{1}(U,u),
\end{equation}
where the second map is well-defined up to conjugacy. Since $F_{x}$ is a complete discrete valuation field with perfect residue field $k_{x}$, we have an exact sequence
\begin{equation*}
1\to I_{x}\to \Gal(F^{s}_{x}/F_{x})\to\Gal(\kbar/k_{x})\to 1. 
\end{equation*}
Under \eqref{incl local Gal}, $I_{x}$ is contained in the normal subgroup $\pi_{1}^{\geom}(U, u)\lhd\pi_{1}(U,u)$. 

When $\textup{char}(k)=p>0$, there is a normal subgroup $I^{w}_{x}\lhd I_{x}$ called the {\em wild inertia group} such that the quotient $I^{t}_{x}:=I_{x}/I^{w}_{x}$ is the maximal prime-to-$p$ quotient of $I_{x}$, called the {\em tame inertia group}. We have a canonical isomorphism of $\Gal(\kbar/k_{x})$-modules
\begin{equation*}
I^{t}_{x}\isom\varprojlim_{(n,p)=1}\mu_{n}(\kbar).
\end{equation*}

Let $\rho:\pi_{1}(U,u)\to H(\Qlbar)$ be an $H$-local system. The {\em local monodromy of $\rho$ at $x\in S$} is the homomorphism $\rho_{x}:=\rho|_{I_{x}}:I_{x}\to H(\Qlbar)$. The local system $\rho$ is said to be {\em tame at $x\in S$} if $\rho_{x}$ factors through the tame inertia group $I^{t}_{x}$. 

\subsubsection{Conductor}\label{sss:cond} For a continuous linear representation of the inertia group on a $\Qlbar$-vector space $V$
\begin{equation*}
\sigma:I_{x}\to \GL(V)
\end{equation*}
one can define its Swan conductor and Artin conductor (see \cite[Chapter 1]{Katz-Gauss}).  We recall their definitions in the case $\sigma(I_{x})$ is finite (see \cite[\S VI.2]{Serre-local}, \cite[\S2]{GR}). Let $D=\sigma(I_{x})=\Gal(L/F^{\ur}_{x})$ for some finite Galois representation $L/F^{\ur}_{x}$ inside $F^{s}_{x}$. There is a filtration of $D$
\begin{equation*}
D=D_{0}\rhd D_{1}\rhd D_{2}\rhd\cdots.
\end{equation*}
For $i\geq0$, $D_{i}$ is the subgroup of $D$ which acts trivially on $\calO_{L}/\frm_{L}^{i}$ ($\frm_{L}$ is the maximal ideal of $\calO_{L}$). This is called the {\em lower numbering filtration} on $D$. In particular, $D_{1}=\sigma^{ss}(I^{w}_{x})=\sigma(I^{w}_{x})$. The {\em Swan conductor} of the representation $\sigma$ is
\begin{equation*}
\Sw(\sigma):=\sum_{i\geq1}\frac{\dim(V/V^{D_{i}})}{[D:D_{i}]}.
\end{equation*}
This turns out to be an integer.
The Swan conductor of $\sigma$ is zero if and only if $\sigma$ is tame.  The {\em Artin conductor} of the representation $\sigma$ is
\begin{equation*}
a(\sigma):=\dim(V/V^{I_{x}})+\Sw(\sigma).
\end{equation*}

If $\sigma=\rho_{x}:I_{x}\to \GL(V)$ is the local monodromy of a local system $\calF\in\Loc(U)$, we also denote $\Sw(\sigma)$ by $\Sw_{x}(\calF)$ and $a(\sigma)$ by $a_{x}(\calF)$.

\subsubsection{The twisted situation}\label{sss:twist loc} Let $\theta_{H}: \pi_{1}(U,u)\to \Aut(H)$ be a homomorphism that factors through a finite quotient $\Gamma=\Gal(U'/U)$ for some finite \'etale Galois cover $U'\to U$. We define a {\em $\theta_{H}$-twisted $H$-local system on $U$} to be a pair $(\calF,\delta)$ where
\begin{itemize}
\item $\calF$ is an $H$-local system on $U'$, viewed as a tensor functor $\Rep(H)\to\Loc(U')$;
\item $\delta$ is a collection of isomorphisms $\delta_{\gamma, V}: \calF_{V^{\gamma}}\cong\gamma^{*}\calF_{V}$, one for each $\gamma\in\Gamma, V\in\Rep(H)$, where $V^{\gamma}$ is the representation of $H$ given by the composition $H\xrightarrow{\gamma}H\to\GL(V)$. The isomorphisms $\{\delta_{\gamma,V}\}$ are required to satisfy the usual cocycle relations with respect to the multiplication in $\Gamma$, and to be compatible with the tensor structure of  $\calF$. 
\end{itemize} 
This definition a priori depends on the choice of $\Gamma$ through which $\theta_{H}$ factors. However, enlarging $\Gamma$ gives an equivalent notion of $\theta_{H}$-twisted $H$-local systems, therefore the definition is independent of the choice of the finite quotient $\Gamma$ of $\pi_{1}(U,u)$. We denote the category of $\theta_{H}$-twisted $H$-local systems on $U$ by $\Loc_{H,\theta_{H}}(U)$.

It is easy to check that there is an equivalence of groupoids
\begin{equation}\label{cont cocycle}
\Loc_{H,\theta_{H}}(U)\cong Z^{1}_{\cont}(\pi_{1}(U,u), H(\Qlbar))/H(\Qlbar)
\end{equation}
where $Z^{1}_{\cont}(\pi_{1}(U,u), H(\Qlbar))$ is the set of continuous 1-cocycles on $\pi_{1}(U,u)$ with values in $H(\Qlbar)$ which carries an action of $\pi_{1}(U,u)$ via $\theta_{H}$. Note that $Z^{1}_{\cont}(\pi_{1}(U,u), H(\Qlbar))$ is in bijection with liftings of $\pi_{1}(U,u)\surj\Gamma$ to $\pi_{1}(U,u)\to H(\Qlbar)\rtimes\Gamma$, and $H(\Qlbar)$ acts on such liftings by conjugation. In particular, isomorphism classes in $\Loc_{H,\theta_{H}}(U)$ are in bijection with $\upH^{1}_{\cont}(\pi_{1}(U,u), H(\Qlbar))$. 

In the twisted situation, the local geometric monodromy of $\calF\in\Loc_{H,\theta_{H}}(U)$ at a point $x\in S$ is a cohomology class $\upH^{1}_{\cont}(I_{x}, H(\Qlbar))$, where $I_{x}$ acts on $H(\Qlbar)$ through $I_{x}\incl \pi_{1}(U,u)\xrightarrow{\theta_{H}}\Aut(H)$.  Moreover, the above construction is functorial with respect to $\pi_{1}(U,u)$-equivariant homomorphisms $H\to H'$.

\subsection{Rigidity for local systems}\label{ss:rls} 
Rigidity of a local system is a geometric property, therefore we assume the base field $k$ to be  algebraically closed in this subsection. Let $X$ be a complete smooth connected algebraic curve over $k$. Let $X$ be a smooth, projective and connected curve over $k$. Fix an open subset $U\subset X$ with finite complement $S$. 

Let $H$ be an algebraic group over $\Qlbar$ with an action $\theta_{H}:\pi_{1}(U,u)\to \Aut(H)$ that factors through a finite quotient $\Gamma=\Gal(U'/U)$.

\begin{defn}[extending Katz {\cite[\S1.0.3]{Katz}}]\label{phy rig} Let $\calF\in\Loc_{H,\theta_{H}}(U)$. Then $\calF$ is {\em physically rigid} if,  for any other object $\calF'\in\Loc_{H,\theta_{H}}(U)$ such that for each $x\in S$, the local geometric monodromy of $\calF'$ is isomorphic to that of $\calF$ (i.e., the two local systems give the same class in $\upH^{1}_{\cont}(I_{x}, H(\Qlbar))$), we have $\calF\cong\calF'$ as objects in $\Loc_{H,\theta_{H}}(U)$.
\end{defn}
Although the definition uses $U$ as an input, the notion of physical rigidity is in fact independent of the open subset $U$: for any nonempty open subset $V\subset U$, $\calF$ is rigid over $U$ if and only if $\calF|_{V}$ is rigid over $V$. Therefore, physical rigidity of an $H$-local system is a property of the Galois representation $\Gal(F^{s}/F)\to H(\Qlbar)\rtimes\Gamma$ obtained by restricting $\rho$ to a geometric generic point of the $X$.

\subsubsection{Motivation for cohomological rigidity} Another notion of rigidity can be thought of as an infinitesimal version of physical rigidity. We motivate the definition by considering local systems in the complex analytic topology. Suppose the algebraic curve $X$ is defined over $\CC$. Let $U=X-S\subset X$ be a Zariski open dense subset and let $U^{\an}$ denote the Riemann surface structure on $U(\CC)$. Let $H$ be a connected reductive group over $\Qlbar$. An $H$-local system $\calF$ in this case is simply a group homomorphism 
\begin{equation*}
\rho_{\calF}:\Gamma^{\top}:=\pi_{1}^{\top}(U^{\an},u)\to H.
\end{equation*}
without continuity conditions. For $x\in S$, the local monodromy of $\calF$ around $x$ is the image of a counterclockwise loop around $x$ (with starts and ends at $u$) in $H$.

In this case, the objects of the category $\Loc_{H}(U^{\an})$ are the $\Qlbar$-points of an algebraic stack $\cLoc_{H}(U^{\an})$ defined over $\Qlbar$. In fact, for simplicity we consider the case where $S\neq\varnothing$, then $\Gamma^{\top}$ is simply a free group $F_{r}$ with $r$ generators. Then $\cLoc_{H}(U^{\an})$ is isomorphic to the quotient stack $[H^{r}/H]$ where $H$ acts diagonally on  $H^{r}$ by conjugation. From this description we readily see that $\cLoc_{H}(U^{\an})$ is a smooth algebraic stack. 

It is convenient to work with a Riemann surface with boundary that is homotopy equivalent to $U^{\an}$. There is a canonical such Riemann surface: the real blow-up $\tilX^{\an}$ of $X^{\an}$ along the points in $S$. The preimage of $x\in S$ in $\tilX^{\an}$ is a circle $\partial_{x}$ which is identified with the circle of unit tangent vectors in $T_{x}X$. The boundary $\partial\tilX^{\an}$ of $\tilX^{\an}$ is the disjoint union of $\partial_{x}$ for $x\in S$. The moduli stack $\cLoc_{H}(\tilX^{\an})$ which is canonically isomorphic to $\cLoc_{H}(U^{\an})$. We also have the moduli stack $\cLoc_{H}(\partial \tilX^{\an})$: it is simply isomorphic to $[H/H]^{S}$, one factor of the adjoint quotient $[H/H]$ for each boundary component $\partial_{x}$. 
Let $H^{\ab}$ be the quotient torus of $H$. We can similarly define $\cLoc_{H^{\ab}}(\tilX^{\an})$ and $\cLoc_{H^{\ab}}(\partial \tilX^{\an})$. We have a commutative diagram relating these moduli stacks
\begin{equation}\label{mor loc}
\xymatrix{\cLoc(\tilX^{\an})\ar[r]^{\res}\ar[d] & \cLoc(\partial\tilX^{\an})\ar[d]\\
\cLoc_{H^{\ab}}(\tilX^{\an})\ar[r] & \cLoc_{H^{\ab}}(\partial \tilX^{\an})}
\end{equation}
where the horizontal maps are given by restricting local systems to the boundary, and vertical maps are induced from the projection $H\to H^{\ab}$. 

Now we fix a monodromy datum $(\calA, C_{S})$:
\begin{itemize}
\item An object $\calA\in\cLoc_{H^{\ab}}(\tilX^{\an})$.
\item $C_{S}=\{C_{x}\}_{x\in S}$, where $C_{x}$ is a conjugacy class in $H$ for each $x\in S$, such that the local monodromy of $\calA|_{\partial_{x}}$ is the image $C^{\ab}_{x}$ of $C_{x}$ in $H^{\ab}$.
\end{itemize}
Define a substack $\calC_{S}=\prod_{x\in S}[C_{x}/H]$ of $\cLoc_{H}(\partial \tilX^{\an})$, and similarly $\calC^{\ab}_{S}=\prod_{x}[C^{\ab}_{x}/H^{\ab}]\subset\cLoc_{H^{\ab}}(\partial \tilX^{\an})$. We define
\begin{eqnarray*}
\cLoc_{H}(\tilX^{\an}, C_{S}):=\calC_{S}\times_{\cLoc_{H}(\partial \tilX^{\an})}\cLoc_{H}(\tilX^{\an})\\
\cLoc_{H^{\ab}}(\tilX^{\an}, C^{\ab}_{S}):=\calC^{\ab}_{S}\times_{\cLoc_{H^{\ab}}(\partial \tilX^{\an})}\cLoc_{H^{\ab}}(\tilX^{\an}).
\end{eqnarray*}
The object $\calA$ lies in $\cLoc_{H^{\ab}}(\tilX^{\an}, C^{\ab}_{S})$. We form the fiber
\begin{equation*}
\cLoc_{H}(\tilX^{\an}, \calA, C_{S}):=\cLoc_{H}(\tilX^{\an}, C_{S})\times_{\cLoc_{H^{\ab}}(\tilX^{\an}, C^{\ab}_{S})}\{\calA\}.
\end{equation*}
Since $\tilX^{\an}$ is homotopy equivalent to $U^{\an}$, the stack $\cLoc_{H}(\tilX^{\an}, C_{S})$ classifies $H$-local systems $\calF$ on $U^{\an}$ with monodromy around $x$ lying in the conjugacy class $C_{x}$ and an isomorphism $\calF^{\ab}\cong\calA$. 

We are interested in the infinitesimal deformations of a point $\calF\in\cLoc_{H}(\tilX^{\an}, \calA, C_{S})$. The tangent complex $T_{\calF}\cLoc_{H}(\tilX^{\an})$ of $\cLoc_{H}(\tilX^{\an})$ at $\calF$ is given by
\begin{equation*}
T_{\calF}\cLoc_{H}(\tilX^{\an})\cong\cohog{*}{\tilX^{\an}, \Ad(\calF)}[1]
\end{equation*}
where the local system $\Ad(\calF)$ is the tensor functor $\calF$ evaluated at the adjoint representation of $H$ on its Lie algebra $\frh$. Let $\calF_{\partial}:=\calF|_{\partial\tilX^{\an}}\in\cLoc_{H}(\partial\tilX^{\an})$. The same calculation shows
\begin{equation*}
T_{\calF_{\partial}}\cLoc_{H}(\partial\tilX^{\an})\cong\cohog{*}{\partial\tilX^{\an}, \Ad(\calF_{\partial})}[1].
\end{equation*}
The tangent complex of $\calC_{S}$ is concentrated in degree -1 and is equal to $\cohog{0}{\partial\tilX^{\an}, \Ad(\calF_{\partial})}[1]$. Therefore we have a distinguished triangle
\begin{equation}\label{partial T}
T_{\calF_{\partial}}\calC_{S}\to T_{\calF_{\partial}}\cLoc_{H}(\partial \tilX^{\an})\to\cohog{1}{\partial\tilX^{\an}, \Ad(\calF_{\partial})}\to
\end{equation}
Since $\cLoc_{H}(\tilX^{\an}, C_{S})$ is defined as the preimage of $\calC_{S}$ under the morphism $\res$ in \eqref{mor loc}, the map $T_{\calF}\cLoc_{H}(\tilX^{\an}, C_{S})\to T_{\calF}\cLoc_{H}(\tilX^{\an})$ has isomorphic cone with the map $T_{\calF_{\partial}}\calC_{S}\to T_{\calF_{\partial}}\cLoc_{H}(\partial \tilX^{\an})$. By \eqref{partial T}, we get a distinguished triangle
\begin{equation}\label{TangF}
T_{\calF}\cLoc_{H}(\tilX^{\an}, C_{S})\to \cohog{*}{\tilX^{\an},\Ad(\calF)}[1]\to\cohog{1}{\partial\tilX^{\an}, \Ad(\calF_{\partial})}\to
\end{equation}
We may express $T_{\calF}\cLoc_{H}(\tilX^{\an}, C_{S})$ in more complex-geometric terms, without appealing to the boundary $\partial\tilX^{\an}$. Let $j:U\incl X$ and $i_{x}:\{x\}\incl X$ ($x\in S$) be the inclusions. We define $j_{!*}\Ad(\calF)$ to be the {\em non-derived} direct image of $\Ad(\calF)$ along $j$. Concretely, the stalk of $j_{!*}\Ad(\calF)$ at $x\in S$ is $(\frh)^{I_{x}}$ where $I_{x}$ is the inertia group at $x$. We have a distinguished triangle of complexes of sheaves on $X^{\an}$
\begin{equation*}
j_{!*}\Ad(\calF)\to j_{*}\Ad(\calF)\to \oplus_{x\in S}i_{x,*}\cohog{1}{\partial_{x},\Ad(\calF_{\partial})}[-1]\to
\end{equation*}
which then gives a distinguished triangle after taking cohomology
\begin{equation*}
\cohog{*}{X^{\an},j_{!*}\Ad(\calF)}[1]\to \cohog{*}{\tilX^{\an}, \Ad(\calF)}[1]\to \cohog{1}{\partial\tilX^{\an}, \Ad(\calF_{\partial})}\to
\end{equation*}
Comparing with \eqref{TangF}, we get a quasi-isomorphism
\begin{equation}\label{Tj!*}
T_{\calF}\cLoc_{H}(\tilX^{\an},C_{S})\cong\cohog{*}{X^{\an}, j_{!*}\Ad(\calF)}[1].
\end{equation}
We use the notation $j_{!*}$ because in this case $j_{!*}\Ad(\calF)[1]$ is also the middle extension of the perverse sheaf $\Ad(\calF)[1]$ from $U$ to $X$. 

Similarly, 
\begin{equation}\label{Tabj}
T_{\calA}\cLoc_{H^{\ab}}(\tilX^{\an},C^{\ab}_{S})\cong\cohog{*}{X^{\an}, j_{!*}\Ad(\calA)}[1].
\end{equation}
Since $H^{\ab}$ is abelian, $\Ad(\calA)$ is the constant sheaf $\frh^{\ab}=\Lie H^{\ab}$, and $j_{!*}\Ad(\calA)$ is the constant sheaf $\frh^{\ab}$ on $X^{\an}$.
Since $\cLoc_{H}(\tilX^{\an}, \calA, C_{S})$ is defined as the fiber of $\cLoc_{H}(\tilX^{\an},C_{S})$ over $\calA$, we have a distinguished triangle
\begin{equation*}
T_{\calF}\cLoc_{H}(\tilX^{\an},\calA, C_{S})\to T_{\calF}\cLoc_{H}(\tilX^{\an},C_{S})\to T_{\calA}\cLoc_{H^{\ab}}(\tilX^{\an},C^{\ab}_{S})\to
\end{equation*}
Comparing \eqref{Tj!*} and $\eqref{Tabj}$ we conclude that
\begin{equation*}
T_{\calF}\cLoc_{H}(\tilX^{\an},\calA, C_{S})\cong\cohog{*}{X^{\an}, j_{!*}\bAd(\calF)}[1]
\end{equation*}
where $\bAd(\calF)=\ker(\Ad(\calF)\to\Ad(\calA))$ is the local system obtained by evaluating the functor $\calF$ on the adjoint representation of $H$ on $\frh^{\der}=\ker(\frh\to\frh^{\ab})$.

As usual, classes in $\cohog{2}{X^{\an}, j_{!*}\bAd(\calF)}$ are obstructions to infinitesimally deform $\calF$ with prescribed local monodromy around $S$ and abelianization $\calF^{\ab}=\calA$. When $\cohog{2}{X^{\an}, j_{!*}\bAd(\calF)}=0$, the moduli stack $\cLoc_{H}(\tilX^{\an}, \calA, C_{S})$ is smooth at $\calF$ and $\cohog{1}{X^{\an}, j_{!*}\bAd(\calF)}$ is its Zariski tangent space at $\calF$. 

Since $\frh^{\der}$ carries an $\Ad(H)$-invariant symmetric bilinear form, $j_{!*}\bAd(\calF)$ is Verdier self-dual and $\cohog{1}{X^{\an}, j_{!*}\bAd(\calF)}$ is a symplectic space. Moreover, the vanishing of $\cohog{2}{X^{\an}, j_{!*}\bAd(\calF)}$ is equivalent to the vanishing of $\cohog{0}{X^{\an}, j_{!*}\bAd(\calF)}$, i.e., if $\calF$ does not have infinitesimal automorphisms then there is no obstruction to its infinitesimal deformation. 

\subsubsection{Cohomological rigidity}\label{sss:cr} Now back to $\ell$-adic local systems. We keep the notations from the beginning of this subsection.  Let $H$ be a connected reductive group over $\Qlbar$ and $\theta_{H}:\pi_{1}(U,u)\to\Aut(H)$ be a homomorphism that factors through a finite quotient $\Gamma$. Let $\calF\in\Loc_{H,\theta_{H}}(U)$, which correspond to a homomorphism $\rho:\pi_{1}(U,u)\to H(\Qlbar)\rtimes\Gamma$. Let $\frh^{\der}=\Lie H^{\der}$. Consider the composition
\begin{equation*}
\bAd(\rho):\pi_{1}(U,u)\xrightarrow{\rho} H(\Qlbar)\rtimes\Gamma\xrightarrow{\xi}\GL(\frh^{\der})
\end{equation*}
where $\xi(h,\gamma)=\Ad(h)\circ\theta_{H}(\gamma)$. This defines a local system on $U$ of rank equal to $\dim H^{\der}$, which we denote by $\bAd(\calF)$. Let $j:U\incl X$ be the open inclusion. The middle extension $j_{!*}\bAd(\calF)$ still makes sense in the $\ell$-adic setting. We arrive at the following definition.
\begin{defn}[extending Katz {\cite[\S5.0.1]{Katz}}]\label{coho rig}
An object $\calF\in\Loc_{H,\theta_{H}}(U)$ is called {\em cohomogically rigid}, if
\begin{equation*}
\Rig(\calF):=\cohog{1}{X,j_{!*}\bAd(\calF)}=0.
\end{equation*}
\end{defn}

\begin{remark} \begin{enumerate}
\item The space $\Rig(\calF)$ only depends on the generic stalk of $\calF$: if we shrink $U$ we get the same middle extension $j_{!*}\bAd(\calF)$ and hence the same vector space $\Rig(\calF)$.
\item The Killing form on $\frh^{\der}$ induces a symplectic form on $\Rig(\calF)$ with values  in $\cohog{2}{X,\Qlbar}\cong \Qlbar(-1)$. In particular, the dimension of $\Rig(\calF)$ is always even. 
\item According to the discussion preceding \S\ref{sss:cr}, assuming $\cohog{2}{X,j_{!*}\bAd(\calF)}$ vanishes as well, we should think of the vanishing of $\Rig(\calF)$ as saying that there are no nontrivial infinitesimal deformation of $\calF$ with prescribed local monodromy around $S$ and prescribed abelianization $\calF^{\ab}$. However, defining the moduli stack of $\ell$-adic local systems is much subtler, and this interpretation only serves as a heuristic.
\end{enumerate}
\end{remark}

The following lemma is easily verified, using the natural transformations $j_{!}\to j_{!*}\to j_{*}$.
\begin{lemma}\label{l:6term}
For any local system $\calL$ on $U$, we have an exact sequence
\begin{equation*}
0\to \cohog{0}{U,\calL}\to \oplus_{x\in S}(\calL_{x})^{I_{x}}\to\cohoc{1}{U,\calL}\to\cohog{1}{U,\calL}\to\oplus_{x\in S}(\calL_{x})_{I_{x}}(-1)\to\cohoc{2}{U,\calL}\to0.
\end{equation*}
Moreover, we have canonical isomorphisms
\begin{eqnarray*}
\cohog{0}{X,j_{!*}\calL}\cong\cohog{0}{U,\calL}\cong(\calL_{u})^{\pi_{1}(U,u)}\\
\cohog{1}{X,j_{!*}\calL}\cong\Im(\cohoc{1}{U,\calL}\to\cohog{1}{U,\calL})\\
\cohog{2}{X,j_{!*}\calL}\cong\cohoc{2}{U,\calL}\cong(\calL_{u})_{\pi_{1}(U,u)}(-1).
\end{eqnarray*}
\end{lemma}

We now give a numerical criterion for cohomological rigidity. 
 
\begin{prop}\label{p:numerics} Let $\calF\in\Loc_{H, \theta_{H}}(U)$. Then $\calF$ is cohomologically rigid if and only if
\begin{equation}\label{num rigid}
\frac{1}{2}\sum_{x\in S}a_{x}(\bAd(\calF))=(1-g_{X})\dim\frh^{\ad}-\dim\cohog{0}{U, \bAd(\calF)}.
\end{equation}
Here $a_{x}(\bAd(\calF))$ is the Artin conductor of $\bAd(\calF)$ at $x$ (see \S\ref{sss:cond}) and  $g_{X}$ is the genus of $X$.
\end{prop} 
\begin{proof} Let $\calL=\bAd(\calF)$. Recall the Grothendieck-Ogg-Shafarevich formula
\begin{equation*}
\chi_{c}(U,\calL):=\sum_{i=0}^{2}(-1)^{i}\dim \cohoc{i}{U,\calL}=\chi_{c}(U)\rank(\calL)-\sum_{x\in S}\Sw_{x}(\calL).
\end{equation*}
where $\chi_{c}(U)=-2g_{X}+2-\#S$. On the other hand, by the exact sequence in Lemma \ref{l:6term} we have
\begin{equation*}
\dim\cohoc{1}{X,j_{!*}\calL}=\dim\cohoc{1}{U,\calL}-\sum_{x\in S}\dim(\calL_{x})^{I_{x}}+\dim\cohog{0}{U,\calL}.
\end{equation*}
Combining these facts we get
\begin{equation*}
\dim\cohoc{1}{X,j_{!*}\calL}=\sum_{x\in S}\left(\dim\calL_{x}/\calL_{x}^{I_{x}}+\Sw_{x}(\calL)\right)+(2g_{X}-2)\rank(\calL)+\dim\cohoc{2}{U,\calL}+\dim\cohog{0}{U,\calL}.
\end{equation*}
Using the relation between Swan and Artin conductors, and using the self-duality of $\calL=\bAd(\calF)$, we get
\begin{equation*}
\dim\cohoc{1}{X,j_{!*}\calL}=\sum_{x\in S}a_{x}(\calL)+(2g_{X}-2)\rank(\calL)+2\dim\cohog{0}{U,\calL}.
\end{equation*}
Therefore the vanishing of $\cohoc{1}{X,j_{!*}\calL}$ is equivalent to the equality \eqref{num rigid}.
\end{proof}

From \eqref{num rigid} we see that cohomologically rigid $H$-local systems exist only when $g_{X}\leq1$. When $g_{X}=1$ and $\calF\in\Loc_{H,\theta_{H}}(U)$ is cohomologically rigid, $\Ad(\calF)$ must be everywhere unramified. There are very few such examples (see \cite[\S1.4]{Katz}). Most examples of rigid local systems are over open subsets of $\PP^{1}$. 

When $H=\GL_{n}$, the two notions of rigidity are related by the following theorem.

\begin{theorem}[Katz {\cite[Theorem 5.0.2]{Katz}}] For $X=\PP^{1}$, cohomological rigidity $\GL_{n}$-local systems (i.e., rank $n$ local systems) are also physical rigidity.
\end{theorem}

\begin{remark} Let $H$ be semisimple. An alternative approach to define the notion of rigidity for a $\theta_{H}$-twisted local system $\calF$ on $U$ over a finite field $k$ is by requiring the adjoint $L$-function of the Galois representation $\rho_{\calF}:\Gamma_{F}\to H(\Qlbar)\rtimes \Gamma$ to be trivial (constant function $1$). This is the approach taken by Gross in \cite{Gross-Trivial L}. When $\cohog{0}{U_{\kbar},\bAd(\calF)}=0$, triviality of the adjoint $L$-function of $\rho_{\calF}$ is equivalent to the cohomological rigidity of $\calF$.  
\end{remark}

\subsection{Rigid local systems of rank two}\label{ss:rigidlocGL2} 
Let $X=\PP^{1}_{k}$ where $k$ is an algebraically closed field with $\chk\neq2$. We shall classify cohomologically rigid $H=\GL_{2}$-local systems (or rank two local systems) over $U=X-S$ for some finite $S\subset|X|$. Let $\calF\in\Loc_{2}(U)$ be {\em irreducible} and cohomologically rigid. We may assume that $S$ is taken to be minimal in the sense that $\calF$ does ramify at all points $x\in S$. We have $\bAd(\calF)=\End^{0}(\calF)$, the local system of traceless endomorphisms of $\calF$. Irreducibility of $\calF$ implies that $\cohog{0}{U,\End^{0}(\calF)}=0$. The formula \eqref{num rigid} then reads
\begin{equation*}
\sum_{x\in S}a_{x}(\End^{0}(\calF))=6.
\end{equation*}

\begin{lemma}\label{l:r2 local} Let $\rho_{x}:I_{x}\to \GL_{2}(\Qlbar)$ be the local monodromy representation of $\calF$ at $x$.
\begin{enumerate}
\item If $\End^{0}(\calF)$ is tame at $x$, then $a_{x}(\End^{0}(\calF))=2$. If moreover $\det(\calF)$ is tame at $x$, then $\calF$ is tame at $x$.
\item If $\End^{0}(\calF)$ is wildly ramified at $x$, then $a_{x}(\End^{0}(\calF))\geq4$. If $a_{x}(\End^{0}(\calF))=4$ and $\det(\calF)$ is tame at $x$, then one of the two cases happens.
\begin{itemize}
\item Both breaks of $\calF$ at $x$ are $1/2$, hence $\Sw_{x}(\calF)=1$. In this case, the local monodromy $\rho_{x}: I_{x}\to \GL_{2}(\Qlbar)$ of $\calF$ is isomorphic to the induction $\Ind_{J_{x}}^{I_{x}}(\chi)$ where $J_{x}\lhd I_{x}$ is the unique subgroup of index two and $\chi:J_{x}\to \Qlbar^{\times}$ is a character.
\item The local monodromy $\rho_{x}:I_{x}\to \GL_{2}(\Qlbar)$ is the direct sum of two characters $\chi_{1},\chi_{2}:I_{x}\to\Qlbar^{\times}$ with $\Sw(\chi_{1})=\Sw(\chi_{2})=1$.
\end{itemize}
\end{enumerate}
\end{lemma}
\begin{proof} Let $\rho^{\der}_{x}:I_{x}\xrightarrow{\rho_{x}}\GL_{2}(\Qlbar)\to\GL(\sl_{2})$ be the local monodromy representation of $\End^{0}(\calF)$ at $x$. The image of the wild inertia $I^{w}_{x}$ under $\rho_{x}$ is a $p$-group, hence lies in a maximal torus of $\GL_{2}$. Therefore $\rho_{x}|_{I^{w}_{x}}\cong\a\oplus\b$ for characters $\a,\b: I^{w}_{x}\to\Qlbar^{\times}$, and $\rho^{\der}_{x}|_{I^{w}_{x}}\cong\a\b^{-1}\oplus\a^{-1}\b\oplus1$.

(1) If $\rho^{\der}_{x}$ is tame, then $\a=\b$. A topological generator $\zeta\in I^{t}_{x}$ maps to a non-identity element in $\PGL_{2}$, whose centralizer is necessarily one-dimensional. Hence $\dim(\sl_{2})^{I_{x}}=1$ and of course $\Sw_{x}(\End^{0}(\calF))=0$. Therefore $a_{x}(\End^{0}(\calF))=3-\dim(\sl_{2})^{I_{x}}+\Sw_{x}(\End^{0}(\calF))=2$. 

If we assume $\det(\rho_{x})$ is tame, then $\a\b=1$. But we also have $\a=\b$, which forces $\a$ and $\b$ both have order at most two, hence trivial. Therefore $\rho_{x}$ is tame.

(2) If $\rho^{\der}_{x}$ is wildly ramified, then $\a\neq\b$. Therefore there is a unique maximal torus $T\subset\GL_{2}$ containing $\rho_{x}(I^{w}_{x})$. We have $\dim(\sl_{2})^{I^{w}_{x}}=1$. The image of $\rho_{x}(I_{x})$ should normalize $\rho_{x}(I^{w}_{x})$, hence it lies in the normalizer $N_{\GL_{2}}(T)$, which gives a map $I^{t}_{x}\to N_{\GL_{2}}(T)/T=\{\pm1\}$. 

Suppose $I^{t}_{x}$ maps onto $N_{\GL_{2}}(T)/T$. Then $(\sl_{2})^{I_{x}}=0$ and since $\Sw_{x}(\End^{0}(\calF))\geq1$, we have $a_{x}(\End^{0}(\calF))=3-\dim(\sl_{2})^{I_{x}}+\Sw_{x}(\End^{0}(\calF))\geq4$. When equality holds, we have $\Sw_{x}(\End^{0}(\calF))=1$. The action of $I^{w}_{x}$ on $\sl_{2}$ is the sum of a trivial character and $\a\b^{-1}$ and $\a^{-1}\b$. Therefore $\a\b^{-1}$ has break $1/2$. Suppose further that $\det(\calF)$ is tame, then $\a\b=1$. Hence $\a\b^{-1}=\a^{2}$ has break $1/2$, which means $\a$ and $\b$ both have break $1/2$. The unique subgroup $J_{x}\lhd I_{x}$ of index two stabilizes the two eigenlines of $\rho_{x}|_{I^{w}_{x}}$, and therefore $\rho_{x}\cong \Ind^{I_{x}}_{J_{x}}(\chi)$ where $\chi$ is the character of $J_{x}$ by which it acts on the $\a$-eigenline of $I^{w}_{x}$.

Suppose $I^{t}_{x}$ maps trivially to $N_{\GL_{2}}(T)/T$. This means $\rho_{x}(I_{x})\subset T$ hence $\rho_{x}\cong\chi_{1}\oplus\chi_{2}$ for characters $\chi_{1},\chi_{2}:I_{x}\to\Qlbar^{\times}$ extending $\a$ and $\b$. In this case $\dim(\sl_{2})^{I_{x}}=1$, and $\Sw_{x}(\End^{0}(\calF))=\Sw(\chi_{1}\chi_{2}^{-1})+\Sw(\chi_{1}^{-1}\chi_{2})\geq2$ (since $\a\neq\b$). Therefore $a_{x}(\End^{0}(\calF))=3-\dim(\sl_{2})^{I_{x}}+\Sw_{x}(\End^{0}(\calF))\geq4$. When equality holds, $\a\b^{-1}$ has break $1$. If moreover $\det(\calF)$ is tame, we have $\a=\b^{-1}$, therefore the break of $\a^{2}$ is 1, hence the breaks of both $\a$ and $\b$ are 1. 
\end{proof}

\begin{prop}\label{p:r2} Let $\calF\in\Loc_{2}(U)$ be an irreducible cohomologically rigid rank two local systems on $U=\PP^{1}-S$ (and $S$ is chosen minimally). Assume $0,\infty\in S$. Then up to an operation $\calF\mapsto\calF\otimes\calL$ for some rank one local system on $U$, and up to an automorphism of $\Gm=\PP^{1}-\{0,\infty\}$, one of the following situations happens.
\begin{enumerate}
\item $S=\{0,1,\infty\}$, $\calF$ is tamely ramified at each $x\in S$ and its monodromy at $1$ is a pseudo-reflection.
\item $S=\{0,\infty\}$, $\calF$ is tame at $0$ and $\Sw_{\infty}(\calF)=1$ with two breaks equal to $1/2$.
\item $S=\{0,\infty\}$, $\calF$ is tame at $0$ and $\Sw_{\infty}(\calF)=1$ with one break equal to $1$ and another break equal to $0$.
\end{enumerate}
\end{prop}
\begin{proof}
The determinant $\det\calF$ is a continuous character of $\chi:\pi_{1}(U)\to\Qlbar^{\times}$, and we can write it as $\chi_{p}\chi^{p}$ where $\chi_{p}$ is the $p$-power part of $\chi$ and $\chi^{p}$ is tame. Since $p>2$, $\chi_{p}$ has a unique square root which gives a rank one local system $\calL$ over $U$. The local system $\calF'=\calF\otimes\calL^{-1}$ is still irreducible and cohomologically rigid, and now its determinant is tame. Therefore we may assume that $\det(\calF)$ is tame.

By Lemma \ref{l:r2 local}, when $\#S\geq2$, there two cases.

(1) $\#S=3$ and $a_{x}(\End^{0}(\calF))=2$ for all $x\in S$. Up to an automorphism of $\Gm$ we may assume $S=\{0,1,\infty\}$. Since $\det(\calF)$ is tame, Lemma \ref{l:r2 local} implies that $\calF$ is tame. We may find a rank one local system $\calL$ on $\PP^{1}-\{0,1\}$ whose local monodromy at $1$ is one of the eigenvalues of the local monodromy of $\calF$ at $1$. Then $\calF\otimes\calL^{-1}$ is tame and its local monodromy at $1$ is a pseudo-reflection.

(2) $\#S=2$ (hence $S=\{0,\infty\}$). By Lemma \ref{l:r2 local}, up to switching  $0$ and $\infty$, we may assume that $\calF$ is $a_{0}(\End^{0}(\calF))=2$ and $a_{\infty}(\End^{0}(\calF))=4$. By Lemma \ref{l:r2 local}(1),  $\calF$ is tame at $0$. There are two cases for $\calF$ at $\infty$ according to Lemma \ref{l:r2 local}(2):
\begin{itemize}
\item Both breaks of $\calF$ at $\infty$ are $1/2$ and hence $\Sw_{\infty}(\calF)=1$.
\item Both breaks of $\calF$ at $\infty$ are $1$. The action of $I_{\infty}$ on $\Qlbar^{2}$ is a direct sum of two characters $\chi_{1}, \chi_{2}:I_{\infty}\to\Qlbar^{\times}$. The character $\chi_{1}$ extends to a rank one local system $\calL_{1}$ on $\Gm$ with tame local monodromy at $\infty$ (Katz's canonical extension, see \cite{Katz-canon}).  Then $\calF\otimes\calL^{-1}_{1}$ is still tame at $0$, with breaks $0$ and $1$ at $\infty$.
\end{itemize}
\end{proof}

\begin{remark} The three cases in Proposition \ref{p:r2} are the irreducible hypergeometric sheaves of rank two constructed by Katz \cite[\S8.2]{Katz-DE}. In fact, if $\calF$ is one of the local systems in Proposition \ref{p:r2}, it is easy to see that $\chi_{c}(\Gm, j_{!*}\calF[1])=1$ in all three cases (where $j:U\incl\Gm$). By \cite[Theorem 8.5.3]{Katz-DE}, $\calF$ is a hypergeometric sheaf. 

As we shall see in \S\ref{ss:desc eigen}, the three local systems in Proposition \ref{p:r2} are the Langlands parameters (restricted to $I_{F}$) of the automorphic representations considered in \S\ref{sss:GL2S3}, \S\ref{sss:GL2S2I} and \S\ref{sss:GL2S2II} respectively.
\end{remark}

\subsection{Rigidity in Inverse Galois Theory}\label{ss:InvGal} 
It is instructive to compare the notion of rigidity for local systems with the notion of a rigid tuple in inverse Galois theory. We give a quick review following \cite[Chapter 8]{Serre-Galois}. 

\begin{defn} Let $H$ be a finite group with trivial center. A tuple of conjugacy classes $(C_{1}, C_{2}, \cdots, C_{n})$ in $H$ is called {\em (strictly) rigid}, if
\begin{itemize}
\item The equation 
\begin{equation}\label{ggg}
g_{1}g_{2}\cdots g_{n}=1
\end{equation}
has a solution with $g_{i}\in C_{i}$, and the solution is unique up to simultaneous $H$-conjugacy;
\item For any solutions $(g_{1},\cdots, g_{n})$ of \eqref{ggg}, $\{g_{i}\}_{i=1,\cdots, n}$ generate $H$.
\end{itemize}
\end{defn}

The connection between rigid tuples and local systems is given by the following theorem. Let $S=\{P_{1},\cdots, P_{n}\}\subset\PP^{1}(\QQ)$, and let $U=\PP^{1}_{\QQ}-S$.

\begin{theorem}[Belyi, Fried, Matzat, Shih, and Thompson] Let $(C_{1},\cdots, C_{n})$ be a rigid tuple in $H$.   Then up to isomorphism there is a unique connected unramified Galois $H$-cover $\pi: Y\to U\otimes_{\QQ}{\Qbar}$ such that a topological generator of the (tame) inertia group at $P_{i}$ acts on $Y$ as an element in $C_{i}$.

Furthermore, if each $C_{i}$ is rational (i.e., $C_{i}$ takes rational values for all irreducible characters of $H$), then the $H$-cover $Y\to U\otimes_{\QQ}{\Qbar}$ is defined over $\QQ$. 
\end{theorem}

From the above theorem we see that the notion of a rigid tuple is an analog of physical rigidity for $H$-local systems when the algebraic group $H$ is a finite group.

Rigid tuples combined with the Hilbert irreducibility theorem solves the inverse Galois problem for $H$.
\begin{cor} Suppose there exists a rational rigid tuple in $H$, then $H$ can be realized as $\Gal(K/\QQ)$ for some Galois number field $K/\QQ$.
\end{cor}
For a comprehensive survey on finite groups that are realized as Galois groups over $\QQ$ using rigid tuples, we refer the readers to the book \cite{MM} by Malle and Matzat.

\section{Calculus of geometric Hecke operators}\label{s:Hk}  

In this section, we will connect the notion of rigidity for automorphic data and the notion of rigidity for local systems together. Guided by the Langlands correspondence,  we formulate a conjectures in \S\ref{ss:conj rigid}. Then we review some basic techniques from the geometric Langlands program in \S\ref{ss:Hk}-\S\ref{ss:Hk eigen} such as the geometric Hecke operators. The main results are Theorem \ref{th:eigen} and Proposition \ref{p:desc eigen} which constructs the Hecke eigen local system under the rigidity assumption. Examples in $\GL_{2}$ from \S\ref{ss:GL2} and \S\ref{ss:rigidlocGL2} are finally connected to one another in \S\ref{ss:finalGL2}. 

\subsection{Langlands correspondence for rigid objects}\label{ss:conj rigid} 
We are in the situation of \S\ref{s:auto}, with $k$ a finite field. 

\subsubsection{The Langlands correspondence for groups over function fields} Let $S\subset |X|$ be a finite set of places containing the ramification locus of $\theta_{X}$, and $U=X-S$. Let $U'$ be the preimage of $U$ in $X'$. Let $W(U,u)\subset\pi_{1}(U,u)$ be the Weil group of $U$ with respect to the base point $u$, i.e., it is the preimage of $\Frob^{\ZZ}$ under the homomorphism $\pi_{1}(U,u)\to\Gk$.

Recall that the Langlands correspondence predicts that to an automorphic representation $\pi$ of $G(\AA_{F})$, unramified outside $S$, one should be able to attach a continuous cocycle $\rho_{\pi}:W(U,u)\to \dG(\Qlbar)$ (with respect to the pinned action $\htheta$ of $\Gamma=\Gal(U'/U)$ on $\dG$) up to $\dG$-conjugacy. Base change to $\kbar$, one should be able attach a $\htheta$-twisted $\dG$-local system $\calF_{\pi}\in\Loc_{\dG,\htheta}(U_{\kbar})$ to $\pi$. The assignment $\pi\mapsto\rho_{\pi}$ should be such that the Satake parameter of $\pi_{x}$ for any $x\notin S$ matches the semisimple part of the conjugacy class of $\rho_{\pi}(\Frob_{x})$ in $\dG$. By the work of V.Lafforgue \cite{VLaff}, when $\pi$ is cuspidal, such a continuous cocycle $\rho_{\pi}$ (hence $\calF_{\pi}$) exists. 

We expect that when $\pi$ is rigid, the local system $\calF_{\pi}$ is also rigid.  More precisely we propose the following conjecture.

\begin{conj}\label{conj:rigid} Let $(\Om, \bK_{S}, \cK_{S},\iota_{S})$ be a weakly rigid geometric automorphic datum, with the corresponding restricted automorphic datum $(\onat,K_{S},\chi_{S})$. Then
\begin{enumerate}
\item  For any $(\onat,K_{S},\chi_{S})$-typical automorphic representation $\pi$, the local system $\calF_{\pi}\in\Loc_{\dG,\htheta}(U_{\kbar})$ attached to it via the Langlands correspondence is cohomologically rigid.
\item If moreover $\bAd(\calF_{\pi})$ (a local system over $U_{\kbar}$ of rank equal to $\dim\dG^{\der}$) does not have nonzero global sections, and that the generic stabilizers of $\Bun_{\cZ}(\bK_{Z,S})$ on $\Bun_{\cG}(\bK_{S})$ are finite, then
\begin{equation}\label{da}
d(\bK^{\ad}_{x})=a_{x}(\bAd(\calF_{\pi})).
\end{equation}
For the definition of $d(\bK^{\ad}_{x})$, see \S\ref{sss:numrig}; for the Artin conductor $a_{x}(\bAd(\calF_{\pi}))$ see \S\ref{sss:cond}.
\end{enumerate}
\end{conj}

\begin{remark} In the situation of Conjecture \ref{conj:rigid}(2), the equality \eqref{da} at all $x\in S$ implies the cohomological rigidity of $\calF_{\pi}$ by the numerical criterion Proposition \ref{p:numerics}. In fact, Conjecture \ref{conj:rigid}(2) was motivated by the similarity between the equations \eqref{cond K} and \eqref{num rigid}, the first one being a numerical condition for the weak rigidity of a geometric automorphic datum and the second a numerical criterion for the cohomological rigidity of a local system. For many known examples of weakly rigid geometric automorphic data, Conjecture \ref{conj:rigid}(2) can be proved for those points $x\in S$ at which $\calF_{\pi}$ is tamely ramified, using the techniques in \cite[\S4]{Y-GenKloo}. See also \cite[\S9]{Y-GenKloo}, where Conjecture \ref{conj:rigid} is fully verified for the local systems constructed in \cite{Y-motive}. When $\calF_{\pi}$ is wildly ramified at $x$, \cite[\S5]{HNY} also verifies \eqref{da} in the examples known as Kloosterman sheaves. 
\end{remark}

\subsection{Geometric Hecke operators}\label{ss:Hk}  
From now on until \S\ref{sss:rat}, $k$ is an algebraically closed field. 

\subsubsection{The geometric Satake equivalence}\label{ss:Sat} Let $L\GG$ be the loop group of $\GG$: this is an ind-scheme representing the functor $R\mapsto \GG(R[[t]])$. Let $L^{+}\GG$ the positive loops of $\GG$: this is a scheme (not of finite type) representing the functor $R\mapsto \GG(R[[t]])$. The fppf quotient $\Gr=L\GG/L^{+}\GG$ is called the {\em affine Grassmannian} of $\GG$.  Then $L^{+}\GG$ acts on $\Gr$ via left translation. The $L^{+}\GG$-orbits on $\Gr$ are indexed by  dominant coweights $\l\in\xcoch(\TT)^+$. The orbit containing the element $t^{\l}\in \TT(k((t)))$ is denoted by $\Gr_{\l}$ and its closure is denoted by $\Gr_{\leq\l}$. We have $\dim\Gr_{\l}=\jiao{2\rho,\l}$, where $2\rho$ denotes the sum of positive roots in $\GG$. The reduced scheme structure on $\Gr_{\leq\l}$ is a projective variety called the affine Schubert variety attached to $\l$. We denote the intersection complex of $\Gr_{\leq\l}$ by $\IC_{\l}$: this is the middle extension of the shifted constant sheaf $\Qlbar[\jiao{2\rho,\l}]$ on $\Gr_{\l}$.

The {\em Satake category} $\Sat=\Perv_{L^{+}G}(\Gr, \Qlbar)$ is the category of $L^{+}G$-equivariant perverse $\Qlbar$-sheaves on $\Gr$ supported on $\Gr_{\leq\l}$ for some $\l$. In \cite{Lu}, \cite{Ginz} and \cite{MV}, it is shown that $\Sat$ carries a natural tensor structure such that the global cohomology functor $h=\cohog{*}{\Gr,-}:\Sat^\geom\to\Vect$ is a fiber functor. It is also shown that the Tannaka dual group of the tensor category $\Sat$ is isomorphic to the Langlands dual group $\dG$. The Tannakian formalism gives the {\em geometric Satake equivalence} of tensor categories
\begin{equation*}
\Sat\cong\Rep(\dG).
\end{equation*}

\subsubsection{Hecke correspondence} We are in the situation of \S\ref{ss:geom data}. Recall an integral model  $\cG$ of $G$ was fixed as in \S\ref{sss:cG}. Let $S\subset|X|$ be a finite set containing the ramification locus $S_{\theta}$ of $\theta_{X}$. Let $U=X-S$ and $U'=\theta^{-1}_{X}(U)$. Then $U'\to U$ is a finite \'etale Galois cover with Galois group $\Gamma$. The base change $\cG\times_{U}U'$ is the constant group scheme $\GG\times U'$. 

Fix a geometric automorphic datum $(\Om, \bK_{S}, \cK_{S},\iota_{S})$.  Let $\bK^{+}_{S}$ be a smaller level as in \S\ref{sss:Kplus}. To alleviate notation, we write
\begin{equation*}
\Bun:=\Bun_{\bG}(\bK_{S}); \quad \Bun^{+}:=\Bun_{\bG}(\bK^{+}_{S}).
\end{equation*}
Consider the following diagram
\begin{equation}\label{Hk diagram}
\xymatrix{& \Hk_{U'}\ar[dl]_{\oll{h}}\ar[dr]^{\orr{h}}\ar[rr]^{\pi} & & U' \\
\Bun^{+} & & \Bun^{+}}
\end{equation}
We explain the meaning of $\Hk_{U'}$ and the various morphisms. The stack $\Hk_{U'}$ classifies the data $(x', \calE_{1},\calE_{2}, \tau)$ where $x'\in U'$ has image $x\in U$, $\calE_{1},\calE_{2}\in\Bun^{+}$ and $\tau:\calE_{1}|_{X-\{x\}}\isom\calE_{2}|_{X-\{x\}}$ is an isomorphism of $\cG$-torsors over $X-\{x\}$ preserving the $\bK^{+}_{y}$-level structures at $y\in S$. The morphisms $\oll{h}$, $\orr{h}$ and $\pi'$ send $(x',\calE_{1},\calE_{2},\tau)$ to $\calE_{1},\calE_{2}$ and $x'$ respectively.

Let $R$ be a $k$-algebra and $x'\in U'(R)$ with image $x\in U(R)$. Let $\calO_{x'}$ be the formal completion of $X'\otimes_{k}R$ along the graph of $x'$. After localizing $R$ we may choose a continuous isomorphism $\alpha: \calO_{x'}\cong R[[t]]$.  Denote the  preimage of $x'$ under $\pi'$ by $\Hk_{x'}$, which in fact only depends on $x$. For a point $(x',\calE_{1},\calE_{2},\tau)\in\Hk_{x'}$, if we fix trivializations of $\calE_{1}$ and $\calE_{2}$ over $\Spec\calO_{x'}\cong\Spec R[[t]]$, the isomorphisms $\tau$ restricted to $\Spec R((t))$ is an isomorphism between the trivial $\GG$-torsors over $\Spec R((t))$, hence given by a point $g_{\tau}\in G(R((t)))$. Changing the trivializations of $\calE_{1}|_{\Spec\calO_{x}}$ and $\calE_{2}|_{\Spec\calO_{x}}$ will result in left and right multiplication of $g_{\tau}$ by an element in $\GG(R[[t]])$ (here we use the fact that $\cG$ becomes a constant group scheme over $U'$). Thus we have an evaluation morphism
\begin{equation}\label{evx}
\ev_{x'}: \Hk_{x'}\to (L^{+}\GG\backslash L\GG/L^{+}\GG)\otimes_{k}R.
\end{equation}
Changing the isomorphism $\alpha$ will change the morphism $\ev_{x'}$ by composing with the action of an element in $\Aut(R[[t]])$ on the target. This construction gives a morphism
\begin{equation}\label{ev Hk}
\ev: \Hk_{U'}\to \left[\frac{L^{+}\GG\backslash L\GG/L^{+}\GG}{\Aut_{k[[t]]}}\right].
\end{equation}
Here $\Aut_{k[[t]]}$ is the pro-algebraic group over $k$ whose $R$-points is the group of continuous $R$-linear ring automorphisms of $R[[t]]$ (with the $t$-adic topology), and it acts on $L\GG$ and $L^{+}\GG$. 

\subsubsection{The geometric Hecke operators}
For each object $V\in\Rep(\dG)$, the corresponding object $\IC_{V}\in\Sat$ under the geometric Satake equivalence is automatically $\Aut_{k[[t]]}$-equivariant, and it defines a complex on the quotient stack $\left[\frac{L^{+}\GG\backslash L\GG/L^{+}\GG}{\Aut_{k[[t]]}}\right]$ which we still denote by $\IC_{V}$.  Let $\IC^{\Hk}_{V}$ be the pullback of $\IC_{V}$ to $\Hk_{U'}$ along the morphism $\ev$ in \eqref{ev Hk}. We define a functor
\begin{eqnarray*}
\frT_{V}: D^{b}_{(\bM_{S},\cK_{S,\Om})}(\Bun^{+})&\to& D^{b}_{(\bM_{S},\calK_{S,\Om})}(U'\times\Bun^{+})\\
\calA &\mapsto& (\pi\times\orr{h})_{!}(\oll{h}^{*}\calA\otimes\IC^{\Hk}_{V}).
\end{eqnarray*}

The formation $V\mapsto \frT_{V}$ is also $\Gamma:=\Gal(U'/U)$-equivariant in the following sense. The group $\Gamma$ acts on $\Hk_{U'}$ because $\Hk_{U'}$ descends to $U$. It also acts on $\Rep(\dG)$:  $\gamma\in\Gamma$ sends a representation $\rho:\dG\to\GL(V)$ to the representation $\dG\xrightarrow{\htheta(\gamma)}\dG\xrightarrow{\rho}\GL(V)$ which we denote by $V^{\gamma}$. Then there is a natural isomorphism of functors
\begin{equation*}
\alpha_{\gamma,V}: \frT_{V^{\gamma}}\cong (\id_{\Bun^{+}}\times\gamma)^{*}\circ\frT_{V}.
\end{equation*}
For $\gamma_{1},\gamma_{2}\in\Gamma$, we have
\begin{equation*}
\alpha_{\gamma_{1}\circ\gamma_{2},V}=(\id_{\Bun^{+}}\times\gamma_{2})^{*}\alpha_{\gamma_{1},V}\circ \alpha_{\gamma_{2},V^{\gamma_{1}}}.
\end{equation*}
The isomorphisms $\alpha_{\gamma,V}$ are obtained from the fact that the morphism $\ev$ in \eqref{ev Hk} is $\Gamma$-equivariant, and the fact that the Satake equivalence $\Sat\cong\Rep(\dG)$ is also $\Gamma$-equivariant (where $\Gamma$ acts on both $\GG$ and $\dG$ by pinned automorphisms). 

To spell out how $\frT_{V}$ is compatible with the tensor structure on $\Rep(\dG)$, it is best to consider more general Hecke operators.

\subsubsection{Iterated Hecke operators}\label{sss:iterated} We refer to \cite[\S2.4-2.7]{GaDJ} for details. For each finite set $I$ we may consider the stack $\Hk_{U'^{I}}$ over $\Bun^{+}\times \Bun^{+}\times (U')^{I}$ classifying the data $(x,\calE_{1},\calE_{2},\tau)$ where $x: I\to U'$ with graph $x_{I}\subset U$, $\calE_{1},\calE_{2}\in\Bun^{+}$ and $\tau:\calE_{1}|_{X-x_{I}}\cong\cE_{2}|_{X-x_{I}}$ is an isomorphism of $\cG$-torsors  preserving level structures. Let $\oll{h}_{I}, \orr{h}_{I}:\Hk_{U'^{I}}\to\Bun^{+}$ and $\pi_{I}:\Hk_{U'^{I}}\to U'^{I}$ be the projections. For any representation $V_{I}\in\Rep(\dG^{I})$, we have an object $\IC^{\Hk}_{V_{I}}$ on $\Hk_{U'^{I}}$ defined using evaluation maps similar to \eqref{ev Hk} indexed by $I$. For any scheme $Y$ over $k$, we may introduce the functor
\begin{eqnarray*}
\frT_{V_{I}, Y}: D^{b}_{(\bM_{S},\cK_{S,\Om})}(Y\times \Bun^{+})&\to& D^{b}_{(\bM_{S},\calK_{S,\Om})}(U'^{I}\times Y\times \Bun^{+})\\
\calA&\mapsto&(\pi_{I}\times p_{Y}\times \orr{h}_{I})_{!}((p_{Y}\times\oll{h}_{I})^{*}\calA\otimes\IC^{\Hk}_{V_{I}})
\end{eqnarray*}
where $p_{Y}:\Hk_{U'^{I}}\times Y\to Y$ is the projection. These functors are functorial in $V_{I}$, and satisfy the following factorization properties.
\begin{enumerate}
\item\label{Tunit} For the trivial representation $\triv\in\Rep(\dG^{I})$, there is a canonical isomorphism $\frT_{\triv, Y}(\calA)\cong\calA\boxtimes\Qlbar$.
\item\label{Tass} For two finite sets $I,J$ and $V_{I}\in\Rep(\dG^{I}), V_{J}\in\Rep(\dG^{J})$, we have a natural isomorphism of functors
\begin{equation}\label{composeTV}
\frT_{V_{I}\boxtimes V_{J},Y}=\frT_{V_{I}, U'^{J}\times Y}\circ \frT_{V_{J}, Y}
\end{equation}
satisfying obvious associativity and unit conditions.
\item\label{Tfactor} For any surjection $\varphi: J\surj I$ of finite sets, we have a corresponding diagonal map $\Delta(\varphi): U'^{I}\incl U'^{J}$. We also have the diagonal map $\Delta(\varphi):\dG^{I}\incl\dG^{J}$ which allows us to view $V_{J}$ as a representation $\Delta(\varphi)^{*}V_{J}\in\Rep(\dG^{I})$. Then there is a natural isomorphism of functors
\begin{equation}\label{factorizeTV}
(\id_{\Bun^{+}}\times \Delta(\varphi)\times\id_{Y})^{*}\frT_{V_{J},Y}\cong\frT_{\Delta(\varphi)^{*}V_{J}, Y}
\end{equation}
compatible with the composition of surjections. 
\end{enumerate}
The first property above is almost a tautology and the second one uses the definition of the convolution product in the Satake category. There are also compatibilities among \eqref{Tunit}-\eqref{Tfactor}, which we do not spell out.

Moreover, the functors $V_{I}\mapsto\frT_{V_{I},Y}$ are $\Gamma$-equivariant as in the case of $\frT_{V}$.

When $Y$ is a point, we write $\frT_{V_{I}, Y}$ as $\frT_{V_{I}}$. Taking $Y$ to be a point, $J=\{1,2\}\surj \{1\}$, letting $V_{J}=V_{1}\boxtimes V_{2}\in\Rep(\dG^{2})$, and combining the isomorphisms \eqref{composeTV} and \eqref{factorizeTV}, we get a canonical isomorphism
\begin{equation}\label{compose TT}
(\frT_{V_{1}, U'}\circ\frT_{V_{2}})|_{\Bun^{+}\times\Delta(U')}\cong \frT_{V_{1}\otimes V_{2}}: D^{b}_{(\bM_{S},\cK_{S,\Om})}(\Bun^{+})\to D^{b}_{(\bM_{S},\calK_{S,\Om})}(\Bun^{+}\times U')
\end{equation}
which is compatible with the associativity of the tensor product $V_{1}\otimes V_{2}\otimes V_{3}$ in $\Rep(\dG)$ in the obvious sense.

\subsubsection{Compatibility with the Kottwitz homomorphism} For $\alpha\in\xch(Z\dG)_{I_{F}}$, let $\Bun_{\alpha}$ (resp. $\Bun^{+}_{\alpha}$) be the preimage of $\alpha$ under the Kottwitz morphism $\kappa:\Bun\to\xch(Z\dG)_{I_{F}}$ (resp. $\kappa^{+}:\Bun^{+}\to\xch(Z\dG)_{I_{F}}$).  In the definition of $D_{\cG,\Om}(\bK_{S}, \cK_{S})$ in terms of $D^{b}_{(\bM_{S}, \cK_{S,\Om})}(\Bun^{+})$, we may work with $\Bun^{+}_{\alpha}$ instead of the whole $\Bun^{+}$ (because $\bM_{S}$ preserves each $\Bun^{+}_{\alpha}$), and define
\begin{equation*}
D_{\alpha}:=D^{b}_{(\bM_{S}, \cK_{S,\Om})}(\Bun^{+}_{\alpha}).
\end{equation*}
Then $D_{\cG,\Om}(\bK_{S}, \cK_{S})$ is the Cartesian product of $D_{\alpha}$ over $\alpha\in\xch(Z\dG)_{I_{F}}$.
For any scheme $Y$ over $k$, we also define
\begin{equation*}
D_{\a}(Y):=D^{b}_{(\bM_{S}, \cK_{S,\Om})}(Y\times \Bun^{+}_{\alpha}).
\end{equation*}
where $\bM_{S}$ acts trivially on $Y$.

For each $x\in X$, we have the local Kottwitz homomorphism $\kappa_{x}: L_{x}G\to \xch(Z\dG)_{I_{x}}$ as in \eqref{local Kott}. In particular, when $x\in U$, $I_{x}$ acts trivially on $\xch(Z\dG)$, and the map $\kappa_{x}$ factors as $L^{+}_{x}G\backslash L_{x}G/L^{+}_{x}G\to\xch(Z\dG)$. However this map is canonical only after choosing a point $x'\in U'$ over $x$. Therefore there is a well-defined map  $\kappa_{\Hk}: \Hk_{U'}\to \xch(Z\dG)$ (while for $\Hk_{U}$ only the map to $\xch(Z\dG)_{I_{F}}$ is well-defined). Similarly, for the iterated Hecke correspondence, we have a map $\kappa_{\Hk,I}:\Hk_{U'^{I}}\to\xch(Z\dG)^{I}\to\xch(Z\dG)$, the last map being the addition. For $\nu\in\xch(Z\dG)$, let $\Hk_{U'^{I},\nu}$ be the preimage of $\nu$ under $\kappa_{\Hk,I}$.

For each $\nu\in\xch(Z\dG)$, let $\Rep(\dG;\nu)$ be the full subcategory of $\Rep(\dG)$ consisting of those $V\in\Rep(\dG)$ with central character $\nu\in\xch(Z\dG)$. Similarly we define $\Rep(\dG^{I}; \nu)$ by looking at how the diagonal copy of $Z\dG$ acts. For $V\in\Rep(\dG;\nu)$, the support of $\IC_{V}\in\Sat$ is on the preimage of $\nu$ under $L^{+}\GG\backslash L\GG/L^{+}\GG\to\xch(Z\dG)$. From this we conclude the following.

\begin{lemma} For $V_{I}\in\Rep(\dG^{I};\nu)$, $\IC^{\Hk}_{V_{I}}$ is supported on $\Hk_{U'^{I},\nu}$, and the geometric Hecke operators restricts to functors
\begin{equation}\label{TVa}
\frT_{V_{I},\a, Y}:D_{\alpha}(Y)\to D_{\alpha+\overline{\nu}}(U'^{I}\times Y) \quad \forall \alpha\in\xch(Z\dG)_{I_{F}}
\end{equation}
where $\overline{\nu}$ is the image of $\nu$ in the quotient $\xch(Z\dG)_{I_{F}}$.
\end{lemma}

\subsection{Hecke eigen properties}\label{ss:Hk eigen}  
\begin{defn}\label{def:eigen} A {\em Hecke eigensheaf} for the geometric automorphic datum $(\Om,\bK_{S}, \cK_{S}, \iota_{S})$ is a triple $(\calA, \calF, \varphi)$ where
\begin{itemize}
\item $\calA\in D_{\cG,\Om}(\bK_{S}, \calK_{S})$.
\item $\calF\in \Loc_{\dG,\htheta}(U)$ called the {\em Hecke eigen local system}. We think of $\calF$ as a $\dG$-local system over $U'$ together with a descent datum $\{\delta_{\gamma, V}: \calF_{V^{\gamma}}\isom\gamma^{*}\calF_{V}\}_{\gamma\in\Gamma, V\in\Rep(\dG, \Qlbar)}$ as in \S\ref{sss:twist loc}.
\item $\varphi$ is a collection of isomorphisms, one for each $V\in\Rep(\dG)$ and functorial in $V$
\begin{equation*}
\varphi_{V}: \frT_{V}(\calA)\cong\calF_{V}\boxtimes\calA.
\end{equation*}
\end{itemize}
The triple $(\calA, \calF, \varphi)$ must satisfy the following conditions.
\begin{enumerate}
\item\label{Etriv} For the trivial representation $\triv=\Qlbar$ of $\dG$, we have the canonical isomorphism $\frT_{\triv}(\calA)\cong \Qlbar\boxtimes\calA$ from \S\ref{sss:iterated}(1) and also $\calF_{\triv}=\Qlbar$. We require that $\varphi_{V}$ be the identity map of $\Qlbar\boxtimes\calA$.
\item For $V_{1}, V_{2}\in\Rep(\dG)$, $\varphi_{V_{1}\otimes V_{2}}$ is equal to the composition
\begin{eqnarray*}
&\frT_{V_{1}\otimes V_{2}}(\calA)&\xrightarrow{\eqref{compose TT}}(\frT_{V_{1},U'}\circ\frT_{V_{2}}(\calA))|_{\Bun^{+}\times \Delta(U')}
\xrightarrow{\frT_{V_{2}, U'}(\varphi_{V_{2}})}\frT_{V_{1},U'}(\calF_{V_{2}}\boxtimes\calA )|_{\Bun^{+}\times \Delta(U')}\\
&\xrightarrow{\id_{\cF_{V_{2}}}\boxtimes\varphi_{V_{1}}}& (\calF_{V_{1}}\boxtimes\calF_{V_{2}}\boxtimes\calA)|_{\Bun^{+}\times \Delta(U')}=(\calF_{V_{1}}\otimes\calF_{V_{2}})\boxtimes\calA\cong\calF_{V_{1}\otimes V_{2}}\boxtimes\calA.
\end{eqnarray*}
The last step uses the tensor structure on the functor $\calF$.
\item\label{Equasi} For any $\gamma\in\Gamma=\Gal(U'/U)$ and $V\in\Rep(\dG)$, we have a commutative diagram
\begin{equation}\label{Gamma Hk}
\xymatrix{\frT_{V^{\gamma}}(\calA)\ar[d]^{\a_{\gamma,V}}\ar[rr]^{\varphi_{V^{\gamma}}} && \calF_{V^{\gamma}}\boxtimes\calA\ar[d]^{\delta_{\gamma, V}\otimes\id_{\calA}}\\
(\id_{\Bun^{+}}\times\gamma)^{*}\frT_{V}(\calA)\ar[rr]^{(\id\times\gamma)^{*}\varphi_{V}} && \gamma^{*}\calF_{V}\boxtimes\calA}
\end{equation}
\end{enumerate}
\end{defn}

These properties imply that $\frT_{V_{1}\boxtimes\cdots\boxtimes V_{n}}(\calA)\cong\calF_{V_{1}}\boxtimes\cdots\boxtimes\calF_{V_{n}}\boxtimes\calA$, compatible with the factorization structures in \S\ref{sss:iterated}.

We also axiomatize a weak version of Definition \ref{def:eigen}. For this we introduce a variant of the notion of $\dG$-local systems.

\begin{defn}\label{def:weak loc} Let $A$ be a set with an action of $\xch(Z\dG)$ (denoted by $\nu:\a \mapsto \nu\cdot \a$, for $\nu\in\xch(Z\dG)$ and $\a\in A$). A $(\dG, A)$-weak local system over $U'$ is the following data. 
\begin{enumerate}
\item For each finite set $I$ and $\a\in A$, there is an additive functor
\begin{eqnarray*}
\cF_{(-),\a}:\Rep(\dG^{I})&\to&\Loc(U'^{I})\\
V_{I}&\mapsto& \cF_{V_{I},\a}.
\end{eqnarray*}
\item\label{unit} For $\triv\in\Rep(\dG^{I})$ the trivial representation, there is a canonical isomorphism $\cF_{\triv,\a}\cong\Qlbar$ (the constant sheaf on $U'^{I}$).
\item\label{Eass} For finite sets $I$ and $J$, $V_{I}\in\Rep(\dG^{I};\mu)$ and $V_{J}\in\Rep(\dG^{J};\nu)$, there is an isomorphism
\begin{equation*}
\cF_{V_{I},\nu\cdot\a}\boxtimes\cF_{V_{J}, \a}\cong \cF_{V_{I}\boxtimes V_{J}, \a}.
\end{equation*}
satisfying obvious associativity for three finite sets $I,J,K$ and for the unit constraint in \eqref{unit}.
\item\label{Efactor} For a surjection $\varphi: J\surj I$ of finite sets, and $V_{J}\in\Rep(\dG^{J};\nu)$ restricting to $\Delta(\varphi)^{*}V_{J}\in\Rep(\dG^{I};\nu)$ via the diagonal $\Delta(\varphi):\dG^{I}\incl\dG^{J}$, there is an isomorphism
\begin{equation*}
\Delta(\varphi)^{*}\cF_{V_{J}, \a}\cong\cF_{\Delta(\varphi)^{*}V_{J},\a}
\end{equation*}
Here $\Delta(\varphi)$ also denotes the diagonal map $U'^{I}\incl U'^{J}$. These isomorphisms should be compatible with compositions of surjections and with the isomorphisms given in \eqref{unit} and \eqref{Eass}.
\end{enumerate}
If,  moreover, the action of $\xch(Z\dG)$ on $A$ factors through the quotient $\xch(Z\dG)_{I_{F}}$, and the functor $\calF_{(-),\a}$ is equipped with a descent datum $\{\delta_{\gamma,V_{I}}\}_{\gamma\in\Gamma, V_{I}\in\Rep(\dG^{I})}$ as in \S\ref{sss:twist loc}, which is compatible with the other isomorphisms we introduced above, $\{\calF_{(-),\a}\}$ is called a {\em $\htheta$-twisted $(\dG,A)$-weak local system}.
\end{defn}

The following lemma records some easy consequences of the definition above.
\begin{lemma}\label{l:weak loc}
Let $\{\calF_{(-),\a}\}$ be a $(\dG,A)$-weak local system on $U'$. Then
\begin{enumerate}
\item\label{EVA1} For $V_{i}\in\Rep(\dG;\nu_{i})$, $i=1,\cdots, n$, we have a functorial isomorphism
\begin{equation}\label{V12}
\calF_{V_{1}\boxtimes\cdots\boxtimes V_{n},\a}\cong\calF_{V_{1},(\nu_{2}+\cdots+\nu_{n})\cdot\a}\boxtimes\cdots\boxtimes\calF_{V_{n-1},\nu_{n}\cdot\a}\boxtimes\calF_{V_{n}, \a}.
\end{equation}
\item\label{EVAext} For any $V_{1}\in\Rep(\dG; \mu), V_{2}\in\Rep(\dG; \nu)$, we have a functorial isomorphism
\begin{equation}\label{extV12}
\calF_{V_{1},\a}\boxtimes\calF_{V_{2},\mu\cdot\a}\cong \calF_{V_{1},\nu\cdot\a}\boxtimes\calF_{V_{2}, \a}.
\end{equation}
\item\label{EVAtensor} For any $V_{1}\in\Rep(\dG; \mu), V_{2}\in\Rep(\dG; \nu)$, we have functorial isomorphisms
\begin{equation*}
\calF_{V_{1},\a}\otimes\calF_{V_{2}, \mu\cdot\a}\cong \calF_{V_{1}\otimes V_{2},\a}\cong \calF_{V_{2}\otimes V_{1}, \a}\cong\calF_{V_{2}, \a}\otimes\calF_{V_{1},\nu\cdot\a}.
\end{equation*}
\item\label{EVAdual} For $V\in\Rep(\dG; \nu)$ with dual $V^{*}\in\Rep(\dG;-\nu)$ and any $\a\in A$, $\calF_{V,\a}$ and $\calF_{V^{*},\nu\cdot\a}$ are dual local systems in a canonical way.
\item\label{EVArank} For $V\in\Rep(\dG)$ and any $\a\in A$, the rank of $\calF_{V,\a}$ is $\dim V$.
\item\label{EVAad} Let $(Z\dG)'\subset Z\dG$ be the subgroup with character group equal to the image of $\xch(Z\dG)\to\Aut(A)$. Let $\dG'=\dG/(Z\dG)'$. Then for each $\a\in A$, the functor $\calF_{(-),\a}$ restricted to $\Rep(\dG')$ gives a $\dG'$-local system over $U'$. Same holds in the $\htheta$-twisted situation.
\end{enumerate}
\end{lemma}
\begin{proof}
\eqref{EVA1} and \eqref{EVAad} are easy and we omit the proof.

\eqref{EVAext} By \eqref{V12}, $\calF_{V_{1},\nu\cdot\a}\boxtimes\calF_{V_{2}, \a}\cong\calF_{V_{1}\boxtimes V_{2}, \a}$ and $\calF_{V_{2},\mu\cdot\a}\boxtimes\calF_{V_{1}, \a}\cong\calF_{V_{2}\boxtimes V_{1}, \a}$. Applying \eqref{Tfactor} to the transposition $\sigma:\{1,2\}\to\{1,2\}$ we get $\sigma^{*}\calF_{V_{2}\boxtimes V_{1}, \a}\cong \calF_{V_{1}\boxtimes V_{2}, \a}$. Combining these facts we get \eqref{extV12}.

\eqref{EVAtensor} Restricting \eqref{extV12} to the diagonal $U'\incl U'\times U'$.

\eqref{EVAdual} We define $\ev:\calF_{V^{*},\nu\cdot\a}\otimes\calF_{V,\a}\to\Qlbar$ as the composition $\calF_{V^{*},\nu\cdot\a}\otimes\calF_{V,\a}\cong\calF_{V^{*}\otimes V,\a}\to\calF_{\triv,\a}\cong\Qlbar$ where we use the evaluation map $\ev:V^{*}\otimes V\to \triv$. We define $\coev: \Qlbar\to\calF_{V,\a}\otimes\calF_{V^{*},\nu\cdot\a}$ to be the composition $\Qlbar\cong\calF_{\triv,\nu\cdot\a}\to\calF_{V\otimes V^{*},\nu\cdot\a}\cong\calF_{V,\a}\otimes\calF_{V^{*},\nu\cdot\a}$ where we use the coevaluation map $\coev: \triv\to V\otimes V^{*}$. We need to check that these maps exhibit $\calF_{V^{*},\nu\cdot\a}$ as the dual object to $\calF_{V,\a}$ in the tensor category $\Loc(U')$. The map $(\id\otimes\ev)\circ(\coev\otimes\id):\calF_{V,\a}\to\calF_{V,\a}$ is the composition
\begin{equation*}
\Qlbar\otimes \calF_{V,\a}\xrightarrow{\coev\otimes\id}\calF_{V\otimes V^{*},\nu\cdot\a}\otimes\calF_{V,\a}\cong\calF_{V,\a}\otimes\calF_{V^{*},\nu\cdot\a}\otimes\calF_{V,\a}\cong\calF_{V,\a}\otimes\calF_{V^{*}\otimes V,\a}\xrightarrow{\id\otimes\ev}\calF_{V,\a}\otimes\Qlbar
\end{equation*}
which is the identity map as part of the compatibility conditions we alluded to in Definition \ref{def:weak loc}. Similarly one checks that the map $(\ev\otimes\id)\circ(\id\otimes\coev):\calF_{V^{*},\nu\cdot\a}\to\calF_{V^{*},\nu\cdot\a}$ is also the identity.

\eqref{EVArank} We may assume $V\in\Rep(\dG;\nu)$. By \eqref{EVAdual}, $\calF_{V,\a}$ and $\calF_{V^{*},\nu\cdot\a}$ are dual local systems, hence have the same rank. On the other hand, $V^{*}\otimes V\in\Rep(\dG^{\ad})$, and by \eqref{EVAad}, $\calF_{V^{*}\otimes V,\a}$ has rank $\dim(V^{*}\otimes V)=(\dim V)^{2}$. Hence $\calF_{V^{*},\nu\cdot\a}\otimes\calF_{V,\a}\cong\calF_{V^{*}\otimes V,\a}$ has rank $(\dim V)^{2}$, and therefore $\calF_{V,\a}$ has rank $\dim V$.
\end{proof}

\begin{cor}\label{c:weak loc} Let $\calF=(\calF_{(-), \a})_{\a \in A}$ be a $(\dG,A)$-weak local system on $U'$. Then
\begin{enumerate}
\item For any $V\in\Rep(\dG)$, if $\a, \b\in A$ are in the same orbit of $\xch(Z\dG)$, then the semisimplifications of $\calF_{V,\a}$ and $\calF_{V,\b}$ are isomorphic (in a non-canonical way).
\item Suppose all $\calF_{V,\a}$ are semisimple local systems, then for any $V_{1}, V_{2}\in\Rep(\dG)$ and $\a\in A$ there is a (non-canonical) isomorphism of local systems $\calF_{V_{1},\a}\otimes\calF_{V_{2},\a}\cong\calF_{V_{1}\otimes V_{2},\a}$.
\end{enumerate}
\end{cor}
\begin{proof} 
(1 )We may assume $V\in\Rep(\dG;\mu)$, $\b=\nu\cdot\a$ for some $\nu\in\xch(Z\dG)$. Pick any $W\in\Rep(\dG;\nu)$.Then we have $\calF_{V,\a}\boxtimes\calF_{W,\mu\cdot\a}\cong\calF_{V,\b}\boxtimes\calF_{W,\a}$ by Lemma \ref{l:weak loc}\eqref{EVAext}. Restricting to $U'\times \{x'\}$ for some geometric point $x'\in U'$, we get $\calF_{V,\a}^{\oplus\dim W}\cong\calF_{V,\b}^{\oplus\dim W}$ using the rank formula in Lemma \ref{l:weak loc}\eqref{EVArank}. From this we conclude that the semisimplifications of $\calF_{V,\a}$ and $\calF_{V,\b}$ are isomorphic.

(2) Again we may assume $V_{2}\in\Rep(\dG;\nu)$. By (1), $\calF_{V_{1},\nu\cdot\a}\cong\calF_{V_{1},\a}$. Therefore the desired isomorphism follows from Lemma \ref{l:weak loc}\eqref{EVAtensor}.
\end{proof}

\begin{defn}\label{def:weak eigen} A {\em weak Hecke eigensheaf} for the geometric automorphic datum $(\Om,\bK_{S}, \cK_{S}, \iota_{S})$ is a triple $(\calA, \calF, \varphi)$ where
\begin{itemize}
\item $\calA\in D_{\cG,\Om}(\bK_{S}, \calK_{S})$, which means an object $\calA_{\a}\in D_{\a}$ for each $\a\in\xch(Z\dG)_{I_{F}}$.
\item $\calF$ is a $\htheta$-twisted $(\dG,\xch(Z\dG)_{I_{F}})$-weak local system with respect to the obvious action of $\xch(Z\dG)$ on $\xch(Z\dG)_{I_{F}}$ by addition. In other words, $\cF$ is a collection $\{\calF_{V,\a}\}$ of local systems over $U'$, for $V\in\Rep(\dG)$ and $\a\in\xch(Z\dG)_{I_{F}}$, together with the extra structures as in Definition \ref{def:weak loc}.
\item $\varphi$ is a collection of natural isomorphisms of functors (for each $\a\in\xch(Z\dG)_{I_{F}}$ and $\nu\in\xch(Z\dG)$ with image $\overline{\nu}\in\xch(Z\dG)_{I_{F}}$):
\begin{equation*}
\varphi_{(-),\a}: \frT_{(-),\a}(\calA_{\a})\cong\calF_{(-),\a}\boxtimes\calA_{\a+\overline{\nu}} :\Rep(\dG^{I};\nu)\to D_{\a+\overline{\nu}}(U'^{I}).
\end{equation*}
The natural isomorphisms $\varphi_{(-),\a}$ should intertwine the factorization structures \eqref{Tunit}\eqref{Tass}\eqref{Tfactor} in \S\ref{sss:iterated} for the iterated Hecke operators and the factorization structures \eqref{unit}\eqref{Eass}\eqref{Efactor} of $\calF$ in Definition \ref{def:weak loc}.
\end{itemize}
\end{defn}

\subsection{Rigid Hecke eigensheaves} 
The main purpose of this subsection is to establish the existence of Hecke eigensheaves in a special situation. 

\subsubsection{Assumptions}\label{sss:ass} We make the following assumptions on the geometric automorphic datum $(\Om,\bK_{S}, \cK_{S}, \iota_{S})$.
\begin{enumerate}
\item\label{Zfin} The affine part of the coarse moduli space of $\Bun^{\n}_{\cZ}(\bK_{Z,S})$ is finite (recall every commutative algebraic group over $k$ is canonically an extension of an abelian scheme by an affine commutative algebraic group; the latter is called the affine part). For example, this is true if $G$ is semisimple, or more generally $\bK_{Z,x}$ has finite index in $L_{x}^{+}\cZ$ for each $x\in S$.
\item\label{uniquerelevant} We assume that for each $\alpha\in \xch(Z\dG)_{I_{F}}$, $\Bun_{\alpha}$ contains a unique $(\Om, \cK_{S})$-relevant $\Bun^{\n}_{\cZ}(\bK_{Z,S})$-orbit $O_{\alpha}$ (see Definition \ref{def:relevant}).
\item Choose a point $\calE_{\alpha}\in O_{\alpha}$. We assume that the group $A_{\alpha}:=A_{\calE_{\a}}$ (see \S\ref{sss:relevant}) is finite.
\item\label{Wunique} Let $\calK_{\alpha}:=\calK_{\calE_{\alpha}}\in\cCS(A_{\alpha})$ (see \S\ref{sss:relevant}). Then $\calK_{\alpha}$ gives a class $\xi_{\alpha}\in\cohog{2}{A_{\alpha}, \Qlbar^{\times}}$ (see Lemma \ref{l:char sh discrete}). We assume that the category $\Rep_{\xi_{\alpha}}(A_{\alpha})$ (see \S\ref{sss:twisted rep}) contains a unique irreducible object $W_{\a}$ up to isomorphism.
\end{enumerate}

Here is the main result of this section.

\begin{theorem}\label{th:eigen} Under the assumptions in \S\ref{sss:ass}, the category $D_{\alpha}$ contains a unique irreducible perverse sheaf $\calA_{\alpha}$ up to isomorphism. Let $\calA\in D_{\cG,\Om}(\bK_{S}, \calK_{S})$ be an object whose restriction to $\Bun^{+}_{\alpha}$ is $\calA_{\alpha}$. Then $\calA$ can be extended to a {\em weak} Hecke eigensheaf $(\calA,\calF,\varphi)$ in which $\calF_{V,\a}$ is a semisimple local system for all $V\in\Rep(\dG)$ and $\a\in\xch(Z\dG)_{I_{F}}$.
\end{theorem}

By Corollary \ref{c:weak loc}(2), each functor $\calF_{(-),\a}:\Rep(\dG)\to \Loc(U')$ is weak tensor in the sense that there exist isomorphisms $\calF_{V_{1},\a}\otimes\calF_{V_{2},\a}\cong\calF_{V_{1}\otimes V_{2},\a}$ which do not necessarily satisfy the associativity constraint.

The proof of Theorem \ref{th:eigen} occupies the rest of the subsection. 

Fix a smaller level $\bK^{+}_{S}\lhd\bK_{S}$ as in \S\ref{sss:Kplus}. For each $\alpha\in\xch(Z\dG)_{I_{F}}$, let $O^{+}_{\alpha}$ be the preimage of $O_{\alpha}$ in $\Bun^{+}_{\alpha}$, which is an $\bM_{S}$-orbit. Choose $\calE^{+}_{\alpha}\in O^{+}_{\alpha}$ over $\calE_{\alpha}$. Then the stabilizer $\bM_{S,\calE^{+}_{\alpha}}$ is identified with $A_{\alpha}$.

\begin{lemma}
The embedding $j_{\alpha}:O_{\alpha}\incl\Bun_{\alpha}$ is open.
\end{lemma}
\begin{proof} Equivalently we need to show that $j^{+}_{\alpha}:O^{+}_{\alpha}\incl\Bun^{+}_{\a}$ is open. Let $\calU\subset\Bun^{+}_{\a}$ be the locus where the stabilizers under $\bM_{S}$ are finite. Then $\calU$ is an open substack of $\Bun^{+}_{\a}$ because the dimension of stabilizers is upper semicontinuous. On the other hand, any $\kbar$-point in $\calU$ is $(\Om,\cK_{S})$-relevant because the relevance condition is vacuous there. Since $O^{+}_{\alpha}$ is the only relevant $\bM_{S}$-orbit, we have $\calU=O^{+}_{\alpha}$, hence $O^{+}_{\alpha}$ is open.
\end{proof}

We have 
\begin{equation*}
D^{b}_{(\bM_{S},\cK_{S,\Om})}(O^{+}_{\a})\cong D^{b}_{(A_{\a}, \cK_{\a})}(\pt)\cong D^{b}(\Rep_{\xi_{\a}}(A_{\a})).
\end{equation*}
The first equivalence follows from the fact that $[\bM_{S}\backslash O^{+}_{\a}]\cong A_{\a}\backslash\pt$; the second equivalence follows from Lemma \ref{l:Ac}. The unique irreducible object $W_{\a}\in \Rep_{\xi_{\a}}(A_{\a})$ gives an irreducible $(\bM_{S},\cK_{S,\Om})$-equivariant local system $\cL_{\a}$ over $O^{+}_{\a}$.

\begin{lemma}\label{l:unique irr} 
\begin{enumerate}
\item We have $j^{+}_{\alpha,!}\cL_{\a}\isom j^{+}_{\alpha,*}\cL_{\a}\in D_{\a}$. We denote both objects by $\calA_{\a}$.
\item For any scheme $Y$ over $k$, the functor
\begin{eqnarray}\label{boxF}
D^{b}(Y)&\to& D_{\a}(Y)=D^{b}_{(\bM_{S},\cK_{S,\Om})}(Y\times\Bun^{+}_{\a})\\
\notag \calA' &\mapsto& \calA'\boxtimes \calA_{\a}
\end{eqnarray}
is an equivalence of categories. In particular, when $Y=\Spec k$, a shift of $\calA_{\a}$ is the unique irreducible perverse sheaf in $D_{\alpha}$.
\end{enumerate}
\end{lemma}
\begin{proof}
(1) Since $O^{+}_{\a}$ is the unique $(\Om,\cK_{S})$-relevant $\bM_{S}$-orbit, any object in $D_{\a}$ must have vanishing stalks and costalks outside $O^{+}_{\alpha}$ by Lemma \ref{l:irr}. Therefore $j^{+}_{\alpha,!}\cL_{\a}\isom j^{+}_{\alpha,*}\cL_{\a}$. 

(2) We shall construct an inverse to the functor \eqref{boxF}. Since $O^{+}_{\a}$ is the unique $(\Om,\cK_{S})$-relevant $\bM_{S}$-orbit, any object in $D_{\a}(Y)$ must have vanishing stalks and costalks outside $Y\times O^{+}_{\a}$, by the same argument as Lemma \ref{l:irr}. Therefore the restriction functor
\begin{equation}\label{res equiv}
(\id_{Y}\times j^{+}_{\a})^{*}: D_{\a}(Y)=D^{b}_{(\bM_{S},\cK_{S,\Om})}(Y\times\Bun^{+}_{\a})\to D^{b}_{(\bM_{S},\calK_{S,\Om})}(Y\times O^{+}_{\a})
\end{equation}
is an equivalence of categories. 

Let $i^{+}_{\a}:\Spec k\to O^{+}_{\a}$ be the point $\calE^{+}_{\a}$. The restriction functor
\begin{equation*}
(\id_{Y}\times i^{+}_{\a})^{*}: D^{b}_{(\bM_{S},\calK_{S,\Om})}(Y\times O^{+}_{\a})\to D^{b}_{(A_{\a},\cK_{\a})}(Y\times\{\calE^{+}_{\a}\})=D^{b}_{(A_{\a},\cK_{\a})}(Y)
\end{equation*}
is again an equivalence of categories. In the last category, $A_{\a}$ acts trivially on $Y$. The forgetful functor identifies $D^{b}_{(A_{\a},\cK_{\a})}(Y)$ with the category of objects in $D^{b}(Y)$ together with a $\xi_{\a}$-twisted action of $A_{\a}$ (for this we need to choose a cocycle representing $\xi_{\a}$). Since $W_{\a}$ is the unique irreducible object in the semisimple category $\Rep_{\xi_{\a}}(A_{\a})$, any object $\calB\in D^{b}_{(A_{\a},\cK_{\a})}(Y)$ is of the form $\calA'\otimes W_{\a}$ for $\calA'=\Hom_{\Rep_{\xi_{\a}}(A_{\a})}(W_{\a}, \calB)\in D^{b}(Y)$. Therefore an inverse to \eqref{boxF} is given by
\begin{equation}\label{HomW}
D_{\a}(Y)\ni\calA\mapsto\Hom_{\Rep_{\xi_{\a}}(A_{\a})}(W_{\a}, (\id_{Y}\times i^{+}_{\a}j^{+}_{\a})^{*}\calA)\in D^{b}(Y).
\end{equation}
\end{proof}

\subsubsection{Construction of $\calF_{V,\a}$} Let $V_{I}\in\Rep(\dG^{I};\nu)$, $\a\in\xch(Z\dG)_{I_{F}}$ and $\b=\a+\overline{\nu}$.  Applying  \eqref{HomW} from the proof of Lemma \ref{l:unique irr} to $Y=U'^{I}$ and $\frT_{V_{I},\a}(\calA_{\alpha})\in D_{\a}(U'^{I})$ (see \eqref{TVa}), we get a canonical isomorphism
\begin{equation*}
\varphi_{V_{I},\a}: \frT_{V_{I},\a}(\calA_{\alpha})\cong\calF_{V_{I},\a}\boxtimes\calA_{\b}.
\end{equation*}
where $\calF_{V_{I},\a}\in D^{b}(U'^{I})$ is defined as
\begin{equation}\label{defineE}
\calF_{V_{I},\a}:=\Hom_{\Rep_{\xi_{\a}}(A_{\a})}\left(W_{\a}, (\id_{U'^{I}}\times i^{+}_{\a}j^{+}_{\a})^{*}\frT_{V_{I},\a}(\calA_{\alpha})\right).
\end{equation}
The properties \eqref{Tunit}\eqref{Tass}\eqref{Tfactor} of the iterated Hecke operators in \S\ref{sss:iterated} implies the properties \eqref{unit}\eqref{Eass}\eqref{Efactor} of $\calF_{V,\a}$ as in Definition \ref{def:weak loc}, except that we do not yet know that $\calF_{V,\a}$ are local systems. Moreover, the $\Gamma$-equivariance of the morphisms $\ev$ (see \eqref{ev Hk}) gives a $\htheta$-twisted descent datum on $\calF_{V_{I},\a}$:
\begin{equation*}
\delta_{\gamma, V_{I},\a}: \calF_{V^{\gamma}_{I},\a}\cong \gamma^{*}\calF_{V_{I}, \a}, \forall \gamma\in\Gamma.
\end{equation*}
The descent datum satisfies the analog of condition \eqref{Equasi} in Definition \ref{def:eigen} (note for $V_{I}\in\Rep(\dG^{I};\nu)$, the central character of $V^{\gamma}$ is $\nu\circ\gamma$, which has the same image in $\xch(Z\dG)_{I_{F}}$, therefore $\calA_{\a+\overline{\nu}}$ appears in both the upper right and lower right corners of \eqref{Gamma Hk}). Therefore to prove Theorem \ref{th:eigen}, it remains to show that $\calF_{V,\a}$ is a local system for each $V\in\Rep(\dG)$, because then by \eqref{V12} all $\calF_{V_{I},\a}$ are local systems. We shall do this in two lemmas.

\begin{lemma}\label{l:Els}
The complex $\calF_{V,\a}$ is a semisimple local system on $U'$ for all $V\in\Rep(\dG; \nu)$ and $\a\in\xch(Z\dG)_{I_{F}}$.
\end{lemma}
\begin{proof} Let $\b=\alpha+\overline{\nu}\in\xch(Z\dG)_{I_{F}}$. Let $\Hk_{U',\a\to\b}$ be a union of components of $\Hk_{U',\nu}$ which map to $\Bun^{+}_{\a}$ under $\oll{h}$ (hence to $\Bun^{+}_{\b}$ under $\orr{h}$). We recall the notion of Universal Local Acyclicity (ULA) from \cite[\S5.1]{BG}. The morphism $\pi\times \oll{h}:\Hk_{U'}\to U'\times\Bun^{+}$ is \'etale locally a trivial fibration with fibers isomorphic to the affine Grassmannian, and under such a trivialization $\IC^{\Hk}_{V}$ becomes $\IC_{V}$ along the fibers. Therefore $\oll{h}^{*}\calA_{\a}\otimes\IC^{\Hk}_{V}$ is ULA with respect to $\pi_{U'}:\Hk_{U', \a\to\b}\to U'$. On the other hand, $\pi\times\orr{h}: \Hk_{U',\a\to\b}\to U'\times\Bun^{+}$ is ind-proper, therefore $\frT_{V}(\calA_{\a})=(\pi\times\orr{h})_{!}(\oll{h}^{*}\calA_{\a}\otimes\IC^{\Hk}_{V})$ is also ULA with respect to the projection to  $U'$. Since $\frT_{V}(\calA_{\a})$ is an external tensor product $\calF_{V,\a}\boxtimes\calA_{\b}$ over $U'\times\Bun^{+}_{\b}$, $\calF_{V,\a}$ is ULA with respect to the identity map on $U'$, i.e., each cohomology sheaf of $\calF_{V,\a}$ is a local system. Therefore, to show that $\calF_{V,\a}$ is a local system,  it suffices to show that the geometric stalk of $\calF_{V,\a}$ at some point $x'\in U'$ is concentrated in degree zero.

Fix a geometric point  $x'\in U'$ and let $\Hk_{x',\a\to\b}$ be the corresponding fiber. Let $\oHk_{x'}$ be the open subset of $\Hk_{x',\a\to\b}$ which is the preimage of $O^{+}_{\a}$ under $\oll{h}$; let $\Hko_{x'}$ be the preimage of $O^{+}_{\b}$ under $\orr{h}$; let $\oHko_{x'}=\oHk_{x'}\cap\Hko_{x'}$. Consider the diagram with a Cartesian square on the left
\begin{equation*}
\xymatrix{ & \oHko_{x'}\ar@{^{(}->}[r]^{j_{\Hk}}\ar[dl]_{\oo{\oll{h}}} & \Hko_{x'}\ar[dr]^{\orr{h}^{\c}}\ar[dl]_{\oll{h}^{\c}}\\
O^{+}_{\a}\ar@{^{(}->}[r]^{j^{+}_{\a}} & \Bun^{+} & & O^{+}_{\b}}
\end{equation*}
Let $\oo{\orr{h}}:\oHko_{x'}\to O^{+}_{\b}$ be the restriction of $\orr{h}$ to $\oHko_{x'}$. To alleviate notation, let $\IC_{x',V}$ denote the restriction of $\IC^{\Hk}_{V}$ to any open substack of the fiber $\Hk_{x',\a\to\b}$. Let $i_{x'}:\{x'\}\incl U'$ be the inclusion. By the definition of the geometric Hecke operators and proper base change, we have
\begin{eqnarray}\label{!Hk}
 \calF_{V,\a, x'}\otimes \cL_{\b}&=&(i_{x'}\times j^{+}_{\b})^{*}\frT_{V,\a}(\calA_{\a})\cong \orr{h}^{\c}_{!}(\oll{h}^{\c,*}j^{+}_{\a,!}\cL_{\a}\otimes\IC_{x',V})\\
\notag &\cong& \orr{h}^{\c}_{!}j_{\Hk,!}(\leftexp{\c}{\oll{h}}^{\c,*}\cL_{\a}\otimes\IC_{x',V})
=\oo{\orr{h}}_{!}(\leftexp{\c}{\oll{h}}^{\c,*}\cL_{\a}\otimes\IC_{x',V}).
\end{eqnarray}
On the other hand, we have another sequence of isomorphisms
\begin{eqnarray}\label{*Hk}
\calF_{V,\a, x'}\otimes\cL_{\b}&=&(i_{x'}\times j^{+}_{\b})^{*}\frT_{V,\a}(\calA_{\a})\cong \orr{h}^{\c}_{*}(\oll{h}^{\c,*}j^{+}_{\a,*}\cL_{\a}\otimes\IC_{x',V})\\
\notag&\cong &\orr{h}^{\c}_{*}j_{\Hk,*}(\leftexp{\c}{\oll{h}}^{\c,*}\cL_{\a}\otimes\IC_{x',V})
=\oo{\orr{h}}_{*}(\leftexp{\c}{\oll{h}}^{\c,*}\cL_{\a}\otimes\IC_{x',V}).
\end{eqnarray}
Here we have used two facts: one is that $\orr{h}^{\c}$ is ind-proper; the other is that $\oll{h}^{\c}:\Hko_{x'}\to \Bun^{+}$ is \'etale locally a trivial fibration, therefore $\oll{h}^{\c,*}j^{+}_{\a,*}\cong j_{\Hk,*}\oll{h}^{\c,*}$.

We define $\GR_{x',\b}$ and its open substack $\oGR_{x',\b}$ by the following Cartesian squares
\begin{equation*}
\xymatrix{& \oGR_{x',\b}\ar@{^{(}->}[r]\ar[dl]_{\leftexp{\c}{\oll{h}}_{\GR}} & \GR_{x', \b}\ar[dl]^{\oll{h}_{\GR}}\ar[dr]\ar[r] & \Hko_{x'}\ar[dr]^{\orr{h}^{\c}}\\
O^{+}_{\a}\ar@{^{(}->}[r]^{j^{+}_{\a}} & \Bun^{+}_{\a} && \{\cE^{+}_{\b}\}\ar[r]^{i^{+}_{\b}} & O^{+}_{\b}}
\end{equation*} 
Since $\orr{h}^{\c}:\Hko_{x'}\to O^{+}_{\b}$ is $\bM_{S}$-equivariant and $O^{+}_{\b}$ is a single $\bM_{S}$-orbit,  \eqref{!Hk} and \eqref{*Hk} can be combined and rewritten as
\begin{equation}\label{!*GR}
\calF_{V,\a,x'}\otimes W_{\b}\cong \cohoc{*}{\oGR_{x',\b}, \leftexp{\c}{\oll{h}}^{*}_{\GR}\cL_{\a}\otimes\IC_{x',V}}\cong \cohog{*}{\oGR_{x',\b}, \leftexp{\c}{\oll{h}}^{*}_{\GR}\cL_{\a}\otimes\IC_{x',V}}.
\end{equation}
The complex $\calH:=\leftexp{\c}{\oll{h}}^{*}_{\GR}\cL_{\a}\otimes\IC_{x',V}$ is perverse (on $\oGR_{x',\b}$). By Lemma \ref{l:affine} below, $\oGR_{x',\b}$ is ind-affine. The support of $\calH$ is a closed subscheme of $\oGR_{x',\b}$ of finite type, and is therefore affine. By \cite[Th\'eor\`eme 4.1.1]{BBD},  $\cohoc{*}{\oGR_{x',\b}, \calH}$ lies in non-negative degrees and $\cohog{*}{\oGR_{x',\b}, \calH}$ lies in non-positive degrees.  By \eqref{!*GR}, these two cohomology groups are isomorphic to each other, hence they are both concentrated in degree zero, and therefore $\calF_{V,\a,x'}$ is concentrated in degree zero for any $x'\in U'$.

Finally we need to argue that $\calF_{V,\a}$ is a semisimple local system. Using the usual spreading-out argument we may reduce to the case where $k$ is a finite field, and we may use the theory of weights for complexes of sheaves. We may normalize the local system $\calL_{\a}$ to be pure of weight zero, and normalize $\IC_{V}$ to be also pure of weight zero. Then $\calA_{\a}=j^{+}_{\a, !}\cL_{\a}=j^{+}_{\a, *}\cL_{\a}$ is pure of weight zero, and $\oll{h}^{*}\calA_{\a}\otimes\IC^{\Hk}_{V}$ is pure of weight zero on $\Hk_{U',\a\to\b}$ (using the fact that $\Hk_{U'}$ is locally a product of $\Bun^{+}\times U'$ and the affine Grassmannian). Applying the ind-proper map $\pi\times\orr{h}$ we conclude that $\frT_{V}(\calA_{\a})\cong\calF_{V,\a}\boxtimes \calA_{\b}$ is a pure complex, hence $\calF_{V,\a}$ is also a pure complex on $U'$. By the decomposition theorem \cite[Corollaire 5.4.6]{BBD}, $\calF_{V,\a}$ is a direct sum of shifted simple perverse sheaves. Since $\calF_{V,\a}$ is a local system, it is a semisimple local system. 
\end{proof}

\begin{lemma}\label{l:affine} For any geometric point $x'\in U'$, the open sub-indscheme $\oGR_{x',\b}\subset\GR_{x',\b}$ is ind-affine.
\end{lemma}
\begin{proof} We give a reinterpretation of $\oGR_{x',\b}$. Consider the functor $\frI_{x',\a\to\b}$ classifying pairs $(m,\iota)$ where $m\in\bM_{S}$ and $\iota: m\cdot\calE^{+}_{\a}|_{X-\{x\}}\isom\calE^{+}_{\b}|_{X-\{x\}}$ ($x$ is the image of $x'$ in $U$). Then $\frI_{x',\a\to\b}$  is represented by an indscheme over $k$. The finite group $A_{\a}$ acts on $\frI_{x',\a\to\b}$ on the right: $a\in A_{\a}\cong\bM_{S,\cE^{+}_{\a}}$ sends $(m, \iota)$ to $(ma, \iota\circ a)$. It is easy to show that the map $\frI_{x',\a\to\b}/A_{\a}\to\oGR_{x',\b}$ sending $(m, \iota)$ to $(m\cdot\calE^{+}_{\a}, \iota)$ is an isomorphism. Therefore it suffices to show that $\frI_{x',\a\to\b}$ is affine. 

The fibers of the natural morphism $\frI_{x',\a\to\b}\to\bM_{S}$ ($(m,\iota)\mapsto m$), if nonempty, are torsors under the automorphism group of $\cE^{+}_{\b}|_{X-\{x\}}$, which is ind-affine. Therefore the fibers of $\frI_{x',\a\to\b}\to\bM_{S}\to\Bun^{\n}_{\cZ}(\bK_{Z,S})$ are ind-affine. Consider the integral model $\cD$ of $D$ (the maximal torus quotient of $G$) as in \S\ref{sss:cG}. Since $Z\incl G\surj D$ is an isogeny, the kernel of $\rho: \Bun^{\n}_{\cZ}(\bK_{Z,S})\to\Bun_{\cZ}\to\Bun_{\cD}$ has affine coarse moduli space. On the other hand, the image of $\frI_{x',\a\to\b}\to\Bun^{\n}_{\cZ}(\bK_{Z,S})\xrightarrow{\rho}\Bun_{\cD}$ lies in the image of $\oGR_{x',\a\to\b}\subset L_{x}G/L_{x}^{+}\cG\to L_{x}D/L_{x}^{+}\cD\to\Bun_{\cD}$, and the latter has discrete reduced scheme structure. We conclusion that the image of $\frI_{x',\a\to\b}\to \Bun^{\n}_{\cZ}(\bK_{Z,S})$ lies in a disjoint union of translations of its affine part, hence finite by Assumption \eqref{Zfin} in \S\ref{sss:ass}. Since the fibers of $\frI_{x',\a\to\b}\to \Bun^{\n}_{\cZ}(\bK_{Z,S})$ are ind-affine, $\frI_{x',\a\to\b}$ is a finite disjoint union of ind-affine schemes, hence ind-affine.  
\end{proof}

\subsection{Description of the eigen local system and examples in $\GL_{2}$}\label{ss:desc eigen}  
We are in the situation of  \S\ref{sss:ass}. We make an extra assumption:
\begin{equation}\label{BunZpt}
\mbox{The coarse moduli space of the stack $\Bun^{\n}_{\cZ}(\bK_{Z,S})$ is a point.}
\end{equation}
In other words, $\Bun^{\n}_{\cZ}(\bK_{Z,S})$ is the classifying space of $\bA_{Z,S}$.
This is satisfied if, for example, $X=\PP^{1}$, $Z=\ZZ\GG\otimes_{k}F$ and $\bK_{Z,x}=L^{+}_{x}\ZZ\GG$ for all $x\in S$. When \eqref{BunZpt} is satisfied, we have $\Bun^{\n}_{\cZ}(\bK^{+}_{Z,S})=\bL_{Z,S}/\bA_{Z,S}$, hence $\bM_{S}=\bL_{S}/\bA_{Z,S}$. The character sheaf $\cK_{S}$ descends to a character sheaf $\cK_{S, \Om}$ on $\bM_{S}$ because of the compatibility isomorphism $\iota_{S}:\Om|_{\bK_{Z,S}}\cong\cK_{S}|_{\bK_{Z,S}}$.

Let $V\in\Rep(\dG;\nu)$, and let $\a\in\xch(Z\dG)_{I_{F}},\b=\a+\overline{\nu}\in\xch(Z\dG)_{I_{F}}$. We will give an explicit description of the Hecke-eigen local systems $\calF_{V,\a}$ defined in \eqref{defineE}.

Because $\Bun^{\n}_{\cZ}(\bK_{Z,S})$ has one geometric point and the assumption \eqref{uniquerelevant} of \S\ref{sss:ass}, $\cE_{\a}$ and $\cE_{\b}$ are the only relevant points on $\Bun_{\a}$ and $\Bun_{\b}$. Let $\Aut_{\a}$ and $\Aut_{\b}$ be the automorphism groups of $\cE_{\a}$ and $\cE_{\b}$, then we have an exact sequence $1\to \bA_{Z,S}\to \Aut_{\a}\to A_{\a}\to1$ and likewise for $\Aut_{\b}$.

Let $\frG_{U',\a\to\b}$ be the  ind-scheme over $U'$ whose fiber at $x'\in U'$ classifies isomorphism $\cE_{\a}|_{X-\{x\}}\isom\cE_{\a}|_{X-\{x\}}$ of $\cG$-torsors preserving $\bK_{S}$-level structures (here $x$ is the image of $x'$ in $U$). Let $\frI_{U',\a\to\b}=\frG_{U',\a\to\b}/\bA_{Z,S}$ where $\bA_{Z,S}$ is acting as automorphisms of $\cE_{\a}$. The group $A_{\a}\times A_{\b}$ acts on $\frI_{U',\a\to\b}$ by $(a_{\a}, a_{\b})\cdot g=a_{\b}ga_{\a}^{-1}$. 

Let $\cE^{+}_{\a}$ and $\cE^{+}_{\b}$ in $\Bun^{+}$ be liftings $\cE_{\a}$ and $\cE_{\b}$. Under such choices, the $\bL_{S}$-reduction of the $\bK_{S}$-level structures of $\cE_{\a}$ and $\cE_{\b}$ are trivialized. Evaluating at $S$ gives a morphism
\begin{equation*}
\ev_{\frI, S}: \frI_{U',\a\to\b}\to\bM_{S}
\end{equation*}
Evaluating an isomorphism $\cE_{\a}|_{X'-\{x'\}}\isom\cE_{\a}|_{X'-\{x'\}}$ at the formal disk around $x'\in U'$ gives
\begin{equation*}
\ev_{\frI}: \frI_{U',\a\to\b}\to \left[\frac{L^{+}\GG\backslash L\GG/L^{+}\GG}{\Aut_{k[[t]]}}\right]
\end{equation*}
in the same way $\ev$ in \eqref{ev Hk} was defined. Let $\IC^{\frI}_{V}:=\ev^{*}_{\frI}\IC_{V}$.

\subsubsection{Central extension}\label{sss:centext} By the discussion in \S\ref{ss:Serrepi1}, the rank one character sheaf $\calK_{S,\Om}$ on $\bM_{S}$ can be obtained from a central extension
\begin{equation*}
1\to C\to \tbM_{S}\xrightarrow{v} \bM_{S}\to 1
\end{equation*}
by taking $\calK_{S,\Om}\cong(v_{!}\Qlbar)_{\chi_{C}}$. Here $\tbM_{S}$ is connected and $C$ is a finite discrete (necessarily abelian) group scheme over $k$, and $\chi_{C}:C\to \Qlbar^{\times}$ is a character. Let 
\begin{equation*}
\tilA_{\a}:=A_{\a}\times_{\bM_{S}}\tbM_{S}; \quad\tilA_{\b}:=A_{\b}\times_{\bM_{S}}\tbM_{S};
 \quad \tfrI_{U',\a\to\b}:=\frI_{U',\a\to\b}\times_{\bM_{S}}\tbM_{S}.
\end{equation*}
Here all the maps to $\bM_{S}$ are evaluation maps. Then $\tilA_{\a}$ is an extension of $A_{\a}$ by $C$; similar remark applies to $\tilA_{\b}$. Again $\tilA_{\a}\times\tilA_{\b}$ acts on $\tfrI_{U',\a\to\b}$, and the action of the diagonal copy of $C$ is trivial.

Recall that a class $\xi_{\a}\in\cohog{2}{A_{\a}, \Qlbar^{\times}}$ is given by the restriction of $\cK_{S,\Om}$ to $A_{\a}$, and we may form the category of twisted representations $\Rep_{\xi_{\a}}(A_{\a})$. We may identify $\Rep_{\xi_{\a}}(A_{\a})$ with the subcategory of $\Rep(\tilA_{\a})$ consisting of representations whose restriction to $C$ is $\chi_{C}$. The unique irreducible object $W_{\a}\in\Rep_{\xi_{\a}}(A_{\a})$ (see Assumption \eqref{Wunique} in \S\ref{sss:ass}) is viewed as an irreducible representation $\tilW_{\a}$ of $\tilA_{\a}$ in this way. 

We denote the pullback of $\IC^{\frI}_{V}$ to $\tfrI_{U',\a\to\b}$ by $\IC^{\tfrI}_{V}$.
Finally, let $\wt\Pi:\tfrI_{U',\a\to\b}\to U'$ be the projection.

\begin{prop}\label{p:desc eigen} We are under the assumptions in \S\ref{sss:ass} and the extra assumption \eqref{BunZpt}. For $V\in\Rep(\dG;\nu)$, the Hecke-eigen local system $\calF_{V,\a}$ constructed in Theorem \ref{th:eigen} can be written as
\begin{equation}\label{compute rho}
\calF_{V,\a}\cong \Hom_{\tilA_{\a}\times\tilA_{\b}}(\tilW^{\vee}_{\a}\boxtimes \tilW_{\b}, \wt\Pi_{!}\IC^{\tfrI}_{V}).
\end{equation}
Here the action of $\tilA_{\a}\times\tilA_{\b}$ on $\tfrI_{U',\a\to\b}$ induces an action  on the complex  $\wt\Pi_{!}\IC^{\tfrI}_{V}\in D^{b}(U')$, and the right side is the multiplicity of the irreducible representation $\tilW^{\vee}_{\a}\boxtimes \tilW_{\b}$ of $\tilA_{\a}\times\tilA_{\b}$ in $\wt\Pi_{!}\IC^{\tfrI}_{V}$, which is still a complex on $U'$.
\end{prop}
\begin{proof} To compute  $\calF_{V,\a}$ we need to restrict $\frT_{V}(\calA_{\a})$ to $U'\times\{\cE^{+}_{\b}\}$. Taking the fiber of $\orr{h}^{-1}(\cE_{\b})$ we get the Beilinson-Drinfeld Grassmannian $\GR_{U',\b}$ whose fiber over $x'\in U'$ is $\GR_{x',\b}$ defined in the proof of Lemma \ref{l:Els}. We also have the open subscheme $\oGR_{U',\b}$ which the preimage of $O^{+}_{\a}$ under $\oll{h}_{\GR}:\GR_{U',\b}\to\Bun^{+}_{\a}$. Restricting $\IC^{\Hk}_{V}$ to $\oGR_{U',\b}$ we get a complex $\IC^{\GR}_{V}$. We have the maps $\leftexp{\c}{\oll{h}}:\oGR_{U',\b}\to O^{+}_{\a}$ and $\leftexp{\c}{\pi}_{\GR}: \oGR_{U',\a\to\b}\to U'$. Then
\begin{equation}\label{WGR}
\calF_{V,\a}\otimes W_{\b}\cong\leftexp{\c}{\pi}_{\GR,!}(\leftexp{\c}{\oll{h}}^{*}\cL_{\a}\otimes\IC^{\GR}_{V}).
\end{equation}

We claim that there is a canonical isomorphism $\oGR_{U',\a\to\b}\cong\tfrI_{U',\a\to\b}/\tilA_{\a}=\frI_{U',\a\to\b}/A_{\a}$. We have introduced the indscheme $\frI_{x',\a\to\b}$ in the proof of Lemma \ref{l:affine}. Since $O_{\a}$ is a single point, $O^{+}_{\a}$ is a single $\bL_{S}$-orbit, any point $\cE^{+}\in O^{+}_{\a}$ is obtained from $\cE^{+}_{\a}$ by changing its $\bK^{+}_{S}$-level structure by an element in $\bL_{S}$ modulo the automorphisms $\Aut_{\a}$. Therefore $\frI_{x',\a\to\b}$ is the fiber of $\frI_{U',\a\to\b}$ over $x'\in U'$. We have argued in the proof of Lemma \ref{l:affine} that $\frI_{x',\a\to\b}/A_{\a}\cong\oGR_{x',\a\to\b}$. That argument works for $x'$ varying over $U'$ and gives the isomorphism $\oGR_{U',\a\to\b}\cong\frI_{U',\a\to\b}/A_{\a}=\tfrI_{U',\a\to\b}/\tilA_{\a}$. Let $\oll{i}: \tfrI_{U',\a\to\b}/\tilA_{\a}\to\BB\tilA_{\a}$  be the canonical projection. Then under the isomorphism $\oGR_{U',\b}\cong\tfrI_{U',\a\to\b}/\tilA_{\a}$, $\leftexp{\c}{\oll{h}}^{*}\cL_{\a}\otimes\IC^{\GR}_{V}$ corresponds to $\oll{i}^{*}\tilW_{\a}\otimes\IC^{\tfrI}_{V}$, viewing $\tilW_{\a}\in\Rep(\tilA_{\a})$ as a sheaf on $\BB\tilA_{\a}$. Therefore we have
\begin{equation*}
\leftexp{\c}{\pi}_{\GR,!}(\leftexp{\c}{\oll{h}}^{*}\cL_{\a}\otimes\IC^{\GR}_{V})\cong \wt\Pi_{!}(\oll{i}^{*}\tilW_{\a}\otimes\IC^{\tfrI}_{V})=\Hom_{\tilA_{\a}}(\tilW^{\vee}_{\a}, \wt\Pi_{!}\IC^{\tfrI}_{V}).
\end{equation*}
On the other hand, by \eqref{WGR}, the above complex is also $\tilA_{\b}$-equivariantly isomorphic to $\calF_{V,\a}\otimes\tilW_{\b}$. Combing these facts we get \eqref{compute rho}.
\end{proof}

In the special case where the $A_{\a}$ and $A_{\b}$ are trivial, we have a simpler description using the maps
\begin{equation}\label{simpler frI}
\xymatrix{& \frI_{U', \a\to\b}\ar[dl]_{\ev_{\frI,S}}\ar[dr]^{\Pi}\\
\bM_{S} & & U'}
\end{equation}

\begin{prop}\label{p:simpler desc} Let $V\in\Rep(\dG;\nu)$, $\a\in\xch(Z\dG)_{I_{F}}$ and $\b=\a+\overline{\nu}$. Under the assumptions in \S\ref{sss:ass}, \eqref{BunZpt} and the extra assumption that $A_{\a}$ and $A_{\b}$ are trivial.  Then the Hecke eigen local system $\calF_{V,\a}$ constructed in Theorem \ref{th:eigen} can be written as
\begin{equation*}
\calF_{V,\a}\cong\Pi_{!}(\ev^{*}_{\frI,S}\cK_{S, \Om}\otimes\IC^{\frI}_{V}).
\end{equation*}
\end{prop}

The proof is similar to that of Proposition \ref{p:desc eigen} and we omit it.

\subsection{Examples for $\GL_{2}$}\label{ss:finalGL2} We shall apply Proposition \ref{p:simpler desc} to the examples considered in \S\ref{ss:GL2}, and check that the Hecke eigen local systems we get do coincide with the classification of cohomologically  rigid rank two local systems in Proposition \ref{p:r2}.

\subsubsection{The example from \S\ref{sss:GL2S3}} The unique $\cK_{S}$-relevant point $\star_{0}$ on $\Bun^{0}_{2}(\bI_{S})$ is the trivial bundle $\calV:=\calO e_{1}\oplus\calO e_{2}$ with lines $\ell_{0}=\jiao{e_{1}}$, $\ell_{1}=\jiao{e_{1}+e_{2}}$ and $\ell_{\infty}=\jiao{e_{2}}$. Let us lift this point to $\Bun^{0}_{2}(\bI^{+}_{S})$, i.e., we choose a basis $u_{x}$ for $\ell_{x}$ which is the spanning vector of $\ell_{x}$ as written above, and let $v_{x}\in\calV_{x}/\ell_{x}$ be given by the image of $e_{2}, e_{2}$ and $e_{1}$ at $0,1$ and $\infty$ respectively. The unique relevant point $\star_{1}$ on $\Bun^{1}_{2}(\bI_{S})$ is the bundle $\calV':=\calO(1) e'_{1}\oplus\calO e'_{2}$ with lines $\ell'_{0}=\jiao{e'_{2}}$, $\ell'_{1}=\jiao{e'_{1}+e'_{2}}$ and $\ell'_{\infty}=\jiao{e'_{2}}$. We similarly lift $\star_{1}$ to $\Bun^{1}_{0}(\bI^{+}_{S})$ by choosing $u'_{x}$ to be the spanning vector of $\ell'_{x}$ as written and $v'_{x}\in\calV'_{x}/\ell'_{x}$ given by the image of $e'_{1}, e'_{2}$ and $e'_{1}$ at $0,1$ and $\infty$ respectively. 

Let $V$ be the standard representation of $\dG=\GL_{2}$. The support $\frG_{V}$ of the sheaf $\IC_{V}$ on the group scheme $\frG_{U, 0\to 1}$ can be described as follows: $\frG_{V}$ classifies embeddings of coherent sheaves $\varphi: \calV\to\calV'$ that send $\ell_{x}$ to $\ell'_{x}$ at each $x\in S$, and is an isomorphism at each $x\in S$. We use the affine coordinate $t$ on $\PP^{1}$ and as such, $\{1,t\}$ can be viewed as a basis for $\Gamma(\PP^{1},\calO(1))$. A map $\calV\to\calV'$ can be written under the bases $e_{i}, e'_{i}$ as a matrix $\varphi=\left(\begin{array}{cc} a_{0}+a_{1}t & b_{0}+b_{1}t \\ c & d\end{array}\right)$ for $a_{0},a_{1}, b_{0}, b_{1},c,d\in k$. The conditions that $\varphi$ should send $\ell_{x}$ to $\ell'_{x}$ imply that $a_{0}=0, b_{1}=0$ and $a_{1}+b_{0}=c+d$. Let us write $a:=a_{1}$ and $b:=b_{0}$. The point where $\varphi$ has a zero is $y=\frac{bc}{ad}$. Since $y\notin S$, we must have $a,b,c,d\neq0$ and $ad\neq bc$. Therefore $\frG_{V}$ is an open subset of $\AA^{3}$. The morphism $\pi:\frG_{V}\incl\frG_{U,0\to1}\to U=\PP^{1}-\{0,1,\infty\}$ sends $(a,b,c,d)\mapsto \frac{bc}{ad}$. Evaluating $\varphi$ at $x\in S$, we get nonzero constants $\alpha^{(1)}_{x}(\varphi)$ and $\a^{(2)}_{x}(\varphi)$ such that $\varphi_{x}(u_{x})=\alpha^{(1)}_{x}(\varphi)u'_{x}$ and $\varphi_{x}(v_{x})=\alpha^{(1)}_{x}(\varphi)v'_{x}$. This gives a map $\a_{x}:\frG_{V}\to \Gm^{2}$ for each $x\in S$. These maps are given in coordinates by
\begin{equation*}
\a_{0}(a,b,c,d)=(c,b);\quad\a_{1}(a,b,c,d)=(a+b,d-b);\quad\a_{0}(a,b,c,d)=(d,a).
\end{equation*}

Since $\bA_{Z,S}=\Gm$ in this case, the scheme $\frI_{V}\subset\frI_{U',0\to1}$ (the support of $\IC^{\frI}_{V}$) is the quotient $\frG_{V}/\Gm$ where $\Gm$ acts on $(a,b,c,d)$ by simultaneous scaling. The tensor product characters sheaf $\boxtimes_{x\in S, \ep\in\{1,2\}}\a^{(\ep),*}_{x}\cK^{(\ep)}_{x}$ descends to a character  sheaf $\cK_{S,\Om}$ on $\bM_{S}\cong\Gm^{6}/\Delta(\Gm)$ thanks to the condition \eqref{det GL2}. We may identify $\frI_{V}$ with a subscheme of $\frG_{V}$ by setting $a+b=c+d=1$. Then
\begin{equation*}
\frI_{V}=\{(a,d)\in(\Gm-\{1\})^{2}|a+d\neq1\}.
\end{equation*}
The maps $\a_{x}$ restricted to $\frI_{V}$ become
\begin{equation*}
\a_{0}(a,d)=(1-d,1-a);\quad\a_{1}(a,d)=(1,a+d-1);\quad\a_{\infty}(a,d)=(d,a).
\end{equation*}
The map $\frI_{V}\to U'=\PP^{1}-\{0,1,\infty\}$ is given by $(a,d)\mapsto(1-a)(1-d)/(ad)$. 

Let us consider the special case where $\chi^{(2)}_{1}$ is trivial. In this case, applying the Lefschetz trace formula to the local system $\calF:=\cF_{V,0}$ as described in Proposition \ref{p:simpler desc}, we get the trace of $\Frob_{x}$ on $\cF$ at $x\in k^{\times}-\{1\}$
\begin{equation*}
f_{\cF,k}(x)=-\sum_{(1-a)(1-d)=adx}\chi^{(1)}_{0}(1-d)\chi^{(2)}_{0}(1-a)\chi^{(1)}_{\infty}(d)\chi^{(2)}_{\infty}(a).
\end{equation*}
Changing variables $a=1/(1-v)$ and $d=1/(1-u)$ for $u,v\in k^{\times}-\{1\}$, we rewrite the above sum as
\begin{equation}\label{preHyp}
f_{\cF,k}(x)=-\chi^{(1)}_{\infty}(-1)\chi^{(2)}_{\infty}(-1)\sum_{uv=x}\chi^{(1)}_{0}(u)(\overline{\chi}^{(1)}_{0}\overline{\chi}^{(1)}_{\infty})(u-1)\chi^{(2)}_{0}(v)(\overline{\chi}^{(2)}_{0}\overline{\chi}^{(2)}_{\infty})(v-1).
\end{equation}

On the other hand, Katz's hypergeometric sheaf $\calH:=\calH(!, \psi; \chi^{(1)}_{0}, \chi^{(2)}_{0}; \overline{\chi}^{(1)}_{\infty}, \overline{\chi}^{(2)}_{\infty})$ is defined as an iterated multiplicative convolution \cite[\S8.2]{Katz-DE}. Its trace function is
\begin{equation*}
f_{\calH,k}(x)=\sum_{x_{1}x_{2}=y_{1}y_{2}x}\psi(x_{1}-y_{1})\chi^{(1)}_{0}(x_{1})\chi^{(1)}_{\infty}(y_{1})\psi(x_{2}-y_{2})\chi^{(2)}_{0}(x_{2})\chi^{(2)}_{\infty}(y_{2}).
\end{equation*}
Making a change of variables $x_{1}=uy_{1}$ and $x_{2}=vy_{2}$ we may rewrite the sum as
\begin{equation}\label{hyp}
\sum_{uv=x}\chi^{(1)}_{0}(u)\chi^{(2)}_{0}(v)\sum_{y_{1}\in k^{\times}}\psi((u-1)y_{1})(\chi^{(1)}_{0}\chi^{(1)}_{\infty})(y_{1})\sum_{y_{2}\in k^{\times}}\psi((v-1)y_{2})(\chi^{(2)}_{0}\chi^{(2)}_{\infty})(y_{2}).
\end{equation}
The two inner sums are equal to $(\overline{\chi}^{(1)}_{0}\overline{\chi}^{(1)}_{\infty})(u-1)G(\psi,\chi^{(1)}_{0}\chi^{(1)}_{\infty})$ and $(\overline{\chi}^{(2)}_{0}\overline{\chi}^{(2)}_{\infty})(v-1)G(\psi,\chi^{(2)}_{0}\chi^{(2)}_{\infty})$, where $G(\psi,\chi)$ stands for the Gauss sum. Comparing \eqref{preHyp} and \eqref{hyp} we get
\begin{equation*}
f_{\calF,k}(x)=-\chi^{(1)}_{\infty}(-1)\chi^{(2)}_{\infty}(-1)G(\psi,\chi^{(1)}_{0}\chi^{(1)}_{\infty})^{-1}G(\psi,\chi^{(2)}_{0}\chi^{(2)}_{\infty})^{-1}f_{\calH,k}(x).
\end{equation*}
This being true for all finite extensions $k'/k$ and $x\in k'^{\times}-\{1\}$, we conclude that $\calF=\calF_{V,0}$ and $\calH=\calH(!, \psi; \chi^{(1)}_{0}, \chi^{(2)}_{0}; \overline{\chi}^{(1)}_{\infty}, \overline{\chi}^{(2)}_{\infty})$ are isomorphic over $U_{\kbar}$. 
 
The case where $\chi^{(2)}_{1}$ is nontrivial can be reduced to this case by twisting $\calF_{V,0}$ by a suitable Kummer sheaf.

\subsubsection{The example from \S\ref{sss:GL2S2I}} The unique $\cK_{S}$-relevant point $\star_{0}\in\Bun^{0}_{2}(\bK_{S})$ is the trivial bundle $\calV=\calO e_{1}\oplus\calO e_{2}$ with the line $\ell_{0}=\jiao{e_{1}}$ at $0$, the vectors $v^{(1)}_{\infty}=e_{2}$ and $v^{(2)}_{\infty}=e_{1}$ viewed as a basis of $\calV_{\infty}/\jiao{v^{(1)}_{\infty}}$.  The unique $\cK_{S}$-relevant point $\star_{1}\in\Bun^{1}_{2}(\bK_{S})$ is the vector bundle $\calV'=\calO(1) e'_{1}\oplus\calO e'_{2}$ with the line $\ell'_{0}=\jiao{e'_{2}}$ at $0$, the vectors $v'^{(1)}_{\infty}=e'_{2}$ and $v'^{(2)}_{\infty}=e'_{1}$ viewed as a basis of $\calV'_{\infty}/\jiao{v'^{(1)}_{\infty}}$.

Let $V$ be the standard representation of $\dG=\GL_{2}$. The support $\frG_{V}$ of the sheaf $\IC_{V}$ on the group scheme $\frG_{U, 0\to 1}$ can be described as follows: $\frG_{V}$ classifies embeddings of coherent sheaves $\varphi: \calV\to\calV'$ that send $\ell_{0}$ to $\ell'_{0}$ and is an isomorphism at $0$, is unipotent upper triangular with respect to the basis $\{v^{(1)}_{\infty}, v^{(2)}_{\infty}\}$ of $\calV_{\infty}$ and the basis $\{v'^{(1)}_{\infty}, v'^{(2)}_{\infty}\}$ of $\calV'_{\infty}$. Again we represent such a map $\calV\to\calV'$ as a matrix $\varphi=\left(\begin{array}{cc} a_{0}+a_{1}t & b_{0}+b_{1}t \\ c & d\end{array}\right)$ for $a_{0},a_{1}, b_{0}, b_{1},c,d\in k$. The conditions on $\varphi$ imply that $\varphi=\left(\begin{array}{cc} t & b \\ c & 1\end{array}\right)$ for $b,c\in\Gm$. In other words, $\frG_{V}=\Gm^{2}$. The point where $\varphi$ has a zero is $y=bc$.  Evaluating $\varphi$ at $0$, we get the map $\ev_{\frG,0}:\frG_{V}\to \bL_{0}=\Gm^{2}$ given by $(b,c)\mapsto (c,b)$. Evaluating $\varphi$ at $\infty$, we see that $\varphi(v^{(1)}_{\infty})=v'^{(1)}_{\infty}+bt^{-1}v'^{(2)}_{\infty}$ and $\varphi(v^{(2)}_{\infty})=v'^{(2)}_{\infty}+cv'^{(1)}_{\infty}$. Therefore the map $\ev_{\frG,\infty}:\frG_{V}\to \bL_{\infty}=\Ga^{2}$ is given by $(b,c)\mapsto(b,c)$.

Notice that in this case $\bA_{Z,S}$ is trivial, so $\frI_{U, 0\to1}=\frG_{U 0\to1}$. The diagram \eqref{simpler frI} now takes the form
\begin{equation*}
\xymatrix{& \Gm^{2}\ar[dl]_{(\sigma, i)}\ar[dr]^{m}\\
\Gm^{2}\times \Ga^{2} & & \Gm}
\end{equation*}
where $\sigma$ is the transposition of the factors of $\Gm^{2}$ and $i$ is the natural embedding $\Gm^{2}\incl\Ga^{2}$; $m:\Gm^{2}\to\Gm$ is the multiplication map. By Proposition \ref{p:simpler desc} we get that the eigen local system for the geometric automorphic datum in \S\ref{sss:GL2S2I} is (notice that $\IC_{V}=\Qlbar[1]$)
\begin{equation*}
\calF_{V,0}=m_{!}(\sigma^{*}\cK_{\chi_{0}}\otimes i^{*}\AS_{\phi})[1]
\end{equation*}
Here $\chi_{0}=(\chi^{(1)}_{0},\chi^{(2)}_{0})$ was part of the geometric automorphic datum at $0$ and $\phi:k\times k\to k$ was the linear function defining part of the geometric automorphic datum at $\infty$.
The Frobenius trace function at $x\in k^{\times}$ is given by
\begin{equation*}
f_{\calF_{V,0}, k}(x)=-\sum_{b,c\in k^{\times}, bc=x}\psi(\phi(b,c))\chi^{(2)}_{0}(b)\chi^{(1)}_{0}(c).
\end{equation*}
In the special case $\chi_{0}=1$ and $\phi(b,c)=b+c$, this is the classical Kloosterman sum up to a sign. For generalizations, see \S\ref{s:Kl}.

\subsubsection{The example from \S\ref{sss:GL2S2II}} The unique $\cK_{S}$-relevant $\star_{0}\in\Bun^{0}_{2}(\bK_{S})$ is the trivial bundle $\calV=\calO e_{1}\oplus\calO e_{2}$ with the line $\ell_{0}=\jiao{e_{1}}$ at $0$, the lines $\ell^{(1)}_{\infty}=\jiao{e_{1}+e_{2}}$ and $\ell^{(2)}_{\infty}=\jiao{e_{2}}$ at $\infty$. The unique $\cK_{S}$-relevant point $\star_{1}\in\Bun^{1}_{2}(\bK_{S})$ is the vector bundle $\calV'=\calO(1) e'_{1}\oplus\calO e'_{2}$ with the line $\ell'_{0}=\jiao{e'_{2}}$ at $0$, the vector $\ell'^{(1)}_{\infty}=\jiao{e'_{1}+e'_{2}}$ and $\ell'^{(2)}_{\infty}=\jiao{e'_{2}}$.

Let $V$ be the standard representation of $\dG=\GL_{2}$. The support $\frG_{V}$ of the sheaf $\IC_{V}$ on the group scheme $\frG_{U, 0\to 1}$ can be described as follows: $\frG_{V}$ classifies embeddings of coherent sheaves $\varphi: \calV\to\calV'$ that send the lines $\ell_{0},\ell^{(1)}_{\infty}$ and $\ell^{(2)}_{\infty}$ to $\ell'_{0},\ell'^{(1)}_{\infty}$ and $\ell'^{(2)}_{\infty}$ respectively, and that $\varphi$ is an isomorphism at $0$ and $\infty$. Again we represent such a map $\calV\to\calV'$ as a matrix $\varphi=\left(\begin{array}{cc} a_{0}+a_{1}t & b_{0}+b_{1}t \\ c & d\end{array}\right)$ for $a_{0},a_{1}, b_{0}, b_{1},c,d\in k$. The conditions on $\varphi$ imply that $\varphi=\left(\begin{array}{cc} (c+d)t & b \\ c & d\end{array}\right)$ for $b,c,d\in\Gm$ and $c+d\neq0$. The evaluation map $\ev_{\frG,0}:\frG_{V}\to \bL_{0}\cong\Gm^{2}$ is given by $(b,c,d)\mapsto (c,b)$. The evaluation map $\ev_{\frG,\infty}:\frG_{V}\to \bL_{\infty}\cong\Gm^{2}\times\Ga^{2}$ is given by $(b,c,d)\mapsto (c+d,d, b,-b)$.  The projection $\Pi_{\frG}:\frG_{V}\to U=\Gm$ is given by $(b,c,d)\mapsto \frac{bc}{(c+d)d}$.

The scheme $\frI_{V}=\frG_{V}/\Gm$ where $\Gm$ acts as simultaneous scaling on $b,c$ and $d$. We may identify $\frI_{V}$ as the subscheme of $\frG_{V}$ with $c+d=1$. Then $\frI_{V}=\{(b,d)\in\Gm\times(\Gm-\{1\})\}$. The map $\ev_{\frI,S}:\frI_{V}\mapsto \bM_{S}=\bL_{S}/\Gm\cong(\Gm^{2}\times\Gm^{2}\times\Ga^{2})/\Gm$ is given by $(b,d)\mapsto [1-d, b, 1, d, b, -b]$. The projection $\Pi:\frI_{V}\to U$ is given by $(b,d)\mapsto b(1-d)/d$. Therefore the Frobenius trace function of $\calF=\calF_{V,0}$ is given by
\begin{equation*}
f_{\calF,k}(x)=-\sum_{b(1-d)=dx}\chi^{(1)}_{0}(1-d)\chi^{(2)}_{\infty}(d)\chi^{(2)}_{0}(b)\psi(\phi^{(1)}(b)-\phi^{(2)}(b)).
\end{equation*}
Making a change variables $b=-u$ and $d=1/(1-v)$, we may rewrite the above sum as
\begin{equation}\label{F12}
f_{\calF,k}(x)=-\chi^{(2)}_{0}(-1)\chi^{(2)}_{\infty}(-1)\sum_{uv=x}\chi^{(1)}_{0}(v)(\overline{\chi}^{(1)}_{0}\overline{\chi}^{(2)}_{\infty})(v-1)\chi^{(2)}_{0}(u)\psi'(u).
\end{equation}
where $\psi'(y)=\psi(\phi^{(2)}(y)-\phi^{(1)}(y))$ is another nontrivial character of $k$.

On the other hand, Katz's hypergeometric sheaf $\calH:=\calH(!, \psi', \chi^{(1)}_{0}, \chi^{(2)}_{0}; \overline{\chi}^{(2)}_{\infty})$ has Frobenius trace function
\begin{equation*}
f_{\calH,k}(x)=\sum_{x_{1}x_{2}=y_{1}x}\psi'(x_{1}-y_{1})\chi^{(1)}_{0}(x_{1})\chi^{(2)}_{\infty}(y_{1})\psi'(x_{2})\chi^{(2)}_{0}(x_{2})
\end{equation*}
Making a change of variables $x_{1}=vy_{1}, x_{2}=u$, we get
\begin{eqnarray}\label{H12}
f_{\calH,k}(x)&=&\sum_{uv=x}\chi^{(1)}_{0}(v)\left(\sum_{y_{1}\in k^{\times}}\psi'((v-1)y_{1})\chi^{(1)}_{0}(y_{1})\chi^{(2)}_{\infty}(y_{1})\right)\psi'(u)\chi^{(2)}_{0}(u)\\
\notag &=&G(\psi',\chi^{(1)}_{0}\chi^{(2)}_{\infty})\sum_{uv=x}\chi^{(1)}_{0}(v)(\overline{\chi}^{(1)}_{0}\overline{\chi}^{(2)}_{\infty})(v-1)\chi^{(2)}_{0}(u)\psi'(u).
\end{eqnarray}
Comparing \eqref{F12} and \eqref{H12} we get
\begin{equation*}
f_{\calF,k}(x)=-\chi^{(2)}_{0}(-1)\chi^{(2)}_{\infty}(-1)G(\psi',\chi^{(1)}_{0}\chi^{(2)}_{\infty})^{-1}f_{\calH,k}(x).
\end{equation*}
This implies that $\calF$ and $\calH$ are isomorphic over $U_{\kbar}$.

\begin{remark} In all three examples, Conjecture \ref{conj:rigid} is easily verified, using the knowledge of Artin conductors worked out in \S\ref{ss:rigidlocGL2}.
\end{remark}

\subsection{Variants and questions}\label{ss:rat}  
We are in the situation of \eqref{sss:ass}.

\subsubsection{A situation where Hecke eigensheaf is guaranteed}\label{sss:more sym} Suppose for some $x\in S$ we are given a pro-algebraic subgroup $\bK^{\sharp}_{x}\subset L_{x}G$ such that 
\begin{itemize}
\item $\bK^{\sharp}_{x}$ contains $\bK_{x}$ as a normal subgroup.
\item The character sheaf $\cK_{x}$ extends to $\cK^{\sharp}_{x}\in\cCS(\bK^{\sharp}_{x})$;
\item The quotient $C_{x}:=\bK^{\sharp}_{x}/\bK_{x}$ is a discrete group over $k$, and the local Kottwitz map $L_{x}G\to\xch(Z\dG)_{I_{F}}$ induces an isomorphism $C_{x}\isom\xch(Z\dG)_{I_{F}}$. 
\end{itemize}

\begin{exam} Let $X=\PP^{1}$, $F=k(t)$ and the quasi-split group $G$ becomes split over $k(t^{1/e})$ for some $e$ prime to $\chk$. Then the integral model $\cG$ is reductive away from $\{0,\infty\}$. Suppose $0\in S$. Let $(\Om,\bK_{S},\cK_{S},\iota_{S})$ be a geometric automorphic datum such that $\bK_{0}$ is a parahoric subgroup and that $N_{L_{0}G}(\bK_{0})/\bK_{0}$ maps isomorphically to $\xch(Z\dG)_{I_{F}}=\xch(Z\dG)_{I_{0}}$. For example we may take $\bK_{0}$ to be an Iwahori subgroup. Suppose also that the character sheaf $\cK_{0}$ extends to $N_{L_{0}G}(\bK_{0})$. Then the group $\bK^{\sharp}_{0}=N_{L_{0}G}(\bK_{0})$ satisfies the conditions in \S\ref{sss:more sym}.
\end{exam}

\begin{theorem}\label{th:eigenvar} Assuming we are in the situation of Theorem \ref{th:eigen} and the extra assumption in \S\ref{sss:more sym} is satisfied. Then the object $\calA\in D_{\cG,\Om}(\bK_{S},\cK_{S})$ can be extended to a Hecke eigensheaf $(\calA,\calF,\varphi)$ in which $\calF_{V}$ is a semisimple local system for all $V\in\Rep(\dG)$. 
\end{theorem}
\begin{proof}[Sketch of proof] Let $\bL^{\sharp}_{x}=\bK^{\sharp}_{x}/\bK^{+}_{x}$ and $\bM^{\sharp}_{S}=\bM_{S}\twtimes{\bL_{x}}\bL^{\sharp}_{x}$. The stack $\Bun_{\cG}(\bK^{+}_{S})$ admits an action of $\bM^{\sharp}_{S}$. Since the character sheaf $\cK_{x}$ extends to $\bK^{\sharp}_{x}$, the character sheaf $\cK_{S,\Om}$ on $\bM_{S}$ also extends to a character sheaf $\cK^{\sharp}_{S,\Om}$ on $\bM^{\sharp}_{S}$. We may then consider the category $D^{\sharp}:=D^{b}_{(\bM^{\sharp}_{S}, \cK^{\sharp}_{S,\Om})}(\Bun^{+})$. Since $\bM^{\sharp}_{S}$ permutes the open and closed substacks $\Bun^{+}_{\a}$ of $\Bun^{+}$ transitively with $\bM_{S}$ the stabilizer of each $\Bun^{+}_{\a}$ ($\a\in\xch(Z\dG)_{I_{F}}$), restricting to a particular component $\Bun^{+}_{0}$ gives an equivalence of categories $D^{\sharp}\isom D^{b}_{(\bM_{S}, \cK_{S,\Om})}(\Bun^{+}_{0})=D_{0}$. Therefore, by Assumptions in \ref{sss:ass}, there is (up to isomorphism) a unique irreducible perverse sheaf $\calA\in D^{\sharp}$, whose restriction to $\Bun^{+}_{\a}$ has to be isomorphic to $\calA_{\a}$ for all $\a\in\xch(Z\dG)_{I_{F}}$. The analog of Lemma \ref{l:unique irr} then says that for any $Y$, the analogous category $D^{\sharp}(Y)$ (with $\Bun^{+}$ replaced by $Y\times\Bun^{+}$) is equivalent to $D^{b}_{c}(Y)$ via the functor $(-)\boxtimes\calA$. Therefore we may define $\calF_{V}\in D^{b}_{c}(U')$ by requiring $\frT_{V}(\calA)\cong \calF_{V}\boxtimes\calA$. By construction $\calF_{V}$ is canonically identified with $\calF_{V,\a}$ for all $\a\in \xch(Z\dG)_{I_{F}}$. The $(\dG,\xch(Z\dG)_{I_{F}})$-weak local system $\{\calF_{V,\a}\}$, under such a canonical identification, becomes an actual $\dG$-local system $\calF_{V}$.
\end{proof}

Next we discuss several questions that naturally arise from Theorem \ref{th:eigen}.

\subsubsection{Rationality issues}\label{sss:rat}
So far in this section we have have been assuming that $k$ is algebraically closed. At various stages of the argument one can in fact work over a more general field $k$, and work with a smaller coefficient field than $\Qlbar$. In our applications (especially to the inverse Galois problem) it is important that the Hecke eigen local systems $\calF$ are defined over a field that is as small as possible.

Suppose we are in the situation of \S\ref{sss:ass} with the extra assumption \eqref{BunZpt}. The description of $\calF_{V,\a}$ given in Proposition \ref{p:desc eigen} (which works over $\kbar$ as stated) allows us to descend it to a smaller base field and a smaller coefficient field as follows.  We use the notation from \S\ref{ss:desc eigen}. We make the following assumptions:
\begin{itemize}
\item For each $\a\in\xch(Z\dG)_{I_{F}}$, the unique $(\Om,\cK_{S})$-relevant point $\calE_{\a}\in\Bun_{\a}$ is defined over $k$. Therefore $A_{\a}$ and its central extension $\tilA_{\a}$ are finite group schemes over $k$. 
\item The representation $\tilW_{\a}$ of $\tilA_{\a}(\kbar)$ can be extended to $\tilA_{\a}(\kbar)\rtimes\Gk$. We choose such an extension for each $\a$ and denote it by $\tilW^{\heartsuit}_{\a}$. By the uniqueness of $\tilW_{\a}$, it always extends to a {\em projective} representation of $\tilA_{\a}(\kbar)\rtimes\Gk$. Therefore this assumption means the vanishing of certain class in $\cohog{2}{k,\Qlbar^{\times}}$.
\item The representations $\tilW^{\heartsuit}_{\a}$ are defined over some field $L$ where $\Ql\subset L\subset\Qlbar$. 
\end{itemize}
We rewrite the right side of \eqref{compute rho} as
\begin{equation}\label{EV rewritten}
\calF_{V,\a}\cong\Hom_{\tilA_{\a}(\kbar)\times\tilA_{\b}(\kbar)}(\tilW^{\heartsuit, \vee}_{\a}\otimes\tilW^{\heartsuit}_{\b}, \wt\Pi_{!}\IC^{\tfrI}_{V}).
\end{equation}
We understand that both $\tilW^{\heartsuit,\vee}_{\a}\otimes\tilW^{\heartsuit}_{\b}$ and $\wt\Pi_{!}\IC^{\tfrI}_{V}$  have $L$-coefficients and view $\wt\Pi_{!}\IC^{\tfrI}_{V}$ as a complex over $U'_{\kbar}$. The above assumptions give a descent datum of the right side of \eqref{EV rewritten} to $U'$ over $k$. Therefore $\calF_{V,\a}$ is defined over $U'$ over $k$, and has $L$-coefficients.

\subsubsection{Rationality of the geometric Satake equivalence} First of all, to define the Satake category, we only need to work with $\Ql$-sheaves, not $\Qlbar$-sheaves. In the version of the geometric Satake equivalence we reviewed in \S\ref{ss:Sat}, we have normalized all $\IC_{V}$ to be pure of weight zero. This involves choosing a square root of the $\ell$-adic cyclotomic character $\chi_{\ell}:\Gk\to\ZZ^{\times}_{\ell}$, and there is no natural choice of such. Assume $k$ is either a finite field or a number field. Following a suggestion of Deligne, one can modify the Satake category to avoid choosing a square root of $\chi_{\ell}$. One defines a subcategory $\Sat^{\Tt}\subset\Sat$ (the superscript stands for Tate), whose objects are finite direct sums of $\IC_{\l}(n)$ ($n\in\ZZ$), where $(n)$ denotes tensoring with the one-dimensional representation $\chi^{n}_{\ell}$ of $\Gk$. The category $\Sat^{\Tt}$ is also a tensor category with $h=\cohog{*}{\Gr,-}$ as a fiber functor. The Tannakian dual of $\Sat^{\Tt}$ is the reductive group $\dG^{\Tt}$ which is related to $\dG$ by
\begin{equation*}
\dG^{\Tt}\cong (\dG\times\Gm^{\Wt})/\mu_{2}.
\end{equation*}
Here $\Gm^{\Wt}$ is  a one-dimensional torus and $\mu_{2}$ is generated by $((-1)^{2\rho},-1)\in\dG\times\Gm^{\Wt}$, where $2\rho\in\xch(\TT)=\xcoch(\dT)$ is the sum of positive roots of $\GG$, viewed as a cocharacter of $\dT$. For each irreducible representation $V\in\Rep(\dG^{\Tt})$, if the central torus $\Gm^{\Wt}$ acts on $V$ via the $w\nth$ power, then the corresponding object $\IC_{V}$ in $\Sat^{\Tt}$ is pure of weight $w$. The one-dimensional representations of $\dG^{\Tt}$ on which $\dG$ acts trivially and $\Gm^{\Wt}$ acts via the $2n\nth$ power corresponds to the sheaf $\Qlbar(-n)$ supported at the point stratum $\Gr_{0}$. For details of the construction, see \cite[\S2.1]{YunKl}.  For a quick algebraic account, see \cite[\S2]{FG}.

\begin{exam}\label{ex:n odd} Let $n\geq2$ be an integer and $G=\PGL_{n}$. We have $\dG=\SL_{n}$. In this case, $(-1)^{2\rho}=(-1)^{n-1}I_{n}\in\dG$. We have
\begin{equation*}
\dG^{\Tt}\cong\begin{cases}\SL_{n}\times (\Gm^{\Wt}/\mu_{2}) & \mbox{ if $n$ is odd;}\\
(\SL_{n}\times\Gm^{\Wt})/\Delta(\mu_{2}) & \mbox{ if $n$ is even.}\end{cases}
\end{equation*}
We define a ``standard representation'' of dimension $n$ and weight $n-1$:
\begin{eqnarray*}
\St&:&\dG^{\Tt}\to \GL_{n}\\
&&(g,t)\mapsto t^{n-1}g\in\GL_{n} (\textup{where }g\in\SL_{n}, t\in\Gm^{\Wt} ).
\end{eqnarray*}
\end{exam}

With $\dG^{\Tt}$ replacing $\dG$, a Hecke eigen local system $\calF$ should be a twisted $\dG^{\Tt}$-local system over $U$ together with extra compatibility datum with Tate twists. More precisely, $\calF$ is a tensor functor $\Rep(\dG^{\Tt})\to\Loc(U')$ , $V\mapsto\calF_{V}$ (with $\htheta$-twisted descent datum), together with an isomorphism $\calF_{\Ql(1)}\cong\Ql(1)$. Here the $\Ql(1)$ in the subscript means the one-dimensional representation of $\dG^{\Tt}$ on which $\dG$ acts trivially and $\Gm^{\Wt}$ acts through the $(-2)\nth$ power, and the other $\Ql(1)$ is the Tate-twisted constant sheaf on $U'$.

Next we discuss local and global monodromy of Hecke eigen local systems. We are back to the situation where $k$ is algebraically closed.
 
\subsubsection{Local monodromy of $\calF$}
In certain cases we can describe the local monodromy of the Hecke eigensheaf $\calF$ at some point $x\in S$. Suppose $\bK_{x}$ is a parahoric subgroup of $L_{x}G$ with reductive quotient $\bL_{x}$, and that $\cK_{x}$ is a character sheaf on $\bK_{x}$ that descends to $\bL_{x}$. In this case the local system $\calF$ is tame at $x$. For simplicity suppose $G$ is split over $F_{x}$, then the restriction of $\cK_{x}$ to a split maximal torus $\TT\subset \bL_{x}$ gives a character sheaf on $\TT$, which is a homomorphism $\pi^{t}_{1}(\Gm)\otimes\xcoch(\TT)\to \Qlbar^{\times}$, or $\rho'_{x}:\pi^{t}_{1}(\Gm)\to\dT(\Qlbar)$. On the other hand, we may canonically identify $\pi^{t}_{1}(\Gm)$ with the tame inertia group $I_{x}$ because both of them are canonically isomorphic to $\varprojlim_{(n,\chk)=1}\mu_{n}(k)$. Therefore $\rho'_{x}$ induces $\rho''_{x}: I_{x}\to \dT(\Qlbar)\incl\dG(\Qlbar)$, which can be shown to be the {\em semisimple part} of the local monodromy of $\calF$ at $x$. When $G$ is not necessarily split over $F_{x}$, the restriction of $\cK_{x}$ to a maximal split torus gives a section $\rho_{x}: I_{x}\to \dT(\Qlbar)\rtimes I_{x}$ up to $\dT$-conjugacy, which is again the semisimple part of the monodromy of $\calF$ at $x$. 

In \cite[\S4.9]{Y-GenKloo} we give a recipe for the unipotent part of the monodromy using Lusztig's theory of two-sided cells in affine Weyl groups. In some cases this can be proved using the techniques in \cite[\S4.10-4.16]{Y-GenKloo}.  

A special case where we know the local monodromy is when $\bK_{x}$ is an Iwahori subgroup and $\cK_{x}$ is trivial. This is the case where the extra condition in \S\ref{sss:more sym} is often satisfied, for example when $\xch(Z\dG)_{I_{x}}=\xch(Z\dG)_{I_{F}}$. In this case the local monodromy of the Hecke eigen local system $\calF$ at $x$ is a regular unipotent element in $\dG$.

\subsubsection{Global monodromy of $\calF$} When we work over a finite field $k$, the proof of Lemma \ref{l:Els} shows that the local systems $\calF_{V,\a}$ are pure. 
Deligne's theorem \ref{th:Del ss} guarantees that the {\em neutral component } of $\dG^{\geom}_{\calF}$ (the global monodromy of $\calF$) is a connected semisimple subgroup of $\dG$.  When we know a local monodromy of $\calF$ is a regular unipotent element and $\dG$ is almost simple, then either $\dG^{\geom,\c}_{\calF}=\dG$, or $\dG^{\geom,\c}_{\calF}=\dG^{\sigma}$ for some pinned automorphism $\sigma$ of $\dG$, or $\dG^{\geom,\c}$ is the image of a principal homomorphism $\SL_{2}\to\dG$ (see \cite[\S13]{FG}). One can often determine $\dG^{\geom}_{\calF}$ in this case by using information from local monodromy at other points. This method enables us to determine the global monodromy of the examples in the next sections \S\ref{s:Kl} and \S\ref{s:3p}.

\section{Kloosterman sheaves as rigid objects over $\PP^{1}-\{0,\infty\}$}\label{s:Kl} 
In this section we review the work \cite{HNY}, in which we used rigid automorphic representations ramified at two places to construct generalizations of Kloosterman sheaves.

\subsection{Kloosterman automorphic data}\label{ss:Kl auto}   
In this section, the curve $X=\PP^{1}$ with function field $F=k(t)$, where $k$ is a finite field. We assume that {\em $\GG$ is almost simple}. Let $e$ be a positive integer prime to $\chk$ such that $k$ contains $e\nth$ roots of unity. Let  $F'=k(t^{1/e})$ and $\theta: \Gal(F^{s}/F)\surj\Gal(F'/F)\cong\mu_{e}(k)\incl\Aut^{\dagger}(\GG)$ be a homomorphism. As in \S\ref{sss:qsplit}, these data determine a quasi-split form $G$ of $\GG$ over $F$ . We have an integral model $\calG$ over $X$ defined in \S\ref{sss:cG}. 

Let $S=\{0,\infty\}$. We shall define a geometric automorphic datum for $G$ with respect to $S$, to be called the {\em Kloosterman automorphic datum}.  

Let $\bK_{0}=\bI_{0}$ be an Iwahori subgroup of $L_{0}G$ and $\bK_{\infty}=\bI^{+}_{\infty}$, the pro-unipotent radical of an Iwahori $\bI_{\infty}\subset L_{\infty}G$.

The reductive quotient of $\bI_{0}$ is identified with the torus $\SS=\TT^{I_{F},\circ}$. A character $\chi: \SS(k)\to \Qlbar^{\times}$ corresponds to a rank one character sheaves $\cK_{0}=\cK_{\chi}$ on $\SS$, which can also be viewed as a character sheaf on $\bK_{0}=\bI_{0}$. 

Let $\bI^{+}_{\infty}$ be the pro-unipotent radical of $\bI_{\infty}$. It is generated by the root groups $(L_{\infty}G)(\wta)$ of all positive affine roots of the loop group $L_{\infty}G$. Note that $\bI_{\infty}/\bI^{+}_{\infty}\cong\SS$. Let $\bI^{++}_{\infty}\subset\bI^{+}_{\infty}$ be the next step in the Moy-Prasad filtration of $\bI_{\infty}$. There is an $\SS$-equivariant isomorphism
\begin{equation}\label{map to simple roots}
\bI^{+}_{\infty}/\bI^{++}_{\infty}\cong\prod_{i=0}^{r}(L_{\infty}G)(\alpha_{i})\end{equation}
where the product runs over simple affine roots. Each $(L_{\infty}G)(\alpha_{i})$ is isomorphic to $\Ga$ over $k$. A linear function $\phi:(\bI^{+}_{\infty}/\bI^{++}_{\infty})(k)\to k$ is said to be {\em generic} if it does not vanish on any of the factors $(L_{\infty}G)(\alpha_{i})$. Fix such a generic linear function $\phi$ and fix a nontrivial additive character $\psi$ of $k$.  The composition $\psi\circ\phi$ gives a character of $\bI^{+}_{\infty}(k)$, and hence a rank one character sheaf $\cK_{\infty}=\cK_{\phi}$ on $\bI^{+}_{\infty}$. 

Finally, $\Bun_{\cZ}(\bK_{Z,S})$ is a point in this case, and Lemma \ref{l:unique Om}  implies that there is a unique choice of $(\Om, \iota_{S})$ (up to isomorphism) making $(\Om, \bK_{S}, \cK_{S},\iota_{S})$ a geometric automorphic datum for $G$ with respect to $S=\{0,\infty\}$. We call it the {\em Kloosterman automorphic datum}. The situation considered in \S\ref{sss:GL2S2I}, when $\GL_{2}$ is replaced with $\SL_{2}$ or $\PGL_{2}$, is a special case of the situation considered here. 

\begin{remark} The automorphic datum $(\bI^{+}_{\infty},\psi\circ\phi)$ at $\infty$ picks out those representations of $G(F_{\infty})$ that contain nonzero eigenvectors of $\bK_{\infty}$ on which $\bK_{\infty}$ acts through the character $\psi\circ\phi$. These representations are first discovered by Gross and Reeder \cite[\S9.3]{GR}, and they call them {\em simple supercuspidal representations}. Such representations appear as direct summands of the compact induction $\textup{c-Ind}^{G(F_{\infty})}_{\bI^{+}_{\infty}(k)}(\psi\circ\phi)$.
\end{remark}

\begin{theorem}[Gross \cite{Gross-Prescribe}, alternative proof by Heinloth-Ng\^o-Yun \cite{HNY}]\label{th:Kl auto} 
\begin{enumerate}
\item Let $\phi$ be a generic linear functional on $(\bI^{+}_{\infty}/\bI^{++}_{\infty})(k)
$ as above. Then the Kloosterman automorphic datum $(\Om, \bK_{S}, \cK_{S},\iota_{S})$ is strongly rigid.
\item Let $(\onat, K_{S}, \chi_{S})$ be the restricted automorphic datum attached to $(\Om, \bK_{S}, \cK_{S},\iota_{S})$, and let $\pi(\chi,\phi)$ be a $(\onat, K_{S}, \chi_{S})$-typical automorphic representation of $G(\AA_{F})$, which is unique up to an unramified twist. When $\chi=1$, the local component at $0$ of $\pi(1,\phi)$ is, up to an unramified twist, the Steinberg representation of $G(F_{0})$.
\end{enumerate}
\end{theorem}

The argument for Part (1) is a generalization of the argument of Proposition \ref{p:GL2S2I}. In particular, we show that each component of $\Bun_{\cG}(\bK_{S})$ contains a unique $\cK_{S}$-relevant point (with trivial automorphism group in this case). 

\subsection{Kloosterman sheaves}   
Theorem \ref{th:eigenvar} is applicable to the Kloosterman automorphic datum. Therefore we have a Hecke eigen local system for the geometric automorphic datum $(\Om, \bK_{S}, \cK_{S},\iota_{S})$, which we denote by $\Kl_{\dG}(\chi,\phi)$ and call it the {\em Kloosterman sheaf} attached to $(\dG,\htheta)$ and the characters $\chi$ and $\phi$. This is a $\htheta$-twisted $\dG$-local system over $\wt{\Gm}$, the $e\nth$ Kummer covering of $\Gm=\PP^{1}_{k}-\{0,\infty\}$. To explain the namesake, we first recall some facts about the classical Kloosterman sheaf defined by Deligne.

\subsubsection{The classical Kloosterman sheaf}
Recall the definition of Kloosterman sums. Let $p$ be a prime number. Fix a nontrivial additive character $\psi:\FF_p\to\Qlbar^{\times}$. Let $n\geq2$ be an integer. Then the $n$-variable Kloosterman sum over $\FF_p$ is a function on $\FF^{\times}_{p}$ whose value at $a\in\FF_{p}^{\times}$ is
\begin{equation*}
\Kl_n(p;a)=\sum_{x_1,\cdots,x_n\in\FF^\times_p;x_1x_{2}\cdots x_{n}=a}\psi(x_1+\cdots+x_n).
\end{equation*}
These exponential sums arise naturally in the study of automorphic forms for $\GL_{n}$.

Deligne \cite{Deligne-ExpSum} gave a geometric interpretation of the Kloosterman sum. He considered the following diagram of schemes over $\FF_{p}$
\begin{equation*}
\xymatrix{& \Gm^n\ar[dl]_{\pi}\ar[dr]^{\sigma}\\ \Gm & & \AA^1}
\end{equation*}
Here $\pi$ is the morphism of taking the product and $\sigma$ is the morphism of taking the sum. 

\begin{defn}[Deligne \cite{Deligne-ExpSum}] The {\em Kloosterman sheaf} is 
\begin{equation*}
\Kl_n:=\bR^{n-1}\pi_!\sigma^*\AS_\psi,
\end{equation*}
over $\Gm=\PP^{1}_{\FF_{p}}-\{0,\infty\}$. Here $\AS_{\psi}$ is the Artin-Schreier sheaf as in Example \ref{ex:AS}. 
\end{defn}
The relationship between the local system $\Kl_{n}$ and the Kloosterman sum $\Kl_{n}(p;a)$ is explained by the following identity
\begin{equation*}
\Kl_n(p;a)=(-1)^{n-1}\Tr(\Frob_a,(\Kl_{n})_a).
\end{equation*}
Here $\Frob_{a}$ is the geometric Frobenius operator acting on the geometric stalk $(\Kl_{n})_{a}$ of $\Kl_{n}$ at $a\in\Gm(\FF_{p})=\FF_{p}^{\times}$.

In \cite[Th\'eor\`eme 7.4, 7.8]{Deligne-ExpSum}, Deligne proved:
\begin{enumerate}
\item $\Kl_n$ is a local system of rank $n$, pure of weight $n-1$. This implies the Weil-type bound $|\Kl_n(p;a)|\leq np^{(n-1)/2}$.
\item $\Kl_n$ is tamely ramified at $0$, and the monodromy is unipotent with a single Jordan block.
\item $\Kl_n$ is totally wild at $\infty$ (i.e., the wild inertia at $\infty$ has no nonzero fixed vector on the stalk of $\Kl_n$), and the Swan conductor $\Sw_\infty(\Kl_n)=1$. 
\end{enumerate}

\begin{remark} For $\dG=\SL_{n}$ (resp. $\Sp_{2n}$), the Kloosterman sheaf $\Kl_{\dG}(1,\phi)$, evaluated at the standard representation of $\dG$, is the same as $\Kl_{n}$ (resp. $\Kl_{2n}$) of Deligne up to a Tate twist. When $\dG=\SO_{2n+1}$ or $G_{2}$,  $\Kl_{\dG}(\chi,\phi)$ was constructed by Katz in \cite{Katz-DE} by different methods (as special cases of hypergeometric sheaves). Our construction of $\Kl_{\dG}(\chi,\phi)$ using geometric automorphic data treats all $\dG$ uniformly, and gives the first examples of motivic local systems with geometric monodromy group $F_{4}, E_{7}$ and $E_{8}$ (see Theorem \ref{th:HNY}(3) below).
\end{remark}

In \cite{HNY} we prove several results on the local and global monodromy of the Kloosterman sheaves $\Kl_{\dG}(\chi,\phi)$. 
\begin{theorem}[Heinloth-Ng\^o-Yun \cite{HNY}]\label{th:HNY} Assume $G$ is split, then $\Kl_{\dG}(\chi,\phi)$ enjoys the following properties.
\begin{enumerate}
\item\label{tame0} $\Kl_{\dG}(\chi,\phi)$ is tame at $0$. A generator of the tame inertia $I^{t}_{0}$ maps to an element in $\dG$ with semisimple part given by $\chi$, viewed as an element in $\dT$. When $\chi=1$, a generator of $I^{t}_{0}$ maps to a regular unipotent element in $\dG$.
\item The local monodromy of $\Kl_{\dG}(\chi,\phi)$ at $\infty$ is a {\em simple wild parameter} in the sense of Gross and Reeder \cite[\S5]{GR}. In particular, $\Sw_{\infty}(\Ad(\Kl_{\dG}(\chi,\phi)))=r$ (the rank of $\dG$). For more details see \S\ref{sss:Kl infty}.
\item If $\chi=1$, then the global geometric monodromy group of $\Kl_{\dG}(1,\phi)$ is a connected almost simple subgroup of $\dG$ of types given by the following table
\begin{center}
\begin{tabular}{|l|l|l|}
\hline
$\dG$ & $\dG^{\geom}_{\Kl_{\dG}(1,\phi)}$ & $\textup{condition on char}(k)$\\\hline
$A_{2n}$ & $A_{2n}$ & $p>2$\\
$A_{2n-1}, C_n$ &  $C_n$ & $p>2$\\ 
$B_n, D_{n+1}$ $(n\geq4)$ & $B_n$ & $p>2$ \\
$E_7$ & $E_7$ & $p>2$\\
$E_8$ & $E_8$ & $p>2$\\
$E_6, F_4$ & $F_4$ & $p>2$\\
$B_3,D_4, G_2$ & $G_2$ & $p>3$\\
\hline
\end{tabular}
\end{center}
\end{enumerate}
\end{theorem}

\subsubsection{Local monodromy at $\infty$}\label{sss:Kl infty} Let us explain in more detail what a {\em simple wild parameter} looks like when $G$ is split, following Gross and Reeder \cite[Proposition 5.6]{GR}.  Assume $p=\textup{char}(k)$ does not divide $\#W$ ($W\cong\WW$ is the Weyl group of $\dG$). Let $\rho_{\infty}:I_{\infty}\to \dG(\Qlbar)$  be the local monodromy of $\Kl_{\dG}(\chi,\phi)$ at $\infty$. Then we have a commutative diagram
\begin{equation*}
\xymatrix{1\ar[r] & I^{w}_{\infty}\ar[r]\ar[d]^{\rho_{\infty}|_{I^{w}_{\infty}}} & I_{\infty}\ar[r]\ar[d]^{\rho_{\infty}} & I^{t}_{\infty}\ar[r]\ar[d] &1\\
1\ar[r]&\dT\ar[r] & N_{\dG}(\dT)\ar[r] & W\ar[r] & 1}
\end{equation*}
The image of $I^{t}_{\infty}$ in $W$ is the cyclic group generated by a Coxeter element $\Cox\in W$, whose order is the Coxeter number $h$ of $G$.  The image $\rho_{\infty}(I^{w}_{\infty})$ is a $\FF_{p}$-vector space equipped with the action of $\Cox$. In fact $\rho_{\infty}(I^{w}_{\infty})\cong\FF_{p}[\zeta_{h}]$,  the extension of $\FF_{p}$ by adjoining $h\nth$ roots of unity, and the Coxeter element acts by multiplication by a primitive $h\nth$ root of unity.

When $p\mid\#W$, a simple wild parameter can be more complicated. For example, when $\dG=\PGL_{2}$ and $p=2$, the image of $\rho_{\infty}$ is isomorphic to the alternating group $A_{4}$ embedded in $\PGL_{2}(\Qlbar)=\SO_{3}(\Qlbar)$ as the symmetry of a regular tetrahedron.

\subsection{Generalizations}  
In \cite{Y-GenKloo} we give further generalizations of Kloosterman sheaves. We work with the same class of quasisplit $G$ as in \S\ref{ss:Kl auto}. We will replace $\bI_{0}$ and $\bI_{\infty}$ by more general parahoric subgroups.

\subsubsection{Admissible parahoric subgroups and epipelagic representations} In their construction of supercuspidal representations, Reeder and Yu \cite{RY} singled out a class of parahoric subgroups $\bP_{\infty}\subset L_{\infty}G$ as follows. Let $\bP^{++}_{\infty}\subset\bP^{+}_{\infty}\subset\bP_{\infty}$ be the first three steps of the Moy-Prasad filtration on $\bP_{\infty}$. Then the vector group $\bV_{\bP}:=\bP^{+}_{\infty}/\bP^{++}_{\infty}$ is a representation of the reductive group $\bL_{\bP}:=\bP_{\infty}/\bP^{+}_{\infty}$. A geometric point in the dual space $\bV^{*}_{\bP}$ is called {\em stable} if its $\bL_{\bP}$-orbit is closed and its stabilizer under $\bL_{\bP}$ is finite. Let $\bV^{*,\st}_{\bP}$ be the open subset of stable points. A parahoric subgroup $\bP_{\infty}\subset L_{\infty}G$ is called {\em admissible} if $\bV^{*,\st}_{\bP}\neq\varnothing$. When $\chk$ is large, conjugacy classes of admissible parahorics are in bijection with regular elliptic numbers of the pair $(\WW,\theta)$. To each stable point $\phi\in\bV^{*}_{\bP}(k)$, Reeder and Yu construct a class of irreducible supercuspidal representations from the compact induction $\textup{c-Ind}_{\bP^{+}_{\infty}(k)}^{G(F_{\infty})}(\psi\circ\phi)$ ($\psi$ is a fixed nontrivial additive character of $k$), and they call them {\em epipelagic representations}.

\subsubsection{Generalized Kloosterman automorphic datum} Let $\bP_{\infty}\subset L_{\infty}G$ be an admissible parahoric subgroup. Let $\bP_{0}\subset L_{0}G$ be a parahoric of the same type as $\bP_{\infty}$. Fix a stable point $\phi\in\bV_{\bP}^{*,\st}(k)$. The character $\psi\circ\phi$ of $\bV_{\bP}(k)$ corresponds to a rank one character sheaf $\cK_{\phi}$ on $\bV_{\bP}$ (and hence on $\bP^{+}_{\infty}$). Fix another character $\chi:\bL_{\bP}(k)\to\Qlbar^{\times}$, which corresponds to a rank one character sheaf $\cK_{\chi}$ on $\bL_{\bP}$ (and hence on $\bP_{0}$). We consider the geometric automorphic datum with respect to $S=\{0,\infty\}$ consisting of
\begin{equation*}
\bK_{S}=(\bP_{0}, \bP^{+}_{\infty}); \quad \cK_{S}=(\cK_{\chi}, \cK_{\phi})
\end{equation*}
and the unique compatible choice of $(\Om,\iota_{S})$ as in the case of Kloosterman automorphic datum. One can show an analog of Theorem \ref{th:Kl auto} which says that  $(\Om,\bK_{S}, \cK_{S}, \iota_{S})$ is strongly rigid.

\subsubsection{Generalized Kloosterman sheaves} Theorem \ref{th:eigenvar} applies to obtain Hecke eigen local systems for the geometric automorphic datum $(\Om,\bK_{S}, \cK_{S}, \iota_{S})$. We denote the resulting Hecke eigen local system by $\Kl_{\dG, \bP}(\chi,\phi)$, which is a $\htheta$-twisted $\dG$-local system over $\wt\Gm$. One new feature of this generalization is that when $\phi$ varies in $\bV^{*,\st}_{\bP}$, the corresponding local systems $\Kl_{\dG, \bP}(\chi,\phi)$  ``glue'' together to give a $\dG$-local system over $\bV^{*,\st}_{\bP}\times\wt\Gm$. We also have an analog of Theorem \ref{th:HNY}\eqref{tame0}: when $G$ is split, the monodromy of $\Kl_{\dG, \bP}(1,\phi)$ at $0$ is tame and unipotent, and the corresponding unipotent class $u_{\bP}$ in $\dG$ can be described purely in terms of $\bP$. For details we refer to \cite{Y-GenKloo}.

\section{Rigid objects over $\PP^{1}_{\QQ}-\{0,1,\infty\}$ and applications}\label{s:3p} 
In this section, we review the work \cite{Y-motive}, in which we use rigid automorphic representations to construct local systems on $\PP^{1}_{\QQ}-\{0,1,\infty\}$. These local systems are the key objects that lead to the answer to Serre's question and the solution of the inverse Galois problem for certain finite simple groups of exceptional Lie type.

\subsection{The geometric automorphic data} 
Let $k$ be a field with $\chk\neq2$. Let $X=\PP^{1}_{k}$ and $S=\{0,1,\infty\}$. Assume $\GG$ is almost simple and simply-connected in this section. 

When the longest element $w_{0}$ in the Weyl group $\WW$ of $\GG$ acts by $-1$ on $\xcoch(\TT)$, we take $G$ to the split group $\GG\otimes_{k}F$. Otherwise, let $F'=k(t^{1/2})$ and let $\theta:\Gamma_{F}\surj\Gal(F'/F)\to \Aut^{\dagger}(\GG)$ map the nontrivial element in $\Gal(F'/F)$ to the unique pinned involution $\sigma$ that acts by $-w_{0}$ on $\xcoch(\TT)$. We define $G$ to be the quasi-split form of $\GG$ over $F$ using $\theta$ as in \S\ref{sss:qsplit}.

Recall that a Chevalley involution of $\GG$ is an involution $\tau$ such that $\dim G^{\tau}$ has the minimal possible dimension, namely $\#\Phi^{+}$ (the number of positive roots of $G$). All Chevalley involutions are conjugate to each other under $G^{\ad}(\kbar)$. The Chevalley involution for $\GG$ is inner if and only if $w_{0}$ acts by $-1$ on $\xcoch(\TT)$.

\subsubsection{A parahoric subgroup}\label{sss:Chev para} Up to conjugacy, there is a unique parahoric subgroup $\bP_{0}\subset L_{0}G$ such that its reductive quotient $\bL_{0}$ is isomorphic to the fixed point subgroup $\GG^{\tau}$ of a Chevalley involution $\tau$. For example, we can take $\bP_{0}$ to be the parahoric subgroup corresponding to the facet containing the element $\rho^{\vee}/2$ in the $S$-apartment of the building of $L_{0}G$, where $\rho^{\vee}$ is half the sum of positive coroots of $\GG$ and $S$ is the maximal split torus of $G$ with $\xcoch(S)=\xcoch(\TT)^{\Gal(F'/F)}$. 

The Dynkin diagram of the reductive quotient $\bL_{0}\cong G^{\tau}$ of $\bP_{0}$ is obtained by removing one or two nodes from the extended Dynkin diagram of $G$. We tabulate the type of $L_{\bP}$ and the node(s) in the affine Dynkin diagram of $L_{0}G$ to be removed in each case. When we say a node is long or short, we mean its corresponding affine simple root is long or short. In the following table $n\geq1$ (and we think of $A_{1}$ as $C_{1}$).

\begin{tabular}{|c|c|c|c|}
\hline
$\GG$ & $G$ split? & $\bL_{0}\cong G^{\tau}$ & nodes to be removed \\ \hline
$A_{2n}$ & no & $B_{n}$ & longest node \\   \hline
$A_{2n+1}$ & no & $D_{n+1}$ & longest node \\   \hline
$B_{2n}$ & yes & $B_{n}\times D_{n}$ & the $(n+1)\nth$ counting from the short node \\ \hline
$B_{2n+1}$ & yes &$B_{n}\times D_{n+1}$ & the $(n+1)\nth$ counting from the short node \\ \hline
$C_{n}$ & yes & $A_{n-1}\times\Gm $ & the two ends \\ \hline
$D_{2n}$ & yes & $D_{n}\times D_{n}$ & the middle node \\ \hline 
$D_{2n+1}$ & no & $B_{n}\times B_{n}$ & the middle node \\ \hline
$E_{6}$ & no &   $C_{4}$  & the long node on one end \\   \hline
$E_{7}$ & yes & $A_{7}$ & the end of the leg of length 1 \\ \hline
$E_{8}$ & yes & $D_{8}$ & the end of the leg of length 2 \\ \hline
$F_{4}$ & yes & $A_{1}\times C_{3}$ & second from the long node end\\ \hline
$G_{2}$ & yes & $A_{1}\times A_{1}$ & middle node \\ \hline
\end{tabular}

Examining all the cases we get
\begin{lemma} If $\GG$ is not of type $C$, then $\bL^{\sc}_{0}\to \bL_{0}$ is a double cover (i.e., the algebraic fundamental group of $\bL_{0}$ has order two).
\end{lemma}
Even if $\GG$ is of type $C_{n}$, $\bL_{0}\cong\GL_{n}$ still admits a unique nontrivial double cover. In all cases, there is a canonical nontrivial double cover $v:\wt{\bL}_{0}\to \bL_{0}$. In particular,
\begin{equation*}
\calK_{0}:=(v_{!}\Qlbar)_{\textup{sgn}}
\end{equation*}
is a rank one character sheaf on $\bL_{0}$ (here $\textup{sgn}$ denotes the nontrivial character of $\ker(v)=\{\pm1\}$). When $k$ is a finite field, we have an exact sequence
\begin{equation*}
1\to\{\pm1\}\to\wt{\bL}_{0}(k)\to \bL_{0}(k)\to\cohog{1}{k,\mu_{2}}=\{\pm1\}.
\end{equation*}
The character $\chi_{0}$ corresponding to $\calK_{0}$ is given by the last arrow above.

\subsubsection{The automorphic datum} Let $\bP_{0}\subset L_{0}G$ be a parahoric subgroup of the type defined in \S\ref{sss:Chev para}. Let $\bP_{\infty}\subset L_{\infty}G$ be a parahoric subgroup of the same type as $\bP_{0}$ (since $L_{0}G$ and $L_{\infty}G$ are isomorphic).  Let $\bI_{1}\subset L_{1}G$ be an Iwahori subgroup. We consider the geometric automorphic datum given by
\begin{eqnarray*}
(\bK_{0},\calK_{0})&=&(\bP_{0},\calK_{0});\\
(\bK_{1}, \calK_{1})&=&(\bI_{1}, \Qlbar);\\
(\bK_{\infty}, \calK_{\infty})&=&(\bP_{\infty}, \Qlbar).
\end{eqnarray*}

Since $\Bun_{\cZ}(\bK_{Z,S})=\BB(\ZZ\GG[2])$, by Lemma \ref{l:unique Om}, the choices of $(\Om,\iota_{S})$ to complete $(\bK_{S},\cK_{S})$ into a geometric automorphic datum are in bijection with the choices of a descent $\overline{\cK}_{0}$ of $\cK_{0}$ to $\bL_{0}/\ZZ\GG[2]$. For this we require that $\wt{\ZZ\GG[2]}=v^{-1}(\ZZ\GG[2])\subset\wt\bL_{0}$ be discrete and we need to choose a character $c:\wt{\ZZ\GG[2]}(k)\to\Qlbar^{\times}$ extending the sign character on $\mu_{2}=\ker(v)$. Therefore, when $\wt{\ZZ\GG[2]}$ contains a factor $\mu_{4}$, we need to assume that $\sqrt{-1}\in k$.

The main technical result of \cite{Y-motive} is the following.
\begin{theorem}[\cite{Y-motive}]\label{th:3pt eigen} Assume $G$ is simply-connected, split and is either simply-laced or of type $G_{2}$. Assume that $k$ contains $\sqrt{-1}$ when $G$ is of type $A_{1}, D_{4n+2}$ or $E_{7}$. Then the geometric automorphic datum $(\Om,\bK_{S},\cK_{S}, \iota_{S})$ satisfies all the assumptions in \S\ref{sss:ass} (note that $\xch(Z\dG)$ is trivial in this case). Therefore Theorem \ref{th:eigen} applies to give a Hecke eigen $\dG$-local system $\calF$ on $U=\PP^{1}_{k}-\{0,1,\infty\}$ attached to the geometric automorphic datum $(\Om,\bK_{S},\cK_{S}, \iota_{S})$.
\end{theorem}

Note that the conditions put on $G$ in the above theorem limits $G$ to be simply-connected of type $A_{1}, D_{2n}, E_{7}, E_{8}$ and $G_{2}$. 

For the unique $(\Om,\cK_{S})$-relevant point $\cE\in\Bun_{G}(\bK_{S})$, we have $\Aut_{G,\bK_{S}}(\cE)\cong\TT[2]$ and $A_{\cE}\cong \TT[2]/\ZZ\GG[2]$. Here we are identifying $\TT$ with a maximal torus of $\bL_{0}$. Let $\tilA_{\cE}=v^{-1}(\TT[2])\subset\wt\bL_{0}$, then we have an exact sequence
\begin{equation*}
1\to \wt{\ZZ\GG[2]}\to \tilA_{\cE}\to A_{\cE}\to1.
\end{equation*}
The character sheaf $\overline{\cK}_{0}\in\cCS(\bL_{0}/\ZZ\GG[2])$ restricts to a cocycle $\xi\in\cohog{2}{A_{\cE}, \Qlbar^{\times}}$ and gives the category  $\Rep_{\xi}(A_{\cE})$. It turns out that $\Rep_{\xi}(A_{\cE})\cong\Rep(\tilA_{\cE}; c)$, the latter being the category of representations of $\tilA_{\cE}$ (a discrete group) on which $\wt{\ZZ\GG[2]}$ acts through the character $c$. The commutator pairing $A_{\cE}\times A_{\cE}\to \mu_{2}=\ker(v)$ is  non-degenerate, and $c|_{\ker(v)}$ is nontrivial. These facts imply that $\Rep(\tilA_{\cE}; c)$ contains a unique irreducible object up to isomorphism. This is also a situation where Proposition \ref{p:desc eigen} applies.

The following theorem summarizes the local and global monodromy of the Hecke eigen $\dG$-local system $\calF$.
\begin{theorem}[{\cite{Y-motive},\cite[\S9]{Y-GenKloo}}] Let $G$ and $\calF$ be as in Theorem \ref{th:3pt eigen}.
\begin{enumerate}
\item The local system $\calF$ is tame.
\item A topological generator of $I^{t}_{1}$ maps to a regular unipotent element in $\dG$.
\item A topological generator of $I^{t}_{\infty}$ maps to a unipotent element whose centralizer in $\dG$ has dimension equal to $\#\Phi^{+}(\dG)$ (unique up to conjugacy).
\item A topological generator of $I^{t}_{0}$ maps to a Chevalley involution in $\dG$.
\item The local system $\calF$ is cohomologically rigid.
\item When $G$ is of type $A_{1}, E_{7}, E_{8}$ and $G_{2}$, the global geometric monodromy of $\calF$ is Zariski dense in $\dG$. When $G$ is of type $D_{2n}$, the Zariski closure of the global geometric monodromy of $\calF$ contains $\SO_{4n-1}\subset\PSO_{4n}=\dG$ of $n\geq3$, and contains $G_{2}\subset \PSO_{8}=\dG$ if $n=2$.
\end{enumerate}
\end{theorem}

\subsection{Applications}  
By a descent argument (using rigidity), we have the following strengthening of Theorem \ref{th:3pt eigen}. 
\begin{theorem}[\cite{Y-motive}]\label{th:3pt eigen rational} Let $k$ be a prime field with $\chk\neq2$ (i.e., $\FF_{p}$ for $p$ an odd prime or $\QQ$). Then the eigen local system $\calF$ in Theorem \ref{th:3pt eigen} can be defined over $k$. Moreover, the monodromy of $\calF$ can be conjugated to $\dG(\Ql)$ inside $\dG(\Qlbar)$.
\end{theorem}

\subsubsection{Application to the construction of motives} Assume $G$ is of type $A_{1}, E_{7}, E_{8}$ or $G_{2}$. Applying the above theorem to $k=\QQ$, we get a $\dG$-local systems $\rho:\pi_{1}(U_{\QQ})\to \dG(\Ql)$ whose geometric monodromy is Zariski dense. For each $\QQ$-point $a\in U(\QQ)=\QQ-\{0,1\}$, restricting $\rho$ to the point $a=\Spec\ \QQ$ gives a continuous Galois representation
\begin{equation}\label{rhoa}
\rho_{a}:\GQ\to\dG(\Ql).
\end{equation}

By Proposition \ref{p:desc eigen}, one sees that for each $V\in\Rep(\dG,\Ql)$, $\rho_{V}$ is obtained as part of the middle dimensional cohomology of some family of varieties over $U$. Using this fact, it can be shown that each $\rho_{a}$ is obtained from motives over $\QQ$ (if $G$ is type $E_{8}$ or $G_{2}$) or $\QQ(i)$ (if $G$ is of type $A_{1}$ or $E_{7}$).

\begin{theorem}[\cite{Y-motive}] Assume $G$ is of type $A_{1}, E_{7}, E_{8}$ or $G_{2}$. There are infinitely many $a\in \QQ-\{0,1\}$ such that the $\rho_{a}$'s are mutually non-isomorphic and all have Zariski dense image in $\dG$. Consequently, there are infinitely many motives over $\QQ$ (if $G$ is type $E_{8}$ or $G_{2}$)  or $\QQ(i)$ (if $G$ is of type $A_{1}$ or $E_{7}$) whose $\ell$-adic motivic Galois group is isomorphic to $\dG$ for any prime $\ell$. 
\end{theorem}
This result then gives an affirmative answer to the $\ell$-adic analog of Serre question (see \S\ref{ss:Serre Q}) for motivic Galois groups of type $E_{7}, E_{8}$ and $G_{2}$.  The case of $G_{2}$ was settled earlier by Dettweiler and Reiter \cite{DR}, using Katz's algorithmic construction of rigid local systems. Our local system $\calF$ in the case $G=G_{2}$ is the same as Dettweiler and Reiter's.

\subsubsection{Application to the inverse Galois problem} Let $\ell$ be a prime number. To emphasize on the dependence on $\ell$, we denote  the Galois representation $\rho_{a}$ in \eqref{rhoa} by $\rho_{a,\ell}$. To solve the inverse Galois problem for the groups $\dG(\FF_{\ell})$,  we would like to choose $a\in \QQ-\{0,1\}$ such that $\rho_{a,\ell}$ has image in $\dG(\ZZ_{\ell})$ (which is always true up to conjugation), and its reduction modulo $\ell$ is surjective. This latter condition is hard to satisfy even if we know that the image of $\rho_{a,\ell}$ is Zariski dense in $\dG(\Ql)$. 

To proceed, let us consider the Betti version of  Theorem \ref{th:3pt eigen} and Theorem \ref{th:3pt eigen rational}. Namely we consider the base field $k=\CC$ and talk about sheaves in $\QQ$-vector spaces on the various complex algebraic moduli stacks. The same argument gives a topological local system
\begin{equation*}
\rho^{\top}: \pi^{\top}_{1}(U_{\CC})\to \dG(\QQ)
\end{equation*}
whose image is Zariski dense.  It makes sense to reduce $\rho^{\top}$ for large enough primes $\ell$
\begin{equation*}
\overline{\rho}^{\top}_{\ell}: \pi^{\top}_{1}(U_{\CC})\to \dG(\FF_{\ell}).
\end{equation*}
A deep theorem of Matthews, Vaserstein and Weisfeiler \cite[Theorem in the Introduction]{MVW} (see also Nori \cite[Theorem 5.1]{Nori}) says  that $\overline{\rho}^{\top}_{\ell}$ is surjective for sufficiently large $\ell$, when $\dG$ is simply-connected. This is the case when $G$ is of type $E_{8}$ and $G_{2}$. Using the comparison between Betti cohomology and $\ell$-adic cohomology, we conclude that for general $a\in \QQ-\{0,1\}$ (general in the sense of Hilbert irreducibility), the reduction $\overline\rho_{a,\ell}$ of $\rho_{a,\ell}$ is also onto $\dG(\FF_{\ell})$. This solves the inverse Galois problem for $E_{8}(\FF_{\ell})$ and $G_{2}(\FF_{\ell})$ for sufficiently large primes $\ell$ (without an effective bound). The local monodromy of $\rho$ also suggests a triple in $E_{8}(\FF_{\ell})$ that might be rigid, see \cite[Conjecture 5.16]{Y-motive}. Recent work of Guralnick and Malle \cite{GM} establishes the rigidity of this triple, hence solves the inverse Galois problem for $E_{8}(\FF_{\ell})$ for all primes $\ell\geq7$.

When $G$ is of type $A_{1}$ or $E_{7}$, $\dG$ is the adjoint form. In this case, the result in \cite{MVW} says that  for sufficiently large prime $\ell$,  the image of $\overline{\rho}^{\top}_{\ell}$ contains the image of $\dG^{\sc}(\FF_{\ell})\to \dG(\FF_{\ell})$. We deduce that the same is true for $\rho_{a,\ell}$ for general $a\in\QQ-\{0,1\}$.

\appendix

\section{Rank one character sheaves}\label{a:ch}  
In this appendix, we study rank one local systems on algebraic groups that behave like characters. Most of the results here are well-known to experts, and our proofs are sketchy.

\subsection{Definitions and basic properties} 
In this subsection $k$ is a perfect field. Let $L$ be an algebraic group over $k$ with the multiplication map $m:L\times L\to L$ and the identity element $e:\Spec k\to L$.

\begin{defn}\label{def:char} A {\em rank one  character sheaf} $\calK$ on $L$ is a local system of rank one on $L$ equipped with two isomorphisms
\begin{eqnarray*}
\mu: m^{*}\calK\isom\calK\boxtimes\calK,\\
u: \Qlbar\isom e^{*}\calK.
\end{eqnarray*}
These isomorphisms should be compatible in the sense that
\begin{eqnarray}\label{umu1}
\mu|_{L\times\{e\}}=\id_{\calK}\otimes u: \calK=\calK\otimes_{\Qlbar}\Qlbar\isom \calK\otimes e^{*}\calK,\\
\label{umu2}\mu|_{\{e\}\times L}=u\otimes\id_{\calK}: \calK=\Qlbar\otimes_{\Qlbar}\calK\isom e^{*}\calK\otimes\calK.
\end{eqnarray}
Furthermore, $\mu$ should make the following diagram commutative
\begin{equation}\label{cocycle}
\xymatrix{(m\times\id_{L})^{*}m^{*}\calK\ar[rr]^{(m\times\id_{L})^{*}\mu}\ar@{=}[d] && (m\times\id_{L})^{*}(\calK\boxtimes\calK)\ar[r] & m^{*}\calK\boxtimes\calK\ar[r]^{\mu\boxtimes\id_{\calK}} & \calK\boxtimes\calK\boxtimes\calK\ar@{=}[d]\\
(\id_{L}\times m)^{*}m^{*}\calK\ar[rr]^{(\id_{L}\times m)^{*}\mu}&& (\id_{L}\times m)^{*}(\calK\boxtimes\calK)\ar[r] & \calK\boxtimes m^{*}\calK\ar[r]^{\id_{\calK}\boxtimes\mu} & \calK\boxtimes\calK\boxtimes\calK}
\end{equation}
\end{defn}
There is an obvious notion of an isomorphism between two rank one character sheaves: it is an isomorphism between local systems intertwining the $\mu$'s and $u$'s. Let $\cCS(L)$ be the category (groupoid) whose objects are rank one character sheaves $(\calK,\mu,u)$ on $L$, and whose morphisms are isomorphisms between them. Then $\cCS(L)$ carries a symmetric monoidal structure given by the tensor product of character sheaves with the constant sheaf $\Qlbar$ (equipped with the tautological $\mu$ and $u$) as the unit object. Let $\CS(L)$ be set of isomorphism classes of objects in $\cCS(L)$, which is an abelian group. 

\begin{remark}\label{r:char}
\begin{enumerate}
\item\label{inversion char}  Let $\iota:L\to L$ be the inversion $g\mapsto g^{-1}$. Then any $\calK\in\cCS(L)$ is equipped with a canonical isomorphism $\iota^{*}\calK\cong \calK^{-1}$ obtained by restricting $\mu$ to the anti-diagonally embedded $L$. 
\item\label{char is a prop} When $L$ is connected, the condition \eqref{cocycle} is automatically satisfied. In this case, the two relations \eqref{umu1} and \eqref{umu2} also guarantee that $\mu$ is commutative, i.e., $\mu=s^{*}\circ \mu$ where $s^{*}$ is the pullback map induced by swapping two factors $s:L\times L\to L\times L$. In this case, a local system $\calK$ of rank one being a character sheaf is a property rather than extra structure on $\calK$: $\calK$ is a character sheaf if and only if $e^{*}\calK$ is isomorphic to the constant sheaf on $\Spec k$ and that for any point $g\in L(\kbar)$, the isomorphism type of $\calK|_{L_{\kbar}}$ is invariant under left and right translation by $g$. On such a local system $\calK$ there is a unique pair $(\mu,u)$ up to isomorphism making $(\cK,\mu,u)$ into a rank one character sheaf.
\item\label{char bc} The automorphism group of a triple $(\calK,\mu,u)\in\cCS(L)$ is 
\begin{equation*}
\Aut(\calK,\mu,u)\cong\Hom(\pi_{0}(L_{\kbar})_{\Gk}, \Qlbar^{\times}). 
\end{equation*} 
When $L$ is connected, $\cCS(L)$ is a groupoid with trivial automorphisms, hence is equivalent to the set $\CS(L)$.

\item\label{base change char} Let $\cCS(L/\kbar)$ be the category of rank one character sheaves over $L_{\kbar}$. We may identify $\cCS(L)$ as the category of $\Gk$-equivariant objects of $\cCS(L/\kbar)$ (which may not induce an injection on the isomorphism classes of objects). When $L$ is connected, the base change map $\CS(L)\to \CS(L/\kbar)$ is injective, and identifies $\CS(L)$ with the $\Gk$-invariants of $\CS(L/\kbar)$.
\end{enumerate}
\end{remark}

We record a few functorial properties of rank one character sheaves.
\begin{lemma} Let $k'/k$ be a finite extension. Let $L$ be an algebraic group over $k'$, and let $\Res^{k'}_{k}L$ be the Weil restriction of $L$ to $k$. Then there is a canonical equivalence of symmetric monoidal categories
\begin{equation*}
\cCS(L)\cong \cCS(\Res^{k'}_{k}L).
\end{equation*}
\end{lemma}
\begin{proof} Let $\Hom_{k}(k',\kbar)$ be the set of $k$-linear embeddings $k'\incl \kbar$.
The base change of $\Res^{k'}_{k}L$ to  $\kbar$ is the Cartesian power $L_{\kbar}^{\Hom_{k}(k', \kbar)}$, and the action of $\Gk$ on it permutes the factors according to its action on $\Hom_{k}(k',\kbar)$. An object $\cK\in\cCS(\Res^{k'}_{k}L)$ is a $\Gk$-equivariant object in $\cCS(L_{\kbar}^{\Hom_{k}(k',\kbar)})$. Any object in $\cCS(L_{\kbar}^{\Hom_{k}(k',\kbar)})$ takes the form $\boxtimes_{\iota\in\Hom_{k}(k',\kbar)}\cK_{\iota}$ where $\cK_{\iota}\in\cCS(L/\kbar)$. The $\Gk$-equivariant structure gives isomorphisms $\cK_{\gamma^{-1}\circ\iota}\cong\gamma^{*}\cK_{\iota}$ for all $\gamma\in\Gk$. Fix an embedding $\iota_{0}\in\Hom_{k}(k',\kbar)$, then the $\Gk$-equivariant structure on $\boxtimes_{\iota\in\Hom_{k}(k',\kbar)}\cK_{\iota}$ is the same as a $\Gal(\kbar/\iota_{0}(k'))$-equivariant structure on $\cK_{\iota_{0}}$, which then gives an object $\cK_{\iota_{0}}\in\cCS(L)$. One can check that this assignment gives the desired equivalence, and is independent of the choice of $\kbar$ and $\iota_{0}$.
\end{proof}

\begin{lemma}\label{l:descent cs} Consider an exact sequence of algebraic groups over $k$:
\begin{equation*}
1\to L_{1}\xrightarrow{i} L_{2}\xrightarrow{\pi} L_{3}\to 1.
\end{equation*}
Let $\cCS(L_{2}; L_{1})$ be the category of pairs $(\cK, \tau)$ where $\cK\in\cCS(L_{2})$ and $\tau: i^{*}\cK\cong\Qlbar$ is an isomorphism in $\cCS(L_{1})$ (here $\Qlbar$ stands for the constant sheaf on $L_{1}$ with the tautological character sheaf structure). Then $\pi^{*}: \cCS(L_{3})\to\cCS(L_{2}; L_{1})$ is an equivalence of symmetric monoidal categories.  
\end{lemma}
\begin{proof} First assume that $L_{1}$ is reduced hence smooth over $k$. Let $(\cK,\tau)\in\cCS(L_{2}; L_{1})$. The character sheaf structure on $\cK$ combined with $\tau$ gives a descent datum of $\cK$ along the smooth morphism $\pi$. By smooth descent of local systems, $(\cK,\tau)$ gives rise to a local system $\overline{\cK}$ on $L_{3}$. It is easy to check that the character sheaf structure of $\cK$ induces one on $\overline{\cK}$. Hence we get a functor $\cCS(L_{2}; L_{1})\to\cCS(L_{3})$ sending $(\cK,\tau)$ to $\overline{\cK}$, and it is straightforward to check that it is inverse to $\pi^{*}$. 

When $L_{1}$ is not necessarily reduced, let $\wt{L}_{3}=L_{2}/L^{\red}_{1}$. We first descend $(\cK,\tau)\in\cCS(L_{2}; L_{1})$ to $\wt{\cK}\in\cCS(\wt{L}_{3})$ by the above argument, then since $\wt{L}_{3}\to L_{3}$ is a homeomorphism for the \'etale topology, $\wt{\cK}$ further descends to $\overline{\cK}\in\cCS(L_{3})$, which gives an inverse functor to $\pi^{*}$.
\end{proof}

\subsection{Relation with Serre's $\pi_{1}$}\label{ss:Serrepi1} 
Let $L$ be a {\em connected} algebraic group over $k$.
Suppose $\calK\in\cCS(L)$ has finite order $n$, then it corresponds to a $\mu_{n}(\Qlbar)$-torsor $L'\to L$ with $L'$ connected. We shall call  $L'$ the associated cover of $(L,\calK)$.

\begin{lemma}
The scheme $L'$ carries a canonical algebraic group structure such that the projection $\pi: L'\to L$ is a group homomorphism such that $\ker(\pi)$ is central in $L'$.
\end{lemma}
\begin{proof}
The associated cover of $(L\times L,\calK\boxtimes\calK)$ is $\pi\times\pi:L'\times L'\to L\times L$. By the functoriality of the construction of associated covers, the isomorphism $\mu: m^{*}\calK\cong\calK\boxtimes\calK$ gives a commutative diagram
\begin{equation*}
\xymatrix{L'\times  L' \ar[r]^{m'}\ar[d]^{\pi\times\pi} & L'\ar[d]^{\pi}\\
L\times L\ar[r]^{m} & L}
\end{equation*}
We use $m'$ to define the multiplication on $L'$. The associativity of $\mu$ shows that $m'$ is associative. The inversion on $L'$ comes from the isomorphism $\iota^{*}\calK\cong \calK^{-1}$ in Remark \ref{r:char}\eqref{inversion char}. The fiber of $L'$ at $e\in L$ is the discrete scheme $\mu_{n}(\Qlbar)$ (over $k$) under the trivialization $u:\Qlbar\cong e^{*}\calK$. We define the identity element of $L'$ to be the point over $e\in L$ corresponding to $1\in\mu_{n}(\Qlbar)$. This completes the construction of the algebraic group structure on $L'$. From the construction of $m'$, left and right multiplication of $\ker(\pi)=\mu_{n}(\Qlbar)$  are both the same as the $\mu_{n}(\Qlbar)$-action on $L'$ coming from the $\mu_{n}(\Qlbar)$-torsor structure. Therefore $\ker(\pi)$ is central in $L'$.
\end{proof}

 Let $\Cov(L)$ be the category consisting of central isogenies $\pi:L'\to L$ of {\em connected} $k$-algebraic groups with $\ker(\pi)$ discrete as a $k$-scheme. The formal (inverse) limit of all objects in $\Cov(L)$ gives the universal central isogeny $\pi^{\univ}: L^{\univ}\to L$. The kernel $\ker(\pi^{\univ})$ is a pro-finite abelian group, and we denote it by $\pi^{\Serre}_{1}(L)$. When $L$ is commutative, this is the same as the $\pi_{1}$ defined by Serre in \cite[\S6.1, Definition 1; \S6.2, Proposition 3]{Serre-Proalg}. Since each object in $\Cov(L)$ is also a finite \'etale cover of $L$, we have a surjection $\pi_{1}(L)^{\ab}\surj\pi^{\Serre}_{1}(L)$. To emphasize the dependence on the base field we write $\Cov(L/k)$ and $\pi_{1}^{\Serre}(L/k)$.

When $k\incl k'$ is a field extension, we have the base change functor $\Cov(L/k)\to \Cov(L/k')$ which is fully faithful. We have $\Cov(L/\kbar)=\varinjlim_{k\subset k'\subset \kbar}\Cov(L/k')$ and hence
\begin{equation*}
\pi_{1}^{\Serre}(L/\kbar)\cong \varprojlim_{k\subset k'\subset\kbar}\pi^{\Serre}_{1}(L/k').
\end{equation*}

\begin{theorem}\label{th:Serre pi1} Let $k$ be any base field and let $L$ be a connected algebraic group over $k$. Then there is a canonical isomorphism
\begin{equation*}
\CS(L)\cong \Hom_{\cont}(\pi^{\Serre}_{1}(L), \Qlbar^{\times})
\end{equation*}
\end{theorem}
\begin{proof}[Sketch of proof]
For those $\calK\in\cCS(L)$ of finite order, we have defined an object $(\pi:L'\to L)\in\Cov(L)$ together with an embedding $\ker(\pi)\incl \Qlbar^{\times}$. This gives a homomorphism $\pi^{\Serre}_{1}(L)\surj \ker(\pi)\incl \Qlbar^{\times}$. This construction clearly passes to the limit and defines a homomorphism $\xi:\CS(L)\to\Hom_{\cont}(\pi^{\Serre}_{1}(L), \Qlbar^{\times})$. Conversely, suppose we are given a character $\chi: \pi^{\Serre}_{1}(L)\to \Qlbar^{\times}$ of finite order. The kernel of $\chi$  corresponds to a central isogeny $\pi:L'\to L$ such that $\ker(\pi)\cong \Im(\chi)$. We then define $\calK_{\chi}=(\pi_{!}\Qlbar)_{\chi}$, the $\chi$-isotypical component of $\pi_{!}\Qlbar$. Since $\pi^{*}\calK_{\chi}\cong \Qlbar$ is a character sheaf, so is $\calK_{\chi}$. This construction  $\chi\mapsto\calK_{\chi}$ also passes to the limit to give a homomorphism $\Hom_{\cont}(\pi^{\Serre}_{1}(L), \Qlbar^{\times})\to\CS(L)$. It is not hard to check that it is inverse to the homomorphism $\xi$.
\end{proof}

\subsection{Connected groups over a finite field} 
In this subsection we assume $k$ is a finite field. 

\subsubsection{Lang torsor} For each $\calK\in\cCS(L)$, the sheaf-to-function correspondence gives a function  $f_{\calK}:L(k)\to \Qlbar^{\times}$. The isomorphisms $\mu: m^{*}\calK\cong\calK\boxtimes \calK$ and $u: \Qlbar\cong e^{*}\cK$ imply that $f_{\calK}$ is a group homomorphism. This way we obtain a homomorphism
\begin{equation*}
f_{L}:\CS(L)\to \Hom(L(k),\Qlbar^{\times}).
\end{equation*}

One the other hand, we have the Lang torsor
\begin{eqnarray}\label{Lang torsor}
\phi_{L}: L&\to& L\\
\notag g&\mapsto& \Frob_{L/k}(g)g^{-1}
\end{eqnarray}
where $\Frob_{L/k}: L\to L$ over $k$ ($\Frob^{*}_{L/k}$ raises functions on $L$ to the $\#k\nth$ power). The morphism $\phi_{L}$ is a right $L(k)$-torsor onto its image: $g\in L(k)$ acts on the source by right multiplication. When $L$ is connected, $\phi_{L}$ is surjective and is a right $L(k)$-torsor. When $L$ is not connected, $\phi_{L}$ may not be surjective. The push-forward sheaf $\phi_{L,!}\Qlbar$ carries an action of $L(k)$. For each character $\chi: L(k)\to \Qlbar^{\times}$, let $\calK_{\chi}=(\phi_{L,!} \Qlbar)_{\chi^{-1}}$ be the $\chi^{-1}$-isotypical direct summand of $\phi_{L,!}\Qlbar$ (we can take the projector in $\Qlbar[L(k)]$ corresponding to $\chi^{-1}$, and apply it to $\phi_{L,!}\Qlbar$ and take the image). When $L$ is connected, we get a homomorphism
\begin{equation}\label{char sh from Lang}
\l_{L}:\Hom(L(k), \Qlbar^{\times})\to \Loc_{1}(L)
\end{equation}
where $\Loc_{1}(L)$ is the group of isomorphism classes of rank one $\Qlbar$-local systems on $L$, with the group structure given by the tensor product.

\begin{lemma}\label{l:fun of Lang} Let $L$ be a connected algebraic group over a finite field $k$. Let $f'_{L}$ denote the sheaf-to-function correspondence $\Loc_{1}(L)\to \Fun(L(k))$. Then the composition $f'_{L}\l_{L}:\Hom(L(k),\Qlbar^{\times})\to\Loc_{1}(L)\to \Fun(L(k))$ is the natural inclusion of $\Hom(L(k),\Qlbar^{\times})$ into $\Fun(L(k))$.
\end{lemma}
\begin{proof}
Let $\chi:L(k)\to \Qlbar^{\times}$ be a character. We shall calculate the trace of the geometric Frobenius $\Frob_{g}$ acting on the stalk of $\calK_{\chi}=(\phi_{L,!}\Qlbar)_{\chi^{-1}}$ at $g\in L(k)$. Let $x\in L(\kbar)$ be such that $\Frob_{L/k}(x)x^{-1}=g$. The preimage $\phi_{L}^{-1}(g)$ consists of points $xa$ where $a\in L(k)$. Since $\Frob_{L/k}(x)=gx$ is also in $\phi_{L}^{-1}(g)$, we have $gx=xa_{0}$ for some $a_{0}\in L(k)$. The arithmetic Frobenius action on $\phi^{-1}_{L}(g)$ is given by $xa\mapsto \Frob_{L/k}(xa)=gxa=xa_{0}a$.  We may identify the stalk of $\phi_{L,!}\Qlbar$ at $g$ with $\Fun(\phi^{-1}_{L}(g))$. Then the stalk of $(\phi_{L,!}\Qlbar)_{\chi^{-1}}$ at $g$ is spanned by the function $e_{g}:xa\mapsto\chi(a)$. The geometric Frobenius $\Frob_{g}$ on $\Fun(\phi^{-1}_{L}(g))$ sends $e_{g}$ to the function $xa\mapsto e_{g}(\Frob_{L/k}(xa))=\chi(a_{0}a)=\chi(a_{0})e_{g}(xa)$. Hence $\Tr(\Frob_{g}, \calK_{\chi})=\chi(a_{0})$. It remains to show that $\chi(a_{0})=\chi(g)$. Note that $a_{0}$ and $g$ are conjugate in $L(\kbar)$ via $x$;  we will show that $a_{0}$ and $g$ are actually conjugate in $L(k)$, and hence have the same value under $\chi$. As usual, $a_{0}$ determines a class  $[a_{0}]\in\cohog{1}{k, Z_{L}(g)}$ which sends the Frobenius element to $\Frob(x)x^{-1}=g$. Since $g$ lies in the neutral component $Z_{L}(g)^{\c}$ of the centralizer $Z_{L}(g)$, the class $[a_{0}]$ in fact comes from a class in $\cohog{1}{k,Z_{L}(g)^{\c}}$, which is trivial by Lang's theorem. Therefore $a_{0}$ and $g$ are conjugate in $L(k)$ and $\chi(a_{0})=\chi(g)$.
\end{proof}

Next we consider the case where $L$ is commutative and connected.

\begin{theorem}\label{th:char sh comm} Let $L$ be a connected commutative algebraic group over $k$. Then $f_{L}$ is an isomorphism of abelian groups
\begin{equation*}
\CS(L)\isom \Hom(L(k), \Qlbar^{\times}).
\end{equation*}
\end{theorem}
\begin{proof} When $L$ is commutative, the map $\l_{L}$ in \eqref{char sh from Lang} has image in  $\CS(L)$. This follows from the fact the Lang isogeny $\phi_{L}$ is a group homomorphism when $L$ is commutative. Thus we have a pair of homomorphisms
\begin{equation*}
f_{L}:\CS(L)\to \Hom(L(k),\Qlbar^{\times}) \textup{   and   }\l_{L}:\Hom(L(k),\Qlbar^{\times})\to\CS(L).
\end{equation*}
By Lemma \ref{l:fun of Lang}, we have $f_{L}\l_{L}=\id$. Therefore it suffices to show that $f_{L}$ is injective to conclude that $f_{L}$ and $\l_{L}$ are inverse to each other. Let $\calK\in\CS(L)$. Let $g\in L$ be a closed point with residue field $k'$. Define the norm $h=\Nm(g):=\prod_{\sigma\in\Gal(k'/k)}\sigma(g)\in L(k)$. We claim that 
\begin{equation}\label{trace norm}
\Tr(\Frob_{g},\calK_{g})=\Tr(\Frob_{h},\calK_{h})
\end{equation}
Once this is proved, the Frobenius trace of $\calK$ at any closed point is determined by the function $f_{L}(\calK)$, and hence $\calK$ is also determined by $f_{L}(\cK)$ by the Chebotarev density theorem. Now we prove \eqref{trace norm}. Let $d=[k':k]$, then $g$ determines a $k$-point $g_{0}\in\Sym^{d}(L)(k)$. Since $L$ is commutative, the $d$-fold multiplication map factors as
\begin{equation*}
L^{d}\xrightarrow{s_{d}}\Sym^{d}(L)\xrightarrow{m_{d}}L.
\end{equation*}
By the definition of character sheaves we have $s_{d}^{*}m_{d}^{*}\calK\cong\calK^{\boxtimes d}$. By adjunction this gives a nonzero map $\alpha:m_{d}^{*}\calK\to (s_{d,*}\calK^{\boxtimes d})^{S_{d}}=:\calK^{(d)}$. It is easy to see that $\calK^{(d)}$ is also a rank one local system, therefore $\alpha$ has to be an isomorphism. We then compute the stalk of $m_{d}^{*}\calK$ and $\calK^{(d)}$ at $g_{0}$. On one hand, $(m_{d}^{*}\calK)_{g_{0}}=\calK_{m_{d}(g_{0})}=\calK_{h}$ hence $\Tr(\Frob_{g_{0}}, (m_{d}^{*}\calK)_{g_{0}})=\Tr(\Frob_{h},\calK_{h})$. On the other hand, the stalk of $\calK^{(d)}$ at $g_{0}$ is  $\otimes_{i=0}^{d-1}\calK_{\Frob^{i}(g)}$. The Frobenius equivariance structure of $\calK$ gives isomorphisms $\calK_{g}\xrightarrow{\iota}\calK_{\Frob(g)}\xrightarrow{\iota}\cdots\xrightarrow{\iota}\calK_{\Frob^{d-1}(g)}\xrightarrow{\iota}\calK_{g}$, and the iteration $\iota^{d}$ of these isomorphisms give the automorphism $\Frob_{g}$ on $\calK_{g}$. Let $v\in\calK_{g}$ be a basis, then $v\otimes\iota(v)\cdots\otimes \iota^{d-1}(v)$ is a basis of $\calK^{(d)}_{g_{0}}$. The action of $\Frob_{g_{0}}$ cyclically permuting the tensor factors, and sends $v\otimes\iota(v)\cdots\otimes \iota^{d-1}(v)$ to $\iota^{d}(v)\otimes\iota(v)\otimes\cdots\otimes\iota^{d-1}(v)=\Frob_{g}(v)\otimes\iota(v)\otimes\cdots\otimes\iota^{d-1}(v)$. Therefore $\Tr(\Frob_{g_{0}},\calK^{(d)}_{g_{0}})=\Tr(\Frob_{g}, \calK_{g})$. Combining these calculations we get \eqref{trace norm}. The theorem is proved.
\end{proof}

\begin{cor} We have isomorphisms
\begin{eqnarray}
\label{Serre pi1 comm} \pi_{1}^{\Serre}(L/k)\cong L(k),\\
\label{Serre pi1 comm limit}\pi_{1}^{\Serre}(L/\kbar)\cong \varprojlim_{k\subset k'\subset \kbar}L(k')
\end{eqnarray}
where the projective system is over finite extensions $k'/k$ and the transition maps are the norm maps $\Nm_{k''/k'}:L(k'')\to L(k')$. 
\end{cor}
\begin{proof}
\eqref{Serre pi1 comm} follows directly from Theorem \ref{th:char sh comm} and Theorem \ref{th:Serre pi1}. When passing to the limit, we need to compute the transition maps $\nu_{k''/k'}:\pi_{1}^{\Serre}(L/k'')\to \pi^{\Serre}_{1}(L/k')$ for $k'\subset k''$ in terms of $L(k')$ and $L(k'')$. By definition, $\nu_{k''/k'}$ is the one induced from the pullback map $\CS(L/k')\to \CS(L/k'')$. During the proof of Theorem \ref{th:char sh comm}, the formula \eqref{trace norm} can be reformulated as a commutative diagram
\begin{equation*}
\xymatrix{\CS(L/k')\ar[r]^{f_{L,k'}}\ar[d]^{\textup{pullback}} & \Hom(L(k'),\Qlbar^{\times})\ar[d]^{\Nm^{*}_{k''/k'}}\\
\CS(L/k'')\ar[r]^{f_{L,k''}} & \Hom(L(k), \Qlbar^{\times})}
\end{equation*}
This implies that $\nu_{k''/k'}$ is given by $\Nm_{k''/k'}:L(k'')\to L(k')$. This proves completes the proof. 
\end{proof}

\begin{exam}\label{ex:Kummer} (Kummer sheaves)
Let $L=\Gm$ be the one-dimensional torus over $k$. For each character $\chi: L(k)\to \Qlbar^{\times}$ we get an object $\calK_{\chi}\in\cCS(\Gm)$. We claim that the canonical surjection $h:\pi_{1}(\Gm/\kbar)\surj\pi_{1}^{\Serre}(L/\kbar)$ induces an isomorphism
\begin{equation}\label{pit pis}
\pi^{t}_{1}(\Gm/\kbar)\isom\pi_{1}^{\Serre}(\Gm/\kbar)
\end{equation}
where  $\pi^{t}_{1}(-)$ denotes the tame fundamental group. On one hand, $\pi_{1}^{\Serre}(L/\kbar)\cong\varprojlim_{k'\subset\kbar}k'^{\times}$ is an inverse limit of finite groups of order prime to $p$. Therefore $\pi_{1}(\Gm/\kbar)\surj\pi_{1}^{\Serre}(L/\kbar)$ factors through the tame quotient. One the other hand,  $\pi_{1}^{t}(\Gm/\kbar)\cong \hZZ'(1):=\varprojlim_{(m,p)=1} \mu_{m}(\kbar)$. Therefore $h$ induces a homomorphism
\begin{equation}\label{inv limit k}
\varprojlim_{(m,p)=1}\mu_{m}(\kbar)\to \varprojlim_{k'\subset\kbar}k'^{\times}.
\end{equation}
What is this homomorphism? This is almost the tautological one: we may replace the limit on the left side by $\varprojlim_{n}\mu_{q^{n}-1}(\kbar)$ (where $q=\#k$) because every integer prime to $p$ is divisible by some $q^{n}-1$. Clearly $\mu_{q^{n}-1}(\kbar)=\FF_{q^{n}}$, and the norm map $\FF_{q^{mn}}^{\times}\to \FF_{q^{n}}^{\times}$ is the same as the power map $[\frac{q^{mn}-1}{q^{n}-1}]:\mu_{q^{mn}-1}(\kbar)\to \mu_{q^{n}-1}(\kbar)$. Therefore the two sides of \eqref{inv limit k} are canonically isomorphic, and one can check that \eqref{inv limit k} is this canonical isomorphism.  This proves \eqref{pit pis}.

The above discussion can easily be generalized to any torus $T$ over $k$. We define the  (prime-to-$p$-part of the) Tate module of $T$ by
\begin{equation*}
\TT'(T):=\varprojlim_{(n,p)=1}T[n](\kbar)
\end{equation*}
as a $\Gk$-module. Note that when $T=\Gm$, $\TT'(\Gm)=\hZZ'(1)$. In general we have $\TT'(T)=\xcoch(T)\otimes_{\ZZ}\hZZ'(1)$ with $\Gk$ diagonally acting on both factors. Then \eqref{pit pis} implies a canonical isomorphisms of $\Gk$-modules
\begin{equation}\label{tame T}
\pi^{t}_{1}(T/\kbar)\cong\TT'(T)\cong \varprojlim_{k'\subset\kbar}T(k')\cong \pi_{1}^{\Serre}(T/\kbar).
\end{equation}
In conclusion, any tame local system of rank one on $T$ is a character sheaf. We call these sheaves {\em Kummer sheaves}.
\end{exam}

\begin{exam}\label{ex:AS}(Artin-Schreier sheaves)
When $L=\Ga$ is the additive group over $k$, for each additive character $\psi:k\to \Qlbar^{\times}$ we have a rank one character sheaf which we denote by $\AS_{\psi}$, the {\em Artin-Schreier} sheaf. More generally, if $V$ is a vector space over $k$, viewed as a commutative group scheme, then for every linear function $\phi\in V^{\vee}$, viewed as a group homomorphism $\phi: V\to\Ga$, the pullback $\phi^{*}\AS_{\psi}$ gives an object in $\cCS(V)$. Fixing $\psi$, every object in $\cCS(V)$ arises this way from a unique $\phi\in V^{\vee}$.
\end{exam}

\begin{exam}(A non-commutative pathology) When $L$ is not necessarily commutative, the image of the map $\l_{L}$ may not lie in $\CS(L)$. For example, take $k=\FF_{2}$ and $L=\SL_{2}$. Then $\SL_{2}(\FF_{2})\cong S_{3}$ has a unique character $\chi$ of order two. Let $\Ga\subset\SL_{2}$ denote the upper triangular unipotent matrices. The order two local system $\calK_{\chi}|_{\Ga}$ is the Artin-Schreier sheaf $\AS_{\psi}$, for the unique nontrivial character $\psi$ of $\FF_{2}$.  However, when we make $t=\diag(t,t^{-1})\in\SL_{2}(\kbar)$ act on $\Ga/\kbar$ by conjugation, the pullback $\Ad(t)^{*}\calK_{\chi}|_{\Ga}$ (over $\Ga/\kbar$) becomes $[t^{2}]^{*}\AS_{\psi}$ (where $[t^{2}]:\Ga\to \Ga$ is multiplication by $t^{2}$), which is not isomorphic to $\AS_{\psi}$ (by looking at the functions they define). But if $m^{*}\calK_{\chi}\cong\calK_{\chi}\boxtimes\calK_{\chi}$, the isomorphism type of $\calK_{\chi}$ would be unchanged under the left and right translation  of $\SL_{2}$ on itself, and in particular invariant under the adjoint action of $\SL_{2}$. Therefore,  $m^{*}\calK_{\chi}$ is not isomorphic to $\calK_{\chi}\boxtimes\calK_{\chi}$ in this case. In what follows we will see how to remedy this pathology.
\end{exam}

Now let $L$ be a connected reductive group over a finite field $k$. Let $L^{\sc}$ be the simply-connected cover of the derived group $L^{\der}$. We have a natural homomorphism $L^{\sc}(k)\to L(k)$ whose image is a normal subgroup of $L(k)$. We may therefore form the quotient $L(k)/L^{\sc}(k)$ (really quotienting by the image), which is abelian.

\begin{theorem}\label{th:char sh red} Let $L$ be a  connected reductive group over $k$. 
\begin{enumerate}
\item We have a natural isomorphism of abelian groups
\begin{equation}\label{CS red}
\CS(L)\isom \Hom(L(k)/L^{\sc}(k), \Qlbar^{\times}).
\end{equation}
\item Let $T$ be a maximal torus in $L$ and $T^{\sc}\subset L^{\sc}$ be its preimage in $L^{\sc}$.  Then we have isomorphisms
\begin{eqnarray}
\label{Serre pi1 red} \pi_{1}^{\Serre}(L/k)\cong T(k)/T^{\sc}(k)\cong L(k)/L^{\sc}(k),\\
\label{Serre pi1 red kbar}\pi_{1}^{\Serre}(L/\kbar)\cong \TT'(T)/\TT'(T^{\sc}).
\end{eqnarray}
\end{enumerate}
\end{theorem}
\begin{proof} First suppose $L=L^{\sc}$. In this case, any central isogeny to $L$ is trivial. Hence $\pi_{1}^{\Serre}(L/k')$ and $\CS(L/k')$ are trivial for any base field $k'$. 

Next consider the general case. Suppose $\calK\in\cCS(L)$. Then its pullback to $L^{\sc}$ has to be the trivial local system by the above discussion. Therefore $f_{L}$ has image in $\Hom(L(k)/L^{\sc}(k), \Qlbar^{\times})$.

The inverse map $\Hom(L(k)/L^{\sc}(k), \Qlbar^{\times})\to\CS(L)$ is again given by (the restriction of) $\l_{L}$. We need to show that if $\chi:L(k)\to \Qlbar^{\times}$ is trivial on $L^{\sc}(k)$, then the corresponding local system $\calK_{\chi}$ is a character sheaf. Consider the intermediate  $L(k)/L^{\sc}(k)$-torsor $\phi'_{L}: L'\to L$ where $L'=L/\Im(L^{\sc}(k)\to L(k))$. For $\chi\in\Hom(L(k)/L^{\sc}(k), \Qlbar^{\times})$, the local system $\calK_{\chi}$ is the corresponding direct summand of $\phi'_{L,*}\Qlbar$. The natural homomorphism $L^{\sc}\to L$ admits a lift $L^{\sc}\to L'$, therefore $\calK_{\chi}$ is trivial when pulled back to $L^{\sc}$. Also it is easy to see that $\calK_{\chi}$ is a Kummer sheaf when restricted to the neutral component of the center $Z^{\circ}$, using Theorem \ref{th:char sh comm}. Consider the isogeny $Z^{\circ}\times L^{\sc}\to L$. In general, whenever we have an isogeny $\pi: \tilL\to L$, if $\pi^{*}\calF$ is a character sheaf, so is $\calF$ (using Remark \ref{r:char}\eqref{char is a prop}). Therefore $\calK_{\chi}\in\cCS(L)$. Checking that the two maps between $\CS(L)$ and $\Hom(L(k)/L^{\sc}(k), \Qlbar^{\times})$ are inverse to each other is left to reader.

The isomorphism \eqref{CS red} together with Theorem \ref{th:Serre pi1} implies that $\pi_{1}^{\Serre}(L)\cong L(k)/L^{\sc}(k)$. The central isogeny $Z^{\c}\times L^{\sc}\to L$ restricts to an isogeny of tori $\varphi: Z^{\c}\times T^{\sc}\to T$. Let $A=\ker(\varphi)$, which is a finite group scheme over $k$ of multiplicative type (but may not be discrete). We get a commutative diagram
\begin{equation*}
\xymatrix{1\ar[r] & \pi_{1}^{\Serre}(Z^{\c}\times T^{\sc})\ar[r]\ar[d] & \pi_{1}^{\Serre}(T) \ar[d]\ar[r] & A(k)\ar@{=}[d]\ar[r] & 1\\
1\ar[r] & \pi_{1}^{\Serre}(Z^{\c}\times L^{\sc})\ar[r] & \pi_{1}^{\Serre}(L)\ar[r] & A(k)\ar[r] & 1}
\end{equation*}
Note that $\pi_{1}^{\Serre}(Z^{\c}\times L^{\sc})\cong\pi^{\Serre}_{1}(Z^{\c})$ since $L^{\sc}$ does not admit nontrivial central isogenies. Also $\pi_{1}^{\Serre}(Z^{\c}\times T^{\sc})\cong \pi_{1}^{\Serre}(Z^{\c})\times\pi_{1}^{\Serre}(T^{\sc})$ by the discussion in Example \ref{ex:Kummer}. Therefore from the above diagram we conclude that
\begin{equation*}
\pi_{1}^{\Serre}(L)\cong\pi_{1}^{\Serre}(T)/\pi_{1}^{\Serre}(T^{\sc}).
\end{equation*}
\eqref{Serre pi1 red} then follows. Using $\TT'(T)\cong\varprojlim_{k'\subset\kbar}T(k')$ and passing to the limit we get \eqref{Serre pi1 red kbar}.
\end{proof}

\begin{theorem}\label{th:char sh} Let $L$ be a connected algebraic group over a finite field $k$.
\begin{enumerate}
\item The homomorphism $f_{L}:\CS(L)\to \Hom(L(k),\Qlbar^{\times})$ is injective.
\item There are surjective homomorphisms
\begin{equation*}
\pi_{1}(L)^{\ab}\surj L(k)^{\ab}\surj \pi_{1}^{\Serre}(L).
\end{equation*}
\end{enumerate}
\end{theorem}
\begin{proof}
(1) Suppose $1\to L'\to L\to L''\to 1$ is an exact sequence of connected algebraic groups over $k$, and that $f_{L'}$ and $f_{L''}$ are known to be injective, we shall show that $f_{L}$ is also injective. Let $\calK\in\CS(L)$ be such that $f_{L}(\calK)=1$. Then $\calK|_{L'}\in\CS(L')$ also lies in the kernel of $f_{L'}$, hence $\calK|_{L'}$ is the constant sheaf by assumption. By Lemma \ref{l:descent cs}, $\calK$ descends to $\calK''\in\cCS(L'')$. The fact that $f_{L}(\calK)=1$ implies that $f_{L''}(\calK'')=1$, hence $\calK''$ is the constant sheaf by assumption. This shows that $\calK$ is also the constant sheaf.

In Theorem \ref{th:char sh comm} we have shown that $f_{L}$ is an isomorphism when $L$ is connected and  commutative; in Theorem \ref{th:char sh red} we have shown that $f_{L}$ is injective if $L$ is connected reductive. Since every connected algebraic group $L$ admits a filtration by normal subgroups with associated graded either connected commutative or connected reductive, we conclude that $f_{L}$ is injective for all connected $L$.

(2) The surjection $\pi_{1}(L)\surj L(k)$ is given by the Lang torsor; the surjection $L(k)^{\ab}\surj \pi_{1}^{\Serre}(L)$ comes from the injection $f_{L}: \CS(L)\cong\Hom_{\cont}(\pi^{\Serre}_{1}(L), \Qlbar^{\times})\incl \Hom(L(k), \Qlbar^{\times})$.
\end{proof}

\subsection{Case of a finite group scheme}  
Let $k$ be any perfect field, and $A$ be a finite group scheme over $k$. 

\begin{lemma}\label{l:char sh discrete} 
\begin{enumerate}
\item There is a canonical isomorphism
\begin{equation*}
\CS(A/\kbar)\cong \cohog{2}{A(\kbar), \Qlbar^{\times}}.
\end{equation*}
\item There is an exact sequence
\begin{equation}\label{CSA}
1\to\cohog{1}{k,\Hom(A(\kbar), \Qlbar^{\times})}\to \CS(A)\to  \cohog{2}{A(\kbar), \Qlbar^{\times}}^{\Gk}.
\end{equation}
\end{enumerate}
\end{lemma}
\begin{proof}
(1) We temporarily assume $k=\kbar$. We may assume $A$ is reduced (passing to $A^{\red}$ does not change the category of rank one character sheaves), and hence identify $A$ with the discrete group $A(\kbar)$. A rank one character sheaf on $A$ is a collection of $1$-dimensional $\Qlbar$-vector spaces $\calK_{a}$, one for each $a\in A$. The datum of $\mu$ is a collection of linear isomorphism $\mu_{a,b}:\calK_{ab}\cong\calK_{a}\otimes\calK_{b}$, one for each pair $(a,b)\in A^{2}$, that satisfy the cocycle relation
\begin{equation}\label{muab}
\mu_{ab,c}\circ(\mu_{a,b}\otimes\id_{\calK_{c}})=\mu_{a,bc}\circ(\id_{\calK_{a}}\otimes\mu_{b,c}), \forall a,b,c\in A.
\end{equation}
The datum of $u$ gives a basis $e_{1}\in\calK_{1}$ such that $\mu_{a,1}=\id_{\calK_{a}}\otimes e_{1}$ and $\mu_{1,a}=e_{1}\otimes\id_{\calK_{a}}$ for all $a\in A$. If we choose a basis $e_{a}\in\calK_{a}$ for every $a$ (keep the same choice for $e_{1}$ given by $u$), then $\mu_{a,b}(e_{ab})=\xi_{a,b}e_{a}\otimes e_{b}$ for some $\xi_{a,b}\in \Qlbar^{\times}$. The cocycle relation \eqref{muab} implies that $\{\xi_{a,b}\}$ is a 2-cocycle of $A$ with values in $\Qlbar^{\times}$ satisfying $\xi_{1,a}=\xi_{a,1}=1$. Changing the choices of $\{e_{a}\}$ changes $\xi_{a,b}$ by a coboundary of a map $\eta:A\to \Qlbar^{\times}$ with $\eta_{1}=1$. Therefore we have a well-defined homomorphism $\alpha:\CS(A/\kbar)\to \cohog{2}{A,\Qlbar^{\times}}$. 

Conversely, from a cocycle $\xi\in Z^{1}(A,\Qlbar^{\times})$, we may first multiply it by a coboundary to make $\xi_{a,1}=\xi_{1,a}=1$ for all $a\in A$. We then define $\calK$ to be the constant sheaf on $A$, and define the map $\mu_{a,b}:\Qlbar^{\times}=\calK_{ab}\to\calK_{a}\otimes\calK_{b}=\Qlbar^{\times}$ as multiplication by $\xi_{a,b}$. We also define $u:\Qlbar\cong \calK_{1}$ as the identity map. The fact that $\xi$ is a cocycle guarantees that $(\calK,\mu,u)$ thus defined is a rank one character sheaf on $A$. Changing $\xi$ by the coboundary of $\eta:A\to \Qlbar^{\times}$ with $\eta_{1}=1$ does not change the isomorphism type of $(\calK,\mu,u)$. This way we have defined a homomorphism $\beta:\cohog{2}{A,\Qlbar^{\times}}\to\CS(A/\kbar)$. It is easy to check that $\alpha$ and $\beta$ are inverse to each other.

(2) We have an  equivalence of categories $\cCS(A)\cong\cCS(A/\kbar)^{\Gk}$, the latter being the category of objects in $\cCS(A/\kbar)^{\Gk}$ together with a $\Gk$-equivariant structure. Since $\cCS(A/\kbar)$ is a Picard category with the group of isomorphism classes of objects $\CS(A/\kbar)\cong\cohog{2}{A(\kbar),\Qlbar^{\times}}$ by (1), and automorphism group $\Hom(A(\kbar), \Qlbar^{\times})$ by Remark \ref{r:char}\eqref{char bc}, the exact sequence \eqref{CSA} follows.
\end{proof}

\subsubsection{Equivariant sheaves}\label{sss:equiv sh} Let $L$ be an algebraic group acting on a scheme $X$ of finite type. Let $\calK\in\cCS(L)$. Let $a:L\times X\to X$ be the action map.  Then an {\em $(L,\calK)$-equivariant} perverse sheaf $\calF$ on $X$ is a pair $(\calF, \a)$  where $\cF$ is a perverse sheaf on $X$ and $\a$ is an isomorphism on $L\times X$
\begin{equation*}
\alpha: a^{*}\calF\cong\calK\boxtimes\calF
\end{equation*}
that restricts to $u\boxtimes\id_{F}$ on $\{e\}\times X$ and such that the composition
\begin{eqnarray*}
&&m^{*}\cK\boxtimes\cF=(m\times\id_{X})^{*}(\cK\boxtimes\cF)\xrightarrow{(m\times\id_{X})^{*}\a^{-1}}(m\times\id_{X})^{*}a^{*}\cF\\
&=&(\id_{L}\times a)^{*}a^{*}\cF\xrightarrow{(\id_{L}\times a)^{*}\a}(\id_{L}\times a)^{*}(\cK\boxtimes\cF)=\cK\boxtimes a^{*}\cF\xrightarrow{\id\boxtimes\a}\cK\boxtimes\cK\boxtimes\cF
\end{eqnarray*}
is equal to $\mu\boxtimes\id_{\cF}$.

In the main body of the paper we need a generalization of the notion of $(L,\calK)$-equivariant perverse sheaves. One can define, using the theory of $\ell$-adic sheaves on Artin stacks as developed by Laszlo and Olsson \cite{LO} or Liu and Zheng \cite{LZ}, a derived category $D^{b}_{(L,\calK)}(X)$ of $(L,\cK)$-equivariant $\Qlbar$-complexes on $X$. This construction works also when $X$ is an Artin stack of finite type.

\subsubsection{Twisted representations}\label{sss:twisted rep} To describe equivariant sheaves on a homogeneous variety, we need the notion of twisted representations of a group which we now recall. Let $\Gamma$ be a group and $\xi\in Z^{2}(\Gamma, \Qlbar^{\times})$ be a cocycle such that $\xi_{1,1}=1$. A $\xi$-twisted representation of $\Gamma$ is a finite-dimensional $\Qlbar$-vector space $V$ with automorphisms $T_{\gamma}:V\to V$, one for each $\gamma\in\Gamma$, such that $T_{1}=\id_{V}$ and
\begin{equation}\label{Twrep}
T_{\gamma\delta}=\xi_{\gamma,\delta}T_{\gamma}T_{\delta}, \forall\gamma,\delta\in\Gamma.
\end{equation}
A $\xi$-twisted representation of $\Gamma$ on $V$ gives a projective representation $\overline{T}:\Gamma\to\PGL(V)$ whose image under the connecting homomorphism $\cohog{1}{\Gamma,\PGL(V)}\to\cohog{2}{\Gamma,\Qlbar^{\times}}$ (associated with $1\to\Qlbar^{\times}\to\GL(V)\to\PGL(V)\to1$) is the class of $\xi$.

Let $\Rep_{\xi}(\Gamma)$ be the category of finite-dimensional $\xi$-twisted representations of $\Gamma$. It is an $\Qlbar$-linear abelian category. If $\xi'$ is another such cocycle in the same cohomology class as $\xi$, then $\xi'=\xi\cdot d\eta$ for some map $\eta: \Gamma\to \Qlbar^{\times}$ with $\eta_{1}=1$, and $\eta$ induces an equivalence $\Rep_{\xi}(\Gamma)\isom\Rep_{\xi'}(\Gamma)$ sending $(V,\{T_{\gamma}\}_{\gamma\in\Gamma})$ to $(V,\{\eta_{\gamma}T_{\gamma}\}_{\gamma\in\Gamma})$.  Therefore the category $\Rep_{\xi}(\Gamma)$ up to equivalence only depends on the cohomology class $[\xi]\in\cohog{2}{\Gamma,\Qlbar^{\times}}$. 

We consider the situation where $X$ is a homogeneous $L$-scheme.
\begin{lemma}\label{l:Ac} Let $k=\kbar$. Let $X$ be a scheme over $k$ with a transitive action of an algebraic group $L$, and let $\cK\in\cCS(L)$. Fix a point $x\in X(k)$ and let $L_{x}$ be its stabilizer with neutral component $L^{\circ}_{x}$. 
\begin{enumerate}
\item If $\calK|_{L_{x}^{\circ}}$ is not isomorphic to the constant sheaf, then $\Loc_{(L,\calK)}(X)$ consists only of the zero object.
\item If $\calK|_{L_{x}^{\circ}}$ is isomorphic to the constant sheaf, then $\calK|_{L_{x}}$ defines a class $\xi\in\cohog{2}{\pi_{0}(L_{x}),\Qlbar^{\times}}$, such that there is an equivalence of categories
\begin{equation*}
\Loc_{(L,\calK)}(X)\cong\Rep_{\xi}(\pi_{0}(L_{x})).
\end{equation*}
\end{enumerate}
\end{lemma}
\begin{proof}
(1) Restricting an object $\calF\in\Loc_{(L,\calK)}(X)$ to $x$ we get an $(L_{x},\cK|_{L_{x}})$-equivariant local system $\calF_{x}$ over the point $x$. The action map of $L_{x}$ on $x$ becomes the structure map $L_{x}\to\Spec k$. The equivariance condition gives an isomorphism between the constant sheaf $\calF_{x}\otimes\Qlbar$ on $L_{x}$ and $\calF_{x}\otimes \cK|_{L_{x}}$. In particular, if $\calF_{x}\neq0$, $\cK|_{L_{x}}$ must be the constant sheaf, and therefore $\cK|_{L^{\c}_{x}}$ is also isomorphic to the constant sheaf. 

(2) Since $\cK|_{L^{\c}_{x}}$ is trivial, $\cK|_{L_{x}}$ descends to a rank one character sheaf $\overline{\cK}$ on $\pi_{0}(L_{x})$, hence giving a class $\xi\in\cohog{2}{\pi_{0}(L_{x}),\Qlbar^{\times}}$ by Lemma \ref{l:char sh discrete}. Restricting a local system to $x$ gives an equivalence of categories $\Loc_{(L,\cK)}(X)\cong\Loc_{(\pi_{0}(L_{x}), \overline{\cK})}(\pt)$. Therefore it suffices to show that $\Loc_{(\Gamma, \cK)}(\pt)\cong\Rep_{\xi}(\Gamma)$ for any finite group $\Gamma$ (viewed as a discrete group scheme over $k=\kbar$) and any rank one character sheaf $\cK\in\cCS(\Gamma)$ giving the class $\xi\in\cohog{2}{\Gamma,\Qlbar^{\times}}$. Given $\calF\in\Loc_{(\Gamma, \cK)}(\pt)$, its stalk is a vector space $V$ equipped with isomorphisms $\varphi_{\gamma}: \cK_{\gamma}\otimes V\to V$, one for each $\gamma\in\Gamma$, satisfying the associativity condition and $\varphi_{1}=\id$. Choosing a trivialization of $\cK_{\gamma}$, the cocycle $\xi$ is defined by the recipe given in Lemma \ref{l:char sh discrete}(1), and $\varphi_{\gamma}$ is identified with an automorphism  $T_{\gamma}$ of $V$. The associativity of $\{\varphi_{\gamma}\}$ becomes the property \eqref{Twrep} for $\{T_{\gamma}\}$; i.e., the $\{T_{\gamma}\}$ give a $\xi$-twisted representation of $\Gamma$ on $V$. Therefore $\Loc_{(\Gamma, \cK)}(\pt)\cong\Rep_{\xi}(\Gamma)$.
\end{proof}


\end{document}